\documentclass{amsart}
\usepackage{amsfonts,amsmath,tikz,xcolor,mathtools,extarrows, amsthm,graphicx,stackrel,amssymb,leftindex}
\usepackage[normalem]{ulem}
\useunder{\uline}{\ul}{}

\usepackage[margin=1in]{geometry}

\usepackage[colorlinks=true,linkcolor=red!70!black,citecolor=green!70!black,urlcolor=magenta!70!black,backref]{hyperref}
\usepackage{colortbl}
\usepackage{hhline}
\usepackage{dsfont}
\usepackage{standalone}
\usepackage{comment}
\usepackage{wasysym} 
\usepackage{bbm} 
\usepackage{stmaryrd} 
\usepackage{enumitem}
\usepackage{caption}
\usepackage{subcaption}
\usepackage[T1]{fontenc} 
\usepackage{siunitx}
\usepackage{makecell}

\usepackage{tikz-cd} 

\usepackage{xargs}
\colorlet{darkblue}{blue!70!black}
\colorlet{darkred}{red!70!black}
\colorlet{darkgreen}{green!70!black}
\colorlet{darkwhite}{white!65!black}
\colorlet{darkorange}{orange!90!black}

\colorlet{darkmagenta}{magenta!70!black}
\colorlet{pink}{green!20!magenta}

\colorlet{lightcyan}{cyan!15!white}
\colorlet{lightyellow}{yellow!30!white}
\colorlet{lightcyanb}{lightcyan!50!black}

\definecolor{purple}{rgb}{0.63, 0.36, 0.94} 

\usetikzlibrary{decorations.markings,decorations.pathreplacing, decorations.pathmorphing, calligraphy,positioning,calc,arrows.meta,cd}

\newenvironment{tric}
    {\begin{tikzpicture}[scale= 0.5, semithick,draw=darkblue,double distance=0.25mm,
        baseline={([yshift=-.8ex]current bounding box.center)}] }
    {\end{tikzpicture}}

\newenvironment{tricF}
    {\begin{tikzpicture}[scale=.5, yscale=.5, semithick,double distance=0.25mm,
        baseline={([yshift=-.8ex]current bounding box.center)}] }
    {\end{tikzpicture}}

\newenvironment{tricpo}
    {\begin{tikzpicture}[scale= 0.35,semithick,draw=darkblue,double distance=1.1,
        baseline={([yshift=-.8ex]current bounding box.center)}] }
    {\end{tikzpicture}}

\newenvironment{tricpob}
    {\begin{tikzpicture}[scale= 0.3,semithick,draw=darkblue,double distance=1.1,
        baseline={([yshift=-.8ex]current bounding box.center)}] }
    {\end{tikzpicture}}

\newenvironment{tricpobs}
    {\begin{tikzpicture}[scale= 0.25,semithick,draw=darkblue,double distance=1.1,
        baseline={([yshift=-.8ex]current bounding box.center)}] }
    {\end{tikzpicture}}

\newenvironment{tricindex}
    {\begin{tikzpicture}[scale= 0.1,semithick,draw=darkblue,double distance=1.1,
        baseline={([yshift=-.8ex]current bounding box.center)}] }
    {\end{tikzpicture}}

\newtheorem{thm}{Theorem}[section]
\newtheorem{lemma}[thm]{Lemma}
\newtheorem{proposition}[thm]{Proposition}
\newtheorem{corollary}[thm]{Corollary}

\newtheorem{conjecture}[thm]{Conjecture}

\newtheorem{notation}[thm]{Notation}
\newtheorem{example}[thm]{Example}

\theoremstyle{definition}
\newtheorem{definition}[thm]{Definition}
\newtheorem{defn}[thm]{Definition}

\theoremstyle{remark}
\newtheorem{remark}[thm]{Remark}

\newcommand{\directedtripleline}[1]{
    \draw[darkblue,line width=2.2pt] #1;
    \draw[double, white] #1;
    \draw[darkblue, postaction={decorate}] #1;
}
\newcommand{\directeddoubleline}[1]{
    \draw[darkblue,double, postaction={decorate}] #1;
}
\newcommand{\directedsingleline}[1]{
    \draw[darkblue, postaction={decorate}] #1;
}

\newcommand{\vara }{X_1}
\newcommand{\varb }{X_2}
\newcommand{\varc }{X_3}
\newcommand{\vard }{X_4}
\newcommand{\vare }{X_5}
\newcommand{\vari }{X_i}
\newcommand{\varj }{X_j}
\newcommand{\vark }{X_k}

\newcommand{\varu}[1]{X_{#1}}

\newcommand{\emptylist}{\varepsilon_\emptyset}
\newcommand{\emptyweb}{\Gamma_\emptyset}

\newcommand{\colora}{\mathit{1}}
\newcommand{\colorb}{\mathit{2}}
\newcommand{\colorc}{\mathit{3}}

\newcommand{\colorn}{\mathit{N}}

\newcommand{\Seq}{\mathsf{Seq}}

\newcommand{\Symf}{\mathrm{Sym}_N} 

\newcommand{\Web}{\mathbf{W}_N}

\newcommand{\Foam}{\mathbf{F}_N}
\newcommand{\FoamFour}{\mathbf{F}_4}
\newcommand{\FoamFive}{\mathbf{F}_5}

\newcommandx{\basicF}[2][1 = 0, 2 = 0]{F^{(#1\text{})}_{#2}}
\newcommandx{\flapcolor}[6][1 = i, 2 = i, 3 = i, 4 = i, 5 = i, 6 =i]{$#1$-$#2$-$#3$-$#4$-$#5$-$#6$}
\newcommandx{\kempe}[2][1 = i, 2 = j]{$#1$-$#2$}

\newcommand{\GL}{\mathfrak{gl}} 
\newcommand{\SL}{\mathfrak{sl}}
\newcommand{\fraksl}{\mathfrak{sl}} 
\newcommand{\foamdeg}{\deg_N} 
\newcommand{\boundarydeg}{\deg_N} 

\newcommand{\DefectThick}{1.7}

\newcommand{\Z}{\mathbb{Z}}

\newcommand{\R}{\mathbb{R}}
\newcommand{\Q}{\mathbb{Q}}

\newcommand{\Hom}{\mathrm{Hom}}
\newcommand{\Fund}{\mathbf{Fund}}

\newcommand{\Wzero}{\mathcal{W}_0}
\newcommand{\Wzeroa}{\mathcal{W}'_0}
\newcommand{\Wone}{\mathcal{W}_1}
\newcommand{\Wtwo}{\mathcal{W}_2}
\newcommand{\Wthree}{\mathcal{W}_3}
\newcommand{\Wfour}{\mathcal{W}_4}
\newcommand{\Wfive}{\mathcal{W}_5}
\newcommand{\Wsix}{\mathcal{W}_6}

\DeclareMathOperator{\idweb}{\mathsf{id}}
\DeclareMathOperator{\idfoam}{\mathit{id}}
\DeclareMathOperator{\HexaBound}{\varepsilon_6}

\newcommand{\palette}{P}
\newcommand{\Inv}{\mathrm{Inv}} 
\newcommand{\webcat}{\mathbf{Web}} 

\title{$6$-valent vertex in the $\GL_N$ web category and its categorification}

\author[J.~Grlj]{Jernej Grlj}
\address{Department of Mathematics\\
 University of Southern California \\
  Los Angeles, California 90089, USA}
  \email{grlj@usc.edu}
  
\author[M.~Khovanov]{Mikhail Khovanov}
\address{Department of Mathematics \\
Johns Hopkins University \\
Baltimore, Maryland 21218, USA}
  \email{khovanov@jhu.edu}
  
 \author[H.~Wu]{Haihan Wu}
\address{Department of Mathematics \\
Johns Hopkins University \\
Baltimore, Maryland 21218, USA}
  \email{haihan.wu.math@gmail.com}

\author[M.~Zhang]{Melissa Zhang}
\address{Department of Mathematics \\
University of California, Davis\\
Davis, California 95616, USA}
  \email{mlzhang@ucdavis.edu}
  
\date{October 13, 2026}

\begin{document}

\begin{abstract}
    We define a $\frac{2\pi}{3}$-rotationally invariant $6$-valent vertex in the $\GL_N$ web category.   When $N=4,5$, we provide a categorification of the $6$-valent vertex using $\GL_N$ foams and decompose the hexagon web into a direct sum of indecomposables. A similar decomposition is conjectured for $N \geq 6$.
\end{abstract}

\maketitle

\tableofcontents

\section{Introduction}

\subsection{Background on webs}
The discovery of the Jones polynomial \cite{JonesPolynomial} triggered developments in low-dimensional topology and representation theory. Reshetikhin and Turaev \cite{RTinv} discovered that the Jones polynomial can be defined by using the braiding structure in a ribbon category, and generalized the Jones polynomial to quantum invariants associated to each simple Lie algebra $\mathfrak{g}$. The ribbon category related to the Jones polynomial can be presented as the Temperley--Lieb category \cite{TemLie}, which is monoidally equivalent to the fundamental representation category of the quantum group $U_q(\mathfrak{sl}_2)$ generated by the $q$-analogue of the vector representation. 

To a rank-2 simple Lie algebra $\mathfrak{g}$, Kuperberg assigned a braided monoidal category, generated by suitable trivalent vertices, which computes link invariants for the rank-2 quantum group $U_q(\mathfrak{g})$ \cite{Kupe-first-G2}.
He later proved that the web category of $\mathfrak{g}$ is equivalent to the fundamental representation category of $U_q(\mathfrak{g})$ \cite{Kuperberg1996}. 
Beyond rank 2, the equivalence between the web category and the fundamental representation category of the corresponding quantum group was proven by Cautis, Kamnitzer, and Morrison for type A in~\cite{Cautis2014}; type C by Bodish, Elias, Rose, and Tatham  in~\cite{bodish2021type}; the quantum orthogonal group by Bodish and the third author in \cite{BW23}; and more recently type B by Bodish, Elias, and Rose in \cite{typeBwebs}.

A morphism in the web category is given by a linear combination of trivalent graphs modulo the defining skein relations. The equivalence between the web category $\webcat_q(\mathfrak{g})$ and its corresponding fundamental representation category $\Fund(U_q(\mathfrak{g}))$ allows us to give a diagrammatic description of a basis for the invariant space. For $\mathfrak{g} = \fraksl_2$, crossingless matching diagrams provide a basis for the invariant space. 
For $\fraksl_3$  (and, more generally, for any rank 2 simple Lie algebra), a basis for the invariant space is provided by non-elliptic webs in \cite{Kuperberg1996}. A set of defining skein relations was conjectured for $\fraksl_4$ in~\cite{DKim} and for $\fraksl_N$ in~\cite{morrison2007diagrammaticcategoryrepresentationtheory}, and later proven in~\cite{Cautis2014}. $\fraksl_N$ web bases were constructed, for example in~\cite[Sections 2.6-2.8]{elias2015lightladdersclaspconjectures}. 
However, a rotation-invariant extension of Kuperberg’s $\fraksl_3$ web basis to higher ranks had been missing until recent progress in the $\fraksl_4$ case by Gaetz, Pechenik, Pfannerer, Striker, and Swanson \cite{SL4Basis}.

\subsection{Background on link homology}
The second author's categorification of the Jones polynomial~\cite{Khovanov2000} started the study of link homology theories. Other early results include a categorification of the $\fraksl_3$ knot invariant, see~\cite{Khovanov2004,Morrison2006OnKC,robert2013thesis, Mackaay2007}, and other related papers; for more details, we refer to the survey \cite{khovanov2025lecturessl3foamslink}. 
A model setup for link homology theories comes from a categorification of the Reshetikhin--Turaev link invariants for the quantum group $U_q(\fraksl_N)$ and coloring of link components by minuscule representations $\Lambda^k_q V$. Here $V$ is the $N$-dimensional fundamental representation and $\Lambda^k_q V$ the $q$-deformed $k$-th exterior power of $V$, for $k=1,\dots, N-1$. 

For these representations, Murakami--Ohtsuki--Yamada (MOY) defined a diagrammatical calculus of intertwiners~\cite{1998HOMFLYPV} which takes positive integral values when evaluated on closed planar diagrams (closed $\SL_N$ webs). 
Link invariants are then obtained by expanding each crossing into an integral linear combination of webs with boundaries.

Originally, a categorification of the MOY graph invariant and associated link homology was achieved by Rozansky and the second author via matrix factorizations 
for the fundamental representation $V$~\cite{KhRo2008}, and extended  by H.~Wu~\cite{MatrixFactorWu} and Y.~Yonezawa~\cite{MatrixFactorYonezawa, Yonezawa2011} to arbitrary exterior powers  $\Lambda^k_q V$. Nowadays, a combinatorial construction of these homology theories uses the Robert--Wagner evaluation \cite{RW-eval-foams} and a TQFT for $\GL_N$
foams as an intermediate step in categorifying the MOY invariants of planar webs. See \cite{khovanov2025lecturessl3foamslink} for more references and details.

When categorified, the MOY integral- and positive-valued  invariants of closed webs (taking values in $\mathbb{Z}_+[q,q^{-1}]$) become state spaces of webs, 
which are free graded modules over the  ground ring of symmetric functions $\Symf$ in $N$ variables. 
The graded rank is the corresponding quantum invariant. These state spaces are then arranged into complexes~\cite{KhRo} whose homology groups categorify the corresponding Reshetikhin--Turaev~\cite{RTinv} and MOY link invariants~\cite{1998HOMFLYPV}. 

It can be shown that the dual canonical basis corresponds to the indecomposable projective objects in a suitable categorification of invariant spaces.
For $\SL_2$ webs, the standard basis of crossingless matchings is the dual canonical basis. Upon categorification, crossingless matching diagrams represent indecomposable projective objects in the self-dual part of the maximal parabolic blocks of $\mathcal{O}$ categorification of the invariant spaces and in the combinatorial categorification of the invariant spaces via arc algebras introduced by the second author~\cite{arcalgebra}.

\def\Benzeneweb
{\begin{tricpob}
\draw [->](-1,1.7)--(0.1,1.7); 
\draw (0,1.7)--(1,1.7);
\draw [->](2,0)--(1.4,-1.02); 
\draw (1.5,-0.85)--(1,-1.7);
\draw [->](-1,-1.7)--(-1.6,-0.68); 
\draw (-1.5,-0.85)--(-2,0); 
\draw[double,decoration={markings,mark=at position 0.65 with {\arrow{>}}},postaction={decorate}]  (-2,0)--(-1,1.7);
\draw[double,decoration={markings,mark=at position 0.75 with {\arrow{>}}},postaction={decorate}] (1,-1.7)--(-1,-1.7);
\draw[double, decoration={markings,mark=at position 0.65 with {\arrow{>}}},postaction={decorate}] (1,1.7)--(2,0);

\draw [->](2,0)--(3.1,0); 
\draw (3,0)--(4,0);
\draw [->](-4,0)--(-2.9,0); 
\draw (-3,0)--(-2,0);
\draw [->](-1,1.7)--(-1.6,2.72);
\draw (-1.5,2.55)--(-2,3.4);
\draw [->](2,-3.4)--(1.4,-2.38);
\draw (1.5,-2.55)--(1,-1.7);
\draw [->](-1,-1.7)--(-1.6,-2.72);
\draw (-1.5,-2.55)--(-2,-3.4);
\draw [->](2,3.4)--(1.4,2.38);
\draw (1.5,2.55)--(1,1.7);
\end{tricpob}
}

\def\BenzenewebM
{\begin{tric}
\draw [decoration={markings,mark=at position 0.55 with {\arrow{>}}},postaction={decorate}] (180:2)..controls(175:1)and(125:1)..(120:2);
\draw [decoration={markings,mark=at position 0.55 with {\arrow{>}}},postaction={decorate}] (60:2)..controls(55:1)and(5:1)..(0:2);
\draw [decoration={markings,mark=at position 0.55 with {\arrow{>}}},postaction={decorate}] (-60:2)..controls(-65:1)and(-115:1)..(-120:2);
\end{tric}
}

\def\Benzeneweba
{\begin{tricpob}
\draw [->](-1,-1.7)--(0.1,-1.7); 
\draw (0,-1.7)--(1,-1.7);
\draw [->](2,0)--(1.4,1.02); 
\draw (1.5,0.85)--(1,1.7);
\draw [->](-1,1.7)--(-1.6,0.68); 
\draw (-1.5,0.85)--(-2,0); 
\draw [double,decoration={markings,mark=at position 0.65 with {\arrow{>}}},postaction={decorate}] (-2,0)--(-1,-1.7);
\draw [double,decoration={markings,mark=at position 0.75 with {\arrow{>}}},postaction={decorate}]  (1,1.7)--(-1,1.7);
\draw [double,decoration={markings,mark=at position 0.65 with {\arrow{>}}},postaction={decorate}] (1,-1.7)--(2,0);

\draw [->](2,0)--(3.1,0); 
\draw (3,0)--(4,0);
\draw [->](-4,0)--(-2.9,0); 
\draw (-3,0)--(-2,0);
\draw [->](-1,1.7)--(-1.6,2.72);
\draw (-1.5,2.55)--(-2,3.4);
\draw [->](2,-3.4)--(1.4,-2.38);
\draw (1.5,-2.55)--(1,-1.7);
\draw [->](-1,-1.7)--(-1.6,-2.72);
\draw (-1.5,-2.55)--(-2,-3.4);
\draw [->](2,3.4)--(1.4,2.38);
\draw (1.5,2.55)--(1,1.7);
\end{tricpob}
}

\def\BenzenewebaM
{\begin{tric}
\draw [decoration={markings,mark=at position 0.55 with {\arrow{>}}},postaction={decorate}] (-180:2)..controls(-175:1)and(-125:1)..(-120:2);
\draw [decoration={markings,mark=at position 0.55 with {\arrow{>}}},postaction={decorate}] (-60:2)..controls(-55:1)and(-5:1)..(0:2);
\draw [decoration={markings,mark=at position 0.55 with {\arrow{>}}},postaction={decorate}] (60:2)..controls(65:1)and(115:1)..(120:2);
\end{tric}
}

\def\BenzenewebMC
{\begin{tric}
\draw [decoration={markings,mark=at position 0.55 with {\arrow{>}}},postaction={decorate}] (180:2)..controls(175:1)and(125:1)..(120:2);
\draw [decoration={markings,mark=at position 0.55 with {\arrow{>}}},postaction={decorate}] (-60:2)..controls(-55:1)and(-5:1)..(0:2);
\draw [decoration={markings,mark=at position 0.55 with {\arrow{>}}},postaction={decorate}] (60:2)--(-120:2);
\end{tric}
}

\def\BenzenewebMB
{\begin{tric}
\draw [decoration={markings,mark=at position 0.55 with {\arrow{>}}},postaction={decorate}] (180:2)..controls(185:1)and(-125:1)..(-120:2);
\draw [decoration={markings,mark=at position 0.55 with {\arrow{>}}},postaction={decorate}] (60:2)..controls(55:1)and(5:1)..(0:2);
\draw [decoration={markings,mark=at position 0.55 with {\arrow{>}}},postaction={decorate}] (-60:2)--(120:2);
\end{tric}
}

\def\BenzenewebMA
{\begin{tric}
\draw [decoration={markings,mark=at position 0.55 with {\arrow{>}}},postaction={decorate}] (60:2)..controls(65:1)and(115:1)..(120:2);
\draw [decoration={markings,mark=at position 0.55 with {\arrow{>}}},postaction={decorate}] (180:2)--(0:2);
\draw [decoration={markings,mark=at position 0.55 with {\arrow{>}}},postaction={decorate}] (-60:2)..controls(-65:1)and(-115:1)..(-120:2);
\end{tric}
}

\def\SixVertex
{\begin{tric}
\draw [decoration={markings,mark=at position 0.6 with {\arrow{>}}},postaction={decorate}] (0,0)--(2,0);
\draw [decoration={markings,mark=at position 0.6 with {\arrow{>}}},postaction={decorate}](-2,0)--(0,0);
\draw [decoration={markings,mark=at position 0.6 with {\arrow{>}}},postaction={decorate}] (1,1.7)--(0,0);
\draw [decoration={markings,mark=at position 0.6 with {\arrow{>}}},postaction={decorate}] (0,0)--(-1,-1.7);
\draw [decoration={markings,mark=at position 0.6 with {\arrow{>}}},postaction={decorate}] (1,-1.7)--(0,0);
\draw [decoration={markings,mark=at position 0.6 with {\arrow{>}}},postaction={decorate}] (0,0)--(-1,1.7);
\end{tric}
}

\subsection{Summary of the results}
In this paper, we use Robert--Wagner's $\GL_N$ foams \cite{RW-eval-foams} for $N \ge 4$ to study the simplest example of a $\GL_N$ basis web that is not indecomposable. In the fundamental representation category of $U_q(\mathfrak{gl}_N)$, $\dim(\Inv((V \otimes V^*)^{\otimes 3}))=6$ where $V$ is the defining representation. However, there are $7$ fully reduced $\GL_N$ webs in the corresponding invariant space, 
shown in Figure \ref{fig:web-names}.
\begin{figure}
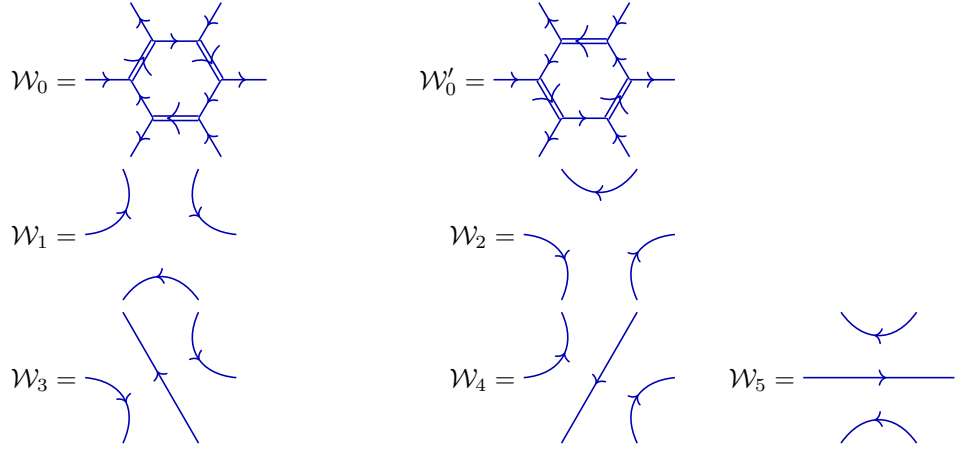

    \centering
    \begin{align*}
        & \Wzero=\Benzeneweb
            & \Wzeroa= \Benzeneweba 
            & \\
        & \Wone= \BenzenewebM  
            & \Wtwo= \BenzenewebaM
            & \\
        & \Wthree=\BenzenewebMB
            & \Wfour=\BenzenewebMC
            & \qquad \Wfive=\BenzenewebMA
    \end{align*}
    \caption{Our notation for the 7 fully reduced $\mathfrak{gl}_N$ webs in $\Inv((V \otimes V^*)^{\otimes 3})$.}
    \label{fig:web-names}
\end{figure}
Therefore, there is a linear relation, shown in Equation~\eqref{GLNBenzene}, among a subset of these $7$ webs, due to the equivalence between the web category and the corresponding representation category.   
\begin{equation}\label{GLNBenzene} 
      \Benzeneweb \  -  [N-3] \ \BenzenewebM
        \  =  \ \Benzeneweba  \ -  [N-3] \  \BenzenewebaM. 
\end{equation}   
This relation poses difficulties in constructing a rotationally invariant $\mathfrak{sl}_4$ web basis \cite{HagemeyerSpiders,SL4Basis}, which indicates that the webs $\Wzero$ and $\Wzeroa$ in Figure \ref{fig:web-names} are not indecomposable objects in the foam category.

In this paper, we introduce a $\frac{2\pi}{3}$-rotationally invariant $6$-valent vertex in the $\GL_N$ web category motivated by  Equation \eqref{GLNBenzene}: 
\begin{equation} \label{GLN6Vertex}
    \Wsix = \ \SixVertex \ :=  \ \Benzeneweb \  - [N-3]\ \BenzenewebM .
\end{equation}

For $N=4,5$, we provide a direct sum decomposition of the hexagon web using $\GL_N$ foams and categorify the $6$-valent vertex $\Wsix$ (Theorems \ref{GL4Decompo} and \ref{GL5Decompo}). 
We also prove that the object corresponding to the $6$-valent vertex in the $\GL_N$ foam category is indecomposable in Theorem \ref{GLNIndecompo} and Remark \ref{GL5Indecompo}, and give a diagrammatic description of a unique basis for the invariant space $\Inv((V \otimes V^*)^{\otimes 3})$, which consists of indecomposable objects in the foam category. 
We expect our categorification of $\Wsix$ to extend to the $\GL_N$ foam category for $N\ge 6$, see Conjecture \ref{gl_n sixvalent seam}. This would lead to a decomposition of the hexagon web when $N \ge 6$ using constructions similar to those presented in this paper.

\subsection{Acknowledgments} J.G. was partially supported by NSF grants DMS-2200419 and DMS-2601214, Simons Foundation Collaboration Grant on New Structures in Low-Dimensional Topology MPS-LDT-00920756 and SFI-MPS-LDT-00014760-01, and Simons Dissertation Fellowship in Mathematics SFI-MPS-SDF-00026689. M.K. was partially supported by NSF grant DMS-2204033
and by the Simons Foundation grant on New Structures in Low-Dimensional Topology. H.W. was supported by NSF grant CCF-2317280 and Simons Foundation grant on New Structures in Low-Dimensional Topology. We would like to thank Elijah Bodish, Ben Elias, Matt Hogancamp, Joel Kamnitzer, Greg Kuperberg, Aaron Lauda, Ziyi Lei, You Qi, Hoel Queffelec, Louis-Hadrien Robert, David Rose, Joshua Swanson, Emmanuel Wagner, and Paul Wedrich for helpful discussions. Some of these discussions took place at the Diagrammatic Categorification workshop at ICERM, which was supported by the NSF Grant No. DMS-2424556. 

\section{\texorpdfstring{$\GL_N$}{GL_N} webs and foams }

We first review the $\GL_N$ web category, defining its generating morphisms, admissible boundary sequences, and fundamental skein relations. Following this, we transition to the topological framework of closed $\GL_N$ foams. We recall the Robert--Wagner evaluation formula and formalize the construction of state spaces for webs with boundaries. Lastly, we recall some facts about Karoubi completions of the relevant categories, which are used to state our results.

\def\SingleEdge
{\begin{tric}
\draw [decoration={markings,mark=at position 0.7 with {\arrow{>}}},postaction={decorate}] (0,0)--(0,1) ;
\end{tric}
}

\def\KSingleEdge
{\begin{tric}
\draw [decoration={markings,mark=at position 0.7 with {\arrow{>}}},postaction={decorate}] (0,0)--(0,1.5) ;
\draw (0,0.75) node[right,black,scale=0.7]{$k$};
\end{tric}
}

\def\DoubleEdge
{\begin{tric}
\draw [double,decoration={markings,mark=at position 0.7 with {\arrow{>}}},postaction={decorate}] (0,0)--(0,1) ;
\end{tric}
}

\def\IHweb
{\begin{tricpo}
\draw[double,thin]
      (0,1)--(0,-1);
\draw[->]  (2,2)--(0.8,1.4) ;
\draw (1,1.5)--(0,1);
\draw[->]  (-2,2)--(-0.8,1.4) ;
\draw (-1,1.5)--(0,1);
\draw[->]  (2,-2)--(0.8,-1.4) ;
\draw (1,-1.5)--(0,-1);
\draw[->]  (-2,-2)--(-0.8,-1.4) ;
\draw (-1,-1.5)--(0,-1);
\end{tricpo}
}

\def\IHweba
{\begin{tricpo}
\draw[double,thin]
      (1,0)--(-1,0);
\draw[->]  (2,2)--(1.4,0.8) ;
\draw (1.5,1)--(1,0);
\draw[->]  (2,-2)--(1.4,-0.8) ;
\draw (1.5,-1)--(1,0);
\draw[->]  (-2,2)--(-1.4,0.8) ;
\draw (-1.5,1)--(-1,0);
\draw[->]  (-2,-2)--(-1.4,-0.8) ;
\draw (-1.5,-1)--(-1,0);
\end{tricpo}
}

\def\ASMweb
{\begin{tricpo}
\draw [->](1,1)--(-0.1,1); 
\draw (0,1)--(-1,1);
\draw [->](1,1)--(1,-0.1); 
\draw (1,0)--(1,-1);
\draw [->](-1,-1)--(-1,0.1); 
\draw (-1,0)--(-1,1); 
\draw [->](-1,-1)--(0.1,-1); 
\draw (0,-1)--(1,-1); 
\draw[double,thin]
      (1,1)--(2,2) (-1,1)--(-2,2) (1,-1)--(2,-2) (-1,-1)--(-2,-2);
\end{tricpo}
}

\def\ASMweba
{\begin{tricpo}
\draw [->](-1,1)--(0.1,1); 
\draw (0,1)--(1,1);
\draw [->](-1,1)--(-1,-0.1); 
\draw (-1,0)--(-1,-1);
\draw [->](1,-1)--(1,0.1); 
\draw (1,0)--(1,1); 
\draw [->](1,-1)--(-0.1,-1); 
\draw (0,-1)--(-1,-1); 
\draw[double,thin]
      (1,1)--(2,2) (-1,1)--(-2,2) (1,-1)--(2,-2) (-1,-1)--(-2,-2);
\end{tricpo}
}

\def\BenzenewebCupA
{\begin{tricpob}
\draw [->](-1,1.7)--(0.1,1.7); 
\draw (0,1.7)--(1,1.7);
\draw [->](2,0)--(1.4,-1.02); 
\draw (1.5,-0.85)--(1,-1.7);
\draw [->](-1,-1.7)--(-1.6,-0.68); 
\draw (-1.5,-0.85)--(-2,0); 
\draw[double,decoration={markings,mark=at position 0.65 with {\arrow{>}}},postaction={decorate}]  (-2,0)--(-1,1.7);
\draw[double,decoration={markings,mark=at position 0.75 with {\arrow{>}}},postaction={decorate}] (1,-1.7)--(-1,-1.7);
\draw[double, decoration={markings,mark=at position 0.65 with {\arrow{>}}},postaction={decorate}] (1,1.7)--(2,0);

\draw [->](2,0)--(3.1,0); 
\draw (3,0)--(4,0);
\draw [->](-4,0)--(-2.9,0); 
\draw (-3,0)--(-2,0);
\draw [->](2,-3.4)--(1.4,-2.38);
\draw (1.5,-2.55)--(1,-1.7);
\draw [->](-1,-1.7)--(-1.6,-2.72);
\draw (-1.5,-2.55)--(-2,-3.4);

\draw [decoration={markings,mark=at position 0.55 with {\arrow{>}}},postaction={decorate}]  (-1,1.7)..controls(-1,3)and(1,3)..(1,1.7);
\end{tricpob}
}

\def\BenzenewebCupB
{\begin{tricpob}
\draw [->](-1,1.7)--(0.1,1.7); 
\draw (0,1.7)--(1,1.7);
\draw [->](2,0)--(1.4,-1.02); 
\draw (1.5,-0.85)--(1,-1.7);
\draw [->](-1,-1.7)--(-1.6,-0.68); 
\draw (-1.5,-0.85)--(-2,0); 
\draw[double,decoration={markings,mark=at position 0.65 with {\arrow{>}}},postaction={decorate}]  (-2,0)--(-1,1.7);
\draw[double,decoration={markings,mark=at position 0.75 with {\arrow{>}}},postaction={decorate}] (1,-1.7)--(-1,-1.7);
\draw[double, decoration={markings,mark=at position 0.65 with {\arrow{>}}},postaction={decorate}] (1,1.7)--(2,0);

\draw [->](-4,0)--(-2.9,0); 
\draw (-3,0)--(-2,0);
\draw [->](-1,1.7)--(-1.6,2.72);
\draw (-1.5,2.55)--(-2,3.4);
\draw [->](2,-3.4)--(1.4,-2.38);
\draw (1.5,-2.55)--(1,-1.7);
\draw [->](-1,-1.7)--(-1.6,-2.72);
\draw (-1.5,-2.55)--(-2,-3.4);

\draw[decoration={markings,mark=at position 0.7 with {\arrow{>}}},postaction={decorate}] (2,0)..controls(3,2)..(1,1.7);
\end{tricpob}
}

\def\BenzenewebClasp
{\begin{tricpob}
\draw [->](-1,1.7)--(0.1,1.7); 
\draw (0,1.7)--(1,1.7);
\draw [->](2,0)--(1.4,-1.02); 
\draw (1.5,-0.85)--(1,-1.7);
\draw [->](-1,-1.7)--(-1.6,-0.68); 
\draw (-1.5,-0.85)--(-2,0); 
\draw[double,decoration={markings,mark=at position 0.65 with {\arrow{>}}},postaction={decorate}]  (-2,0)--(-1,1.7);
\draw[double,decoration={markings,mark=at position 0.75 with {\arrow{>}}},postaction={decorate}] (1,-1.7)--(-1,-1.7);
\draw[double, decoration={markings,mark=at position 0.65 with {\arrow{>}}},postaction={decorate}] (1,1.7)--(2,0);

\draw (-3,0)--(-2,0);
\draw [->](-1,1.7)--(-1.6,2.72);
\draw (-1.5,2.55)--(-2,3.4);
\draw [->](-1,-1.7)--(-1.6,-2.72);
\draw (-1.5,-2.55)--(-2,-3.4);

\draw [decoration={markings,mark=at position 0.65 with {\arrow{>}}},postaction={decorate}](2,0)--(3.5,0); 
\draw [->](-4,0)--(-2.9,0); 
\draw [decoration={markings,mark=at position 0.6 with {\arrow{>}}},postaction={decorate}](3.5,2.6)..controls(3,2.6)and(1.5,2.5)..(1,1.7);
\draw [decoration={markings,mark=at position 0.6 with {\arrow{>}}},postaction={decorate}](3.5,-2.6)..controls(3,-2.6)and(1.5,-2.5)..(1,-1.7);

\draw[darkred, thick] (3.5,3.4)--(3.5,-3.4)--(4,-3.4)--(4,3.4)--cycle; 
\end{tricpob}
}

\def\Benzenewebsmall
{\begin{tricindex}
\draw [->](-1,1.7)--(0.1,1.7); 
\draw (0,1.7)--(1,1.7);
\draw [->](2,0)--(1.4,-1.02); 
\draw (1.5,-0.85)--(1,-1.7);
\draw [->](-1,-1.7)--(-1.6,-0.68); 
\draw (-1.5,-0.85)--(-2,0); 
\draw[double,decoration={markings,mark=at position 0.65 with {\arrow{>}}},postaction={decorate}]  (-2,0)--(-1,1.7);
\draw[double,decoration={markings,mark=at position 0.75 with {\arrow{>}}},postaction={decorate}] (1,-1.7)--(-1,-1.7);
\draw[double, decoration={markings,mark=at position 0.65 with {\arrow{>}}},postaction={decorate}] (1,1.7)--(2,0);

\draw [->](2,0)--(3.1,0); 
\draw (3,0)--(4,0);
\draw [->](-4,0)--(-2.9,0); 
\draw (-3,0)--(-2,0);
\draw [->](-1,1.7)--(-1.6,2.72);
\draw (-1.5,2.55)--(-2,3.4);
\draw [->](2,-3.4)--(1.4,-2.38);
\draw (1.5,-2.55)--(1,-1.7);
\draw [->](-1,-1.7)--(-1.6,-2.72);
\draw (-1.5,-2.55)--(-2,-3.4);
\draw [->](2,3.4)--(1.4,2.38);
\draw (1.5,2.55)--(1,1.7);
\end{tricindex}
}

\def\BenzenewebM
{\begin{tric}
\draw [decoration={markings,mark=at position 0.55 with {\arrow{>}}},postaction={decorate}] (180:2)..controls(175:1)and(125:1)..(120:2);
\draw [decoration={markings,mark=at position 0.55 with {\arrow{>}}},postaction={decorate}] (60:2)..controls(55:1)and(5:1)..(0:2);
\draw [decoration={markings,mark=at position 0.55 with {\arrow{>}}},postaction={decorate}] (-60:2)..controls(-65:1)and(-115:1)..(-120:2);
\end{tric}
}

\def\BenzenewebMCupA
{\begin{tric}
\draw [decoration={markings,mark=at position 0.53 with {\arrow{>}}},postaction={decorate}] (180:2)..controls(-1,0)and(-1,1)..(0,1) ..controls(1,1)and(1,0)..(0:2);
\draw [decoration={markings,mark=at position 0.55 with {\arrow{>}}},postaction={decorate}] (-60:2)..controls(-65:1)and(-115:1)..(-120:2);
\end{tric}
}

\def\BenzenewebMCupB
{\begin{tric}
\draw [decoration={markings,mark=at position 0.55 with {\arrow{>}}},postaction={decorate}] (180:2)..controls(175:1)and(125:1)..(120:2);
\draw [decoration={markings,mark=at position 0.55 with {\arrow{>}}},postaction={decorate}] (-60:2)..controls(-65:1)and(-115:1)..(-120:2);

\draw [decoration={markings,mark=at position 0.55 with {\arrow{>}}},postaction={decorate}] 
(1.8,1) arc (0:360:0.8);
\end{tric}
}

\def\BenzenewebMClasp
{\begin{tric}
\draw [decoration={markings,mark=at position 0.55 with {\arrow{>}}},postaction={decorate}] (180:2)..controls(175:1)and(125:1)..(120:2);
\draw [decoration={markings,mark=at position 0.55 with {\arrow{>}}},postaction={decorate}] (1.7,1.5)..controls(0.5,1.5)and(0.5,-0.2)..(1.7,-0.2);
\draw [decoration={markings,mark=at position 0.55 with {\arrow{>}}},postaction={decorate}] (1.7,-1)..controls(1,-1)and(-115:1.2)..(-120:2);

\draw[darkred, thick] (2,2.05)--(2,-2.05)--(1.7,-2.05)--(1.7,2.05)--cycle; 
\end{tric}
}

\def\BenzenewebMAClasp
{\begin{tric}
\draw [decoration={markings,mark=at position 0.55 with {\arrow{>}}},postaction={decorate}] (180:2)--(0:1.7);
\draw [decoration={markings,mark=at position 0.55 with {\arrow{>}}},postaction={decorate}] (1.7,-1)..controls(1,-1)and(-115:1.2)..(-120:2);
\draw [decoration={markings,mark=at position 0.55 with {\arrow{>}}},postaction={decorate}] (1.7,1)..controls(1,1)and(115:1.2)..(120:2);

\draw[darkred, thick] (2,2.05)--(2,-2.05)--(1.7,-2.05)--(1.7,2.05)--cycle; 
\end{tric}
}

\def\BenzenewebMBClasp
{\begin{tric}
\draw [decoration={markings,mark=at position 0.55 with {\arrow{>}}},postaction={decorate}] (180:2)..controls(185:1)and(-125:1)..(-120:2);
\draw [decoration={markings,mark=at position 0.55 with {\arrow{>}}},postaction={decorate}] (1.7,-1)..controls(0,-1)and(-0.7,1)..(120:2);
\draw [decoration={markings,mark=at position 0.55 with {\arrow{>}}},postaction={decorate}] (1.7,1.5)..controls(0.5,1.5)and(0.5,-0.2)..(1.7,-0.2);

\draw[darkred, thick] (2,2.05)--(2,-2.05)--(1.7,-2.05)--(1.7,2.05)--cycle; 
\end{tric}
}

\def\BenzenewebMC
{\begin{tric}
\draw [decoration={markings,mark=at position 0.55 with {\arrow{>}}},postaction={decorate}] (180:2)..controls(175:1)and(125:1)..(120:2);
\draw [decoration={markings,mark=at position 0.55 with {\arrow{>}}},postaction={decorate}] (-60:2)..controls(-55:1)and(-5:1)..(0:2);
\draw [decoration={markings,mark=at position 0.55 with {\arrow{>}}},postaction={decorate}] (60:2)--(-120:2);
\end{tric}
}

\def\BenzenewebMCClasp
{\begin{tric}
\draw [decoration={markings,mark=at position 0.55 with {\arrow{>}}},postaction={decorate}] (180:2)..controls(175:1)and(125:1)..(120:2);
\draw [decoration={markings,mark=at position 0.55 with {\arrow{>}}},postaction={decorate}] (1.7,1)..controls(0,1)and(-0.7,-1)..(-120:2);
\draw [decoration={markings,mark=at position 0.55 with {\arrow{>}}},postaction={decorate}] (1.7,-1.5)..controls(0.5,-1.5)and(0.5,0.2)..(1.7,0.2);

\draw[darkred, thick] (2,2.05)--(2,-2.05)--(1.7,-2.05)--(1.7,2.05)--cycle; 
\end{tric}
}

\def\BenzenewebaCupA
{\begin{tricpob}
\draw [->](-1,-1.7)--(0.1,-1.7); 
\draw (0,-1.7)--(1,-1.7);
\draw [->](2,0)--(1.4,1.02); 
\draw (1.5,0.85)--(1,1.7);
\draw [->](-1,1.7)--(-1.6,0.68); 
\draw (-1.5,0.85)--(-2,0); 
\draw [double,decoration={markings,mark=at position 0.65 with {\arrow{>}}},postaction={decorate}] (-2,0)--(-1,-1.7);
\draw [double,decoration={markings,mark=at position 0.75 with {\arrow{>}}},postaction={decorate}]  (1,1.7)--(-1,1.7);
\draw [double,decoration={markings,mark=at position 0.65 with {\arrow{>}}},postaction={decorate}] (1,-1.7)--(2,0);

\draw [->](2,0)--(3.1,0); 
\draw (3,0)--(4,0);
\draw [->](-4,0)--(-2.9,0); 
\draw (-3,0)--(-2,0);
\draw [->](2,-3.4)--(1.4,-2.38);
\draw (1.5,-2.55)--(1,-1.7);
\draw [->](-1,-1.7)--(-1.6,-2.72);
\draw (-1.5,-2.55)--(-2,-3.4);

\draw [decoration={markings,mark=at position 0.55 with {\arrow{>}}},postaction={decorate}]  (-1,1.7)..controls(-0.8,3)and(0.8,3)..(1,1.7);
\end{tricpob}
}

\def\BenzenewebaCupB
{\begin{tricpob}
\draw [->](-1,-1.7)--(0.1,-1.7); 
\draw (0,-1.7)--(1,-1.7);
\draw [->](2,0)--(1.4,1.02); 
\draw (1.5,0.85)--(1,1.7);
\draw [->](-1,1.7)--(-1.6,0.68); 
\draw (-1.5,0.85)--(-2,0); 
\draw [double,decoration={markings,mark=at position 0.65 with {\arrow{>}}},postaction={decorate}] (-2,0)--(-1,-1.7);
\draw [double,decoration={markings,mark=at position 0.75 with {\arrow{>}}},postaction={decorate}]  (1,1.7)--(-1,1.7);
\draw [double,decoration={markings,mark=at position 0.65 with {\arrow{>}}},postaction={decorate}] (1,-1.7)--(2,0);

\draw [->](-4,0)--(-2.9,0); 
\draw (-3,0)--(-2,0);
\draw [->](-1,1.7)--(-1.6,2.72);
\draw (-1.5,2.55)--(-2,3.4);
\draw [->](2,-3.4)--(1.4,-2.38);
\draw (1.5,-2.55)--(1,-1.7);
\draw [->](-1,-1.7)--(-1.6,-2.72);
\draw (-1.5,-2.55)--(-2,-3.4);

\draw[decoration={markings,mark=at position 0.75 with {\arrow{>}}},postaction={decorate}] (2,0)..controls(3,2)..(1,1.7);
\end{tricpob}
}

\def\BenzenewebaClasp
{\begin{tricpob}
\draw [->](-1,-1.7)--(0.1,-1.7); 
\draw (0,-1.7)--(1,-1.7);
\draw [->](2,0)--(1.4,1.02); 
\draw (1.5,0.85)--(1,1.7);
\draw [->](-1,1.7)--(-1.6,0.68); 
\draw (-1.5,0.85)--(-2,0); 
\draw [double,decoration={markings,mark=at position 0.65 with {\arrow{>}}},postaction={decorate}] (-2,0)--(-1,-1.7);
\draw [double,decoration={markings,mark=at position 0.75 with {\arrow{>}}},postaction={decorate}]  (1,1.7)--(-1,1.7);
\draw [double,decoration={markings,mark=at position 0.65 with {\arrow{>}}},postaction={decorate}] (1,-1.7)--(2,0);

\draw [->](-4,0)--(-2.9,0); 
\draw (-3,0)--(-2,0);
\draw [->](-1,1.7)--(-1.6,2.72);
\draw (-1.5,2.55)--(-2,3.4);
\draw [->](-1,-1.7)--(-1.6,-2.72);
\draw (-1.5,-2.55)--(-2,-3.4);

\draw [decoration={markings,mark=at position 0.65 with {\arrow{>}}},postaction={decorate}](2,0)--(3.5,0); 
\draw [->](-4,0)--(-2.9,0); 
\draw [decoration={markings,mark=at position 0.6 with {\arrow{>}}},postaction={decorate}](3.5,2.6)..controls(3,2.6)and(1.5,2.5)..(1,1.7);
\draw [decoration={markings,mark=at position 0.6 with {\arrow{>}}},postaction={decorate}](3.5,-2.6)..controls(3,-2.6)and(1.5,-2.5)..(1,-1.7);

\draw[darkred, thick] (3.5,3.4)--(3.5,-3.4)--(4,-3.4)--(4,3.4)--cycle; 
\end{tricpob}
}

\def\BenzenewebaM
{\begin{tric}
\draw [decoration={markings,mark=at position 0.55 with {\arrow{>}}},postaction={decorate}] (-180:2)..controls(-175:1)and(-125:1)..(-120:2);
\draw [decoration={markings,mark=at position 0.55 with {\arrow{>}}},postaction={decorate}] (-60:2)..controls(-55:1)and(-5:1)..(0:2);
\draw [decoration={markings,mark=at position 0.55 with {\arrow{>}}},postaction={decorate}] (60:2)..controls(65:1)and(115:1)..(120:2);
\end{tric}
}

\def\BenzenewebaMCupA
{\begin{tric}
\draw [decoration={markings,mark=at position 0.55 with {\arrow{>}}},postaction={decorate}] (-180:2)..controls(-175:1)and(-125:1)..(-120:2);
\draw [decoration={markings,mark=at position 0.55 with {\arrow{>}}},postaction={decorate}] (-60:2)..controls(-55:1)and(-5:1)..(0:2);
\draw [decoration={markings,mark=at position 0.55 with {\arrow{>}}},postaction={decorate}] 
(0.8,0.5) arc (360:0:0.8);
\end{tric}
}

\def\BenzenewebaMCupB
{\begin{tric}
\draw [decoration={markings,mark=at position 0.55 with {\arrow{>}}},postaction={decorate}] (-180:2)..controls(-175:1)and(-125:1)..(-120:2);
\draw [decoration={markings,mark=at position 0.55 with {\arrow{>}}},postaction={decorate}] (-60:2)..controls(0:1.5)and(60:1.5)..(120:2);
\end{tric}
}

\def\BenzenewebaMClasp
{\begin{tric}
\draw [decoration={markings,mark=at position 0.55 with {\arrow{>}}},postaction={decorate}] (-180:2)..controls(-175:1)and(-125:1)..(-120:2);
\draw [decoration={markings,mark=at position 0.55 with {\arrow{>}}},postaction={decorate}] (1.7,-1.5)..controls(0.5,-1.5)and(0.5,0.2)..(1.7,0.2);
\draw [decoration={markings,mark=at position 0.55 with {\arrow{>}}},postaction={decorate}] (1.7,1)..controls(1,1)and(115:1.2)..(120:2);

\draw[darkred, thick] (2,2.05)--(2,-2.05)--(1.7,-2.05)--(1.7,2.05)--cycle; 
\end{tric}
}

\def\Skeina
{\begin{tric}
\draw [decoration={markings,mark=at position 0.52 with {\arrow{>}}},postaction={decorate}] (0.85,0) circle (1.2);
\end{tric}
}

\def\SkeinaGL
{\begin{tric}
\draw [decoration={markings,mark=at position 0.52 with {\arrow{>}}},postaction={decorate}] (0.85,0) circle (1.2);
\draw (-0.4,0) node[left, black, scale=0.7] {$k$};
\end{tric}
}

\def\Skeinb
{\begin{tric}
\draw [decoration={markings,mark=at position 0.52 with {\arrow{>}}},postaction={decorate},double] (0,0) circle (1.2);
\end{tric}
}

\def\Skeinc
{\begin{tric}
\draw [double,decoration={markings,mark=at position 0.65 with {\arrow{>}}},postaction={decorate}](0,0.5)--(0,1.5); 
  \draw [double,decoration={markings,mark=at position 0.65 with {\arrow{>}}},postaction={decorate}]  (0,-1.5)--(0,-0.5);
  
\draw[decoration={markings,mark=at position 0.55 with {\arrow{>}}},postaction={decorate}]  (0,-0.5)..controls(-0.5,-0.5)and(-0.5,0.5)..(0,0.5) ;
\draw [decoration={markings,mark=at position 0.55 with {\arrow{>}}},postaction={decorate}](0,-0.5)..controls(0.5,-0.5)and(0.5,0.5)..(0,0.5);
\end{tric}
}

\def\SkeinBiGL
{\begin{tric}
\draw [decoration={markings,mark=at position 0.65 with {\arrow{>}}},postaction={decorate}](0,0.5)--(0,1.5); 
  \draw [decoration={markings,mark=at position 0.65 with {\arrow{>}}},postaction={decorate}]  (0,-1.5)--(0,-0.5);
  
\draw[decoration={markings,mark=at position 0.55 with {\arrow{>}}},postaction={decorate}]  (0,-0.5)..controls(-0.5,-0.5)and(-0.5,0.5)..(0,0.5) ;
\draw [decoration={markings,mark=at position 0.55 with {\arrow{>}}},postaction={decorate}](0,-0.5)..controls(0.5,-0.5)and(0.5,0.5)..(0,0.5);

\draw (-0.1,1)node[left,black,scale=0.7]{$k+l$};
\draw (-0.1,-1)node[left,black,scale=0.7]{$k+l$};
\draw (-0.5,0)node[left,black,scale=0.7]{$k$};
\draw (0.5,0)node[right,black,scale=0.7]{$l$};
\end{tric}
}

\def\SkeinBiGLC
{\begin{tric}
\draw [decoration={markings,mark=at position 0.65 with {\arrow{>}}},postaction={decorate}](0,0.5)--(0,1.5); 
  \draw [decoration={markings,mark=at position 0.65 with {\arrow{>}}},postaction={decorate}]  (0,-1.5)--(0,-0.5);
  
\draw[decoration={markings,mark=at position 0.55 with {\arrow{>}}},postaction={decorate}]  (0,0.5)..controls(-0.5,0.5)and(-0.5,-0.5)..(0,-0.5) ;
\draw [decoration={markings,mark=at position 0.55 with {\arrow{>}}},postaction={decorate}](0,-0.5)..controls(0.5,-0.5)and(0.5,0.5)..(0,0.5);

\draw (0,1)node[left,black,scale=0.7]{$k$};
\draw (0,-1)node[left,black,scale=0.7]{$k$}; 
\draw (-0.5,0)node[left,black,scale=0.7]{$l$};
\draw (0.5,0)node[right,black,scale=0.7]{$k+l$};
\end{tric}
}

\def\Skeinca
{\begin{tric}
\draw[double,decoration={markings,mark=at position 0.55 with {\arrow{>}}},postaction={decorate}] (0,-1.5)--(0,1.5);
\end{tric}
}

\def\Skeindot
{\begin{tric}
\draw [double,decoration={markings,mark=at position 0.65 with {\arrow{>}}},postaction={decorate}](0,0.5)--(0,1.5); 
  \draw [double,decoration={markings,mark=at position 0.65 with {\arrow{>}}},postaction={decorate}]  (0,-1.5)--(0,-0.5);
   \draw [double,decoration={markings,mark=at position 0.65 with {\arrow{>}}},postaction={decorate}]  (0,0.5)--(0,-0.5);
\filldraw[darkred] (0,0.5)circle (3pt);
\filldraw[darkred] (0,-0.5)circle (3pt);
\end{tric}
}

\def\SkeincB
{\begin{tric}
\draw [decoration={markings,mark=at position 0.6 with {\arrow{>}}},postaction={decorate}](0,0.7)--(0,1.5); 
  \draw [decoration={markings,mark=at position 0.6 with {\arrow{>}}},postaction={decorate}]  (0,-1.5)--(0,-0.7);
\draw[decoration={markings,mark=at position 0.55 with {\arrow{>}}},postaction={decorate}]  (0,0.7)..controls(-0.5,0.7)and(-0.5,-0.7)..(0,-0.7) ;
\draw [double,decoration={markings,mark=at position 0.65 with {\arrow{>}}},postaction={decorate}](0,-0.7)..controls(0.5,-0.7)and(0.5,0.7)..(0,0.7);
\end{tric}
}

\def\SkeincaB
{\begin{tric}
\draw[decoration={markings,mark=at position 0.55 with {\arrow{>}}},postaction={decorate}] (0,-1.5)--(0,1.5);
\end{tric}
}

\def\SkeinBiGLB
{\begin{tric}
\draw[decoration={markings,mark=at position 0.55 with {\arrow{>}}},postaction={decorate}] (0,-1.5)--(0,1.5);
\draw (-0.1,0) node[left,black,scale=0.7] {$k+l$};
\end{tric}
}

\def\SkeinBiGLD
{\begin{tric}
\draw[decoration={markings,mark=at position 0.55 with {\arrow{>}}},postaction={decorate}] (0,-1.5)--(0,1.5);
\draw (-0.1,0) node[left,black,scale=0.7] {$k$};
\end{tric}
}

\def\Skeine
{\begin{tric}
\draw [scale=0.7,decoration={markings,mark=at position 0.6 with {\arrow{>}}},postaction={decorate}] (-1,1)--(-1,-1);
\draw  [scale=0.7,decoration={markings,mark=at position 0.6 with {\arrow{>}}},postaction={decorate}] (1,-1)--(1,1);
\draw  [scale=0.7,decoration={markings,mark=at position 0.65 with {\arrow{>}}},postaction={decorate}] (-2,-2)--(-1,-1);
\draw [scale=0.7,decoration={markings,mark=at position 0.65 with {\arrow{>}}},postaction={decorate}] (-1,1)--(-2,2);
\draw  [scale=0.7,decoration={markings,mark=at position 0.65 with {\arrow{>}}},postaction={decorate}] (2,2)--(1,1);
\draw  [scale=0.7,decoration={markings,mark=at position 0.65 with {\arrow{>}}},postaction={decorate}] (1,-1)--(2,-2);

\draw [double,scale=0.7,decoration={markings,mark=at position 0.65 with {\arrow{>}}},postaction={decorate}] (1,1)--(-1,1);
\draw [double,scale=0.7,decoration={markings,mark=at position 0.65 with {\arrow{>}}},postaction={decorate}] (-1,-1)--(1,-1);
\end{tric}
}

\def\Skeinea
{\begin{tric}
\draw[scale=0.5] (-2,-1.732)--(-1,0)--(-2,1.732) (2,1.732)--(1,0)--(2,-1.732) (-1,0)--(1,0);
\end{tric}
}

\def\Skeinec
{\begin{tric}
\draw  [scale=0.6,decoration={markings,mark=at position 0.55 with {\arrow{>}}},postaction={decorate}]
(-2,-2)..controls(-1,-1)and(-1,1)..(-2,2) ;
\draw  [scale=0.6,decoration={markings,mark=at position 0.55 with {\arrow{>}}},postaction={decorate}]
(2,2)..controls(1,1)and(1,-1)..(2,-2);
\end{tric}
}

\def\Skeined
{\begin{tric}
\draw[scale=0.6,decoration={markings,mark=at position 0.55 with {\arrow{>}}},postaction={decorate}] 
(2,2)..controls(1,1)and(-1,1)..(-2,2);
\draw [scale=0.6,decoration={markings,mark=at position 0.55 with {\arrow{>}}},postaction={decorate}]
(-2,-2)..controls(-1,-1)and(1,-1)..(2,-2) ;
\end{tric}
}

\def\SkeinSquare
{\begin{tric}
\draw [scale=0.7,decoration={markings,mark=at position 0.6 with {\arrow{>}}},postaction={decorate}] (-1,-1)--(-1,1);
\draw  [scale=0.7,decoration={markings,mark=at position 0.6 with {\arrow{>}}},postaction={decorate}] (1,1)--(1,-1);
\draw [scale=0.7,decoration={markings,mark=at position 0.65 with {\arrow{>}}},postaction={decorate}] (1,1)--(-1,1);
\draw [scale=0.7,decoration={markings,mark=at position 0.65 with {\arrow{>}}},postaction={decorate}] (-1,-1)--(1,-1);

\draw  [double,scale=0.7,decoration={markings,mark=at position 0.75 with {\arrow{>}}},postaction={decorate}] (-2,-2)--(-1,-1);
\draw [double,scale=0.7,decoration={markings,mark=at position 0.75 with {\arrow{>}}},postaction={decorate}] (-1,1)--(-2,2);
\draw  [double,scale=0.7,decoration={markings,mark=at position 0.75 with {\arrow{>}}},postaction={decorate}] (2,2)--(1,1);
\draw  [double,scale=0.7,decoration={markings,mark=at position 0.75 with {\arrow{>}}},postaction={decorate}] (1,-1)--(2,-2);
\end{tric}
}

\def\SkeinSquareA
{\begin{tric}
\draw [scale=0.7,decoration={markings,mark=at position 0.6 with {\arrow{>}}},postaction={decorate}] (-1,1)--(-1,-1);
\draw  [scale=0.7,decoration={markings,mark=at position 0.6 with {\arrow{>}}},postaction={decorate}] (1,-1)--(1,1);
\draw [scale=0.7,decoration={markings,mark=at position 0.65 with {\arrow{>}}},postaction={decorate}] (1,-1)--(-1,-1);
\draw [scale=0.7,decoration={markings,mark=at position 0.65 with {\arrow{>}}},postaction={decorate}] (-1,1)--(1,1);

\draw  [double,scale=0.7,decoration={markings,mark=at position 0.75 with {\arrow{>}}},postaction={decorate}] (-2,2)--(-1,1);
\draw [double,scale=0.7,decoration={markings,mark=at position 0.75 with {\arrow{>}}},postaction={decorate}] (-1,-1)--(-2,-2);
\draw  [double,scale=0.7,decoration={markings,mark=at position 0.75 with {\arrow{>}}},postaction={decorate}] (2,-2)--(1,-1);
\draw  [double,scale=0.7,decoration={markings,mark=at position 0.75 with {\arrow{>}}},postaction={decorate}] (1,1)--(2,2);

\draw  [double,scale=0.7,decoration={markings,mark=at position 0.75 with {\arrow{>}}},postaction={decorate}] (-2,2)--(-3,3);
\draw [double,scale=0.7,decoration={markings,mark=at position 0.75 with {\arrow{>}}},postaction={decorate}] (-3,-3)--(-2,-2);
\draw  [double,scale=0.7,decoration={markings,mark=at position 0.75 with {\arrow{>}}},postaction={decorate}] (2,-2)--(3,-3);
\draw  [double,scale=0.7,decoration={markings,mark=at position 0.75 with {\arrow{>}}},postaction={decorate}] (3,3)--(2,2);

\filldraw[darkred,scale=0.7] (2,-2)circle (4pt);
\filldraw[darkred,scale=0.7] (2,2)circle (4pt);
\filldraw[darkred,scale=0.7] (-2,-2)circle (4pt);
\filldraw[darkred,scale=0.7] (-2,2)circle (4pt);
\end{tric}
}

\def\SkeinIHout
{\begin{tric}
\draw [double,decoration={markings,mark=at position 0.75 with {\arrow{>}}},postaction={decorate}](0,0.5)--(0,1.5); 
\draw [double,decoration={markings,mark=at position 0.75 with {\arrow{>}}},postaction={decorate}]  (0,0.5)--(0,-0.5);

\draw [decoration={markings,mark=at position 0.65 with {\arrow{>}}},postaction={decorate}]  (0,-0.5)--(1,-1);
\draw [decoration={markings,mark=at position 0.65 with {\arrow{>}}},postaction={decorate}]  (0,-0.5)--(-1,-1);

\draw [decoration={markings,mark=at position 0.65 with {\arrow{>}}},postaction={decorate}]  (0,1.5)--(1,2);
\draw [decoration={markings,mark=at position 0.65 with {\arrow{>}}},postaction={decorate}]  (0,1.5)--(-1,2);
\filldraw[darkred] (0,0.5)circle (3pt);
\end{tric}
}

\def\SkeinIHouta
{\begin{tric}
\draw [double,decoration={markings,mark=at position 0.75 with {\arrow{>}}},postaction={decorate}](0.5,0)--(1.5,0); 
\draw [double,decoration={markings,mark=at position 0.75 with {\arrow{>}}},postaction={decorate}]  (0.5,0)--(-0.5,0);

\draw [decoration={markings,mark=at position 0.65 with {\arrow{>}}},postaction={decorate}]  (-0.5,0)--(-1,1);
\draw [decoration={markings,mark=at position 0.65 with {\arrow{>}}},postaction={decorate}]  (-0.5,0)--(-1,-1);

\draw [decoration={markings,mark=at position 0.65 with {\arrow{>}}},postaction={decorate}]  (1.5,0)--(2,1);
\draw [decoration={markings,mark=at position 0.65 with {\arrow{>}}},postaction={decorate}]  (1.5,0)--(2,-1);
\filldraw[darkred] (0.5,0)circle (3pt);
\end{tric}
}

\def\SkeinIHin
{\begin{tric}
\draw [double,decoration={markings,mark=at position 0.75 with {\arrow{>}}},postaction={decorate}](0,1.5)--(0,0.5); 
\draw [double,decoration={markings,mark=at position 0.75 with {\arrow{>}}},postaction={decorate}] (0,-0.5)--(0,0.5);

\draw [decoration={markings,mark=at position 0.65 with {\arrow{>}}},postaction={decorate}]  (1,-1)--(0,-0.5);
\draw [decoration={markings,mark=at position 0.65 with {\arrow{>}}},postaction={decorate}]  (-1,-1)--(0,-0.5);

\draw [decoration={markings,mark=at position 0.65 with {\arrow{>}}},postaction={decorate}]  (1,2)--(0,1.5);
\draw [decoration={markings,mark=at position 0.65 with {\arrow{>}}},postaction={decorate}]  (-1,2)--(0,1.5);
\filldraw[darkred] (0,0.5)circle (3pt);
\end{tric}
}

\def\SkeinIHina
{\begin{tric}
\draw [double,decoration={markings,mark=at position 0.75 with {\arrow{>}}},postaction={decorate}](1.5,0)--(0.5,0); 
\draw [double,decoration={markings,mark=at position 0.75 with {\arrow{>}}},postaction={decorate}] (-0.5,0)--(0.5,0);

\draw [decoration={markings,mark=at position 0.65 with {\arrow{>}}},postaction={decorate}]  (-1,1)--(-0.5,0);
\draw [decoration={markings,mark=at position 0.65 with {\arrow{>}}},postaction={decorate}]  (-1,-1)--(-0.5,0);

\draw [decoration={markings,mark=at position 0.65 with {\arrow{>}}},postaction={decorate}]  (2,1)--(1.5,0);
\draw [decoration={markings,mark=at position 0.65 with {\arrow{>}}},postaction={decorate}]  (2,-1)--(1.5,0);
\filldraw[darkred] (0.5,0)circle (3pt);
\end{tric}
}

\def\Skeinf
{\begin{tric}
\draw [scale=0.7,decoration={markings,mark=at position 0.6 with {\arrow{>}}},postaction={decorate}] (-1,-1)--(-1,1);
\draw  [scale=0.7,decoration={markings,mark=at position 0.6 with {\arrow{>}}},postaction={decorate}] (1,1)--(1,-1);
\draw [scale=0.7,decoration={markings,mark=at position 0.65 with {\arrow{>}}},postaction={decorate}] (1,1)--(-1,1);

\draw [double,scale=0.7,decoration={markings,mark=at position 0.65 with {\arrow{>}}},postaction={decorate}] (1,-1)--(-1,-1);

\draw  [scale=0.7,decoration={markings,mark=at position 0.75 with {\arrow{>}}},postaction={decorate}] (-1,-1)--(-2,-2);
\draw  [scale=0.7,decoration={markings,mark=at position 0.75 with {\arrow{>}}},postaction={decorate}] (2,-2)--(1,-1);

\draw [double,scale=0.7,decoration={markings,mark=at position 0.75 with {\arrow{>}}},postaction={decorate}] (-1,1)--(-2,2);
\draw  [double,scale=0.7,decoration={markings,mark=at position 0.75 with {\arrow{>}}},postaction={decorate}] (2,2)--(1,1);
\end{tric}
}

\def\Skeinfa
{\begin{tric}
\draw [scale=0.8,decoration={markings,mark=at position 0.55 with {\arrow{>}}},postaction={decorate}](-0.5,0)--(1.5,0); 

\draw [scale=0.8,decoration={markings,mark=at position 0.65 with {\arrow{>}}},postaction={decorate}] (-0.5,0)--(-1.25,-1.25);
\draw [scale=0.8,decoration={markings,mark=at position 0.65 with {\arrow{>}}},postaction={decorate}]  (2.25,-1.25)--(1.5,0);

\draw [scale=0.8,double,decoration={markings,mark=at position 0.7 with {\arrow{>}}},postaction={decorate}]  (-1,1)--(-0.5,0);
\draw [scale=0.8,double,decoration={markings,mark=at position 0.65 with {\arrow{>}}},postaction={decorate}]  (1.5,0)--(2,1);
\draw [scale=0.8,double,decoration={markings,mark=at position 0.7 with {\arrow{>}}},postaction={decorate}]  (-1,1)--(-1.5,2);
\draw [scale=0.8,double,decoration={markings,mark=at position 0.65 with {\arrow{>}}},postaction={decorate}]  (2.5,2)--(2,1);

\filldraw[scale=0.8,darkred] (2,1)circle (3.5pt)  (-1,1)circle (3.5pt);
\end{tric}
}

\def\Skeinfb
{\begin{tric}
\draw[double,scale=0.6,decoration={markings,mark=at position 0.55 with {\arrow{>}}},postaction={decorate}] 
(2,2)..controls(1,1)and(-1,1)..(-2,2);
\draw [scale=0.6,decoration={markings,mark=at position 0.55 with {\arrow{>}}},postaction={decorate}]
(2,-2)..controls(1,-1)and(-1,-1)..(-2,-2);
\end{tric}
}

\subsection{\texorpdfstring{$\GL_N$}{GL_N} webs}

While much of the literature deals with $\SL_N$ webs, we work with $\GL_N$ webs in order to use the foam framework from~\cite{RW-eval-foams} when categorifying the webs. While evaluations for $\SL_2$ and $\SL_3$ foams are straightforward, it is unclear how to modify the Robert--Wagner evaluation formula from $\GL_N$ webs to $\SL_N$ webs when $N\ge 4$. 
Most statements about $\SL_N$ webs can be modified to the $\GL_N$ setting, see \cite{Lauda2019, Queffelec2016, Cautis2014}. We will now introduce the notation and gather some of the main properties of webs.

Recall the quantum integer $[n]:= \dfrac{q^n-q^{-n}}{q-q^{-1}}$. 
We write  $ [n] ! :=  [1] [2] [3] ... [n] $ and 
$ \begin{bmatrix}
n \\
k 
\end{bmatrix} := \cfrac{[n]!}{[k]![n-k]!} $.


\def\TypeAVert
{\begin{tric}
\draw [scale=1.2,decoration={markings,mark=at position 0.7 with {\arrow{>}}},postaction={decorate}] (0,0)--(90:1) ;
\draw [scale=1.2,decoration={markings,mark=at position 0.7 with {\arrow{>}}},postaction={decorate}] (210:1)--(0,0) ;
\draw [scale=1.2,decoration={markings,mark=at position 0.7 with {\arrow{>}}},postaction={decorate}] (330:1)--(0,0);

\draw [scale=1.2](0.05,0.5)node[black,right,scale=0.7]{$k+l$}
      (210:0.5)node[black,anchor=north,scale=0.7]{$k$}
      (330:0.5)node[black,anchor=north,scale=0.7]{$l$}; 
\end{tric}
}

\def\TypeAVertA
{\begin{tric}
\draw [scale=1.2,decoration={markings,mark=at position 0.7 with {\arrow{>}}},postaction={decorate}] (-90:1)--(0,0) ;
\draw [scale=1.2, decoration={markings,mark=at position 0.7 with {\arrow{>}}},postaction={decorate}] (0,0)--(-210:1) ;
\draw [scale=1.2, decoration={markings,mark=at position 0.7 with {\arrow{>}}},postaction={decorate}] (0,0)--(-330:1);
\draw [scale=1.2](0.05,-0.5)node[black,right,scale=0.7]{$k+l$}
      (-215:0.6)node[black,above,scale=0.7]{$k$}
      (-325:0.6)node[black,above,scale=0.7]{$l$}; 
\end{tric}
}

\def\TypeATag
{\begin{tric}
\draw [decoration={markings,mark=at position 0.7 with {\arrow{>}}},postaction={decorate}](-0.5,-1)--(0,0);
\draw [decoration={markings,mark=at position 0.7 with {\arrow{>}}},postaction={decorate}](0.5,1)--(0,0);
\draw (0,0)--(-0.4,0.2);
\draw (-0.2,-0.6)node[black,right,scale=0.7]{$k$}; 
\draw (0.3,0.5)node[black,right,scale=0.7]{$n-k$}; 
\end{tric}
}

\def\TypeATagA
{\begin{tric}
\draw [decoration={markings,mark=at position 0.7 with {\arrow{>}}},postaction={decorate}](0,0)--(-0.5,-1);
\draw [decoration={markings,mark=at position 0.7 with {\arrow{>}}},postaction={decorate}](0,0)--(0.5,1);
\draw (0,0)--(-0.4,0.2);
\draw (-0.2,-0.6)node[black,right,scale=0.7]{$k$}; 
\draw (0.3,0.5)node[black,right,scale=0.7]{$n-k$}; 
\end{tric}
}

\def\TypeAIH
{\begin{tric}
\draw[scale=0.7,decoration={markings,mark=at position 0.7 with {\arrow{>}}},postaction={decorate}](-1.732,-2)--(0,-1);

\draw[scale=0.7,decoration={markings,mark=at position 0.7 with {\arrow{>}}},postaction={decorate}](1.732,-2)--(0,-1);

\draw[scale=0.7,decoration={markings,mark=at position 0.7 with {\arrow{>}}},postaction={decorate}](0,1)--(-1.732,2) ;

\draw[scale=0.7,decoration={markings,mark=at position 0.7 with {\arrow{>}}},postaction={decorate}](0,1)--(1.732,2);

\draw[scale=0.7,decoration={markings,mark=at position 0.7 with {\arrow{>}}},postaction={decorate}](0,-1)-- (0,1);

\draw[scale=0.7](0.5,0) node[right,midway,black,scale=0.7]{$2$};
\end{tric}
}

\def\TypeACircleNeed
{\begin{tric}
\draw[scale=0.7,decoration={markings,mark=at position 0.5 with {\arrow{>}}},postaction={decorate}] (-1,0) circle (1.5);
\end{tric}
}

\def\TypeACircleK
{\begin{tric}
\draw[scale=0.7,decoration={markings,mark=at position 0.5 with {\arrow{>}}},postaction={decorate}] (-1,0) circle (1.5);
\draw (-1.9,0) node[scale=0.7,left,black]{$k$};
\end{tric}
}

\def\TypeACircleKRe
{\begin{tric}
\draw[scale=0.7,decoration={markings,mark=at position 0.5 with {\arrow{<}}},postaction={decorate}] (-1,0) circle (1.5);
\draw (-1.9,0) node[scale=0.7,left,black]{$k$};
\end{tric}
}

\def\TypeABigonNeedA
{\begin{tric}
\draw[double, scale= 0.8,decoration={markings,mark=at position 0.65 with {\arrow{>}}},postaction={decorate}] (0,0.7)--(0,2) ;

\draw[double, scale= 0.8,decoration={markings,mark=at position 0.65 with {\arrow{>}}},postaction={decorate}](0,-2)--(0,-0.7); 
 
\draw[scale= 0.8,decoration={markings,mark=at position 0.6 with {\arrow{>}}},postaction={decorate}]     (0,-0.7)..controls(-0.5,-0.7)and(-0.5,0.7)..(0,0.7) ;

\draw[scale= 0.8,decoration={markings,mark=at position 0.6 with {\arrow{>}}},postaction={decorate}](0,-0.7)..controls(0.5,-0.7)and(0.5,0.7)..(0,0.7);
\end{tric}
}

\def\TypeABigonDefA
{\begin{tric}
\draw[scale= 0.8,decoration={markings,mark=at position 0.65 with {\arrow{>}}},postaction={decorate}] (0,0.7)--(0,2); 
\draw (0.1,1.1) node[black,right,scale=0.7]{$k+l$}; 

\draw[scale= 0.8,decoration={markings,mark=at position 0.65 with {\arrow{>}}},postaction={decorate}](0,-2)--(0,-0.7);
\draw (0.1,-1.2) node[black,right,scale=0.7]{$k+l$} ; 
 
\draw[scale= 0.8,decoration={markings,mark=at position 0.6 with {\arrow{>}}},postaction={decorate}]     (0,-0.7)..controls(-0.5,-0.7)and(-0.5,0.7)..(0,0.7) ;
\draw (-0.45,0) node[black,left,scale=0.7]{$k$} ;

\draw[scale= 0.8,decoration={markings,mark=at position 0.6 with {\arrow{>}}},postaction={decorate}](0,-0.7)..controls(0.5,-0.7)and(0.5,0.7)..(0,0.7) ;
\draw (0.45,0) node[black,right,scale=0.7]{$l$} ;
\end{tric}
}

\def\TypeABigonDefB
{\begin{tric}
\draw[scale= 0.8,decoration={markings,mark=at position 0.65 with {\arrow{>}}},postaction={decorate}] (0,0.7)--(0,2); 
\draw (0.1,1.1) node[black,right,scale=0.7]{$k$}; 

\draw[scale= 0.8,decoration={markings,mark=at position 0.65 with {\arrow{>}}},postaction={decorate}](0,-2)--(0,-0.7);
\draw (0.1,-1.2) node[black,right,scale=0.7]{$k$} ; 
 
\draw[scale= 0.8,decoration={markings,mark=at position 0.6 with {\arrow{>}}},postaction={decorate}]     (0,-0.7)..controls(-0.5,-0.7)and(-0.5,0.7)..(0,0.7) ;
\draw (-0.45,0) node[black,left,scale=0.7]{$k+l$} ;

\draw[scale= 0.8,decoration={markings,mark=at position 0.6 with {\arrow{<}}},postaction={decorate}](0,-0.7)..controls(0.5,-0.7)and(0.5,0.7)..(0,0.7) ;
\draw (0.45,0) node[black,right,scale=0.7]{$l$} ;
\end{tric}
}

\def\TypeABigonNeedB
{\begin{tric}
\draw[scale= 0.8,decoration={markings,mark=at position 0.65 with {\arrow{>}}},postaction={decorate}] (0,0.7)--(0,2) ;

\draw[scale= 0.8,decoration={markings,mark=at position 0.65 with {\arrow{>}}},postaction={decorate}](0,-2)--(0,-0.7); 
 
\draw[scale= 0.8,decoration={markings,mark=at position 0.6 with {\arrow{>}}},postaction={decorate}]     (0,0.7)..controls(-0.5,0.7)and(-0.5,-0.7)..(0,-0.7) ;

\draw[double,scale= 0.8,decoration={markings,mark=at position 0.6 with {\arrow{>}}},postaction={decorate}](0,-0.7)..controls(0.5,-0.7)and(0.5,0.7)..(0,0.7);
\end{tric}
}

\def\TypeADoubleID
{\begin{tric}
\draw[double, scale= 0.8,decoration={markings,mark=at position 0.65 with {\arrow{>}}},postaction={decorate}] (0,-2)--(0,2);
\end{tric}
}

\def\TypeAkplID
{\begin{tric}
\draw[scale= 0.8,decoration={markings,mark=at position 0.6 with {\arrow{>}}},postaction={decorate}] (0,-2)--(0,2);
\draw (0.1,0) node[black,right,scale=0.7]{$k+l$};
\end{tric}
}

\def\TypeAkID
{\begin{tric}
\draw[scale= 0.8,decoration={markings,mark=at position 0.6 with {\arrow{>}}},postaction={decorate}] (0,-2)--(0,2);
\draw (0.1,0) node[black,right,scale=0.7]{$k$};
\end{tric}
}

\def\TypeASingleID
{\begin{tric}
\draw[scale= 0.8,decoration={markings,mark=at position 0.65 with {\arrow{>}}},postaction={decorate}] (0,-2)--(0,2);
\end{tric}
}

\def\Flowassoci
{\begin{tric}
\draw[scale= 0.8,decoration={markings,mark=at position 0.65 with {\arrow{>}}},postaction={decorate}] (0,0)..controls(0,0.5)and(-0.2,0.7)..(-0.5,1);
\draw[scale= 0.8,decoration={markings,mark=at position 0.65 with {\arrow{>}}},postaction={decorate}] (-1,0)..controls(-1,0.5)and(-0.8,0.7)..(-0.5,1);
\draw[scale= 0.8,decoration={markings,mark=at position 0.65 with {\arrow{>}}},postaction={decorate}] 
      (-0.5,1)..controls(-0.5,1.5)and(-0.3,1.7)..(0,2) ;  
 \draw[scale= 0.8,decoration={markings,mark=at position 0.65 with {\arrow{>}}},postaction={decorate}]     (1,0)..controls(1,1)and(0.6,1.5)..(0,2);
  \draw[scale= 0.8,decoration={markings,mark=at position 0.65 with {\arrow{>}}},postaction={decorate}]     (0,2)--(0,3);
     \draw[scale= 0.8]  (0,0)node[below,black,scale=0.7]{$l$}
      (-1,0)node[below,black,scale=0.7]{$k$}
      (1,0)node[below,black,scale=0.7]{$m$}
     (-0.5,1.5)node[left,black,scale=0.7]{$k+l$}
      (0,3)node[above,black,scale=0.7]{$k+l+m$};
\end{tric}
}

\def\Flowassocia
{\begin{tric}
\draw[scale= 0.8,decoration={markings,mark=at position 0.65 with {\arrow{>}}},postaction={decorate}] (0,0)..controls(0,0.5)and(0.2,0.7)..(0.5,1);
 \draw[scale= 0.8,decoration={markings,mark=at position 0.65 with {\arrow{>}}},postaction={decorate}]      (1,0)..controls(1,0.5)and(0.8,0.7)..(0.5,1);
 \draw[scale= 0.8,decoration={markings,mark=at position 0.65 with {\arrow{>}}},postaction={decorate}] 
      (0.5,1)..controls(0.5,1.5)and(0.3,1.7)..(0,2) ;
              
  \draw[scale= 0.8,decoration={markings,mark=at position 0.65 with {\arrow{>}}},postaction={decorate}]     (-1,0)..controls(-1,1)and(-0.6,1.5)..(0,2);
 \draw[scale= 0.8,decoration={markings,mark=at position 0.65 with {\arrow{>}}},postaction={decorate}]      (0,2)--(0,3);
\draw [scale=0.8]      (0,0)node[below,black,scale=0.7]{$l$}
      (1,0)node[below,black,scale=0.7]{$m$}
      (-1,0)node[below,black,scale=0.7]{$k$}
     (0.5,1.5) node[right,black,scale=0.7]{$l+m$}
      (0,3)node[above,black,scale=0.7]{$k+l+m$};
\end{tric}
}

\def\TypeARungSwap{
\begin{tric}
\draw[scale=1.2,decoration={markings,mark=at position 0.5 with {\arrow{>}}},postaction={decorate}]  (0,0)--(0,1); 
\draw [scale=1.2](-0.1,0.4) node[left,scale=0.7,black]{$k$} ;
\draw [scale=1.2,decoration={markings,mark=at position 0.45 with {\arrow{>}}},postaction={decorate}] 
(1,0)--(1,1) ;
\draw [scale=1.2](1.1,0.4) node[right,scale=0.7,black]{$l$};
\draw [scale=1.2,decoration={markings,mark=at position 0.55 with {\arrow{>}}},postaction={decorate}] 
       (0,1)--(0,2);
\draw  [scale=1.2](-0.1, 1.5) node[left,scale=0.7,black]{$k-b$};     
\draw [scale=1.2,decoration={markings,mark=at position 0.6 with {\arrow{>}}},postaction={decorate}]      (1,1)--(1,2) ;
\draw  [scale=1.2] (1.1,1.45) node[right,scale=0.7,black]{$l+b$};
\draw [scale=1.2,decoration={markings,mark=at position 0.7 with {\arrow{>}}},postaction={decorate}]      
       (0,2)--(0,3);
\draw[scale=1.2] (-0.1,2.6)  node[left,scale=0.7,black]{$k-b+a$};
\draw [scale=1.2, decoration={markings,mark=at position 0.7 with {\arrow{>}}},postaction={decorate}]  
       (1,2)--(1,3) ;
\draw [scale=1.2](1.1,2.6)  node[right,scale=0.7,black]{$l+b-a$};
\draw [scale=1.2,decoration={markings,mark=at position 0.55 with {\arrow{<}}},postaction={decorate}]  (0,2.2)--(1,1.8) ;
\draw [scale=1.2](0.5,2.05)node[above,scale=0.7,black]{$a$}; 
\draw [scale=1.2,decoration={markings,mark=at position 0.55 with {\arrow{>}}},postaction={decorate}] (0,0.8)--(1,1.2) ;
\draw [scale=1.2](0.5,0.9) node[below,scale=0.7,black]{$b$}; 
\end{tric}
}

\def\TypeARungSwapA{
\begin{tric}
\draw[xscale=1.4, scale=1.2,decoration={markings,mark=at position 0.5 with {\arrow{>}}},postaction={decorate}]  (0,0)--(0,1); 
\draw [xscale=1.4,scale=1.2](-0.1,0.4) node[left,scale=0.7,black]{$k$} ;
\draw [xscale=1.4,scale=1.2,decoration={markings,mark=at position 0.45 with {\arrow{>}}},postaction={decorate}] 
(1,0)--(1,1) ;
\draw [xscale=1.4,scale=1.2](1.1,0.4) node[right,scale=0.7,black]{$l$};
\draw [xscale=1.4,scale=1.2,decoration={markings,mark=at position 0.55 with {\arrow{>}}},postaction={decorate}] 
       (0,1)--(0,2);
\draw  [xscale=1.4,scale=1.2](-0.1, 1.5) node[left,scale=0.7,black]{$k+a-t$};     
\draw [xscale=1.4,scale=1.2,decoration={markings,mark=at position 0.6 with {\arrow{>}}},postaction={decorate}]      (1,1)--(1,2) ;
\draw  [xscale=1.4,scale=1.2] (1.1,1.45) node[right,scale=0.7,black]{$l-a+t$};
\draw [xscale=1.4,scale=1.2,decoration={markings,mark=at position 0.7 with {\arrow{>}}},postaction={decorate}]      
       (0,2)--(0,3);
\draw[xscale=1.4,scale=1.2] (-0.1,2.6)  node[left,scale=0.7,black]{$k-b+a$};
\draw [xscale=1.4,scale=1.2, decoration={markings,mark=at position 0.7 with {\arrow{>}}},postaction={decorate}]  
       (1,2)--(1,3) ;
\draw [xscale=1.4,scale=1.2](1.1,2.6)  node[right,scale=0.7,black]{$l+b-a$};
\draw [xscale=1.4,scale=1.2,decoration={markings,mark=at position 0.55 with {\arrow{<}}},postaction={decorate}]  (1,2.2)--(0,1.8) ;
\draw [xscale=1.4,scale=1.2](0.5,2.05)node[above,scale=0.7,black]{$b-t$}; 
\draw [xscale=1.4,scale=1.2,decoration={markings,mark=at position 0.55 with {\arrow{>}}},postaction={decorate}] (1,0.8)--(0,1.2) ;
\draw [xscale=1.4,scale=1.2](0.5,0.9) node[below,scale=0.7,black]{$a-t$}; 
\end{tric}
}

\def\TypeACCSwapA{
\begin{tric}
\draw [xscale=1.4,scale=1.2,decoration={markings,mark=at position 0.55 with {\arrow{>}}},postaction={decorate}] (0,0)--(0,1)node[shift={(-0.1,0)},midway,left,scale=0.7,black]{$k$}; 
\draw [xscale=1.4,scale=1.2,decoration={markings,mark=at position 0.55 with {\arrow{<}}},postaction={decorate}]      (1,0)--(1,1) node[shift={(0.1,0)},midway,right,scale=0.7,black]{$l$} ;
\draw [xscale=1.4,scale=1.2,decoration={markings,mark=at position 0.55 with {\arrow{>}}},postaction={decorate}]
       (0,1)--(0,2) node[shift={(-0.1,0)},midway,left,scale=0.7,black]{$k-1$};
\draw [xscale=1.4,scale=1.2,decoration={markings,mark=at position 0.55 with {\arrow{<}}},postaction={decorate}]       
       (1,1)--(1,2) node[shift={(0.1,0)},midway,right,scale=0.7,black]{$l-1$};
\draw [xscale=1.4,scale=1.2,decoration={markings,mark=at position 0.55 with {\arrow{>}}},postaction={decorate}]   (0,2)--(0,3) node[shift={(-0.1,0)},midway,left,scale=0.7,black]{$k$};
\draw [xscale=1.4,scale=1.2,decoration={markings,mark=at position 0.55 with {\arrow{<}}},postaction={decorate}]   (1,2)--(1,3) node[shift={(0.1,0)},midway,right,scale=0.7,black]{$l$};
\draw [xscale=1.4,scale=1.2,decoration={markings,mark=at position 0.55 with {\arrow{>}}},postaction={decorate}] (0,0.85)..controls(0.3,1.35)and(0.7,1.35)..  
       (1,0.85)node[shift={(0,-0.1)},midway,below,scale=0.7,black]{$1$};
\draw [xscale=1.4,scale=1.2,decoration={markings,mark=at position 0.55 with {\arrow{<}}},postaction={decorate}]       (0,2.15)..controls(0.3,1.65)and(0.7,1.65)..(1,2.15)node[shift={(0,0.1)},midway,above,scale=0.7,black]{$1$}; 
\end{tric}
}

\def\TypeACCSwap{
\begin{tric}
\draw  [xscale=1.4,scale=1.2,decoration={markings,mark=at position 0.55 with {\arrow{>}}},postaction={decorate}] (0,0)--(0,1)node[shift={(-0.1,0)},midway,left,scale=0.7,black]{$k$};
\draw  [xscale=1.4,scale=1.2,decoration={markings,mark=at position 0.55 with {\arrow{<}}},postaction={decorate}] 
       (1,0)--(1,1) node[shift={(0.1,0)},midway,right,scale=0.7,black]{$l$};
\draw  [xscale=1.4,scale=1.2,decoration={markings,mark=at position 0.55 with {\arrow{>}}},postaction={decorate}]        
       (0,1)--(0,2) node[shift={(-0.1,0)},midway,left,scale=0.7,black]{$k+1$};
\draw  [xscale=1.4,scale=1.2,decoration={markings,mark=at position 0.55 with {\arrow{<}}},postaction={decorate}]        
       (1,1)--(1,2) node[shift={(0.1,0)},midway,right,scale=0.7,black]{$l+1$};
\draw  [xscale=1.4,scale=1.2,decoration={markings,mark=at position 0.55 with {\arrow{>}}},postaction={decorate}]        
       (0,2)--(0,3) node[shift={(-0.1,0)},midway,left,scale=0.7,black]{$k$};
\draw  [xscale=1.4,scale=1.2,decoration={markings,mark=at position 0.55 with {\arrow{<}}},postaction={decorate}]        
       (1,2)--(1,3) node[shift={(0.1,0)},midway,right,scale=0.7,black]{$l$};
\draw  [xscale=1.4,scale=1.2,decoration={markings,mark=at position 0.55 with {\arrow{<}}},postaction={decorate}]  (0,1)..controls(0.3,0.5)and(0.7,0.5)..  
       (1,1)node[shift={(0,-0.1)},midway,below,scale=0.7,black]{$1$};
\draw  [xscale=1.4,scale=1.2,decoration={markings,mark=at position 0.55 with {\arrow{>}}},postaction={decorate}]        
       (0,2)..controls(0.3,2.5)and(0.7,2.5)..(1,2)node[shift={(0,0.1)},midway,above,scale=0.7,black]{$1$}; 
\end{tric}
}

\def\TypeACCSwapB{
\begin{tric}
\draw [scale=1.2] (0,0)--(0,1) 
       (1,0)--(1,1) 
       (0,2)--(0,3) 
       (1,2)--(1,3) ;
\draw  [scale=1.2,decoration={markings,mark=at position 0.55 with {\arrow{>}}},postaction={decorate}]  (0,1)--(0,2) node[shift={(-0.1,0)},midway,left,scale=0.7,black]{$k$};
\draw [scale=1.2,decoration={markings,mark=at position 0.55 with {\arrow{<}}},postaction={decorate}]       (1,1)--(1,2) node[shift={(0.1,0)},midway,right,scale=0.7,black]{$l$} ;  
\end{tric}
}

\begin{defn} [{\cite[Section 2.2]{Cautis2014}}] \label{TypeAWebs}
 Define the $\GL_N$ web category, denoted by 
 $\Web=\webcat_q(\GL_{N})$, to be the pivotal $\mathbb{C}(q)$-linear category monoidally generated by the objects $k^{\pm}$, for $k \in \{1,2,...,N\}$,
 where $k^-$ is understood as the dual of $k=k^{+}$. Its morphisms are generated by the following two types of vertices:
 \begin{align*}
     \TypeAVert \in {\Hom}_{\Web}(k\otimes l, k+l) ,  \quad 
     \TypeAVertA \in {\Hom}_{\Web}(k+l, k \otimes l) , \\ 
 \end{align*}
for $l \in \mathbb{Z}_{\ge 0}$ and with all labels on the edges of a diagram at most $N$.

Morphisms are modded out by the tensor ideal generated by the following skein relations:
\begin{align} \label{DefSkeinGLN}
&\TypeACircleK=\TypeACircleKRe= \begin{bmatrix} N \\ k \end{bmatrix} , \quad \TypeABigonDefA= \begin{bmatrix} k+l \\ k \end{bmatrix} \ \TypeAkplID \ , \quad \TypeABigonDefB=  \begin{bmatrix} N-k \\ l \end{bmatrix} \ \TypeAkID ,\\
&\Flowassoci=\Flowassocia, \TypeARungSwap = \sum_{t} \begin{bmatrix} k-l+a-b \\ t \end{bmatrix} \TypeARungSwapA, \notag \\
& \TypeACCSwap=\TypeACCSwapA + [N-k-l]\TypeACCSwapB \notag 
\end{align}
\end{defn}

\def\SingleEdgeLabel
{\begin{tric}
\draw [scale=1.2,decoration={markings,mark=at position 0.7 with {\arrow{>}}},postaction={decorate}] (0,0)--(90:1) ;

\draw [scale=1.2](0.05,0.5)node[black,right,scale=0.7]{$1$}; 
\end{tric}
}

\def\SingleEdgeNoLabel
{\begin{tric}
\draw [scale=1.2,decoration={markings,mark=at position 0.7 with {\arrow{>}}},postaction={decorate}] (0,0)--(90:1) ;
\end{tric}
}

\def\DoubleEdgeLabel
{\begin{tric}
\draw [scale=1.2,decoration={markings,mark=at position 0.7 with {\arrow{>}}},postaction={decorate}] (0,0)--(90:1) ;

\draw [scale=1.2](0.06,0.5)node[black,right,scale=0.7]{$2$};
\end{tric}
}

\def\DoubleEdgeNoLabel
{\begin{tikzpicture}[scale=0.7, decoration={
    markings,
    mark=at position 0.6 with {\arrow{>}}}, baseline={([yshift=-.8ex]current bounding box.center)}]
\directeddoubleline{(0,0)--(90:0.8)};
\end{tikzpicture}
}

\def\TripleEdgeLabel
{\begin{tric}
\draw [scale=1.2,decoration={markings,mark=at position 0.7 with {\arrow{>}}},postaction={decorate}] (0,0)--(90:1) ;

\draw [scale=1.2](0.06,0.5)node[black,right,scale=0.7]{$3$};
\end{tric}
}

\def\TripleEdgeNoLabel
{\begin{tikzpicture}[scale=0.7, decoration={
    markings,
    mark=at position 0.6 with {\arrow{>}}}, baseline={([yshift=-.8ex]current bounding box.center)}]
\directedtripleline{(0,0)--(90:0.8)};
\end{tikzpicture}
}

\begin{notation}
        In this paper, we denote \ $\SingleEdgeLabel$ by \ $\SingleEdgeNoLabel$\ ,  \ $\DoubleEdgeLabel$ by \ $\DoubleEdgeNoLabel$, and \ $\TripleEdgeLabel$ by \ $\TripleEdgeNoLabel$. 
\end{notation}

We call a planar graph with trivalent vertices $\TypeAVert$ and  $\TypeAVertA$ a \emph{$\GL_{N}$ web}, or simply a \emph{web}. 
Denote by $\Seq_{N}$ the set of finite sequences of elements of $\{1^{\pm},2^{\pm},...,N^{\pm}\}$. Any element $\varepsilon \in \Seq_{N}$ can be written as \(\varepsilon = a_1^{\epsilon_1} a_2^{\epsilon_2} \cdots a_r^{\epsilon_r}\), we will use this notation throughout the paper. We identify these sequences with objects of $\Web$. 
The empty sequence is included in $\Seq_{N}$ and is denoted by $\emptylist$. 

Given two sequences $\varepsilon_0,\varepsilon_1\in \Seq_{N}$, 
a $\GL_{N}$ web from $\varepsilon_0$ to $\varepsilon_1$ is a planar graph $W$ as above in $\R\times [0,1]$  with $ W \cap (\R \times \{i\}) = \varepsilon_i \times \{i\}$ for $i=0,1$. 
Such graphs are considered up to planar isotopies rel boundary. 
A web from $\varepsilon_0$ to $\varepsilon_1$ defines a morphism in ${\Hom}_{\Web}(\varepsilon_0,\varepsilon_1)$. 
Denote by $\idweb_{\varepsilon_0}$ the web from $\varepsilon_0$ to $\varepsilon_0$ given by vertical strands, the identity morphism in ${\Hom}_{\Web}(\varepsilon_0,\varepsilon_0)$.

Given a web $W_0$ from $\varepsilon_0$ to $\varepsilon_1$ and a web $W_1$ from $\varepsilon_1$ to $\varepsilon_2$, denote by $W_1W_0$ the web from $\varepsilon_0$ to $\varepsilon_2$ by stacking $W_1$ on top of $W_0$. The morphism defined by $W_1W_0$ in ${\Hom}_{\Web}(\varepsilon_0,\varepsilon_2)$ agrees with the composition of the morphisms defined by $W_1$ and $W_0$. 

We call a boundary sequence $\varepsilon \in \Seq_N$ \emph{admissible} when the net flow of labels is balanced, ensuring that the total sum of labels carrying a positive sign equals the total sum of labels carrying a negative sign. A sequence $\varepsilon \in \Seq_N$ is a boundary of a $\GL_N$ web from $\emptylist$ to $\varepsilon$ if and only if it is admissible.

A \emph{closed web} $\Gamma$ is a web from $\emptylist$ to $\emptylist$, which defines a morphism in ${\Hom}_{\Web}(\emptylist,\emptylist)$, denoted by $\langle \Gamma \rangle \idweb_{\emptylist}$.

Given a web $W$ from $\varepsilon_0$ to $\varepsilon_1$, denote by $\overline{W}$ the web from $\varepsilon_1$ to $\varepsilon_0$ obtained by taking the reflection of $W$ about a horizontal line and reversing the orientation of each edge.

\def\BenzenewebDisk
{\begin{tricpob}
\begin{scope}[yscale=-1]
\draw [->](-1,1.7)--(0.1,1.7); 
\draw (0,1.7)--(1,1.7);
\draw [->](2,0)--(1.4,-1.02); 
\draw (1.5,-0.85)--(1,-1.7);
\draw [->](-1,-1.7)--(-1.6,-0.68); 
\draw (-1.5,-0.85)--(-2,0); 
\draw[double,decoration={markings,mark=at position 0.65 with {\arrow{>}}},postaction={decorate}]  (-2,0)--(-1,1.7);
\draw[double,decoration={markings,mark=at position 0.75 with {\arrow{>}}},postaction={decorate}] (1,-1.7)--(-1,-1.7);
\draw[double, decoration={markings,mark=at position 0.65 with {\arrow{>}}},postaction={decorate}] (1,1.7)--(2,0);

\draw [->](2,0)--(3.1,0); 
\draw (3,0)--(4,0);
\draw [->](-4,0)--(-2.9,0); 
\draw (-3,0)--(-2,0);
\draw [->](-1,1.7)--(-1.6,2.72);
\draw (-1.5,2.55)--(-2,3.4);
\draw [->](2,-3.4)--(1.4,-2.38);
\draw (1.5,-2.55)--(1,-1.7);
\draw [->](-1,-1.7)--(-1.6,-2.72);
\draw (-1.5,-2.55)--(-2,-3.4);
\draw [->](2,3.4)--(1.4,2.38);
\draw (1.5,2.55)--(1,1.7);

\draw [black, dashed] (0,0) circle (4); 
\draw (0,4) node[scale=1.5]{$\mathbf{\ast}$};
\end{scope}
\end{tricpob}
}

\def\BenzenewebBent
{\begin{tricpob}
\begin{scope}[yscale=-1]
\draw [->](-1,1.7)--(0.1,1.7); 
\draw (0,1.7)--(1,1.7);
\draw [->](2,0)--(1.4,-1.02); 
\draw (1.5,-0.85)--(1,-1.7);
\draw [->](-1,-1.7)--(-1.6,-0.68); 
\draw (-1.5,-0.85)--(-2,0); 
\draw[double,decoration={markings,mark=at position 0.65 with {\arrow{>}}},postaction={decorate}]  (-2,0)--(-1,1.7);
\draw[double,decoration={markings,mark=at position 0.75 with {\arrow{>}}},postaction={decorate}] (1,-1.7)--(-1,-1.7);
\draw[double, decoration={markings,mark=at position 0.65 with {\arrow{>}}},postaction={decorate}] (1,1.7)--(2,0);

\draw [decoration={markings,mark=at position 0.65 with {\arrow{>}}},postaction={decorate}] (2,0)..controls(3,0)and(4,-2.5)..(4,-3.4); 
\draw [decoration={markings,mark=at position 0.65 with {\arrow{>}}},postaction={decorate}] (-4,-3.4)..controls(-4,-2.5)and(-3,0)..(-2,0); 

\draw [decoration={markings,mark=at position 0.65 with {\arrow{>}}},postaction={decorate}] (-1,1.7)..controls(-1.6,2.72)and(-5.5,0)..(-5.5,-3.4);
\draw [decoration={markings,mark=at position 0.65 with {\arrow{>}}},postaction={decorate}] (5.5,-3.4)..controls(5.5,0)and(1.6,2.72)..(1,1.7);

\draw [->](2,-3.4)--(1.4,-2.38);
\draw (1.5,-2.55)--(1,-1.7);
\draw [->](-1,-1.7)--(-1.6,-2.72);
\draw (-1.5,-2.55)--(-2,-3.4);

\draw [black, dashed] (-6.5,-3.4)--(6.5,-3.4); 
\end{scope}
\end{tricpob}
}

\def\BenzeneDual
{\begin{tricpob}
\draw [->](-1,1.7)--(0.1,1.7); 
\draw (0,1.7)--(1,1.7);
\draw [->](2,0)--(1.4,-1.02); 
\draw (1.5,-0.85)--(1,-1.7);
\draw [->](-1,-1.7)--(-1.6,-0.68); 
\draw (-1.5,-0.85)--(-2,0); 
\draw[double,decoration={markings,mark=at position 0.65 with {\arrow{>}}},postaction={decorate}]  (-2,0)--(-1,1.7);
\draw[double,decoration={markings,mark=at position 0.75 with {\arrow{>}}},postaction={decorate}] (1,-1.7)--(-1,-1.7);
\draw[double, decoration={markings,mark=at position 0.65 with {\arrow{>}}},postaction={decorate}] (1,1.7)--(2,0);

\draw [decoration={markings,mark=at position 0.65 with {\arrow{>}}},postaction={decorate}] (2,0)..controls(3,0)and(4,-2.5)..(4,-3.4); 
\draw [decoration={markings,mark=at position 0.65 with {\arrow{>}}},postaction={decorate}] (-4,-3.4)..controls(-4,-2.5)and(-3,0)..(-2,0); 

\draw [decoration={markings,mark=at position 0.65 with {\arrow{>}}},postaction={decorate}] (-1,1.7)..controls(-1.6,2.72)and(-5.5,0)..(-5.5,-3.4);
\draw [decoration={markings,mark=at position 0.65 with {\arrow{>}}},postaction={decorate}] (5.5,-3.4)..controls(5.5,0)and(1.6,2.72)..(1,1.7);

\draw [->](2,-3.4)--(1.4,-2.38);
\draw (1.5,-2.55)--(1,-1.7);
\draw [->](-1,-1.7)--(-1.6,-2.72);
\draw (-1.5,-2.55)--(-2,-3.4);

\draw [black, dashed] (-6.5,-3.4)--(6.5,-3.4); 
\end{tricpob}
}

\def\BenzeneDualColor
{\begin{tricpob}
\draw [->](-1,1.7)--(0.1,1.7); 
\draw (-0.2,1.9) node[above,black,scale=0.7]{$e_a$}; 
\draw (0,1.7)--(1,1.7);

\draw [->](2,0)--(1.4,-1.02); 
\draw (1.5,-1.2) node[right,black,scale=0.7]{$e_a$}; 
\draw (1.5,-0.85)--(1,-1.7);

\draw [->](-1,-1.7)--(-1.6,-0.68);
\draw (-1.4,-1.2) node[left,black,scale=0.7]{$e_a$}; 
\draw (-1.5,-0.85)--(-2,0); 

\draw[double,decoration={markings,mark=at position 0.65 with {\arrow{>}}},postaction={decorate}]  (-2,0)--(-1,1.7);
\draw[double,decoration={markings,mark=at position 0.75 with {\arrow{>}}},postaction={decorate}] (1,-1.7)--(-1,-1.7);
\draw[double, decoration={markings,mark=at position 0.65 with {\arrow{>}}},postaction={decorate}] (1,1.7)--(2,0);

\draw [decoration={markings,mark=at position 0.65 with {\arrow{>}}},postaction={decorate}] (2,0)..controls(3,0)and(4,-2.5)..(4,-3.4); 
\draw [decoration={markings,mark=at position 0.65 with {\arrow{>}}},postaction={decorate}] (-4,-3.4)..controls(-4,-2.5)and(-3,0)..(-2,0); 

\draw [decoration={markings,mark=at position 0.65 with {\arrow{>}}},postaction={decorate}] (-1,1.7)..controls(-1.6,2.72)and(-5.5,0)..(-5.5,-3.4);
\draw [decoration={markings,mark=at position 0.65 with {\arrow{>}}},postaction={decorate}] (5.5,-3.4)..controls(5.5,0)and(1.6,2.72)..(1,1.7);

\draw [->](2,-3.4)--(1.4,-2.38);
\draw (1.5,-2.55)--(1,-1.7);
\draw [->](-1,-1.7)--(-1.6,-2.72);
\draw (-1.5,-2.55)--(-2,-3.4);

\draw (-2,-3.4) node[below,black,scale=0.7]{$e_i$}; 
\draw  (2,-3.4) node[below,black,scale=0.7]{$e_i$};
\draw (-4,-3.4) node[below,black,scale=0.7]{$e_i$};
\draw  (4,-3.4) node[below,black,scale=0.7]{$e_i$};
\draw  (-6,-3.4) node[below,black,scale=0.7]{$e_i$};
\draw  (6,-3.4) node[below,black,scale=0.7]{$e_i$};
\end{tricpob}
}

Given  $\varepsilon\in \Seq_{N}$, we call a web from $\emptylist$ to $\varepsilon$ a \emph{web with boundary $\varepsilon$}. 
A web $W$ with boundary $\varepsilon$ can alternatively be seen as a web in the $2$-disk $\mathbb{D}^2$, where the boundary $\varepsilon$ of $W$ is given by a finite set of points on a marked circle 
$\mathbb{S}^1=(\partial\mathbb{D}^2, *)$ equipped with labels and signs agreeing with $\varepsilon$.
(By convention, $*$ does not meet $\partial W$.)
We may depict this equivalence as follows:
\[
\scalebox{0.8}{\BenzenewebDisk \qquad \textit{\Large$\sim$} \qquad \BenzenewebBent}
\]

\begin{proposition}[{\cite[Theorem 3.3.1]{Cautis2014}}]
\label{WebRep}
There is a monoidal functor
\[
\Phi:\webcat_q(\GL_{N}) \rightarrow \Fund (U_q(\GL_{N}))
\]
sending $k$ to $V_{\omega_k}$, which is an equivalence of $\mathbb{C}(q)$-linear pivotal categories.
\end{proposition}

$\Fund(U_q(\GL_N))$ is the fundamental representation category of $U_q(\GL_N)$, the full subcategory of the representation category of $U_q(\GL_N)$, which is monoidally generated by the fundamental representations $\wedge_q^k\mathbb{C}_q^N$ and their duals.

Proposition \ref{WebRep} is due to Cautis--Kamnitzer--Morrison \cite{Cautis2014} where they state the results for $U_q(\SL_N)$. The $U_q(\GL_N)$ version appeared in \cite{MixedQHowe,SuperQHowe,ExtremalGLN}.

\begin{proposition} 
    The following relations follow from Relations \eqref{DefSkeinGLN}.  
    \begin{equation} \label{NeededSkein}
        \TypeACircleNeed= [N] , \quad \TypeABigonNeedA=[2] \ \TypeADoubleID \ , \quad \TypeABigonNeedB=  [N-1] \ \TypeASingleID .\ 
    \end{equation}   
\end{proposition}

\def\TypeABraidWeb
{\begin{tric}
\draw (-1.5,1.5)--(1.5,-1.5) ;
\draw[double=darkblue,ultra thick,white,line width=3pt] (1.5,1.5)--(-1.5,-1.5);
\draw [->](0.7,0.7)--(1,1)node[below,black,pos=0.4,scale=0.7]{$1$};
\draw [->](-0.7,0.7)--(-1,1)node[below,black,pos=0.4,scale=0.7]{$1$};
\end{tric}
}

\def\TypeAIdWeb{
\begin{tric}
\draw[scale=0.7,decoration={markings,mark=at position 0.6 with {\arrow{>}}},postaction={decorate}](-2,-2)..controls(-1,-1)and(-1,1)..(-2,2)node[right,black,midway,scale=0.7]{$1$}  ;
\draw  [scale=0.7,decoration={markings,mark=at position 0.6 with {\arrow{>}}},postaction={decorate}] (2,-2)..controls(1,-1)and(1,1)..(2,2)node[left,black,midway,scale=0.7]{$1$};
\end{tric}
}

\def\TypeAIH
{\begin{tric}
\draw[scale=0.7,decoration={markings,mark=at position 0.7 with {\arrow{>}}},postaction={decorate}](-1.732,-2)--(0,-1)
node[below,black,midway,scale=0.7]{$1$};

\draw[scale=0.7,decoration={markings,mark=at position 0.7 with {\arrow{>}}},postaction={decorate}](1.732,-2)--(0,-1)
node[below,black,midway,scale=0.7]{$1$};

\draw[scale=0.7,decoration={markings,mark=at position 0.7 with {\arrow{>}}},postaction={decorate}](0,1)--(-1.732,2)
node[above,black,pos=0.6,scale=0.7]{$1$};

\draw[scale=0.7,decoration={markings,mark=at position 0.7 with {\arrow{>}}},postaction={decorate}](0,1)--(1.732,2)
node[above,black,pos=0.6,scale=0.7]{$1$};

\draw[double, scale=0.7,decoration={markings,mark=at position 0.7 with {\arrow{>}}},postaction={decorate}](0,-1)-- (0,1);
\end{tric}
}

\subsection{Closed \texorpdfstring{$\GL_N$}{GL_N} foams}

In this section we recall the definition of $\GL_N$ foams and their evaluation formula, following \cite{RW-eval-foams}.

\begin{definition} \label{defSL4foam}

A closed $\GL_N$ foam $F$ is a topological space obtained by gluing together a finite collection of facets $\Sigma_i$, $i$ in some finite indexing set $I$, along with an embedding $\iota: F \rightarrow \mathbb{R}^3$. 
Each facet $\Sigma_i$ is an oriented compact connected surface with a label $k$ where $k \in \{ 1,2, \dots , N \}$. 
Facets are glued along \emph{bindings} and sometimes meet at \emph{singular vertices}, such that orientations are compatible  along bindings and vertices,
according to the local models depicted in Figure \ref{figure basic foam}.
\end{definition}

\def\BindingLineOrientGLN
{
\begin{tikzpicture}[yscale=.7, scale=.75, baseline={([yshift=-.8ex]current bounding box.center)}]
    \def \t{8}; 
    \def \r{3}; 
    \def \h{4}; 
    \def \a {(\t:\r cm)};
    \def \b {(\t+120: \r cm)};
    \def \c {(\t+240: \r cm)};
    \draw[thick,darkgreen] (0,0) -- \a;
    \draw[thick,darkgreen, dashed] (0,0) -- (\t+120: .5*\r cm);
    \draw[thick,darkgreen] (\t+120: .5*\r cm) -- (\t+120: \r cm);
    \draw[thick,darkgreen] (0,0) -- \c;
    \foreach \p in {\b,\c}{
        \draw[thick,darkgreen] \p -- +(0,\h);
    }
    \foreach \p in {(0,0)}{
    \draw[thick,darkgreen] \p -- +(0,\h);
    }
    \foreach \p in {\a}{
        \draw[thick,darkgreen] \p -- +(0,\h);
    }
    \begin{scope}[yshift=\h cm]
        \draw[thick,darkgreen] (0,0) -- \a;
        \draw[thick,darkgreen] (0,0) -- \b;
        \draw[thick,darkgreen] (0,0) -- \c;
    \end{scope}
    \draw[thick, darkorange, ->,yscale=1.4]  (\r/2.5,\h/2.6) arc (180:450:0.25);
    \draw[thick, darkorange, ->,yscale=1.4] (-\r/2,\h/1.5) arc (180:450:0.25);
    \draw[thick, darkorange, ->,yscale=1.4] (-\r/3.25,-0.25) arc (180:450:0.25);
    \draw (\r/2.5,\h/3.2) node[red, scale=0.7]{$a+b$};
    \draw (-\r/2,\h/1.5) node[red, scale=0.7]{$b$};
    \draw (-\r/3.4,-1.2) node[red, scale=0.7]{$a$};
\end{tikzpicture}
}

\def\BindingLineOrientEdge
{
\begin{tikzpicture}[yscale=.7, scale=.75, baseline={([yshift=-.8ex]current bounding box.center)}]
    \def \t{8}; 
    \def \r{3}; 
    \def \h{4}; 
    \def \a {(\t:\r cm)};
    \def \b {(\t+120: \r cm)};
    \def \c {(\t+240: \r cm)};
    \draw[thick,darkgreen] (0,0) -- \a;
    \draw[thick,darkgreen, dashed] (0,0) -- (\t+120: .5*\r cm);
    \draw[thick,darkgreen] (\t+120: .5*\r cm) -- (\t+120: \r cm);
    \draw[thick,darkgreen] (0,0) -- \c;
    \foreach \p in {\b,\c}{
        \draw[thick,darkgreen] \p -- +(0,\h);
    }
    \foreach \p in {(0,0)}{
    \draw[decoration={markings,mark=at position 0.52 with {\arrow{>}}},postaction={decorate}, thick,darkgreen] \p -- +(0,\h);
    }
    \foreach \p in {\a}{
        \draw[thick,darkgreen] \p -- +(0,\h);
    }
    \begin{scope}[yshift=\h cm]
        \draw[thick,darkgreen] (0,0) -- \a;
        \draw[thick,darkgreen] (0,0) -- \b;
        \draw[thick,darkgreen] (0,0) -- \c;
    \end{scope}
    \draw[thick, darkorange, ->,yscale=1.4]  (\r/2.5,\h/2.6) arc (180:450:0.25);
    \draw[thick, darkorange, ->,yscale=1.4] (-\r/2,\h/1.5) arc (180:450:0.25);
    \draw[thick, darkorange, ->,yscale=1.4] (-\r/3.25,-0.25) arc (180:450:0.25);
    \draw (\r/2.5,\h/3.2) node[red, scale=0.7]{$a+b$};
    \draw (-\r/2,\h/1.5) node[red, scale=0.7]{$b$};
    \draw (-\r/3.4,-1.2) node[red, scale=0.7]{$a$};
\end{tikzpicture}
}

\def\PurpleVertexOrient
{\begin{tikzpicture}[scale=.75, baseline={([yshift=-.8ex]current bounding box.center)}]
    \begin{scope}[xshift=3cm]
        \draw[double,darkgreen] (-2,-1) -- (0,1) 
            -- (0,-1.5) -- (-2,-3.5) -- (-2,-1);
        \draw[red, line width= \DefectThick ] (-1,1) -- (-1,-2.5);
    \end{scope}
    \begin{scope}[xshift=1cm, darkgreen]
        \draw (-2.3,-1) -- (-0.3,1) -- (4.3,1) -- (2.3,-1) -- (-2.3,-1);
    \end{scope}
    \draw[double,darkgreen] (-0.3,0) rectangle (4.3,2.3);
    \draw[red, line width= \DefectThick ] (2,0) -- (2,2.3);
    \filldraw[purple] (2,0) circle (2pt);
    \draw[thick, darkorange, ->] (2.7,-1.5) arc (0:270:0.25);
    \draw[thick, darkorange, <-] (1.5,-2.2) arc (-90:180:0.25);
    \draw[thick, darkorange, <-] (0,-0.7) arc (-90:180:0.25);
    \draw[thick, darkorange, ->] (3.7,-0.4) arc (0:270:0.25);
    \draw[thick, darkorange, ->] (0.4,0.4) arc (180:450:0.25);
    \draw[thick, darkorange, <-] (3.8,0.65) arc (90:360:0.25);
    \draw[thick, darkorange, ->] (2.85,1.75) arc (0:270:0.25);
    \draw[thick, darkorange, <-] (1.4,1.5) arc (-90:180:0.25);
\end{tikzpicture}
}

\def\PurpleVertexGLN
{\begin{tikzpicture}[baseline={([yshift=-.8ex]current bounding box.center)}]
    \begin{scope}[xshift=3cm]
        \draw[thick,darkgreen] (-0.2,-1) 
            -- (-0.2,-1.5) -- (-1.8,-3.5) -- (-1.8,-1)--(-1,0);
        \draw[dashed,thick,darkgreen](-1,0)--(-0.2,1)--(-0.2,-1);    
    \end{scope}
    \begin{scope}[thick,xshift=1cm, darkgreen]
        \draw (-1,0)--(-1.8,-1)--(2.2,-1)--(3,0);   
        \draw [dashed](-1,0)--(-0.2,1)--(3.8,1)--(3,0);  
        \draw (3,0)--(3.8,1)--(3,1); 
    \end{scope}
    \draw[thick,darkgreen] (0,0) rectangle (4,2);
    \draw (1,-0.3) node[red,scale=0.7]{$a+b+c$};
    \draw (1.7,0.3) node[red,scale=0.7]{$a$};
    \draw (2.4,-0.3) node[red,scale=0.7]{$c$};
    \draw (3,0.3) node[red,scale=0.7]{$b$};
    \draw[thick, darkorange, <-] (-0.1,-0.6) arc (180:450:0.25);
    \draw[thick, darkorange, <-] (0.6,0.4) arc (180:450:0.25);
    \draw[thick, darkorange, ->] (3.7,0.5) arc (0:270:0.25);
    \draw[thick, darkorange, <-] (2.45,-0.55) arc (180:450:0.25);
    \draw (1.8,-1.5) node[red,scale=0.7]{$a+b$};
    \draw (2.5,1.5) node[red,scale=0.7]{$b+c$};
    \draw[thick, darkorange, <-] (1.5,-2) arc (180:450:0.25);
    \draw[thick, darkorange, <-] (1.6,1.5) arc (180:450:0.25);
\end{tikzpicture}
}

\begin{figure}[htbp]
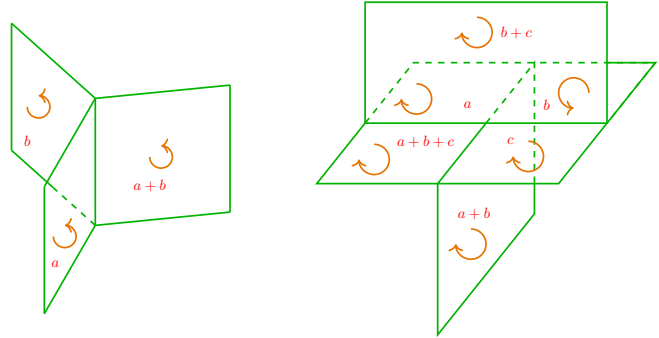

    \centering
    \scalebox{0.8}{%
        $\displaystyle    \BindingLineOrientGLN \qquad \qquad \PurpleVertexGLN$
    }
    \caption{An $(a,b,a+b)$-binding (on the left) and an $(a,b,c,a+b,b+c,a+b+c)$-singular vertex (on the right).}
    \label{figure basic foam}
\end{figure}

\begin{definition}\label{localmodeldef}
We use a few more intuitive terms to describe foams.

\begin{itemize}

\item  We call the label $k$ of a facet the \emph{thickness} of the facet, and also refer to a facet with thickness $k$ as a \emph{$k$-facet}. When we later define colorings on foams, a $k$-facet will be colored with $k$ pairwise distinct pigments. 

\item The $1$-dimensional glued seam among $a$-facet, $b$-facet, and $(a+b)$-facet is called a \emph{$(a,b,a+b)$-binding}, or a \emph{binding} for short.

\item The $0$-dimensional glued region in the intersection of all the facets in the local model is called a \emph{singular point}.  A singular point is adjacent to six facets. There are four seams (bindings) near each singular point.  

\end{itemize}
    
\end{definition}

\def\BindingLine
{
\begin{tikzpicture}[yscale=.7, scale=.75]
    \def \t{10}; 
    \def \r{2}; 
    \def \h{4}; 
    \def \a {(\t:\r cm)};
    \def \b {(\t+120: \r cm)};
    \def \c {(\t+240: \r cm)};
    \draw[double, darkgreen] (0,0) -- \a;
    \draw[darkgreen, dashed] (0,0) -- (\t+120: .5*\r cm);
    \draw[darkgreen] (\t+120: .5*\r cm) -- (\t+120: \r cm);
    \draw[darkgreen] (0,0) -- \c;
    \foreach \p in {\b,\c}{
        \draw[darkgreen] \p -- +(0,\h);
    }
    \foreach \p in {(0,0)}{
    \draw[thick, darkgreen] \p -- +(0,\h);
    }
    \foreach \p in {\a}{
        \draw[double, darkgreen] \p -- +(0,\h);
    }
    \begin{scope}[yshift=\h cm]
        \draw[double, darkgreen] (0,0) -- \a;
        \draw[darkgreen] (0,0) -- \b;
        \draw[darkgreen] (0,0) -- \c;
    \end{scope}
\end{tikzpicture}
}

\def\BindingLineOrient
{
\begin{tikzpicture}[yscale=.7, scale=.75, baseline={([yshift=-.8ex]current bounding box.center)}]
    \def \t{8}; 
    \def \r{3}; 
    \def \h{4}; 
    \def \a {(\t:\r cm)};
    \def \b {(\t+120: \r cm)};
    \def \c {(\t+240: \r cm)};
    \draw[double, darkgreen] (0,0) -- \a;
    \draw[darkgreen, dashed] (0,0) -- (\t+120: .5*\r cm);
    \draw[darkgreen] (\t+120: .5*\r cm) -- (\t+120: \r cm);
    \draw[darkgreen] (0,0) -- \c;
    \foreach \p in {\b,\c}{
        \draw[darkgreen] \p -- +(0,\h);
    }
    \foreach \p in {(0,0)}{
    \draw[thick, darkgreen] \p -- +(0,\h);
    }
    \foreach \p in {\a}{
        \draw[double, darkgreen] \p -- +(0,\h);
    }
    \begin{scope}[yshift=\h cm]
        \draw[double, darkgreen] (0,0) -- \a;
        \draw[darkgreen] (0,0) -- \b;
        \draw[darkgreen] (0,0) -- \c;
    \end{scope}
    \draw[thick, darkgreen, ->,yscale=1.4] (\r/2.5,\h/2.6) arc (180:450:0.25);
       \draw[thick, darkgreen, ->,yscale=1.4] (-\r/2,\h/1.5) arc (180:450:0.25);
       \draw[thick, darkgreen, ->,yscale=1.4] (-\r/3.25,-0.25) arc (180:450:0.25);
\end{tikzpicture}
}

\def\GLBindingLine
{
\begin{tikzpicture}[yscale=.7, scale=.65]
    \def \t{10}; 
    \def \r{2}; 
    \def \h{4}; 
    \def \a {(\t:\r cm)};
    \def \b {(\t+120: \r cm)};
    \def \c {(\t+240: \r cm)};
    \draw[darkgreen] (0,0) -- \a;
    \draw[darkgreen, dashed] (0,0) -- (\t+120: .5*\r cm);
    \draw[darkgreen] (\t+120: .5*\r cm) -- (\t+120: \r cm);
    \draw[darkgreen] (0,0) -- \c;
    \foreach \p in {\b,\c}{
        \draw[darkgreen] \p -- +(0,\h);
    }
    \foreach \p in {(0,0)}{
    \draw[thick, darkgreen] \p -- +(0,\h);
    }
    \foreach \p in {\a}{
        \draw[darkgreen] \p -- +(0,\h);
    }
    \draw (-0.3,\h/2.5) node[black,scale=0.7]{$a$};
    \draw (1,\h/1.75) node[black,scale=0.7]{$a+b$};
    \draw (-1,\h/1.5) node[black,scale=0.7]{$b$};
    \begin{scope}[yshift=\h cm]
        \draw[darkgreen] (0,0) -- \a;
        \draw[darkgreen] (0,0) -- \b;
        \draw[darkgreen] (0,0) -- \c;
    \end{scope}
\end{tikzpicture}
}

\def\DefectLine
{\begin{tikzpicture}
    \draw[double,darkgreen] (0,0) rectangle (2,1.5);
    \draw[red, line width=\DefectThick] (1,0) -- (1,1.5);
\end{tikzpicture}
}

\def\DefectLineOrient
{\begin{tikzpicture} [baseline={([yshift=-.8ex]current bounding box.center)}]
    \draw[double,darkgreen] (0,0) rectangle (3,2);
    \draw[red, line width=\DefectThick] (1.5,0) -- (1.5,2);
    \draw[thick, darkgreen, ->] (0.5,1) arc (180:450:0.2);
    \draw[thick, darkgreen, <-] (2.2,1.2) arc (90:360:0.2);
\end{tikzpicture}
}

\def\DefectLineOrientEdge
{\begin{tikzpicture} [baseline={([yshift=-.8ex]current bounding box.center)}]
    \draw[double,darkgreen] (0,0) rectangle (3,2);
    \draw[red, line width=\DefectThick,decoration={markings,mark=at position 0.52 with {\arrow{>}}},postaction={decorate}] (1.5,0) -- (1.5,2);
    \draw[thick, darkgreen, ->] (0.5,1) arc (180:450:0.2);
    \draw[thick, darkgreen, <-] (2.2,1.2) arc (90:360:0.2);
\end{tikzpicture}
}

\def\PurpleVertex
{\begin{tikzpicture}
    \begin{scope}[xshift=1.5cm, scale=.5]
        \draw[double,darkgreen] (-2,-1) -- (0,1) 
            -- (0,-2) -- (-2,-4) -- (-2,-1);
        \draw[red, line width= \DefectThick ] (-1,1) -- (-1,-3);
    \end{scope}
    \begin{scope}[xshift=.5cm, darkgreen, scale=.5]
        \draw (-2,-1) -- (0,1) -- (4,1) -- (2,-1) -- (-2,-1);
    \end{scope}
    \draw[double,darkgreen] (0,0) rectangle (2,1.5);
    \draw[red, line width= \DefectThick ] (1,0) -- (1,1.5);
    \filldraw[purple] (1,0) circle (2pt);
\end{tikzpicture}
}

\begin{remark}
 Given the orientation of a facet around a singular vertex, the compatibility condition for orientations around a binding uniquely determines the orientations of the rest of the facets around the singular vertex.
\end{remark}

\begin{remark}
    In a picture of $1$-facet, we allow the label $1$ to be omitted. Moreover, we will sometimes depict a $2$-facet by drawing two parallel facets close to one another. Any number denoting the thickness of a facet will always be written in red;  we preserve black text for \textit{colorings} (Definition \ref{def:coloring}). 
\end{remark}

\begin{definition}
    The orientation of the $(a,b,a+b)$-binding is induced by the orientations of the facets as in Figure \ref{figure of oriented seam}. 
    The orientation of the binding agrees with the orientations of the $a$-facet and the $b$-facet but is opposite to the orientation of the $(a+b)$-facet.
    
\begin{figure}[htbp]
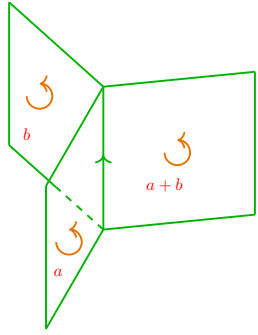

    \centering
    \scalebox{0.9}{%
        $\displaystyle \BindingLineOrientEdge$
    }
    \caption{The induced orientation on the $(a,b,a+b)$-binding.}
    \label{figure of oriented seam}
\end{figure}

\end{definition}

\begin{remark}
    In what follows we distinguish lower case variables $x_i$ and upper case variables $\vari$. Namely, $x_i$'s will be used to denote abstract symmetric polynomials, whereas $\vari$'s are assigned to \textit{pigments} (to be defined below) and plugged into the former abstract symmetric polynomials. Importantly, $x_i$'s have degree $1$ as polynomials, but $\vari$'s are assumed to have degree $2$. We will write $\Symf$ to denote the symmetric polynomials in either set of variables when clear from the context.
\end{remark}

\begin{notation} 
A $k$-facet of a foam may be decorated with dots labeled by homogeneous symmetric polynomials in $k$ variables, with $\Q$-coefficients. On a single facet, a dot is labeled by a polynomial in $x$. On a $k$-facet, a dot is labeled by a homogeneous symmetric polynomial $s(x_1, x_2, \ldots, x_k)$ in $k$ variables $x_1, x_2, \ldots, x_k$. Multiple dots may appear on the same facet. 

We denote by $e_k$ the elementary symmetric polynomial of degree $k$ in all $N$ variables:
\[
e_k = \sum_{1 \leq i_1 < i_2 < \cdots < i_k \leq N} \prod_{j=1}^k x_{i_j} 
\in \mathbb{Q}[x_1 , x_2 , x_3 , \dots , x_N ]^{\mathfrak{S}_N}=\Symf.
\]
\end{notation}

\def\BindingLineOrient
{
\begin{tikzpicture}[yscale=.7, scale=.75]
    \def \t{10}; 
    \def \r{2}; 
    \def \h{4}; 
    \def \a {(\t:\r cm)};
    \def \b {(\t+120: \r cm)};
    \def \c {(\t+240: \r cm)};
    \draw[double, darkgreen] (0,0) -- \a;
    \draw[darkgreen, dashed] (0,0) -- (\t+120: .5*\r cm);
    \draw[darkgreen] (\t+120: .5*\r cm) -- (\t+120: \r cm);
    \draw[darkgreen] (0,0) -- \c;
    \foreach \p in {\b,\c}{
        \draw[darkgreen] \p -- +(0,\h);
    }
    \foreach \p in {(0,0)}{
    \draw[thick, darkgreen] \p -- +(0,\h);
    }
    \foreach \p in {\a}{
        \draw[double, darkgreen] \p -- +(0,\h);
    }
    \begin{scope}[yshift=\h cm]
        \draw[double, darkgreen] (0,0) -- \a;
        \draw[darkgreen] (0,0) -- \b;
        \draw[darkgreen] (0,0) -- \c;
    \end{scope}
\end{tikzpicture}
}

\def\DefectLineOrient
{\begin{tikzpicture}
    \draw[double,darkgreen] (0,0) rectangle (2,1.5);
    \draw[red, thick] (1,0) -- (1,1.5);
\end{tikzpicture}
}

\subsection{\texorpdfstring{$\GL_N$}{GL_N} foam evaluation}\label{sl4foameevaluationchapter}

Let $\palette:= \{\colora,\colorb,\colorc,\dots, \colorn \}$ denote the set of \emph{pigments}, which are also referred to as \emph{colors}. These correspond to the distinct roots \(\vari\) of \(\GL_N\) (see \cite{RW-eval-foams}).

\begin{definition}
\label{def:coloring}
A \emph{coloring} of a closed foam $F$ is a map $c$ from the set of facets of $F$ to the set of subsets of $\palette$ such that the following conditions hold. 

\begin{itemize}
    \item Each $k$-facet is assigned $k$ distinct pigments. We say a facet $\Sigma$ is \emph{colored}
    with pigment $i$ if $i \in c(\Sigma)$.
    \item Let $\Sigma_a, \Sigma_b,$ and $\Sigma_{a+b}$ denote facets meeting at an $(a,b,a+b)$-binding, where $\Sigma_k$ is a $k$-facet. If $ \{i_1,i_2,\dots,i_a\} = c(\Sigma_a)$ and $\{j_1,j_2,\dots j_b\}= c(\Sigma_b)$, then $c(\Sigma_a)\cap c(\Sigma_b)=\emptyset $ and $c(\Sigma_{a+b}) =c(\Sigma_a)\sqcup c(\Sigma_b) $. 
\end{itemize}

\end{definition}

\begin{definition}
Given a coloring $c$ of a closed foam $F$, 
let $F_i(c)$ denote the \emph{monochrome surface} embedded in $F$ consisting of all facets colored with pigment $i$.

For two distinct pigments $i,j$, the \emph{bichrome surface} $F_{ij}(c)$ is defined to be the symmetric difference of the surfaces $F_i$ and $F_j$ union the seams. In other words, $F_{ij}(c)$ consists of facets colored with pigment $i$ or $j$ but not with both $i$ and $j$.

\end{definition}

\begin{definition}
Consider a foam $F$ decorated by dots. Color $F$ by coloring $c$. Consider a facet $f$ of $F$. If $f$ is decorated by a single dot labeled by $x^m$, and $f$ is colored by $i$, then define $P_f(c(f))=\vari^m$. If $f$ is a $k$-facet  decorated with $s(x_1, x_2, \ldots, x_k)$, and $f$ is colored by $\{i_1, i_2 \ldots, i_k\}$, then define $P_f(c(f))= s(\varu{i_1}, \varu{i_2}, \ldots, \varu{i_k})$. In general, take the product of these polynomials over all dots on the facet $f$ to define $P_f(c(f))$. Note that every \(\varu{i}\), \(1\le i\le N\), has degree \(2\) (see~\cite[Section 2.2]{RW-eval-foams}).
\end{definition}

\begin{definition}\label{foam eval definition}
If $(F,c)$ is a colored decorated $\GL_N$ foam, define
    \begin{align*}
    s(F,c) &= \sum_{i=1}^N {\cfrac{i}{2} \cdot \chi\left( {F_i(c)}\right)}  + \sum_{1\leq i < j \leq N} \theta^+_{ij}(F,c), \\
      P(F,c) &= \prod_{f \textrm{ facet of $F$}} P_f(c(f)),\\
      Q(F,c) &= \prod_{1\leq i < j \leq N} (\vari-\varj)^{\frac{\chi(F_{ij}(c))}{2}}, \text{ and } \\
      \langle F,c \rangle &= (-1)^{s(F,c)} \frac{P(F,c)}{Q(F,c)}.
    \end{align*}
    If $F$ is a decorated foam, we define the \emph{evaluation} of the foam $F$ by
    \begin{equation}
    \label{eq_def_eval}
    \langle F \rangle := \sum_{c \textrm{ coloring of $F$}} \langle F,c \rangle. 
    \end{equation}
\end{definition}

\begin{definition}[{\cite[Definition 2.1]{RW-eval-foams}}]
    Let $F$ be a $\GL_N$ foam and $c$ a coloring of $F$.
    The \emph{set of $i,j$ circles} $\theta_{ij}$ is the set of connected components of 
    $F_i(c) \cap F_j(c) \cap F_{ij}(c)$. We call an $i,j$ circle \emph{positive} if its local seam section from Figure \ref{figure of oriented seam} is colored by $i$ on the $a$-facet  and by $j$ on the $b$-facet in coloring $c$. 
    In other words, we assign a positive sign to a circle (or seam) if, using the left hand rule along the circle (seam), we observe the facets in the order $(a+b, a, b)$. 
    We denote the number of positive $i,j$ circles by $\theta_{ij}^+$.
\end{definition}

The foam evaluation of a closed foam results in a symmetric polynomial $\langle F \rangle \in \Symf (\varu{1}, \varu{2}, \ldots,\varu{k})$ \cite[Proposition 2.18]{RW-eval-foams}. To aid in long computations we will need the following lemma. 

\begin{lemma}
 [{\cite[Lemma 2.19]{RW-eval-foams}}] \label{rwlemma2.19}
    Let $F$ be a $\GL_N$ foam and let $c$ be a coloring of $F$. Suppose $c'$ is a coloring related to $c$ by a Kempe move \cite[Remark 2.8]{RW-eval-foams}. That is, there is a closed subsurface $S_c$ of the bichrome surface $F_{12}(c)$ such that if we exchange all instances of colors $1$ and $2$ on $S_c$ the color we obtain is $c'$. Then 
    \[s(F,c) \equiv s(F, c') + \chi(S_c) \mod 2.\]
\end{lemma}
Understanding how Kempe moves affect other colors besides $1,2$ follows from the proof of \cite[Lemma 2.19]{RW-eval-foams} and we refer the reader there for details. 

\subsection{\texorpdfstring{$\GL_N$}{GL_N} foams with boundary}\label{foam with boundary}

A closed foam comes with an embedding $\iota : F \longrightarrow  \mathbb{R}^3$. For $r \in \R$, let $T_r$ be the plane $\{(x,y,z) \ | \  z=r\}$. We say that $T_r$ intersects $\iota(F)$ \emph{generically} if the following conditions hold: 
\begin{itemize}
    \item  $\iota(F) \cap T_r $ is a closed $\GL_N$ web $\Gamma$ in $T_r$, where the orientations of the edges of $\Gamma$ are induced by the orientations of the facets of $\iota(F)$. 
    \item There exists $\epsilon>0$, such that $
   \iota(F) \cap \{(x,y,z) \ |\  x,y\in \mathbb{R}, \ z \in (r-\epsilon, r+\epsilon)\} $ is homeomorphic to $\Gamma\times(0,1)$.
\end{itemize}

In other words, $T_r$ intersects $\iota(F)$ generically if $T_r\cap \iota(F)$ does not contain any singular points of $\iota(F)$ and $T_r$ intersects the binding lines and the facets of $\iota(F)$ transversely.

\begin{definition}\label{def:foam with closed boundary}
Given an embedding of a closed foam $\iota : F \longrightarrow  \mathbb{R}^3, $
define a foam with boundary $(\Gamma_0,\Gamma_1)$ as the intersection of $\iota(F)$ with $\mathbb{R}^2 \times [0,1]= \{ (x,y,z)\ |\  x,y\in \mathbb{R}, z \in [0, 1] \}$ such that $T_0$ and $T_1$ intersect $\iota(F)$ generically and $T_i \cap \iota(F)= \Gamma_i$ are closed webs, for $i=0,1.$ 
\end{definition} 

We refer to a foam with boundary $(\Gamma_0, \Gamma_1)$ as a \emph{$(\Gamma_0,\Gamma_1)$-foam}. We will occasionally refer to foams with boundary as bordered foams. 
Later, when we use foams with corners, we will adopt the same notation of $(\Gamma_0,\Gamma_1)$-foam, see Definition \ref{def:foamfromto}. 

\begin{remark}
    A foam with boundary $(\Gamma_0,\Gamma_1)$ is thought of as a singular cobordism from $\Gamma_0$ to $\Gamma_1$. We choose the orientations of the edges of $\Gamma_1$ to agree with the edge orientations induced from the orientations of the facets of $\iota(F)$, and the orientations of the edges of $\Gamma_0$ to be opposite to the edge orientations induced by the orientations of the facets of $\iota(F)$.

    Consider a foam $U$ with boundary $(\Gamma_0,\Gamma_1)$ and a foam $V$ with boundary $(\Gamma_1,\Gamma_2)$. We define the composition of $V$ and $U$, denoted by $V\circ U$, as the foam with boundary $(\Gamma_0,\Gamma_2)$ obtained by stacking $V$ on top of $U$ and rescaling to fit in $\mathbb{R}^2\times [0,1]$.

\end{remark}

\begin{definition} [{\cite[Definition 2.3]{RW-eval-foams}}] \label{dfn:degreefoam}

The \emph{degree} of a closed $\GL_N$ foam $F$ is defined to be 
\[
    \foamdeg(F) := -\sum_{f\textrm{ facet}}d(f) + \sum_{\substack{i \textrm{ interval} \\ \textrm{binding}}}d(i) -\sum_{\substack{p \textrm{ singular}\\ \textrm{point}}} d(p) + \sum_{\substack{s \textrm{ labeling}\\ \textrm{of a dot}}} \text{\(\deg(s)\)},
\]
with terms defined as follows. 
\begin{itemize}
    \item For each facet $f$ with label $a$, set $d(f)=a(N-a)\chi(f)$, where $\chi(f)$ is the Euler characteristic of $f$.
    \item For each interval (i.e.\ not a circle) binding $i$ with adjacent facets labeled $(a, b, a+b)$, set $d(i)= ab + (a+b)(N-a -b)$.
    \item For each singular vertex $p$ with adjacent facets labeled $(a, b, c, a+b, b+c, a+b+c)$, set $d(p) = ab + bc+ cr + ra + ac+ br $,  where $r =N -a -b -c $.
    \item Finally, \(\deg(s)\) is the degree of the homogeneous symmetric polynomial \(s \in \Symf(X_1, \ldots, X_N)\).
\end{itemize}

\end{definition}

The same formula is used to define the degree of a $\GL_N$ foam with boundary; see the discussion in \cite[Section 4]{RW-eval-foams}. 

Alternatively, the degree of a foam can be computed by fixing a coloring and summing up the negatives of Euler characteristics of bichrome surfaces, and the contribution from decorations (see~\cite[Lemma 2.14]{RW-eval-foams}). As an explicit formula, for any choice of coloring $c$ of a foam $F$ we have 
 \begin{align}\label{eq_degree_euler}
     \foamdeg(F) &= -\sum_{1 \leq i < j \leq N} \chi(F_{ij}(c)) + \sum_{f \textrm{ facet of $F$}} \deg(P_f). 
 \end{align}
 (If a foam admits no coloring, it induces the 0 map on the state spaces of its boundary webs.) 
Comparing this to Definition \ref{foam eval definition}, we note that the factor of $\frac{1}{2}$ in the exponent of $Q$ is balanced by $\vari$ all having degree $2$. This is designed to make~\cite[Lemma 2.14]{RW-eval-foams} work.

This approach also generalizes to foams with boundary. In that case the foam must first be closed. The grading convention for bordered foams 
in \cite{RW-eval-foams} agrees with the definition of $\foamdeg$ from \cite[Lemma 3.4]{Queffelec2016}. 
Later, we will also work with cornered foams, whose degree will be computed by first closing the cornered foam (Definition \ref{def:foamfromto}) by $\idfoam_W$; the degree computation does not depend on the choice of web closure $W$.

\begin{definition}\label{def:foam between closed webs}
    Given a closed web $\Gamma$, the state space of $\Gamma$, denoted by $\mathcal{F}(\Gamma)$, 
    is a $\Symf$-module generated by foams with boundary $(\emptyweb,\Gamma)$ modulo the foam evaluation formula. 
 
 In other words, $\mathcal{F}(\Gamma)$ is a $\Symf$-module generated by symbols $\mathcal{F}(U)$ for all foams $U$ with boundary $(\emptyweb,\Gamma)$, where a relation $\sum_i a_i \mathcal{F}(U_i)=0$ holds in $\mathcal{F}(\Gamma)$ for $a_i \in \Symf$ and 
 $U_i$ 
 if and only if 
 \begin{equation*} 
   \sum_i a_i \langle V \circ U_i \rangle = 0 
 \end{equation*} 
 for any foam $V$ with boundary $(\Gamma,\emptyweb)$. Note that  
 $\langle  V \circ U_i \rangle \in \Symf$ is the evaluation of the closed foam $V\circ U_i.$ 

 The state space $\mathcal{F}(\Gamma)$ is naturally a graded $\Symf$-module. The degree of an element $\mathcal{F}(U)$ in $\mathcal{F}(\Gamma)$ is given by the degree of the foam $U$, and $\deg(a \mathcal{F}(U))=\deg(a)+\deg(\mathcal{F}(U))$. 

There is a direct sum decomposition of the graded abelian group  
 \[ \mathcal{F}(\Gamma) = \bigoplus_{k} \mathcal{F}(\Gamma)_k,\] where $\mathcal{F}(\Gamma)_k$ is the subspace of elements in $\mathcal{F}(\Gamma)$ of degree $k$, union with the $0$ element.  
 
 We denote by $\mathcal{F}(\Gamma)\{i\}$ the graded $\Symf$-module obtained by raising the grading of $\mathcal{F}(\Gamma)$ by $i$:
\[ \mathcal{F}(\Gamma)\{i\} = \bigoplus_{k} \mathcal{F}(\Gamma)\{i\}_k, \quad \mathcal{F}(\Gamma)\{i\}_k=  \mathcal{F}(\Gamma)_{k-i}.  \] 
\end{definition}

\begin{proposition}
    A foam $V$ with boundary $(\Gamma_0,\Gamma_1)$ induces a grading-preserving homomorphism 
    \begin{align*}
    \mathcal{F}(V): \mathcal{F}(\Gamma_0) &\longrightarrow 
    \mathcal{F}(\Gamma_1)\{-\foamdeg(V)\} \\
     \mathcal{F}(U) &\mapsto  \mathcal{F}(V\circ U).
    \end{align*} 

\end{proposition}

\begin{proof}
For any $(\emptyweb,\Gamma_0)$-foam $U$, $\mathcal{F}(U) \in \mathcal{F}(\Gamma_0)_{\foamdeg(U)}$ since $U$ has degree $\foamdeg(U)$. 
Since $V\circ U$ is a $(\emptyweb,\Gamma_1)$-foam with degree $\foamdeg(V) + \foamdeg(U)$,  
\[ \mathcal{F}(V\circ U)\in \mathcal{F}(\Gamma_1)_{\foamdeg(V)+\foamdeg(U)}
=\mathcal{F}(\Gamma_1)\{-\foamdeg(V)\}_{\foamdeg(U)}. \]
\end{proof}

\def\SquareClosedTop{
\begin{tikzpicture}[scale=0.7, decoration={
    markings,
    mark=at position 0.5 with {\arrow{>}}}, baseline={([yshift=-.8ex]current bounding box.center)}]

\def \fs{.4}; 
\def \sw {(-1+\fs,-1)};
\def \nw {(-1,1)};
\def \ne {(1,1)};
\def \se {(1-\fs,-1)};

    \directedtripleline{\se -- \ne};
    \directedsingleline{\sw -- \se};
    \directedsingleline{\sw -- \nw};
    \directedsingleline{\ne -- \nw};
    
    \directeddoubleline{\nw to[out=180,in=180, looseness=3] (-1+\fs,-1)};
    \directeddoubleline{(1,1) to[out=0,in=0, looseness=3] (1-\fs,-1) };

\end{tikzpicture}
}

\def\SquareClosedBottom{
\begin{tikzpicture}[scale=0.7, decoration={
    markings,
    mark=at position 0.5 with {\arrow{>}}}, baseline={([yshift=-.8ex]current bounding box.center)}]

\def \fs{.4}; 
\def \sw {(-1+\fs,-1)};
\def \nw {(-1,1)};
\def \ne {(1,1)};
\def \se {(1-\fs,-1)};


\directedtripleline{\sw -- \nw};
\directedsingleline{\se -- \sw};
\directedsingleline{\se -- \ne};
\directedsingleline{\nw -- \ne};

\directeddoubleline{\nw to[out=180,in=180, looseness=3] (-1+\fs,-1)};
\directeddoubleline{(1,1) to[out=0,in=0, looseness=3] (1-\fs,-1) };

\end{tikzpicture}
}

\def\SquareFoamClosed{
\begin{tikzpicture}[scale=0.7,decoration={
    markings,
    mark=at position 0.5 with {\arrow{>}}}, baseline={([yshift=-.8ex]current bounding box.center)}]

\def \fs{.4}; 
\def \sw {(-1+\fs,-1)};
\def \nw {(-1,1)};
\def \ne {(1,1)};
\def \se {(1-\fs,-1)};


\directedtripleline{\sw -- \nw};
\directedsingleline{\se -- \sw};
\directedsingleline{\se -- \ne};
\directedsingleline{\nw -- \ne};

\directeddoubleline{\nw to[out=180,in=180, looseness=3] (-1+\fs,-1)};
\directeddoubleline{(1,1) to[out=0,in=0, looseness=3] (1-\fs,-1) };

\def \h{5};

\draw[darkgreen] \nw -- (-1,1+\h);
\draw[darkgreen] \ne -- (1,1+\h);
\draw[darkgreen] \sw -- (-1+\fs,-1+\h);
\draw[darkgreen] \se -- (1-\fs,-1+\h);
\draw[darkgreen] (-2.62,0.1) -- (-2.62, 0.1+\h);
\draw[darkgreen] (2.62,0.1) -- (2.62, 0.1+\h);

\begin{scope}[yshift=.5*\h cm]
    \draw[darkgreen] \nw -- \ne -- \se -- \sw -- \nw;
\end{scope}

\begin{scope}[yshift=.5*\h cm]
    \node[scale=0.7] at (-2,0) {\color{red} $2$};
    \node[scale=0.7] at (2,.0) {\color{red} $2$};
    \node[scale=0.7] at (0,0) {\color{red} $2$};
\end{scope}
\begin{scope}[yshift=.25*\h cm]
    \node[scale=0.7] at (-.8,0) {\color{red} $3$};
    \node[scale=0.7] at (.8,.0) {\color{red} $1$};
    \node[scale=0.7] at (0,0) {\color{red!40!white} $(1)$};
\end{scope}
\node[scale=0.7] at (0,0) {\color{red} $1$};
\begin{scope}[yshift=.75*\h cm]
    \node[scale=0.7] at (-.8,0) {\color{red} $1$};
    \node[scale=0.7] at (.8,.0) {\color{red} $3$};
    \node[scale=0.7] at (0,1) {\color{red} $1$};
\end{scope}


\begin{scope}[yshift= \h cm]
    \directedtripleline{\se -- \ne};
    \directedsingleline{\sw -- \se};
    \directedsingleline{\sw -- \nw};
    \directedsingleline{\ne -- \nw};
    
    \directeddoubleline{\nw to[out=180,in=180, looseness=3] (-1+\fs,-1)};
    \directeddoubleline{(1,1) to[out=0,in=0, looseness=3] (1-\fs,-1) };

\end{scope}

\end{tikzpicture}
}

\def\SquareFoamState{
\begin{tikzpicture}[scale=0.7,decoration={
    markings,
    mark=at position 0.5 with {\arrow{<}}}, baseline={([yshift=-.8ex]current bounding box.center)}]


\def \fs{.2}; 
\def \sw {(-1+\fs,-1)};
\def \nw {(-1,1)};
\def \ne {(1,1)};
\def \se {(1-\fs,-1)};

\directedtripleline{\nw -- \sw};
\directedsingleline{\sw -- \se};
\directedsingleline{\ne -- \se};
\directedsingleline{\ne -- \nw};

\directeddoubleline{\sw to[out=180,in=180, looseness=3] (-1,1)};
\directeddoubleline{\se to[out=0,in=0, looseness=3] (1,1)};

\draw[darkgreen] (-2.69,0) to[out=270, in=270, looseness=2] (2.69,0);

\def \vx{2.135}; 
\draw[darkgreen] \sw to[out=260, in=-10] (-1*\vx,-2);
\draw[dashed, darkgreen] \nw to[out=260, in=40] (-1*\vx,-2);
\begin{scope}[xscale=-1]
    \draw[darkgreen] \sw to[out=260, in=-10] (-1*\vx,-2);
    \draw[dashed, darkgreen] \nw to[out=260, in=40] (-1*\vx,-2);
\end{scope}

\node[scale=0.7] at (0,-2) {\color{red} $1$};
\node[scale=0.7] at (-2,0) {\color{red} $2$};
\node[scale=0.7] at (2,0) {\color{red} $2$};
\node[scale=0.7] at (-1.3,-1.3) {\color{red} $3$};
\node[scale=0.7] at (1.3,-1.3) {\color{red} $1$};

\end{tikzpicture}
}

\def\SquareFoamStateCompo{
\begin{tikzpicture}[scale=0.7, decoration={
    markings,
    mark=at position 0.5 with {\arrow{>}}}, baseline={([yshift=-.8ex]current bounding box.center)}]

\def \fs{.4}; 
\def \sw {(-1+\fs,-1)};
\def \nw {(-1,1)};
\def \ne {(1,1)};
\def \se {(1-\fs,-1)};


\directedtripleline{\sw -- \nw};
\directedsingleline{\se -- \sw};
\directedsingleline{\se -- \ne};
\directedsingleline{\nw -- \ne};

\directeddoubleline{\nw to[out=180,in=180, looseness=3] (-1+\fs,-1)};
\directeddoubleline{(1,1) to[out=0,in=0, looseness=3] (1-\fs,-1) };

\def \h{5};

\draw[darkgreen] \nw -- (-1,1+\h);
\draw[darkgreen] \ne -- (1,1+\h);
\draw[darkgreen] \sw -- (-1+\fs,-1+\h);
\draw[darkgreen] \se -- (1-\fs,-1+\h);
\draw[darkgreen] (-2.62,0.1) -- (-2.62, 0.1+\h);
\draw[darkgreen] (2.62,0.1) -- (2.62, 0.1+\h);

\begin{scope}[yshift=.5*\h cm]
    \draw[darkgreen] \nw -- \ne -- \se -- \sw -- \nw;
\end{scope}

\begin{scope}[yshift=.5*\h cm]
    \node[scale=0.7] at (-2,0) {\color{red} $2$};
    \node[scale=0.7] at (2,.0) {\color{red} $2$};
    \node[scale=0.7] at (0,0) {\color{red} $2$};
\end{scope}
\begin{scope}[yshift=.25*\h cm]
    \node[scale=0.7] at (-.8,0) {\color{red} $3$};
    \node[scale=0.7] at (.8,.0) {\color{red} $1$};
    \node[scale=0.7] at (0,0) {\color{red!40!white} $(1)$};
\end{scope}
\node[scale=0.7] at (0,0) {\color{red} $1$};
\begin{scope}[yshift=.75*\h cm]
    \node[scale=0.7] at (-.8,0) {\color{red} $1$};
    \node[scale=0.7] at (.8,.0) {\color{red} $3$};
    \node[scale=0.7] at (0,1) {\color{red} $1$};
\end{scope}


\begin{scope}[yshift= \h cm]
    \directedtripleline{\se -- \ne};
    \directedsingleline{\sw -- \se};
    \directedsingleline{\sw -- \nw};
    \directedsingleline{\ne -- \nw};
    
    \directeddoubleline{\nw to[out=180,in=180, looseness=3] (-1+\fs,-1)};
    \directeddoubleline{(1,1) to[out=0,in=0, looseness=3] (1-\fs,-1) };

\end{scope}


\begin{scope}[yshift=-\h, decoration={
    markings,
    mark=at position 0.5 with {\arrow{<}}}]

\def \dx{.07}; 
\def \dy{.25}; 
\draw[darkgreen] (-2.69+\dx,\dy) to[out=270, in=270, looseness=2] (2.69-\dx,\dy);

\def \vx{1.85}; 
\draw[darkgreen] (-1+\fs,-.85) to[out=260, in=-10] (-1*\vx,-2);
\draw[dashed, darkgreen] \nw to[out=260, in=40] (-1*\vx,-2);
\begin{scope}[xscale=-1]
    \draw[darkgreen] (-1+\fs,-.85) to[out=260, in=-10] (-1*\vx,-2);
    \draw[dashed, darkgreen] \nw to[out=260, in=40] (-1*\vx,-2);
\end{scope}

\node[scale=0.7] at (0,-2) {\color{red} $1$};
\node[scale=0.7] at (-2,0) {\color{red} $2$};
\node[scale=0.7] at (2,0) {\color{red} $2$};
\node[scale=0.7] at (-1.05,-1.3) {\color{red} $3$};
\node[scale=0.7] at (1.05,-1.3) {\color{red} $1$};

\end{scope}

\end{tikzpicture}
}

\begin{example}\label{eg:closed web foam}
Consider closed webs 
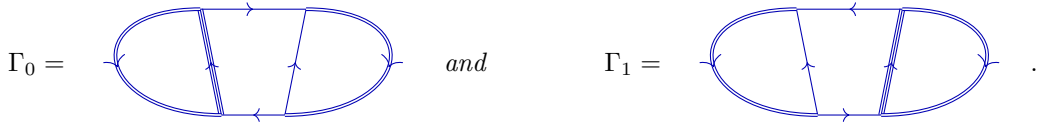
\begin{align*}
    \Gamma_0 =\SquareClosedBottom \text{ and } & & \Gamma_1 =\SquareClosedTop.
\end{align*}
If $V$ is a foam with boundary $(\Gamma_0,\Gamma_1)$ and $U$ is a foam with boundary $(\emptyweb,\Gamma_0)$. Then $V\circ U$ is a foam with boundary $(\emptyweb,\Gamma_1)$. See Figure \ref{figure:composition of foam}.\\
\begin{figure}[htbp]
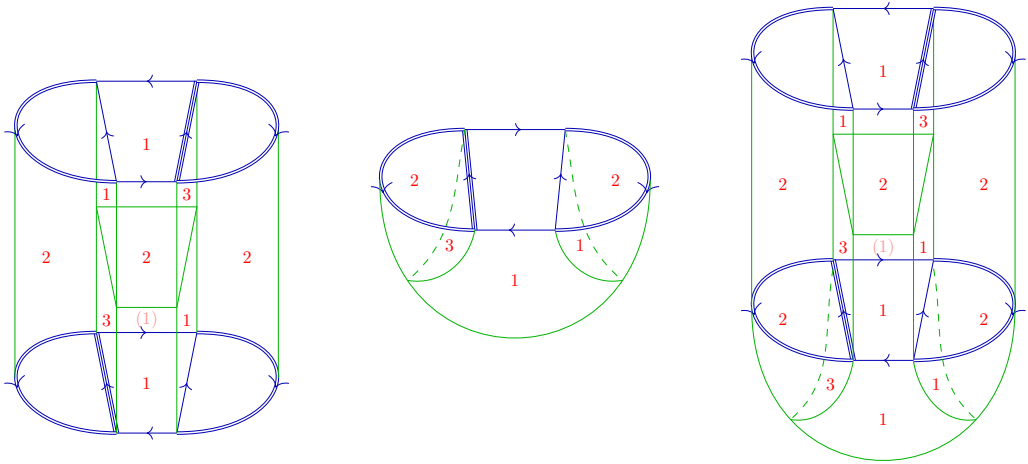

    \centering
\begin{align*}
  \scalebox{0.95} {$\displaystyle
   \SquareFoamClosed  \quad \SquareFoamState  \quad \SquareFoamStateCompo  $}
\end{align*}
    \caption{Depicted from left to right are foams $V$, $U$, and $V \circ U$.}
    \label{figure:composition of foam}
\end{figure}
\end{example}

Now we consider foams between $\GL_N$ webs with a fixed boundary instead of foams between closed $\GL_N$ webs.

\begin{definition}\label{def:foamfromto}
Consider two webs $W_0$ and $W_1$ with the same boundary $\varepsilon$. 
Define a foam $F$ from $W_0$ to $W_1$ as a foam $F_\cup$
with boundary $(\emptyweb, \overline{W_0} W_1  )$ cut along a tubular neighborhood of $\varepsilon$ relative to $F_\cup$. We will consider the boundary of $F$ in a natural way as
\[
\partial F = (\overline{W_0}\times \{0\}) \cup (W_1\times \{1\}) \cup \left( \bigsqcup_{p\in  \varepsilon} p\times [0,1] \right) . 
\]
The orientation of $W_1$ and $\varepsilon\times[0,1]$ agrees with the orientation of the facets of $F$, and the orientation of $W_0$ is opposite to the orientation of the facets of $F$. A coloring of $F$ is a coloring of $F_\cup$ that respects the cutting we performed.

When $\varepsilon \neq \emptylist$, we call $F$ a \emph{foam with corners} or a \emph{cornered foam}.
\end{definition}
We will occasionally denote a foam from $W_0$ to $W_1$ as $(W_0,W_1)$-foam, just as in the case of boarded foams.

The following example illustrates how a foam with boundary $(\emptyweb, \overline{W_0} W_1  )$ corresponds to a foam with corners from $W_0$ to $W_1$. This will serve as a running example through the rest of this section.

\begin{example}\label{eg:foam with corners}
\def\SqOpenTop{
\begin{tikzpicture}[ scale=0.5, decoration={
    markings,
    mark=at position 0.5 with {\arrow{>}}},   
    baseline={([yshift=-.8ex]current bounding box.center)} ] 

\def \fs{.4}; 
\def \sw {(-1+\fs,-1)};
\def \nw {(-1,1)};
\def \ne {(1,1)};
\def \se {(1-\fs,-1)};

\def \topdx{1.5};
\def \botdx{1};

    \directedtripleline{\se -- \ne};
    \directedsingleline{\sw -- \se};
    \directedsingleline{\sw-- \nw};
    \directedsingleline{\ne -- \nw};
    
    \directeddoubleline{\nw  -- (-1-\topdx,1)};
    \directeddoubleline{(-1-\botdx,-1) -- \sw };
    \begin{scope}[xscale=-1]
        \directeddoubleline{\nw  -- (-1-\topdx,1)};
        \directeddoubleline{(-1-\botdx,-1) -- \sw};
    \end{scope}
\end{tikzpicture}
}

\def\SqOpenBottom{
\begin{tikzpicture}[ scale=0.5, decoration={
    markings,
    mark=at position 0.5 with {\arrow{>}}},   
    baseline={([yshift=-.8ex]current bounding box.center)} ] 

\def \fs{.4}; 
\def \sw {(-1+\fs,-1)};
\def \nw {(-1,1)};
\def \ne {(1,1)};
\def \se {(1-\fs,-1)};


\directedtripleline{\sw -- \nw};
\directedsingleline{\se -- \sw};
\directedsingleline{\se -- \ne};
\directedsingleline{\nw -- \ne};

\def \topdx{1.5};
\def \botdx{1};
\directeddoubleline{\nw  -- (-1-\topdx,1)};
\directeddoubleline{(-1-\botdx,-1) -- \sw};
\begin{scope}[xscale=-1]
    \directeddoubleline{\nw  -- (-1-\topdx,1)};
    \directeddoubleline{(-1-\botdx,-1) -- \sw};
\end{scope}
\end{tikzpicture}
}

\def\SquareFoamOpen{
\begin{tikzpicture}[scale=0.75, decoration={
    markings,
    mark=at position 0.5 with {\arrow{>}}},   
    baseline={([yshift=-.8ex]current bounding box.center)} ] 

\def \fs{.4}; 
\def \sw {(-1+\fs,-1)};
\def \nw {(-1,1)};
\def \ne {(1,1)};
\def \se {(1-\fs,-1)};


\directedtripleline{\sw -- \nw};
\directedsingleline{\se -- \sw};
\directedsingleline{\se -- \ne};
\directedsingleline{\nw -- \ne};

\def \topdx{1.5};
\def \botdx{1};
\directeddoubleline{\nw  -- (-1-\topdx,1)};
\directeddoubleline{(-1-\botdx,-1) -- \sw};
\begin{scope}[xscale=-1]
    \directeddoubleline{\nw  -- (-1-\topdx,1)};
    \directeddoubleline{(-1-\botdx,-1) -- \sw};
\end{scope}

\def \h{5};

\draw[darkgreen] \nw -- (-1,1+\h);
\draw[darkgreen] \ne -- (1,1+\h);
\draw[darkgreen] \sw -- (-1+\fs,-1+\h);
\draw[darkgreen] \se -- (1-\fs,-1+\h);
\draw[darkgreen] (-1-\topdx,1) -- (-1-\topdx,1+\h);
\draw[darkgreen] (-1-\botdx,-1) -- (-1-\botdx,-1+\h);
\begin{scope}[xscale=-1]
    \draw[darkgreen] (-1-\topdx,1) -- (-1-\topdx,1+\h);
    \draw[darkgreen] (-1-\botdx,-1) -- (-1-\botdx,-1+\h);
\end{scope}

\begin{scope}[yshift=.5*\h cm]
    \draw[darkgreen] \nw -- \ne -- \se -- \sw -- \nw;
\end{scope}

\begin{scope}[yshift=.5*\h cm]
    \node[scale=0.7] at (0,0) {\color{red} $2$};
\end{scope}
\begin{scope}[yshift=.25*\h cm]
    \node[scale=0.7] at (-.8,0) {\color{red} $3$};
    \node[scale=0.7] at (.8,.0) {\color{red} $1$};
    \node[scale=0.7] at (0,0) {\color{red!40!white} $(1)$};
\end{scope}
\node[scale=0.7] at (0,0) {\color{red} $1$};
\node[scale=0.7] at (-1.25,0) {\color{red} $2$};
\node[scale=0.7] at (1.25,0) {\color{red} $2$};
\begin{scope}[yshift=.75*\h cm]
    \node[scale=0.7] at (-.8,0) {\color{red} $1$};
    \node[scale=0.7] at (.8,.0) {\color{red} $3$};
    \node[scale=0.7] at (0,1) {\color{red} $1$};
\end{scope}
\begin{scope}[yshift=\h cm]
    \node[scale=0.7] at (-1.75,0) {\color{red} $2$};
    \node[scale=0.7] at (1.75,0) {\color{red} $2$};
\end{scope}


\begin{scope}[yshift= \h cm]
    \directedtripleline{\se -- \ne};
    \directedsingleline{\sw -- \se};
    \directedsingleline{\sw-- \nw};
    \directedsingleline{\ne -- \nw};
    
    \directeddoubleline{\nw  -- (-1-\topdx,1)};
    \directeddoubleline{(-1-\botdx,-1) -- \sw };
    \begin{scope}[xscale=-1]
        \directeddoubleline{\nw  -- (-1-\topdx,1)};
        \directeddoubleline{(-1-\botdx,-1) -- \sw};
    \end{scope}

\end{scope}

\end{tikzpicture}
}

\def\SquareFoamBent{
\begin{tikzpicture}[decoration={
    markings,
    mark=at position 0.5 with {\arrow{>}}},  
    baseline={([yshift=-.8ex]current bounding box.center)}]

\def \s{.6}; 
\def \wsw{(0,0)};
\def \wnw{(0+\s, 1)};
\def \wse{(1,0)};
\def \wne{(1+\s, 1)};
\def \esw{(2,0)};
\def \enw{(2+\s, 1)};
\def \ese{(3,0)};
\def \ene{(3+\s, 1)};


\def \cl{2.5}; 
\draw[darkgreen] (-.41,-.27) to[out=270,in=270,looseness=\cl] ++(3.84,0);
\draw[darkgreen, densely dashed] (.19,1.27) to[out=270,in=270,looseness=\cl] ++(3.84,0);
\draw[darkgreen] \wnw to[out=270,in=270,looseness=\cl] ++(3,0);
\draw[darkgreen] \wsw to[out=270,in=270,looseness=\cl] ++(3,0);
\draw[darkgreen] \wne to[out=270,in=270,looseness=\cl] ++(1,0);
\draw[darkgreen] \wse to[out=270,in=270,looseness=\cl] ++(1,0);

\def \a{(1.5,-.73)};
\def \b{(1.5+\s, .27)};
\def \c{(1.5,-.73-1.47)};
\def \d{(1.5+\s, .27-1.47)};
\draw[darkgreen] \a -- \b -- \d -- \c -- \a;

\begin{scope}[xshift=-.15cm, yshift=-.55cm]
    \draw[darkgreen] \enw -- \esw;
\end{scope}
\begin{scope}[xshift=-.4cm, yshift=-1.55cm]
    \draw[darkgreen] \ene -- \ese;
\end{scope}

\directedsingleline{\wsw -- \wnw};
\directedtripleline{\wse -- \wne};
\directedsingleline{\enw -- \esw};
\directedtripleline{\ene -- \ese};
\directeddoubleline{\wne -- \enw};
\directeddoubleline{\esw -- \wse};
\directedsingleline{\wne -- \wnw};
\directedsingleline{\wsw -- \wse};
\directedsingleline{\enw -- \ene};
\directedsingleline{\ese -- \esw};

\def \l{2}; 
\def \t{20}; 
\directeddoubleline{\wnw to[out=180-\t,in=\t,looseness=\l] (3+\s, 1)};
\directeddoubleline{\ese to[out=-\t,in=180+\t,looseness=\l] (0, 0)};

\node[scale=0.7] at (1.5+\s, 1.3) {\color{red} $2$};
\node[scale=0.7] at (1.5+\s, .7) {\color{red} $2$};
\node[scale=0.7] at (1.5, -2.7) {\color{red} $2$};
\node[scale=0.7] at (1.5, -.3) {\color{red} $2$};
\node[scale=0.7] at (1.5+.5*\s, -1) {\color{red} $2$};
\node[scale=0.7] at (1.5,.25) {\color{red} $3$};
\node[scale=0.7] at (3.35,.25) {\color{red} $3$};

\end{tikzpicture}
}

Consider the following two webs with boundary ${2^-}{2^+}{2^+}{2^-}$:
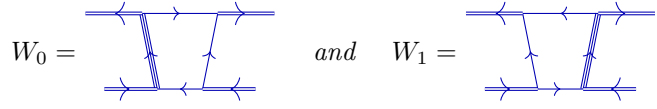
\begin{align*}
    W_0=\SqOpenBottom \quad \text{ and } \quad  W_1=\SqOpenTop
\end{align*} 
A foam $F$ from $W_0$ to $W_1$ is obtained from a foam $F_\cup$ with boundary $(\emptyweb, \overline{W_0} W_1  )$. See Figure \ref{figureWoW1overline} for an example. 
Colloquially, we say that $F_\cup$ is obtained by \emph{bending up} the foam $F$. 

\begin{figure}[htbp]
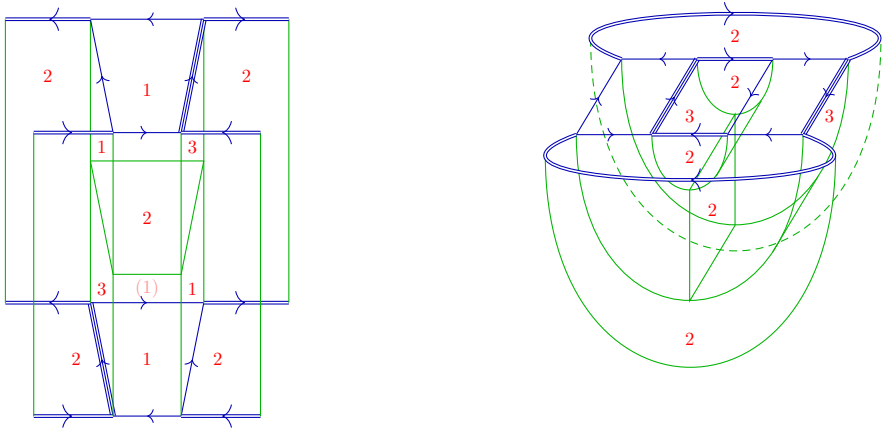

    \centering
\begin{align*}
    \scalebox{1}{$
   \SquareFoamOpen
$} &&\scalebox{1}{$
   \SquareFoamBent
$}
\end{align*}
    \caption{Depicted on the left is a foam $F$ from $W_0$ to $W_1$. On the right the foam $F$ is viewed as a foam $F_\cup$ with boundary $(\emptyweb, \overline{W_0} W_1 )$.}
    \label{figureWoW1overline}
\end{figure}
\end{example}

There are various ways to concatenate cornered foams (with suitably compatible boundary).
We present two of them in the next definition, also allowing linear combinations of foams. 

\begin{definition} \label{Foam-Compo-Conca}
Given webs $W_0, W_1,$ and $W_2$ with the same boundary $\varepsilon$, consider a foam $M_0$ from $W_0$ to $W_1$ and a foam $M_1$ from $W_1$ to $W_2$. Define the composition of $M_1$ with $M_0$, denoted by $M_1 \circ M_0$, as the foam from $W_0$ to $W_2$ obtained by stacking $M_1$ on top of $M_0$. Consider a finite $\Symf$-linear combination $f_0= \sum_i u_i \cdot \leftindex^i M_0 $ of foams $\leftindex^i M_0$ from $W_0$ to $W_1$ and a finite $\Symf$-linear combination $f_1=\sum_{j} v_j \cdot \leftindex^j M_1$ of foams $\leftindex^j M_1$ from $W_1$ to $W_2$, define the composition of $f_1$ with $f_0$ as:  
\[ f_1 \circ f_0= \sum_{i,j} v_j u_i \cdot \leftindex^j M_1 \circ  \leftindex^i M_0. \]

Given webs $W_0, W_1, W_0',$ and $W_1'$ with the same boundary $\varepsilon$, consider a foam $M_0$ from $W_0$ to $W_1$ and a foam $M_0'$ from $W_0'$ to $W_1'$. Denote by $\overline{M_0} \ast M_0'$ the foam with boundary $(\overline{W_0}W_0',\overline{W_1}W_1')$, obtained by reversing the orientations of all the facets of $M_0$ and then gluing the boundary facets of the foams together along $\varepsilon \times [0,1]$. Consider a finite $\Symf$-linear combination $f_0= \sum_i u_i \cdot \leftindex^i M_0$ of foams $\leftindex^i M_0$ from $W_0$ to $W_1$ and a finite $\Symf$-linear combination ${f_0}'=\sum_{j} v_j \cdot \leftindex^j M_0'$ of foams $\leftindex^j M_0'$ from $W_0'$ to $W_1'$, define $\overline{f_0} \ast f_0'$ as:  
\[ \overline{f_0} \ast f_0'  = \sum_{i,j} u_i v_j  \cdot \overline{\leftindex^i M_0} \ast \leftindex^j M_0'. \] 
\end{definition}

\def\DoubleID{
\begin{tikzpicture}[ scale=0.2, decoration={
    markings,
    mark=at position 0.5 with {\arrow{>}}},   
    baseline={([yshift=-.8ex]current bounding box.center)} ] 
\directeddoubleline{(-2,-2)..controls(-1,-1)and(-1,1)..(-2,2) };
\directeddoubleline{(2,-2)..controls(1,-1)and(1,1)..(2,2)} ;
\end{tikzpicture}
}

\begin{example}\label{eg:different foam clousure}
    Recall the webs and foams from Example \ref{eg:foam with corners} and Example \ref{eg:closed web foam}. Denote by $W_2$ the web $\DoubleID$ with boundary $2^-2^+2^-2^+$. Using the notation from the aforementioned Examples, we notice that $\Gamma_0= \overline{W_2} W_0 $, $\Gamma_1= \overline{W_2} W_1 $. Moreover, we can see that $V=\overline{\idfoam_{W_2}} \ast F $.
\end{example}

We will now adapt the discussion following Definition \ref{dfn:degreefoam} to cornered foams. In particular are going to define the degree of a foam with corners. The original definition is due to \cite[Definition 3.3]{Queffelec2016}.

\begin{lemma}
\label{lem:euler-char-boundary-conditions}
    Let $\varepsilon = a_1^{\epsilon_1}a_2^{\epsilon_2}\cdots a_k^{\epsilon_k} \in \Seq_{N}$. 
    Let $W$ be a web with boundary $\varepsilon$ and choose any coloring $c$ for the identity foam $F:= \idfoam_W$. Then 
    \[
        \sum_{1 \leq i < j \leq N} \chi(F_{ij}) 
        = \frac{1}{2} \sum_{l =1}^k a_l (N - a_l).
    \]
    In particular, this quantity depends only on the boundary sequence $\varepsilon$. 
    \begin{proof}
        Each $F_{ij}$ is of the form $A_{ij} \times I$ where $A_{ij}$ is a 1-manifold properly embedded in the disk with $k$ marked points $\{p_1,\ldots, p_k\} = \partial W$; $\partial A_{ij}$ is a subset of these points. 
        For $1 \leq l \leq k$, let $P_l$ denote the set of $a_l$ distinct pigments on the facet of $F$ abutting $p_l \times I$. 
        
        The 1-manifold $A_{ij}$ consists of arcs and closed circles. 
        Since $\chi(S^1 \times I) = 0$, only the arcs contribute to $\chi(F_{ij}) = \chi(A_{ij})$. 
        Let $C_l$ denote the set of pigments 
        Now $p_l \in \partial A_{ij}$ if and only if one of the following is true: 
            \begin{itemize}
                \item  $i \in P_{l}$ and $j \not\in P_{l}$, or
                \item $i \not\in P_l$ and $j \in P_l$.
            \end{itemize}
        Thus 
        \[
            \sum_{1 \leq i < j \leq N} |\partial A_{ij}| = \sum_{l =1}^k a_l (N - a_l).
        \]
        The claim follows from the fact that each arc has Euler characteristic 1, and has two endpoints.
    \end{proof}
\end{lemma}

Lemma \ref{lem:euler-char-boundary-conditions} justifies the following definition and notation. 

\begin{definition}\label{def:deg of web with boundary}
For an admissible boundary sequence $\varepsilon =  a_1^{\epsilon_1}a_2^{\epsilon_2}\cdots a_k^{\epsilon_k} \in \Seq_N$, define
\[
    \boundarydeg(\varepsilon) := \frac{1}{2} \sum_{l=1}^k a_l (N - a_l).
\]
\end{definition}

This quantity appears in \cite[Section 2]{Robert2015} and corresponds to the dimension of complex Grassmannians.
It will serve as a correction term in the degree of a foam with corners. We first record the degree of the closed-up foam $F_\cup$.

\begin{remark}
\label{rem:bending-up-foamdeg}
Let $W, W'$ be webs with boundary $\varepsilon$, and let $F$ be a $(W,W')$-foam.
Let $F_\cup$ be the foam with boundary in $\Hom_{\Foam}(\emptyweb, \overline{W} W')$ 
associated to $F$ (see Definition \ref{def:foamfromto}).
Let $c$ be a coloring of $F$.
Lemma \ref{lem:euler-char-boundary-conditions} shows that
\[
    \foamdeg(F_\cup)= -\sum_{1 \leq i < j \leq N} \chi(F_{ij}(c)) + \sum_{f \textrm{ facet of $F$}} \deg(P_f).
\]
\end{remark}

Using $\foamdeg(F_\cup)$ directly as the degree of a cornered foam would not assign degree $0$ to identity foams; the correction term $\boundarydeg(\varepsilon)$ remedies this. Accordingly, we define the degree of a cornered foam $F$ by horizontally composing $F$ (using $*$) with any compatible identity foam. The following lemma shows that this does not depend on the choice of identity foam.

\begin{lemma} \label{lem:foam_degree_with corners}
    Let $W_0, W_1, W_2,$ and $ W_3$ be webs with the same boundary $\varepsilon$. Let $F$ be a foam from $W_0$ to $W_1$. Then
    \[
   \foamdeg(\overline{\idfoam_{W_2}} \ast F) =\foamdeg(\overline{\idfoam_{W_3}} \ast F)
    \]
\end{lemma}
\begin{proof}
    This lemma is a direct consequence of \cite[Lemma 3.4]{Rose_2016}.
\end{proof}

\begin{definition}\label{def:degree foam with corners}
    Let $W_0, W_1$ and $ W_2$ be webs with the same boundary $\varepsilon$. Let $F$ be a foam from $W_0$ to $W_1$. Define
    \begin{align*}
        \foamdeg(F) := \foamdeg(\overline{\idfoam_{W_2}}\ast F ).
    \end{align*}
\end{definition}

{This definition arranges so that the identity foam on any web $W$ with boundary $\varepsilon$ has degree $0$.}

The following corollary follows from \cite[Section 3]{Queffelec2016}, and provides an alternative formula:
\begin{corollary}\label{cor:degree foam with corner alt}
     Let $W_0, W_1$ be webs with boundary $\varepsilon$. Let $F$ be a foam from $W_0$ to $W_1$. Then
    \begin{align*}
        \foamdeg(F) =-\sum_{1 \leq i < j \leq N} \chi(F_{ij}(c)) + \sum_{f \textrm{ facet of $F$}} \deg(P_f) + \boundarydeg(\varepsilon).
    \end{align*}
\end{corollary}
\begin{proof}
    This follows by the definition of Euler characteristics, Remark \ref{rem:bending-up-foamdeg}, and Definition \ref{def:foamfromto}.
\end{proof}

{Equivalently, the degree of $F$ and the degree of its closure are related by $\foamdeg(F_\cup) + \boundarydeg(\varepsilon) = \foamdeg(F)$ as the following example demonstrates.}

\begin{example}
    Given the foams in Example \ref{eg:foam with corners}, \ref{eg:closed web foam}, and \ref{eg:different foam clousure}, we have
    \begin{align*}
        \foamdeg(F) &= \foamdeg(\overline{\idfoam_{W_2}} \ast F )=\foamdeg V = \foamdeg(F_\cup)+\foamdeg(2^-2^+2^+2^-) \\
        &= -4N+4 + 4(N-2) \\
        &= -4.
    \end{align*}
\end{example}

The following category should be compared with the formalization found in the $2$-categories from \cite[Definition 3.1]{Queffelec2016}. 

\begin{definition}
Fix an admissible boundary sequence $\varepsilon \in \Seq_N$. The \emph{$\mathfrak{gl}_N$ foam category with boundary $\varepsilon$}, denoted $\Foam^\varepsilon$, is defined as follows. Its objects are the webs with boundary $\varepsilon$. Given two such webs $W_0$ and $W_1$, a morphism from $W_0$ to $W_1$ is a finite formal $\Symf$-linear combination of foams from $W_0$ to $W_1$,
\[
    f=\sum_i k_i M_i,
\]
where each $k_i\in \Symf$ and each $M_i$ is a foam from $W_0$ to $W_1$. These morphism spaces are taken modulo the following evaluation relation: For foams $M_i$ from $W_0$ to $W_1$ and coefficients $a_i\in \Symf$, we set $\sum_i a_i M_i =0$ if and only if $\sum_i a_i \langle V M_i\rangle =0$ for every foam $V$ with boundary $(\overline{W_0}W_1,\emptyweb)$, where $\langle VM_i\rangle\in \Symf$ denotes the evaluation of the closed foam $VM_i$.
\end{definition}

The category $\Foam^\varepsilon$ is graded. The degree of an individual term $k_iM_i$ is
\[
    \deg(k_iM_i) := \deg(k_i)+\foamdeg(M_i),
\]
and for each $d\in \mathbb{Z}$ the homogeneous degree $d$ morphism space $\Hom_{\Foam^\varepsilon}^d(W_0,W_1)$ is spanned by those finite sums $\sum_i k_iM_i$ in which every nonzero term has total degree $d$. Thus
\[
    \Hom_{\Foam^\varepsilon}(W_0,W_1)
    =
    \bigoplus_{d\in \mathbb{Z}}
    \Hom_{\Foam^\varepsilon}^d(W_0,W_1)
\]
is a graded $\Symf$-module.

Note that in comparison to Definition \ref{def:foam between closed webs}, we omit the analogue of $\mathcal{F}$ in order to make the text in the next sections cleaner. 
We will use $\mathcal{E}$ for elements in the Karoubi envelope, which we define now:

\begin{definition}\label{def:Karoubi of cornered foams}
Let $\operatorname{Kar}(\Foam^\varepsilon)$ denote the \emph{Karoubi completion} (or \emph{Karoubi envelope}) of $\Foam^\varepsilon$. An object of $\operatorname{Kar}(\Foam^\varepsilon)$ is a pair $(W,P)$, where $W$ is a web with boundary $\varepsilon$ and
\[
    P\in \Hom^0_{\Foam^\varepsilon}(W, W)
\]
is a homogeneous degree-zero idempotent  (so that \(P\circ P=P\). We denote this object by $\mathcal{E}(P)$. In particular, if $P=\idfoam_W$, we write
\[
    \mathcal{E}(W):=\mathcal{E}(\idfoam_W).
\]

A morphism
\[
    f\colon \mathcal{E}(P_0)=(W_0,P_0)
    \longrightarrow
    \mathcal{E}(P_1)=(W_1,P_1)
\]
is a morphism $f\in \Hom_{\Foam^\varepsilon}(W_0,W_1)$ satisfying
\[
    P_1\circ f=f=f\circ P_0.
\]

There exists a formal grading of objects in $\operatorname{Kar}(\Foam^\varepsilon)$ which respects the degree of morphisms between the objects. 
An object $\mathcal{E}(P)$ with formal grading $i\in \mathbb{Z}$ is denoted by
\[
    \mathcal{E}(P)\{i\}.
\]
Similarly, write
\[
    \mathcal{E}(W)\{i\}:=\mathcal{E}(\idfoam_W)\{i\}.
\]

The morphism space between objects $\mathcal{E}(P_0)\{i\}$ and $\mathcal{E}(P_1)\{j\}$ is defined by
\[
\Hom_{\operatorname{Kar}(\Foam^\varepsilon)}
\bigl(\mathcal{E}(P_0)\{i\},\mathcal{E}(P_1)\{j\}\bigr)
=
P_1\,
\Hom_{\Foam^\varepsilon}^{\,i-j}(W_0,W_1)\,
P_0. 
\]

In other words, a morphism
\[
    f\colon \mathcal{E}(P_0)\{i\}
    \longrightarrow
    \mathcal{E}(P_1)\{j\}
\]
is a homogeneous morphism $f\in \Hom_{\Foam^\varepsilon}(W_0,W_1)$ such that
\[
   \deg(f)=i-j 
   \quad 
   \text{and} 
   \quad 
   P_1\circ f=f=f\circ P_0.
\]
\end{definition}

\begin{remark}
   The formal grading of objects in  $\operatorname{Kar}(\Foam^\varepsilon)$ matches the grading shift of the corresponding state spaces. To be more precise, given $\mathcal{E}(P_0)=(W_0,P_0)$ and 
   $\mathcal{E}(P_1)=(W_1,P_1) \in \operatorname{Kar}(\Foam^\varepsilon)$, consider a morphism
   \[
    f \colon \mathcal{E}(P_0)\{i\}
    \longrightarrow
    \mathcal{E}(P_1)\{j\},
\] 
where $f=M$ and $\deg(f)=\foamdeg(M)= i-j$.
Given any web $W_2$ with boundary $\varepsilon$, 
since $\foamdeg(\overline{\idfoam_{W_2}} \ast M )= \foamdeg(M)= i-j$, we have the grading-preserving homomorphism
\[ 
    \mathcal{F}(M): \mathcal{F}(\overline{\idfoam_{W_2}}\ast P_0 )(\mathcal{F}( \overline{W_2} \ast W_0 ))\{i\} \longrightarrow 
    \mathcal{F}(\overline{\idfoam_{W_2}} \ast P_1)(\mathcal{F}(\overline{W_2} \ast W_1 ))\{j\}.
\]

\end{remark}

\begin{remark}
In the web category viewed after decategorification, a formal linear combination of webs corresponds categorically to a direct sum of the corresponding objects in $\Foam^\varepsilon$ or in $\operatorname{Kar}(\Foam^\varepsilon)$. More precisely, a sum
\[
    W_1+\cdots+W_r
\]
is categorified by the direct sum
\[
    \mathcal{E}(W_1)\oplus \cdots \oplus \mathcal{E}(W_r).
\]

If grading shifts are present, then a graded term $q^i W$ corresponds to the shifted object $\mathcal{E}(W)\{i\}$. Thus formal sums of webs in the decategorified web category are realized by direct sums of the associated objects in the foam category or its Karoubi completion.
\end{remark}

\subsection{Indecomposable objects in $Kar(\Foam^\varepsilon)$}
We now present some results about the indecomposable objects in the Karoubi envelope. While the results are known (see for example \cite{BarNatan2006} for a treatment in a similar context), we reinterpret them in our setting.  

\begin{definition}\label{def:indecomposable_Web}
An object $\mathcal{E}(P)=(W,P) \in Kar(\Foam^\varepsilon)$ is \emph{decomposable} if it is isomorphic to a direct sum of objects, i.e.\ if there exist
\begin{itemize}
    \item objects $\mathcal{E}(P_1)=(W_1,P_1)$, $\mathcal{E}(P_2)=(W_2,P_2),$ 
    \item non-zero morphisms $F_i$ (non-zero linear combinations of foams) from $W$ to $W_i$ for $i=1,2$,
    \item and morphisms $I_i$ from $W_i$ to $W$ for $i=1,2$
\end{itemize}
such that the following foam relations hold:
\begin{align}
    F_i &= {F_i} \circ P = {P_i} \circ {F_i}, & I_i = P\circ {I_i} = {I_i} \circ {P_i}, \text{ for } i=1,2, \label{DecompRel} \\
    P &= {I_1}\circ {F_1}+{I_2} \circ{F_2},  \notag\\ 
    {P_1} &= {F_1}\circ {I_1} , & {P_2} = {F_2} \circ{I_2},  \notag\\
    0 &= {F_1} \circ {I_2} , & 0 = {F_2} \circ{I_1} \notag
\end{align}
We say $\mathcal{E}(P)$ is \emph{indecomposable} if it is not decomposable. 
\end{definition}

Note that the above definition implicitly requires the foams $F_1$ and $F_2$ appearing in the definition to satisfy certain conditions on their degrees. We will not need those details and we leave them implicit in the next proposition as well.

\begin{proposition} \label{IndecompoCrit}
    An object \(\mathcal{E}(P)=(W,P) \in Kar(\Foam^{\varepsilon})\) is indecomposable if
    \[
        P\Hom^0_{\Foam^{\varepsilon}}(W,W)P=\mathbb{Q}P.
    \]
    For \(P=\idfoam_W\), this becomes \(\Hom^0_{\Foam^{\varepsilon}}(W,W)=\mathbb{Q}\idfoam_W\).
\end{proposition}

\begin{proof}
    The inclusion \(\Foam^\varepsilon \subset Kar(\Foam^{\varepsilon})\) sends an object \(W\) to \((W,\idfoam_W)\) and sends a morphism to itself, so it suffices to prove the statement for the Karoubi envelope.
    
    Suppose that $(W,P)$ satisfies the condition in the proposition but is decomposable. Then there exist $\mathcal{E}(P_1)=(W_1,P_1)$, $ \mathcal{E}(P_2)=(W_2,P_2),$ and morphisms $F_i,I_i$ for $i=1,2$, such that relations in Equation \eqref{DecompRel} are satisfied. 
    Since \({I_1}\circ{F_1}\in P\Hom^0_{\Foam^{\varepsilon}}(W,W)P\), there exists \(k\in\mathbb{Q}\) such that
    \[
        {I_1}\circ{F_1}=P\circ {I_1}\circ {F_1}\circ P=kP.
    \]
   Then 
    \[ k{F_2} = {F_2} \circ  kP = {F_2} \circ {I_1} \circ {F_1} =0,  \]
    So $k=0$ since $F_2$ is non-zero. Hence ${I_1} \circ {F_1}= k P = 0$. By a similar argument, we get ${I_2} \circ {F_2} = 0$. This leads to    
    \[P= {I_1} \circ {F_1}+ {I_2} \circ {F_2}= 0, \] 
    which contradicts the assumption that $(W,P)$ is decomposable, completing the proof. 
\end{proof}

We also collect here some facts that we will need in the rest of the paper. First recall the following theorem and its corollary.

\begin{proposition}  [{\cite[Theorem 3.30]{RW-eval-foams}}] \label{StateSpaceDim}
For any closed $\GL_N$ web $\Gamma$, the graded rank of $\mathcal{F}(\Gamma)$ is equal to $\langle \Gamma \rangle$.
\end{proposition}

\begin{corollary} \label{WebPairDim}
        Given two webs $W_1$ and $W_2$ with boundary $\varepsilon$, the graded rank of $\Hom_{\Foam^{\varepsilon}}(W_1,W_2)$ is given by $q^{\boundarydeg(\varepsilon)}\langle \overline{W_1}W_2\rangle$. 
\end{corollary}

The following lemma follows from Corollary \ref{WebPairDim} and Proposition \ref{IndecompoCrit}.

\begin{lemma} \label{WebPairIndecompo}
    A \(\GL_N\) web \(W\) with boundary \(\varepsilon\) is indecomposable if
    \[
        q^{\boundarydeg(\varepsilon)}\langle \overline{W}W \rangle \in 1+q\mathbb{Z}[q].
    \]
    Equivalently, \(\langle \overline{W}W \rangle\) has coefficient \(1\) in degree \(-\boundarydeg(\varepsilon)\) and has no terms in smaller degrees.
\end{lemma}

\section{Definition and properties of the \texorpdfstring{$6$}{6}-valent vertex}\label{SixVertexProof}

Recent work \cite{SL4Basis} successfully uses $6$-valent vertices for a construction of a preferred $\mathrm{SL}(4)$ web basis. In this section we define the $6$-valent vertex for $\GL_N$ webs and provide the properties of these webs that we will use. These statements are proven using standard web calculations. This section prepares the ground for the introduction of the categorified $6$-valent vertex later on.

\subsection{Hexagon relation in the \texorpdfstring{$\GL_N$}{GL_N} web category} 

In this subsection we recall the \emph{benzene} and \emph{matching} webs, along with some of their properties. 
Additionally, we derive the relation that we will use to define the $6$-valent vertex. 

\begin{notation}\label{notation webs 6pt boundary}
    We recall the notation from the introduction (see Figure \ref{fig:web-names} and Equation \eqref{GLN6Vertex}) for the following webs in $\Hom_{\Web}(\varepsilon_\emptyset, (1^+\otimes 1^-)^{\otimes 3})$. We denote by $\HexaBound$ the boundary sequence $(1^+\otimes 1^-)^{\otimes 3}$ of the webs.
      \begin{align*}
      &\Wzero=\Benzeneweb  , \,  \Wzeroa= \Benzeneweba   , \,  \Wone= \BenzenewebM    ,\,  \Wtwo= \BenzenewebaM  , \\
  & \Wthree=\BenzenewebMB  ,\,  \Wfour=\BenzenewebMC  , \, \Wfive=\BenzenewebMA,  \\
  &\Wsix =  \SixVertex  =   \Benzeneweb   -   \BenzenewebM .
  \end{align*}  
\end{notation}

\begin{proposition}\label{hexagon rotation web}
The following relation holds in $\Web$:
\begin{equation} \label{HexagonRelation}
      \Benzeneweb \  -  [N-3] \ \BenzenewebM
        \  =  \ \Benzeneweba  \ -  [N-3] \  \BenzenewebaM. 
      \end{equation} 
\end{proposition}

\begin{proof}
 We know that $\Hom_{U_q(\GL_N)}(\mathbb{C}(q), (V_{\omega_1}\otimes V_{\omega_1}^*)^{\otimes 3})$ is $6$-dimensional from the tensor decomposition formulas. So $\Hom_{\Web}(\varepsilon_\emptyset, (1^+\otimes 1^-)^{\otimes 3})$ is also $6$-dimensional by Proposition \ref{WebRep}. Hence there is a linear relation between the $7$ webs $\Wzero, \Wzeroa, \Wone, \Wtwo, \Wthree, \Wfour,$ and $\Wfive$.

  Suppose for $k_0, k_0', k_1, k_2, k_3, k_4, k_5$ not all $0$, 
  \begin{equation}\label{LinearComboHexa}
      k_0 \Wzero  + k_0' \Wzeroa  + k_1 \Wone  + k_2 \Wtwo 
      + k_3 \Wthree  + k_4 \Wfour  + k_5 \Wfive = 0. 
\end{equation}

Attaching a $\GL_N$ clasp (see \cite{elias2015lightladdersclaspconjectures} for a definition of clasps) to each web in Equation \eqref{LinearComboHexa}, we get 
 \begin{align*}
      k_0 \ \BenzenewebClasp \ + k_0' \ \BenzenewebaClasp  \ + k_1 \ \BenzenewebMClasp  \  + k_2 \  & \BenzenewebaMClasp \\
      + k_3 \ \BenzenewebMAClasp \ + k_4 \ \BenzenewebMBClasp \ + k_5 \ & \BenzenewebMCClasp = 0. 
  \end{align*} 
  
 This implies that $k_3=0$. Similarly, we can show that $k_4=k_5=0$. Now we attach a cap diagram to each web in Equation \eqref{LinearComboHexa}: 
\begin{align*}
     & k_0 \ \BenzenewebCupA \ + k_0' \ \BenzenewebaCupA  \ + k_1\ \BenzenewebMCupA  \  + k_2 \  \BenzenewebaMCupA  = 0,  \\
     & k_0 \ \BenzenewebCupB \ + k_0' \ \BenzenewebaCupB  \ + k_1 \ \BenzenewebMCupB  \  + k_2 \  \BenzenewebaMCupB  = 0.  
  \end{align*} 

We simplify the above relations by applying Relations \eqref{NeededSkein}, and obtain the following system of equations:
\begin{align*}
    &[2][N-2]k_0+[N-1]k_0'+k_1=0, \\
    &[2]k_0+[N-1][N-2]k_0'+[N]k_2=0,\\
    &[N-1]k_0+[2][N-2]k_0'+k_2=0, \\
    &[N-1][N-2]k_0+[2]k_0'+[N]k_1=0. 
\end{align*}
So $k_0'=-k_0$, $k_1=-k_2=-[N-3]k_0$. Since $k_0, k_0', k_1, k_2, k_3, k_4,$ and $ k_5$ are not all $0$, we have proven Equation \eqref{HexagonRelation}.
\end{proof}

\subsection{\texorpdfstring{$6$}{6}-valent vertex}

Next we introduce the $6$-valent vertex in the $\GL_N$ web category $\Web$. We will show that the $6$-valent vertex is positive and list the graphical calculations  needed later for the categorification of the $6$-valent vertex. 

\def\SixVertex
{\begin{tric}
\draw [decoration={markings,mark=at position 0.6 with {\arrow{>}}},postaction={decorate}] (0,0)--(2,0);
\draw [decoration={markings,mark=at position 0.6 with {\arrow{>}}},postaction={decorate}](-2,0)--(0,0);
\draw [decoration={markings,mark=at position 0.6 with {\arrow{>}}},postaction={decorate}] (1,1.7)--(0,0);
\draw [decoration={markings,mark=at position 0.6 with {\arrow{>}}},postaction={decorate}] (0,0)--(-1,-1.7);
\draw [decoration={markings,mark=at position 0.6 with {\arrow{>}}},postaction={decorate}] (1,-1.7)--(0,0);
\draw [decoration={markings,mark=at position 0.6 with {\arrow{>}}},postaction={decorate}] (0,0)--(-1,1.7);
\end{tric}
}

\def\HexaColoring
{\begin{tric}
\draw [scale=0.7, decoration={markings,mark=at position 0.65 with {\arrow{>}}},postaction={decorate}] 
      (0,0)..controls(1.7,-0.6)and(2,-1.4)..(2,-2);
\draw [scale=0.7, decoration={markings,mark=at position 0.5 with {\arrow{>}}},postaction={decorate}]      
      (1.2,-2)..controls(1.2,-1.4)and(0.8,-0.6)..(0,0);
\draw [scale=0.7, decoration={markings,mark=at position 0.65 with {\arrow{>}}},postaction={decorate}]       
      (0,0)..controls(-0.8,-0.6)and(-1.2,-1.4)..(-1.2,-2);
\draw [scale=0.7, decoration={markings,mark=at position 0.45 with {\arrow{>}}},postaction={decorate}]       
     (-2,-2)..controls(-2,-1.4)and(-1.7,-0.6)..(0,0);
\draw [scale=0.7, decoration={markings,mark=at position 0.65 with {\arrow{>}}},postaction={decorate}]       
      (0,0)..controls(0.25,-0.6)and(0.4,-1.4)..(0.4,-2);
\draw [scale=0.7, decoration={markings,mark=at position 0.55 with {\arrow{>}}},postaction={decorate}]       
     (-0.4,-2) ..controls(-0.4,-1.4)and(-0.25,-0.6)..(0,0);
\end{tric}
}

\begin{defn}
    By Equation \eqref{HexagonRelation}, we can define the following $6$-valent vertex, denoted by $\Wsix$, which is $\cfrac{2\pi}{3}$-rotation-invariant and reflection-invariant along the edges:
      \begin{equation} \label{GLN6VertexDef}
      \SixVertex \ :=  \ \Benzeneweb \  -  [N-3] \ \BenzenewebM
        \  =  \ \Benzeneweba  \ -  [N-3] \  \BenzenewebaM
      \end{equation} 
\end{defn}

We write down concretely the morphisms in $\Fund(U_q(\GL_{N}))$ corresponding to $\overline{\Wzero}$, $\overline{\Wone}$, and $\overline{\Wsix}$ when $q=1$. We then also write down the vector calculations for generic $q$ and observe that the $6$-valent vertex $\Wsix$ is positive. 

\begin{proposition} \label{VertexColoring}
  For $\overline{\Wzero}, \overline{\Wone},$ and $ \overline{\Wsix} \in \Hom_{\Web}(\HexaBound, \emptylist)$, 
  $\Phi\left(\overline{\Wzero}\right), \Phi\left(\overline{\Wone}\right), $ and $ \Phi\left(\overline{\Wsix}\right)\in \Hom_{\Fund(U_q(\GL_{N}))}((V\otimes V^*)^{\otimes 3}, \mathbb{C}(q))$. Given basis vectors $e_i, 1\le i \le N$ of the $\GL_{N}$ defining representation $V$, when $q=1$, $\Phi\left(\overline{\Wzero}\right), \Phi\left(\overline{\Wone}\right), $ and $ \Phi\left(\overline{\Wsix}\right)$ act on the basis as follows, where $i,j$ and $k$ are different indices.  

  \begin{align*}
  \Phi\left(\overline{\Wzero}\right) : (V\otimes V^*)^{\otimes 3} &\longrightarrow \mathbb{C} \\
  (e_i\otimes e_i^*)^{\otimes 3} & \mapsto N-1 \\
  (e_i\otimes e_i^*)^{\otimes 2}\otimes e_j\otimes e_j^* & \mapsto N-2 \\
   e_i\otimes e_i^* \otimes e_j\otimes e_j^* \otimes e_k\otimes e_k^*  & \mapsto N-3 \\
   e_i\otimes e_j^* \otimes e_k \otimes e_k^* \otimes e_j\otimes e_i^*  & \mapsto 1 \\
    e_i\otimes e_j^* \otimes e_k\otimes e_i^* \otimes e_j\otimes e_k^*  & \mapsto -1 \\
    \text{else} & \mapsto 0 
\end{align*}

\begin{align*}
  \Phi\left(\overline{\Wone}\right) : (V\otimes V^*)^{\otimes 3} &\longrightarrow \mathbb{C} \\
  (e_i\otimes e_i^*)^{\otimes 3} & \mapsto 1 \\
   (e_i\otimes e_i^*)^{\otimes 2}\otimes e_j\otimes e_j^* & \mapsto 1 \\
   e_i\otimes e_i^* \otimes e_j\otimes e_j^* \otimes e_k\otimes e_k^*  & \mapsto 1 \\
    \text{else} & \mapsto 0 
\end{align*}

\begin{align*}
  \Phi\left(\overline{\Wsix}\right) : (V\otimes V^*)^{\otimes 3} &\longrightarrow \mathbb{C} \\
  (e_i\otimes e_i^*)^{\otimes 3} & \mapsto 2 \\
  (e_i\otimes e_i^*)^{\otimes 2}\otimes e_j\otimes e_j^* & \mapsto 1 \\
   e_i\otimes e_i^* \otimes e_j\otimes e_j^* \otimes e_k\otimes e_k^*  & \mapsto 0 \\
   e_i\otimes e_j^* \otimes e_k \otimes e_k^* \otimes e_j\otimes e_i^*  & \mapsto 1 \\
    e_i\otimes e_j^* \otimes e_k\otimes e_i^* \otimes e_j\otimes e_k^*  & \mapsto -1 \\
     \text{else} & \mapsto 0 
\end{align*}
\end{proposition}

\begin{proof}
The proposition follows directly from vector calculations. For example, we compute $\Phi\left(\overline{\Wzero}\right) \left( (e_i\otimes e_i^*)^{\otimes 3} \right)$ by observing that $\BenzeneDualColor$ gives a non-zero vector assignment which evaluates to $1$ when $a\ne i$.
\end{proof}

\def\VertWzero
{\begin{tric}
\draw [scale=0.7, decoration={markings,mark=at position 0.65 with {\arrow{>}}},postaction={decorate}] 
      (0,-3.5)--(0,-2.5);
\draw [scale=0.7, decoration={markings,mark=at position 0.65 with {\arrow{>}}},postaction={decorate}] 
      (3,-3.5)--(3,-1);
\draw [scale=0.7, decoration={markings,mark=at position 0.65 with {\arrow{>}}},postaction={decorate}] 
      (-3,-3.5)--(-3,-1);
      
\draw [scale=0.7, decoration={markings,mark=at position 0.65 with {\arrow{>}}},postaction={decorate}] 
      (0,-2.5)--(-3,-1);
\draw [scale=0.7] (-1.6,-1.95)node[black,scale=0.7,below]{$1$};          
\draw [scale=0.7, decoration={markings,mark=at position 0.65 with {\arrow{>}}},postaction={decorate}] 
      (0,-2.5)--(3,-1);
\draw [scale=0.7] (1.6,-1.95)node[black,scale=0.7,below]{$N-2$};     

\draw [scale=0.7, decoration={markings,mark=at position 0.65 with {\arrow{>}}},postaction={decorate}] 
      (3,-1)--(3,1)node[midway,black,scale=0.7,left]{$N-1$};     
\draw [scale=0.7, decoration={markings,mark=at position 0.65 with {\arrow{>}}},postaction={decorate}] 
      (-3,-1)--(-3,1);
\draw [scale=0.7] (-3.2,0) node[black,scale=0.7,left]{$2$};   

\draw [scale=0.7, decoration={markings,mark=at position 0.65 with {\arrow{>}}},postaction={decorate}] 
      (-3,1)--(0,2.5);
\draw [scale=0.7] (-1.6,1.95)node[black,scale=0.7,above]{$1$};       
\draw [scale=0.7, decoration={markings,mark=at position 0.65 with {\arrow{>}}},postaction={decorate}] 
      (3,1)--(0,2.5);
\draw [scale=0.7] (1.5,2)node[black,scale=0.7,above]{$N-2$}; 

\draw [scale=0.7, decoration={markings,mark=at position 0.65 with {\arrow{>}}},postaction={decorate}] 
      (0,2.5)--(0,3.5);
\draw [scale=0.7, decoration={markings,mark=at position 0.65 with {\arrow{>}}},postaction={decorate}] 
      (3,1)--(3,3.5); 
\draw [scale=0.7, decoration={markings,mark=at position 0.65 with {\arrow{>}}},postaction={decorate}] 
      (-3,1)--(-3,3.5); 

\draw[scale=0.7] (-3,-3.5)node[black,scale=0.7,below]{$1$}
                 (3,-3.5)node[black,scale=0.7,below]{$1$}
                 (0,-3.5)node[black,scale=0.7,below]{$N-1$}
                 (-3,3.5)node[black,scale=0.7,above]{$1$}
                 (3,3.5)node[black,scale=0.7,above]{$1$}
                 (0,3.5)node[black,scale=0.7,above]{$N-1$};      
\end{tric}
}

\def\VertWone
{\begin{tric}
\draw [scale=0.7, decoration={markings,mark=at position 0.65 with {\arrow{>}}},postaction={decorate}] 
      (-2,-3.5)--(-2,3.5);
      
\draw [scale=0.7, decoration={markings,mark=at position 0.55 with {\arrow{>}}},postaction={decorate}]      
(0,-3.5)..controls(0,-2)and(0,-1)..(1,-1);
\draw [scale=0.7, decoration={markings,mark=at position 0.55 with {\arrow{>}}},postaction={decorate}]      
(2,-3.5)..controls(2,-2)and(2,-1)..(1,-1);

\draw [scale=0.7, decoration={markings,mark=at position 0.45 with {\arrow{<}}},postaction={decorate}]      
(0,3.5)..controls(0,2)and(0,1)..(1,1);
\draw [scale=0.7, decoration={markings,mark=at position 0.45 with {\arrow{<}}},postaction={decorate}]      
(2,3.5)..controls(2,2)and(2,1)..(1,1);

\draw [scale=0.7, decoration={markings,mark=at position 0.45 with {\arrow{>}}},postaction={decorate}]  (1,-1)--(1,1) node[black, midway, scale=0.7,right ]{$N$};

\draw[scale=0.7] (-2,-3.5)node[black,scale=0.7,below]{$1$}
                 (2,-3.5)node[black,scale=0.7,below]{$1$}
                 (0,-3.5)node[black,scale=0.7,below]{$N-1$}
                 (-2,3.5)node[black,scale=0.7,above]{$1$}
                 (2,3.5)node[black,scale=0.7,above]{$1$}
                 (0,3.5)node[black,scale=0.7,above]{$N-1$};
\end{tric}
}

\begin{proposition}\label{QVertexColoring}
View $\Wzero$, $\Wone$, and $\Wsix$ as webs in $\Hom_{\Web}(1^{+}(N-1)^{+}1^{+}, 1^{+}(N-1)^{+}1^{+})$ in the following way:  
\[\Wsix=\Wzero- [N-3] \Wone := \VertWzero- [N-3] \VertWone.   \]
Denote $\bigwedge_{n \ne i}e_n \in \wedge_q^{N-1} V$ by $\widehat{e_i}$. For different indices $i,j$ and $k$, we have
  \begin{align*}
  (e_i \otimes \widehat{e_i} \otimes e_i )^*\big(\Phi\left(\Wsix\right)(e_i \otimes \widehat{e_i} \otimes e_i )\big)&= [2],\\
 (e_i \otimes \widehat{e_i} \otimes e_j )^*\big(\Phi\left(\Wsix\right)(e_i \otimes \widehat{e_i} \otimes e_j )\big)&= q^{\frac{j-i}{|j-i|}},\\
  (e_i \otimes \widehat{e_k} \otimes e_k )^*\big(\Phi\left(\Wsix\right)(e_i \otimes \widehat{e_j} \otimes e_j )\big)&= 0,\\
  (e_i \otimes \widehat{e_j} \otimes e_k )^*\big(\Phi\left(\Wsix\right)(e_i \otimes \widehat{e_j} \otimes e_k )\big)&= q^{\frac{k-i}{|k-i|}},\\
(e_i \otimes \widehat{e_j} \otimes e_k )^*\big(\Phi\left(\Wsix\right)(e_k \otimes \widehat{e_j} \otimes e_i )\big)&= -1, \\
\text{else} &= 0.  \\
\end{align*}
Fix arbitrary indices $i_1,i_2, \dots, i_6 $, and denote by $\kappa_n^t$ the integer coefficient in front of each term $q^t$ in $(e_{i_1} \otimes \widehat{e_{i_2}} \otimes e_{i_3} )^*\big(\Phi\left(\mathcal{W}_n\right)(e_{i_4} \otimes \widehat{e_{i_5}} \otimes e_{i_6} )\big)$. In other words, let 
\[(e_{i_1} \otimes \widehat{e_{i_2}} \otimes e_{i_3} )^*\big(\Phi\left(\mathcal{W}_n\right)(e_{i_4} \otimes \widehat{e_{i_5}} \otimes e_{i_6} )\big)= \sum_t \kappa_n^t q^t.\]
For any given $n$, any non-zero $\kappa_n^t$ has the same sign. 
Furthermore, $\kappa_0^t$ has the same sign as $\kappa_6^t$ when $\kappa_6^t \ne 0$, and $|\kappa_0^t|\ge |\kappa_6^t|  $. 
\end{proposition}

\begin{proof}
This proposition follows directly from vector calculations. The formulas for vector calculations can be found in~\cite[Section A.2]{RobertWagner2020}, for example.   
\end{proof}

\begin{corollary}
    Let $\Gamma$ be a closed web built from $\Wsix$ and the usual trivalent vertices from Definition \ref{TypeAWebs}, then $\Phi(\Gamma)\in \Z_+[q,q^{-1}]$.  
\end{corollary}

\begin{proof}
We replace each $\Wsix$ in $\Gamma$ by $\Wzero$ and denote the resulting closed web by $\Gamma_0$. First we know that $\Phi(\Gamma_0)\in \Z_+[q,q^{-1}]$ due to positivity of the planar web evaluation in type $A$.
By Proposition \ref{QVertexColoring}, any non-zero integer coefficient in $\Phi(\Gamma)$ has the same sign as the corresponding coefficient in $\Phi(\Gamma_0)$, so $\Phi(\Gamma) \in \Z_+[q,q^{-1}]$.
\end{proof}

\begin{notation}\label{flapcolor on web}
    We use \flapcolor[i_1][i_2][i_3][i_4][i_5][i_6] to denote the coloring $e_{i_1}^* \otimes e_{i_2} \otimes e_{i_3}^* \otimes e_{i_4} \otimes e_{i_5}^* \otimes e_{i_6}$ of the web. 
\end{notation}

\def\HexaPairing
{\begin{tric}
\begin{scope}[yscale=-1]
\draw [scale=0.7, decoration={markings,mark=at position 0.5 with {\arrow{<}}},postaction={decorate}] 
      (0,0)..controls(1.7,-0.6)and(2,-1.4)..(2,-2)..controls(2,-2.6)and(1.7,-3.4)..(0,-4);
\draw [scale=0.7, decoration={markings,mark=at position 0.5 with {\arrow{<}}},postaction={decorate}]      
     (0,-4)..controls(0.8,-3.4)and(1.2,-2.6).. (1.2,-2)..controls(1.2,-1.4)and(0.8,-0.6)..(0,0);
\draw [scale=0.7, decoration={markings,mark=at position 0.5 with {\arrow{<}}},postaction={decorate}]       
      (0,0)..controls(-0.8,-0.6)and(-1.2,-1.4)..(-1.2,-2)..controls(-1.2,-2.6)and(-0.8,-3.4)..(0,-4);
\draw [scale=0.7, decoration={markings,mark=at position 0.5 with {\arrow{<}}},postaction={decorate}]       
     (0,-4)..controls(-1.7,-3.4)and(-2,-2.6)..(-2,-2)..controls(-2,-1.4)and(-1.7,-0.6)..(0,0);
\draw [scale=0.7, decoration={markings,mark=at position 0.5 with {\arrow{<}}},postaction={decorate}]       
      (0,0)..controls(0.25,-0.6)and(0.4,-1.4)..(0.4,-2)..controls(0.4,-2.6)and(0.25,-3.4)..(0,-4);
\draw [scale=0.7, decoration={markings,mark=at position 0.5 with {\arrow{<}}},postaction={decorate}]      
     (0,-4)..controls(-0.25,-3.4)and(-0.4,-2.6)..(-0.4,-2) ..controls(-0.4,-1.4)and(-0.25,-0.6)..(0,0);
\end{scope}    
\end{tric}
}

\def\BenzeneBound
{\begin{tricpobs}
\draw [->](-1,1.7)--(0.1,1.7); 
\draw (0,1.7)--(1,1.7);
\draw [->](2,0)--(1.4,-1.02); 
\draw (1.5,-0.85)--(1,-1.7);
\draw [->](-1,-1.7)--(-1.6,-0.68); 
\draw (-1.5,-0.85)--(-2,0); 
\draw[double,decoration={markings,mark=at position 0.65 with {\arrow{>}}},postaction={decorate}]  (-2,0)--(-1,1.7);
\draw[double,decoration={markings,mark=at position 0.75 with {\arrow{>}}},postaction={decorate}] (1,-1.7)--(-1,-1.7);
\draw[double, decoration={markings,mark=at position 0.65 with {\arrow{>}}},postaction={decorate}] (1,1.7)--(2,0);

\draw [->](2,0)--(3.1,0); 
\draw (3,0)--(4,0);
\draw [->](-4,0)--(-2.9,0); 
\draw (-3,0)--(-2,0);
\draw [->](-1,1.7)--(-1.6,2.72);
\draw (-1.5,2.55)--(-2,3.4);
\draw [->](2,-3.4)--(1.4,-2.38);
\draw (1.5,-2.55)--(1,-1.7);
\draw [->](-1,-1.7)--(-1.6,-2.72);
\draw (-1.5,-2.55)--(-2,-3.4);
\draw [->](2,3.4)--(1.4,2.38);
\draw (1.5,2.55)--(1,1.7);

\draw[double,decoration={markings,mark=at position 0.65 with {\arrow{>}}},postaction={decorate}] (-2,3.4)--(-4,0);
\draw[double,decoration={markings,mark=at position 0.65 with {\arrow{>}}},postaction={decorate}] (4,0)--(2,3.4);
\draw[double,decoration={markings,mark=at position 0.65 with {\arrow{>}}},postaction={decorate}] (-2,-3.4)--(2,-3.4);

\draw[decoration={markings,mark=at position 0.55 with {\arrow{>}}},postaction={decorate}] (2,3.4)--(-2,3.4);
\draw[decoration={markings,mark=at position 0.55 with {\arrow{>}}},postaction={decorate}] (-4,0)--(-2,-3.4);
\draw[decoration={markings,mark=at position 0.55 with {\arrow{>}}},postaction={decorate}] (2,-3.4)--(4,0);
\end{tricpobs}
}

\def\BenzeneMixBound
{\begin{tricpobs}
\draw [->](-1,1.7)--(0.1,1.7); 
\draw (0,1.7)--(1,1.7);
\draw [->](2,0)--(1.4,-1.02); 
\draw (1.5,-0.85)--(1,-1.7);
\draw [->](-1,-1.7)--(-1.6,-0.68); 
\draw (-1.5,-0.85)--(-2,0); 
\draw[double,decoration={markings,mark=at position 0.65 with {\arrow{>}}},postaction={decorate}]  (-2,0)--(-1,1.7);
\draw[double,decoration={markings,mark=at position 0.75 with {\arrow{>}}},postaction={decorate}] (1,-1.7)--(-1,-1.7);
\draw[double, decoration={markings,mark=at position 0.65 with {\arrow{>}}},postaction={decorate}] (1,1.7)--(2,0);

\draw [->](2,0)--(3.1,0); 
\draw (3,0)--(4,0);
\draw [->](-4,0)--(-2.9,0); 
\draw (-3,0)--(-2,0);
\draw [->](-1,1.7)--(-1.6,2.72);
\draw (-1.5,2.55)--(-2,3.4);
\draw [->](2,-3.4)--(1.4,-2.38);
\draw (1.5,-2.55)--(1,-1.7);
\draw [->](-1,-1.7)--(-1.6,-2.72);
\draw (-1.5,-2.55)--(-2,-3.4);
\draw [->](2,3.4)--(1.4,2.38);
\draw (1.5,2.55)--(1,1.7);

\draw[decoration={markings,mark=at position 0.65 with {\arrow{<}}},postaction={decorate}] (-2,3.4)--(-4,0);
\draw[decoration={markings,mark=at position 0.65 with {\arrow{<}}},postaction={decorate}] (4,0)--(2,3.4);
\draw[decoration={markings,mark=at position 0.65 with {\arrow{<}}},postaction={decorate}] (-2,-3.4)--(2,-3.4);

\draw[double,decoration={markings,mark=at position 0.55 with {\arrow{<}}},postaction={decorate}] (2,3.4)--(-2,3.4);
\draw[double,decoration={markings,mark=at position 0.55 with {\arrow{<}}},postaction={decorate}] (-4,0)--(-2,-3.4);
\draw[double,decoration={markings,mark=at position 0.55 with {\arrow{<}}},postaction={decorate}] (2,-3.4)--(4,0);
\end{tricpobs}
}

\begin{proposition}\label{prop:closed sixvertex evaluation}
The following webs from $\Web$ evaluate as stated:
\begin{align*}
 \BenzeneBound& = [N-2]^2 [N-1]^2[N]+ [2]^2[N-1]^2[N] + [N-2][2][N-1][N]  \\
 &= \begin{dcases}
       2 q^{-9} + \sum_{k>-9} a_k q^k, & N=4,  \\
      q^{11-5N}+ \sum_{k>11-5N} b_k q^k  , & N \ge 5. \\
    \end{dcases}       \\
 \BenzeneMixBound & = [2]^2[N-1]^2[N]+ [N-2][2][N-1][N] + [N-1]^2[N]+[N-3][2]^2[N-1][N] \\  
 & =  q^{3-3N} + \sum_{k>3-3N} c_k q^k. \\
 \HexaPairing& = [N]^3 + 3 [N][N-1]^2 + [N][3] =  q^{3-3N} + \sum_{k>3-3N} d_k q^k.
\end{align*}
\end{proposition}

\begin{proof}
Apply the skein relations in Definition \ref{TypeAWebs}.    
\end{proof}

\begin{proposition} \label{CartanMatrix}
Given webs $\Wone,\Wtwo,\dots, \Wsix \in \Hom_{\Web}(\emptylist, \HexaBound) $, the pairing matrix is $\left(\overline{\mathcal{W}_i}\mathcal{W}_j\right)_{i,j}=q^{3-3N} (I_6 + \left(p_{i,j}(q)\right)_{i,j})$, where $p_{i,j}(q)\in q\Z[q]$ are polynomials in $q$ without constant terms.  
\end{proposition}

\def\flower
{\begin{tric}
\begin{scope}[yscale=-1]
\draw[decoration={markings,mark=at position 0.4 with {\arrow{<}}},postaction={decorate}]  (0,0)..controls(1,-3)and(-1,-3)..(0,0);
\draw[decoration={markings,mark=at position 0.4 with {\arrow{<}}},postaction={decorate}]  (0,0)..controls(4.5,-3)and(1.5,-3)..(0,0);
\draw[decoration={markings,mark=at position 0.65 with {\arrow{>}}},postaction={decorate}]  (0,0)..controls(-4.5,-3)and(-1.5,-3)..(0,0);
\end{scope}
\end{tric}
}

\begin{proof}

By graphical calculations, 
   \begin{align*}
   & \overline{\Wsix} \Wsix   = \HexaPairing = [N]^3 + 3 [N][N-1]^2 + [N][3] =  q^{3-3N} + \sum_{k>3-3N} c_k q^k :=A(q), \\
   & \overline{\Wone} \Wsix = \flower=\bigl([N-1] + [2][N]\bigr)[N] := B(q).
   \end{align*}

 Similarly, we obtain the following matrix by graphical calculations.  
\begin{align*}
&\left(\overline{\mathcal{W}_i}\mathcal{W}_j\right)_{i,j} = \begin{pmatrix}
[N]^3 & [N] & [N]^2 & [N]^2 & [N]^2 & B(q) \\
[N] & [N]^3 & [N]^2 & [N]^2 & [N]^2 & B(q) \\
[N]^2 & [N]^2 & [N]^3 & [N] & [N] & B(q) \\
[N]^2 & [N]^2 & [N] & [N]^3 & [N] & B(q) \\
[N]^2 & [N]^2 & [N] & [N] & [N]^3 & B(q) \\
B(q) & B(q) & B(q) & B(q) & B(q) & A(q)
\end{pmatrix}\\
&=q^{3-3N} \begin{pmatrix}
1 + p_{1,1}(q) & p_{1,2}(q) & p_{1,3}(q) & p_{1,4}(q) & p_{1,5}(q) & p_{1,6}(q) \\
p_{2,1}(q) & 1+ p_{2,2}(q) & p_{2,3}(q) & p_{2,4}(q) & p_{2,5}(q) & p_{2,6}(q) \\
p_{3,1}(q) & p_{3,2}(q) & 1+p_{3,3}(q) & p_{3,4}(q) & p_{3,5}(q) & p_{3,6}(q)  \\
p_{4,1}(q) & p_{4,2}(q) & p_{4,3}(q) & 1+p_{4,4}(q) & p_{4,5}(q) & p_{4,6}(q) \\
p_{5,1}(q) & p_{5,2}(q) & p_{5,3}(q) & p_{5,4}(q) & 1+p_{5,5}(q) & p_{5,6}(q) \\
p_{6,1}(q) & p_{6,2}(q) & p_{6,3}(q) & p_{6,4}(q) & p_{6,5}(q) & 1 + p_{6,6}(q) 
\end{pmatrix} \\
&=q^{3-3N} (I_6 + \left(p_{i,j}(q)\right)_{i,j}),\quad  p_{i,j}(q)\in q\Z[q]. 
\end{align*}  
\end{proof}

After categorification, this matrix, with the scalar $q^{3-3N}$ removed, is the graded Cartan matrix of the $\operatorname{Kar}(\Foam^\varepsilon)$ for $\HexaBound = 1^+1^-1^+1^-1^+1^-$. That is, the matrix describes graded ranks of homomorphisms between the six non-isomorphic indecomposable objects of $\operatorname{Kar}(\Foam^\varepsilon)$. Any indecomposable object in that category is isomorphic to the grading shift of one of the six objects $\Wone,\dots, \Wsix$. Entries of the matrix describe graded ranks of $\Hom$ spaces between these objects, as graded modules over the ground ring.

\section{Categorifying the \texorpdfstring{$6$}{6}-valent vertex with \texorpdfstring{$\GL_N$}{GLN} foams}

We now categorify the $6$-vertex introduced in the previous section and the relations involving it. The categorical interpretation leads to decompositions of webs into indecomposable summands. 
We then present the categorification for the case of $\GL_4$ and $\GL_5$ foams.

\def\SingleCy
{\begin{tricF}[scale=.5, yscale=.5]
    \def \r{2}; 
    \def \h{14}; 
    \draw[dotted,darkblue] (\r, -1*\h) arc (0:180:\r cm);
    \draw [darkblue,decoration={markings,mark=at position 0.5 with {\arrow{>}}},postaction={decorate}](-\r, -\h) arc (-180:0:\r cm);
    \draw [darkgreen] (\r,0) -- (\r,-\h);
    \draw [darkgreen] (-\r,0) -- (-\r,-\h);
    \draw[darkblue,decoration={markings,mark=at position 0.75 with {\arrow{>}}},postaction={decorate}] (0,0) circle (\r cm);
\end{tricF}}

\def\SingleCupCapAb{
\begin{tricF}[scale=.5, yscale=.5]
    \def \r{2}; 
    \def \h{14}; 
    \draw[darkgreen] (-\r,0) .. controls (-1,-.5*\h) and (1,-.5*\h) .. (\r,0);
    \draw[darkblue,decoration={markings,mark=at position 0.75 with {\arrow{>}}},postaction={decorate}] (0,0) circle (\r cm);
    \draw[dotted,darkblue] (\r, -1*\h) arc (0:180:\r cm);
    \draw[darkblue,decoration={markings,mark=at position 0.5 with {\arrow{>}}},postaction={decorate}] (-\r, -\h) arc (-180:0:\r cm);
    \draw[darkgreen] (-\r,-\h) .. controls (-1,-.5*\h) and (1,-.5*\h) .. (\r,-\h);
    \node at (0, -.25*\h) {$\bullet$};
    \node at (0.5, -.25*\h)[scale=0.7] {$x^i$};
    \node at (0, -.75*\h) {$\bullet$};
    \node at (0.5, -.75*\h)[scale=0.7] {$x^j$};
\end{tricF}
}

\def\matchingbubble{
\begin{tikzpicture}[yscale=.5,
    decoration={
    markings,
    mark=at position 0.7 with {\arrow{>}}},
    baseline=2cm
    ]
\def \h{5};
\def \t{45}; 
\def \r{1.2}; 
\def \R{2}; 

\def \T{\t-30};
\begin{scope}[yscale=-1]
    \foreach \k in {0,2,4}{
        \directedsingleline{ (\T+\k * 60:\R cm) 
        .. controls (0,0) ..  (\T + \k * 60 + 60:\R cm)};
    }
    \draw[darkgreen!50] (.15,-.585*\h) to[out=\T, in=-90, looseness=.4] (\T+30:.44 cm);
    \draw[darkgreen!50] (-.25,-.605*\h) to[out=\T+120, in=-90, looseness=.5] (\T+150:.44 cm);
    \draw[darkgreen!50] (0.07,-.675*\h) to[out=\T-75, in=-90, looseness=.5] (\T+270:.44 cm);
    \draw[darkgreen!50, dashed] 
            (.15,-.585*\h)
            -- (-.25,-.605*\h)
            -- (0.07,-.675*\h)
            -- (.15,-.585*\h)
            ;

    \foreach \k in {0,1,2,3,4,5}{
        \draw[darkgreen] (\T + \k * 60:2cm) -- +(0,-\h);
    }
    \begin{scope}[yshift=-\h cm]
        \foreach \k in {0,2,4}{
            \directeddoubleline{(\T+\k * 60:\r cm) --  (\T + \k * 60 + 60:\r cm)};
            \directedsingleline{ (\T+\k*60 - 60:\r cm) -- (\T +\k*60:\r cm)};
        }
        \foreach \k in {0,2,4}{
            \directedsingleline{(\T+\k*60:\R cm) -- (\T+\k*60:\r cm)};
        }
        \foreach \k in {1,3,5}{
            \directedsingleline{ (\T+\k*60:\r cm) -- (\T+\k*60:\R cm)};
        }
    \end{scope}
    \begin{scope}[yshift=-\h cm]
        \foreach \k in {0,2,4}{
            \filldraw[blue!20, fill opacity=0.5] (\T+\k * 60:\r cm) 
                .. controls (0,.5*\h) .. (\T + \k * 60 + 60:\r cm) -- (\T+\k * 60:\r cm) ;
            \draw[blue] (\T+\k * 60:\r cm) 
                .. controls (0,.5*\h) .. (\T + \k * 60 + 60:\r cm);
        }
    \end{scope}
\end{scope}

\begin{scope}[yshift=10cm]
    \foreach \k in {0,2,4}{
        \directedsingleline{ (\t + \k * 60 + 60:\R cm) 
        .. controls (0,0) .. (\t+\k * 60:\R cm)} ;
    }
    \begin{scope}[yshift=-\h cm]
        \foreach \k in {0,2,4}{
            \draw[double] (\t+\k * 60:\r cm) --  (\t + \k * 60 + 60:\r cm);
            \draw (\t+\k*60 - 60:\r cm) -- (\t +\k*60:\r cm);
        }
        \foreach \k in {0,1,2,3,4,5}{
            \draw (\t+\k*60:\r cm) -- (\t+\k*60:\R cm);
        }
    \end{scope}
    \begin{scope}[yshift=-\h cm]
        \foreach \k in {0,2,4}{
            \filldraw[blue!20, fill opacity=0.5] (\t+\k * 60:\r cm) 
                .. controls (0,.5*\h) .. (\t + \k * 60 + 60:\r cm) -- (\t+\k * 60:\r cm) ;
            \draw[blue] (\t+\k * 60:\r cm) 
                .. controls (0,.5*\h) .. (\t + \k * 60 + 60:\r cm);
        }
    \end{scope}
    \draw[darkgreen!50] (.03,-.575*\h) to[out=\t, in=-90, looseness=.4] (\t+30:.44 cm);
    \draw[darkgreen!50] (-.265,-.635*\h) to[out=\t+120, in=-90, looseness=.5] (\t+150:.44 cm);
    \draw[darkgreen!50] (.2,-.66*\h) to[out=\t-75, in=-90, looseness=.5] (\t+270:.44 cm);
    \begin{scope}[yshift=-.625*\h cm]
        \def \d {.25}; 
        \draw[darkgreen!50, dashed] (82:\d cm) 
            -- (190:\d cm)
            -- (315:\d cm)
            -- (82:\d cm) 
            ;
    \end{scope}
    \foreach \k in {0,1,2,3,4,5}{
        \draw[darkgreen] (\t + \k * 60:2cm) -- +(0,-\h);
    }
\end{scope}

\filldraw (0,6.9) circle (2pt) node[above] {$x^{2N-8-k}$};

\end{tikzpicture}
}

\def\matchingbubbleA
{
\begin{tikzpicture}[decoration={
    markings,
    mark=at position 0.7 with {\arrow{>}}},
    xscale=.5,
    yscale=.25,
    baseline=.5cm
    ]

\def \h{5};
\def \t{45}; 
\def \r{1.2}; 
\def \R{2}; 
\def \T{\t-30};

\foreach \k in {0,2,4}{
    \directedsingleline{ (\T + \k * 60 + 60:\R cm)
    .. controls (0,0) .. (\T+\k * 60:\R cm)  };
}
\foreach \k in {0,1,2,3,4,5}{
     \draw[darkgreen] (\T + \k * 60 + 60:\R cm)
    -- ++(0,6);
}
\draw[darkgreen] (62:.41*\r cm) -- ++(0,6);
\draw[darkgreen] (155:.38*\r cm) -- ++(0,6);
\begin{scope}[yshift=6cm]
    \foreach \k in {0,2,4}{
        \directedsingleline{ (\T + \k * 60 + 60:\R cm)
        .. controls (0,0) .. (\T+\k * 60:\R cm)  };
    }
\end{scope}
\end{tikzpicture}
}

\def\matchingbubbleB
{
\begin{tikzpicture}[decoration={
    markings,
    mark=at position 0.7 with {\arrow{>}}},
    xscale=.5,
    yscale=.25,
    baseline=.5cm
    ]

\def \h{5};
\def \t{45}; 
\def \r{1.2}; 
\def \R{2}; 
\def \T{\t-30};

\foreach \k in {0,2,4}{
    \directedsingleline{ (\T + \k * 60 + 60:\R cm)
    .. controls (0,0) .. (\T+\k * 60:\R cm)  };
}
\foreach \k in {0,1,2,3,4,5}{
     \draw[darkgreen] (\T + \k * 60 + 60:\R cm)
    -- ++(0,6);
}
\draw[darkgreen] (62:.41*\r cm) -- ++(0,6);
\draw[darkgreen] (155:.38*\r cm) -- ++(0,6);
\begin{scope}[yshift=6cm]
    \foreach \k in {0,2,4}{
        \directedsingleline{ (\T + \k * 60 + 60:\R cm)
        .. controls (0,0) .. (\T+\k * 60:\R cm)  };
    }
\end{scope}
\begin{scope}[yscale=2] 
    \filldraw (-1,2) circle (2pt) node[below] {$x$};
\end{scope}
\end{tikzpicture}
}

\def\matchingbubbleC{
\begin{tikzpicture}[decoration={
    markings,
    mark=at position 0.7 with {\arrow{>}}},
    xscale=.5,
    yscale=.25,
    baseline=.5cm
    ]

\def \h{5};
\def \t{45}; 
\def \r{1.2}; 
\def \R{2}; 
\def \T{\t-30};

\foreach \k in {0,2,4}{
    \directedsingleline{ (\T + \k * 60 + 60:\R cm)
    .. controls (0,0) .. (\T+\k * 60:\R cm)  };
}
\foreach \k in {0,1,2,3,4,5}{
     \draw[darkgreen] (\T + \k * 60 + 60:\R cm)
    -- ++(0,6);
}
\draw[darkgreen] (62:.41*\r cm) -- ++(0,6);
\draw[darkgreen] (155:.38*\r cm) -- ++(0,6);
\begin{scope}[yshift=6cm]
    \foreach \k in {0,2,4}{
        \directedsingleline{ (\T + \k * 60 + 60:\R cm)
        .. controls (0,0) .. (\T+\k * 60:\R cm)  };
    }
\end{scope}
\begin{scope}[yscale=2] 
    \filldraw (-.1,1) circle (2pt) node[below] {$x$};
\end{scope}
\end{tikzpicture}
}

\def\matchingbubbleD{
\begin{tikzpicture}[decoration={
    markings,
    mark=at position 0.7 with {\arrow{>}}},
    xscale=.5,
    yscale=.25,
    baseline=.5cm
    ]

\def \h{5};
\def \t{45}; 
\def \r{1.2}; 
\def \R{2}; 
\def \T{\t-30};

\foreach \k in {0,2,4}{
    \directedsingleline{ (\T + \k * 60 + 60:\R cm)
    .. controls (0,0) .. (\T+\k * 60:\R cm)  };
}
\foreach \k in {0,1,2,3,4,5}{
     \draw[darkgreen] (\T + \k * 60 + 60:\R cm)
    -- ++(0,6);
}
\draw[darkgreen] (62:.41*\r cm) -- ++(0,6);
\draw[darkgreen] (155:.38*\r cm) -- ++(0,6);
\begin{scope}[yshift=6cm]
    \foreach \k in {0,2,4}{
        \directedsingleline{ (\T + \k * 60 + 60:\R cm)
        .. controls (0,0) .. (\T+\k * 60:\R cm)  };
    }
\end{scope}
\begin{scope}[yscale=2] 
    \filldraw (1.7,2) circle (2pt) node[below] {$x$};
\end{scope}
\end{tikzpicture}
}

\def\AlphaZ {
\begin{tikzpicture}[draw=darkblue, scale=.4, 
    mid arrow/.style={
            postaction={
                decorate,
                decoration={
                    markings,
                    mark=at position 0.5 with {      
                        \arrow[##1]{>}
                    }
                }
            }
        },
        mid arrow/.default=black, baseline={([yshift=-.8ex]current bounding box.center)}
    ]
\def \r{2}; 
\def \R{4}; 
\def\t{4}; 
\begin{scope}[yscale=.5]
    \def \v {(0,-6)}; 
    \foreach \a in {0,2,4}{
        \draw[double, postaction={mid arrow}] (60*\a+\t:\r) --  (60*\a+60+\t:\r);
        \draw[postaction={mid arrow}] (60*\a+\t:\R) -- (60*\a+\t:\r);
    }
    \foreach \a in {1,3,5}{
        \draw[postaction={mid arrow}] (60*\a+\t:\r) -- (60*\a+60+\t:\r);
        \draw[postaction={mid arrow}] (60*\a+\t:\r) -- (60*\a+\t:\R);
    }
    \foreach \a in {0,1,2,3,4,5}{
        \draw[darkgreen] (60*\a+\t:\r) -- \v;
        \draw[darkgreen, very thin] (60*\a+\t:\R) -- ($(60*\a+\t:\R) + (0,-10)$);
    }
\end{scope}
\draw[very thick, orange] (0,-3) -- (0,-5);
\begin{scope}[yscale=.5, yshift=-10cm]
    \foreach \a in {0,2,4}{
        \draw[postaction={mid arrow}] (60*\a+\t:\R) -- (60*\a+\t:0);
    }
    \foreach \a in {1,3,5}{
        \draw[postaction={mid arrow}] (60*\a+\t:0) -- (60*\a+\t:\R);
    }
\end{scope}
\end{tikzpicture}
}

\def\TAlphaZ {
\begin{tikzpicture}[draw=darkblue, scale=.4, 
    mid arrow/.style={
            postaction={
                decorate,
                decoration={
                    markings,
                    mark=at position 0.5 with {      
                        \arrow[##1]{>}
                    }
                }
            }
        },
        mid arrow/.default=black, baseline={([yshift=-.8ex]current bounding box.center)}
    ]
\def \r{2}; 
\def \R{4}; 
\def\t{184}; 
\begin{scope}[yscale=.5]
    \def \v {(0,-6)}; 
    \foreach \a in {0,2,4}{
        \draw[double, postaction={mid arrow}] (60*\a+\t:\r) --  (60*\a+60+\t:\r);
        \draw[postaction={mid arrow}] (60*\a+\t:\R) -- (60*\a+\t:\r);
    }
    \foreach \a in {1,3,5}{
        \draw[postaction={mid arrow}] (60*\a+\t:\r) -- (60*\a+60+\t:\r);
        \draw[postaction={mid arrow}] (60*\a+\t:\r) -- (60*\a+\t:\R);
    }
    \foreach \a in {0,1,2,3,4,5}{
        \draw[darkgreen] (60*\a+\t:\r) -- \v;
        \draw[darkgreen, very thin] (60*\a+\t:\R) -- ($(60*\a+\t:\R) + (0,-10)$);
    }
\end{scope}
\draw[very thick, orange] (0,-3) -- (0,-5);
\begin{scope}[yscale=.5, yshift=-10cm]
    \foreach \a in {0,2,4}{
        \draw[postaction={mid arrow}] (60*\a+\t:\R) -- (60*\a+\t:0);
    }
    \foreach \a in {1,3,5}{
        \draw[postaction={mid arrow}] (60*\a+\t:0) -- (60*\a+\t:\R);
    }
\end{scope}
\end{tikzpicture}
}

\def\TBetaZ{
\begin{tikzpicture}[draw=darkblue, scale=-.4, 
    mid arrow/.style={
            postaction={
                decorate,
                decoration={
                    markings,
                    mark=at position 0.5 with {      
                        \arrow[##1]{>}
                    }
                }
            }
        },
        mid arrow/.default=black,  baseline={([yshift=-.8ex]current bounding box.center)}
    ]
\def \r{2}; 
\def \R{4}; 
\def\t{4}; 
\begin{scope}[yscale=.5]
    \def \v {(0,-6)}; 
    \foreach \a in {0,2,4}{
        \draw[double, postaction={mid arrow}] (60*\a+\t:\r) --  (60*\a+60+\t:\r);
        \draw[postaction={mid arrow}] (60*\a+\t:\R) -- (60*\a+\t:\r);
    }
    \foreach \a in {1,3,5}{
        \draw[postaction={mid arrow}] (60*\a+\t:\r) -- (60*\a+60+\t:\r);
        \draw[postaction={mid arrow}] (60*\a+\t:\r) -- (60*\a+\t:\R);
    }
    \foreach \a in {0,1,2,3,4,5}{
        \draw[darkgreen] (60*\a+\t:\r) -- \v;
        \draw[darkgreen, very thin] (60*\a+\t:\R) -- ($(60*\a+\t:\R) + (0,-10)$);
    }
\end{scope}
\draw[very thick, orange] (0,-3) -- (0,-5);
\begin{scope}[yscale=.5, yshift=-10cm]
    \foreach \a in {0,2,4}{
        \draw[postaction={mid arrow}] (60*\a+\t:\R) -- (60*\a+\t:0);
    }
    \foreach \a in {1,3,5}{
        \draw[postaction={mid arrow}] (60*\a+\t:0) -- (60*\a+\t:\R);
    }
\end{scope}
\end{tikzpicture}
}

\def\BetaZ{
\begin{tikzpicture}[draw=darkblue, scale=-.4, 
    mid arrow/.style={
            postaction={
                decorate,
                decoration={
                    markings,
                    mark=at position 0.5 with {      
                        \arrow[##1]{>}
                    }
                }
            }
        },
        mid arrow/.default=black,  baseline={([yshift=-.8ex]current bounding box.center)}
    ]
\def \r{2}; 
\def \R{4}; 
\def\t{184}; 
\begin{scope}[yscale=.5]
    \def \v {(0,-6)}; 
    \foreach \a in {0,2,4}{
        \draw[double, postaction={mid arrow}] (60*\a+\t:\r) --  (60*\a+60+\t:\r);
        \draw[postaction={mid arrow}] (60*\a+\t:\R) -- (60*\a+\t:\r);
    }
    \foreach \a in {1,3,5}{
        \draw[postaction={mid arrow}] (60*\a+\t:\r) -- (60*\a+60+\t:\r);
        \draw[postaction={mid arrow}] (60*\a+\t:\r) -- (60*\a+\t:\R);
    }
    \foreach \a in {0,1,2,3,4,5}{
        \draw[darkgreen] (60*\a+\t:\r) -- \v;
        \draw[darkgreen, very thin] (60*\a+\t:\R) -- ($(60*\a+\t:\R) + (0,-10)$);
    }
\end{scope}
\draw[very thick, orange] (0,-3) -- (0,-5);
\begin{scope}[yscale=.5, yshift=-10cm]
    \foreach \a in {0,2,4}{
        \draw[postaction={mid arrow}] (60*\a+\t:\R) -- (60*\a+\t:0);
    }
    \foreach \a in {1,3,5}{
        \draw[postaction={mid arrow}] (60*\a+\t:0) -- (60*\a+\t:\R);
    }
\end{scope}
\end{tikzpicture}
}

\def\AlphaZTBetaZDef{
\begin{tikzpicture}[draw=darkblue, scale=.4, xscale=-1, yscale=1.5,
    mid arrow/.style={
            postaction={
                decorate,
                decoration={
                    markings,
                    mark=at position 0.5 with {      
                        \arrow[##1]{>}
                    }
                }
            }
        },
        mid arrow/.default=black, 
        baseline={([yshift=-.8ex]current bounding box.center)}
    ]
\def \r{2}; 
\def \R{4}; 
\def\t{-7}; 
\begin{scope}[yscale=.5, yshift=-10cm]
    \foreach \a in {0,2,4}{
        \draw[postaction={mid arrow}] (60*\a+60+\t:\r) -- (60*\a+\t:\r)  ;
        \draw[postaction={mid arrow}] (60*\a+\t:\R) -- (60*\a+\t:\r);
    }
    \foreach \a in {1,3,5}{
        \draw[postaction={mid arrow}, double] (60*\a+60+\t:\r) -- (60*\a+\t:\r) ;
        \draw[postaction={mid arrow}] (60*\a+\t:\r) -- (60*\a+\t:\R);
    }
\end{scope}
\begin{scope}[yshift=-2.5 cm, yscale=.5]
    \foreach \a in {0,1,2,3,4,5}{
        \draw (60*\a+\t:\r) --  (60*\a+60+\t:\r);
        \draw[dotted] (60*\a+\t:\R) -- (60*\a+\t:\r);
    }
\end{scope}
\begin{scope}[yscale=.5]
    \foreach \a in {0,2,4}{
        \draw[double, postaction={mid arrow}] (60*\a+\t:\r) --  (60*\a+60+\t:\r);
        \draw[postaction={mid arrow}] (60*\a+\t:\R) -- (60*\a+\t:\r);
    }
    \foreach \a in {1,3,5}{
        \draw[postaction={mid arrow}] (60*\a+\t:\r) -- (60*\a+60+\t:\r);
        \draw[postaction={mid arrow}] (60*\a+\t:\r) -- (60*\a+\t:\R);
    }
    \foreach \a in {0,1,2,3,4,5}{
        \draw[darkgreen, very thin] (60*\a+\t:\R) -- ($(60*\a+\t:\R) + (0,-10)$);
        \draw[darkgreen, very thin] (60*\a+\t:\r) -- ($(60*\a+\t:\r) + (0,-10)$);
    }
\end{scope}
\end{tikzpicture}
}

\def\AlphaZTBetaZ{
\begin{tikzpicture}[draw=darkblue, scale=.4, 
    mid arrow/.style={
            postaction={
                decorate,
                decoration={
                    markings,
                    mark=at position 0.5 with {      
                        \arrow[##1]{>}
                    }
                }
            }
        },
        mid arrow/.default=black,   baseline={([yshift=-.8ex]current bounding box.center)}
    ]
\begin{scope}[xscale=-1] 
    \def \r{2}; 
    \def \R{4}; 
    \def\t{-4}; 
    \begin{scope}[yscale=.5]
        \def \v {(0,-6)}; 
        \foreach \a in {0,2,4}{
            \draw[double, postaction={mid arrow}] (60*\a+\t:\r) --  (60*\a+60+\t:\r);
            \draw[postaction={mid arrow}] (60*\a+\t:\R) -- (60*\a+\t:\r);
        }
        \foreach \a in {1,3,5}{
            \draw[postaction={mid arrow}] (60*\a+\t:\r) -- (60*\a+60+\t:\r);
            \draw[postaction={mid arrow}] (60*\a+\t:\r) -- (60*\a+\t:\R);
        }
        \foreach \a in {0,1,2,3,4,5}{
            \draw[darkgreen] (60*\a+\t:\r) -- \v;
            \draw[darkgreen, very thin] (60*\a+\t:\R) -- ($(60*\a+\t:\R) + (0,-10)$);
        }
    \end{scope}
    \draw[very thick, orange] (0,-3) -- (0,-5);
    \begin{scope}[yscale=.5, yshift=-10cm]
        \foreach \a in {0,2,4}{
            \draw[postaction={mid arrow}] (60*\a+\t:\R) -- (60*\a+\t:0);
        }
        \foreach \a in {1,3,5}{
            \draw[postaction={mid arrow}] (60*\a+\t:0) -- (60*\a+\t:\R);
        }
    \end{scope}
\end{scope}
\begin{scope}[scale=-1, yshift=10cm]
    \def \r{2}; 
    \def \R{4}; 
    \def\t{4}; 
    \begin{scope}[yscale=.5]
        \def \v {(0,-6)}; 
        \foreach \a in {0,2,4}{
            \draw[double, postaction={mid arrow}] (60*\a+\t:\r) --  (60*\a+60+\t:\r);
            \draw[postaction={mid arrow}] (60*\a+\t:\R) -- (60*\a+\t:\r);
        }
        \foreach \a in {1,3,5}{
            \draw[postaction={mid arrow}] (60*\a+\t:\r) -- (60*\a+60+\t:\r);
            \draw[postaction={mid arrow}] (60*\a+\t:\r) -- (60*\a+\t:\R);
        }
        \foreach \a in {0,1,2,3,4,5}{
            \draw[darkgreen] (60*\a+\t:\r) -- \v;
            \draw[darkgreen, very thin] (60*\a+\t:\R) -- ($(60*\a+\t:\R) + (0,-10)$);
        }
    \end{scope}
    \draw[very thick, orange] (0,-3) -- (0,-5);
    \begin{scope}[yscale=.5, yshift=-10cm]
        \foreach \a in {0,2,4}{
            \draw[postaction={mid arrow}] (60*\a+\t:\R) -- (60*\a+\t:0);
        }
        \foreach \a in {1,3,5}{
            \draw[postaction={mid arrow}] (60*\a+\t:0) -- (60*\a+\t:\R);
        }
    \end{scope}
\end{scope}
\end{tikzpicture}
}

\def\TAlphaZBetaZDef{
\begin{tikzpicture}[draw=darkblue, scale=.4, xscale=-1, yscale=1.5,
    mid arrow/.style={
            postaction={
                decorate,
                decoration={
                    markings,
                    mark=at position 0.5 with {      
                        \arrow[##1]{>}
                    }
                }
            }
        },
        mid arrow/.default=black, 
        baseline={([yshift=-.8ex]current bounding box.center)}
    ]
\def \r{2}; 
\def \R{4}; 
\def\t{180-7}; 
\begin{scope}[yscale=.5, yshift=-10cm]
    \foreach \a in {0,2,4}{
        \draw[postaction={mid arrow}] (60*\a+\t:\r) --  (60*\a+60+\t:\r) ;
        \draw[postaction={mid arrow}] (60*\a+\t:\r) -- (60*\a+\t:\R);
    }
    \foreach \a in {1,3,5}{
        \draw[postaction={mid arrow}, double] (60*\a+\t:\r) --  (60*\a+60+\t:\r);
        \draw[postaction={mid arrow}] (60*\a+\t:\R) -- (60*\a+\t:\r);
    }
\end{scope}
\begin{scope}[yshift=-2.5 cm, yscale=.5]
    \foreach \a in {0,1,2,3,4,5}{
        \draw (60*\a+\t:\r) --  (60*\a+60+\t:\r);
        \draw[dotted] (60*\a+\t:\R) -- (60*\a+\t:\r);
    }
\end{scope}
\begin{scope}[yscale=.5]
    \foreach \a in {0,2,4}{
        \draw[double, postaction={mid arrow}] (60*\a+60+\t:\r) -- (60*\a+\t:\r) ;
        \draw[postaction={mid arrow}]  (60*\a+\t:\r) -- (60*\a+\t:\R);
    }
    \foreach \a in {1,3,5}{
        \draw[postaction={mid arrow}] (60*\a+60+\t:\r) -- (60*\a+\t:\r);
        \draw[postaction={mid arrow}] (60*\a+\t:\R) -- (60*\a+\t:\r);
    }
    \foreach \a in {0,1,2,3,4,5}{
        \draw[darkgreen, very thin] (60*\a+\t:\R) -- ($(60*\a+\t:\R) + (0,-10)$);
        \draw[darkgreen, very thin] (60*\a+\t:\r) -- ($(60*\a+\t:\r) + (0,-10)$);
    }
\end{scope}
\end{tikzpicture}
}

\def\TAlphaZBetaZ{
\begin{tikzpicture}[draw=darkblue, scale=.4, 
    mid arrow/.style={
            postaction={
                decorate,
                decoration={
                    markings,
                    mark=at position 0.5 with {      
                        \arrow[##1]{>}
                    }
                }
            }
        },
        mid arrow/.default=black,   baseline={([yshift=-.8ex]current bounding box.center)}
    ]
\begin{scope}[xscale=-1] 
    \def \r{2}; 
    \def \R{4}; 
    \def\t{180-4}; 
    \begin{scope}[yscale=.5]
        \def \v {(0,-6)}; 
        \foreach \a in {0,2,4}{
            \draw[double, postaction={mid arrow}] (60*\a+\t:\r) --  (60*\a+60+\t:\r);
            \draw[postaction={mid arrow}] (60*\a+\t:\R) -- (60*\a+\t:\r);
        }
        \foreach \a in {1,3,5}{
            \draw[postaction={mid arrow}] (60*\a+\t:\r) -- (60*\a+60+\t:\r);
            \draw[postaction={mid arrow}] (60*\a+\t:\r) -- (60*\a+\t:\R);
        }
        \foreach \a in {0,1,2,3,4,5}{
            \draw[darkgreen] (60*\a+\t:\r) -- \v;
            \draw[darkgreen, very thin] (60*\a+\t:\R) -- ($(60*\a+\t:\R) + (0,-10)$);
        }
    \end{scope}
    \draw[very thick, orange] (0,-3) -- (0,-5);
    \begin{scope}[yscale=.5, yshift=-10cm]
        \foreach \a in {0,2,4}{
            \draw[postaction={mid arrow}] (60*\a+\t:\R) -- (60*\a+\t:0);
        }
        \foreach \a in {1,3,5}{
            \draw[postaction={mid arrow}] (60*\a+\t:0) -- (60*\a+\t:\R);
        }
    \end{scope}
\end{scope}
\begin{scope}[scale=-1, yshift=10cm]
    \def \r{2}; 
    \def \R{4}; 
    \def\t{180+4}; 
    \begin{scope}[yscale=.5]
        \def \v {(0,-6)}; 
        \foreach \a in {0,2,4}{
            \draw[double, postaction={mid arrow}] (60*\a+\t:\r) --  (60*\a+60+\t:\r);
            \draw[postaction={mid arrow}] (60*\a+\t:\R) -- (60*\a+\t:\r);
        }
        \foreach \a in {1,3,5}{
            \draw[postaction={mid arrow}] (60*\a+\t:\r) -- (60*\a+60+\t:\r);
            \draw[postaction={mid arrow}] (60*\a+\t:\r) -- (60*\a+\t:\R);
        }
        \foreach \a in {0,1,2,3,4,5}{
            \draw[darkgreen] (60*\a+\t:\r) -- \v;
            \draw[darkgreen, very thin] (60*\a+\t:\R) -- ($(60*\a+\t:\R) + (0,-10)$);
        }
    \end{scope}
    \draw[very thick, orange] (0,-3) -- (0,-5);
    \begin{scope}[yscale=.5, yshift=-10cm]
        \foreach \a in {0,2,4}{
            \draw[postaction={mid arrow}] (60*\a+\t:\R) -- (60*\a+\t:0);
        }
        \foreach \a in {1,3,5}{
            \draw[postaction={mid arrow}] (60*\a+\t:0) -- (60*\a+\t:\R);
        }
    \end{scope}
\end{scope}
\end{tikzpicture}
}

\def\AlphaO{
\begin{tikzpicture}[scale=0.85,yscale=.75,xscale=-1, draw=darkblue,
    decoration={
    markings},
    baseline={([yshift=-.8ex]current bounding box.center)}
    ]
\def \h{5};
\def \t{230}; 
\def \r{1.2}; 
\def \R{2}; 
%
\begin{scope}[yshift=-\h cm]
    \foreach \k in {0,2,4}{
        \draw[double, postaction={decorate}, decoration={
            markings,
            mark=at position 0.5 with {\arrow{>}}
        }] (\t+\k * 60:\r cm) --  (\t + \k * 60 + 60:\r cm);
        \draw[postaction={decorate}, decoration={
            markings,
            mark=at position 0.5 with {\arrow{>}}
        }] (\t+\k*60 - 60:\r cm) -- (\t +\k*60:\r cm);
    }
    \foreach \k in {0,2,4}{
        \draw[postaction={decorate}, decoration={
            markings,
            mark=at position 0.7 with {\arrow{<}}
        }] (\t+\k*60:\r cm) -- (\t+\k*60:\R cm);
    }
    \foreach \k in {1,3,5}{
        \draw[postaction={decorate}, decoration={
            markings,
            mark=at position 0.7 with {\arrow{>}}
        }] (\t+\k*60:\r cm) -- (\t+\k*60:\R cm);
    }
\end{scope}
\begin{scope}[yshift=-\h cm]
    \foreach \k in {0,2,4}{
        \filldraw[green!20, fill opacity=0.5] (\t+\k * 60:\r cm) 
            .. controls (0,.5*\h) .. (\t + \k * 60 + 60:\r cm) -- (\t+\k * 60:\r cm) ;
        \draw[green] (\t+\k * 60:\r cm) 
            .. controls (0,.5*\h) .. (\t + \k * 60 + 60:\r cm);
    }
\end{scope}
\draw[darkgreen!50] (-0.045,-.675*\h)--(\t+30:.44 cm);
\draw[darkgreen!50] (.22,-.6*\h) to[out=\t+120, in=-90, looseness=.5] (\t+150:.44 cm);
\draw[darkgreen!50] (-.18,-.585*\h) to[out=\t-75, in=-90, looseness=.5] (\t+270:.44 cm);
\begin{scope}[yshift=-.625*\h cm]
    \def \d {.25}; 
    \draw[darkgreen!50, dashed] (-98:\d cm) 
        -- (24:\d cm)
        -- (135:\d cm)
        -- (-98:\d cm) 
        ;
\end{scope}
\foreach \k in {0,1,2,3,4,5}{
    \draw[darkgreen] (\t + \k * 60:2cm) -- +(0,-\h);
}
\foreach \k in {0,2,4}{
\draw[
        postaction={decorate}, decoration={
            markings,
            mark=at position 0.7 with {\arrow{>}}
        }
     ]
    (\t+\k * 60:\R cm) 
    .. controls (0,0) ..  (\t + \k * 60 + 60:\R cm);
}
\end{tikzpicture}
}

\def\AlphaTwoA{
\begin{tikzpicture}[scale=0.85,yscale=.75,xscale=-1, draw=darkblue,
    decoration={
    markings},
    baseline={([yshift=-.8ex]current bounding box.center)}
    ]
\def \h{5};
\def \t{230}; 
\def \r{1.2}; 
\def \R{2}; 
%
\filldraw[black] (1.6,-1.2) circle (2pt) node[black,scale=0.7, above]{$x$};
\begin{scope}[yshift=-\h cm]
    \foreach \k in {0,2,4}{
        \draw[double, postaction={decorate}, decoration={
            markings,
            mark=at position 0.5 with {\arrow{>}}
        }] (\t+\k * 60:\r cm) --  (\t + \k * 60 + 60:\r cm);
        \draw[postaction={decorate}, decoration={
            markings,
            mark=at position 0.5 with {\arrow{>}}
        }] (\t+\k*60 - 60:\r cm) -- (\t +\k*60:\r cm);
    }
    \foreach \k in {0,2,4}{
        \draw[postaction={decorate}, decoration={
            markings,
            mark=at position 0.7 with {\arrow{<}}
        }] (\t+\k*60:\r cm) -- (\t+\k*60:\R cm);
    }
    \foreach \k in {1,3,5}{
        \draw[postaction={decorate}, decoration={
            markings,
            mark=at position 0.7 with {\arrow{>}}
        }] (\t+\k*60:\r cm) -- (\t+\k*60:\R cm);
    }
\end{scope}
\begin{scope}[yshift=-\h cm]
    \foreach \k in {0,2,4}{
        \filldraw[green!20, fill opacity=0.5] (\t+\k * 60:\r cm) 
            .. controls (0,.5*\h) .. (\t + \k * 60 + 60:\r cm) -- (\t+\k * 60:\r cm) ;
        \draw[green] (\t+\k * 60:\r cm) 
            .. controls (0,.5*\h) .. (\t + \k * 60 + 60:\r cm);
    }
\end{scope}
\draw[darkgreen!50] (-0.045,-.675*\h)--(\t+30:.44 cm);
\draw[darkgreen!50] (.22,-.6*\h) to[out=\t+120, in=-90, looseness=.5] (\t+150:.44 cm);
\draw[darkgreen!50] (-.18,-.585*\h) to[out=\t-75, in=-90, looseness=.5] (\t+270:.44 cm);
\begin{scope}[yshift=-.625*\h cm]
    \def \d {.25}; 
    \draw[darkgreen!50, dashed] (-98:\d cm) 
        -- (24:\d cm)
        -- (135:\d cm)
        -- (-98:\d cm) 
        ;
\end{scope}
\foreach \k in {0,1,2,3,4,5}{
    \draw[darkgreen] (\t + \k * 60:2cm) -- +(0,-\h);
}
\foreach \k in {0,2,4}{
\draw[
        postaction={decorate}, decoration={
            markings,
            mark=at position 0.7 with {\arrow{>}}
        }
     ]
    (\t+\k * 60:\R cm) 
    .. controls (0,0) ..  (\t + \k * 60 + 60:\R cm);
}
\end{tikzpicture}
}

\def\AlphaTwoB{
\begin{tikzpicture}[scale=0.85,yscale=.75,xscale=-1, draw=darkblue,
    decoration={
    markings},
    baseline={([yshift=-.8ex]current bounding box.center)}
    ]
\def \h{5};
\def \t{230}; 
\def \r{1.2}; 
\def \R{2}; 
%
\filldraw[black] (0.15,-1.5) circle (2pt) node[black,scale=0.7, above]{$x$};
\begin{scope}[yshift=-\h cm]
    \foreach \k in {0,2,4}{
        \draw[double, postaction={decorate}, decoration={
            markings,
            mark=at position 0.5 with {\arrow{>}}
        }] (\t+\k * 60:\r cm) --  (\t + \k * 60 + 60:\r cm);
        \draw[postaction={decorate}, decoration={
            markings,
            mark=at position 0.5 with {\arrow{>}}
        }] (\t+\k*60 - 60:\r cm) -- (\t +\k*60:\r cm);
    }
    \foreach \k in {0,2,4}{
        \draw[postaction={decorate}, decoration={
            markings,
            mark=at position 0.7 with {\arrow{<}}
        }] (\t+\k*60:\r cm) -- (\t+\k*60:\R cm);
    }
    \foreach \k in {1,3,5}{
        \draw[postaction={decorate}, decoration={
            markings,
            mark=at position 0.7 with {\arrow{>}}
        }] (\t+\k*60:\r cm) -- (\t+\k*60:\R cm);
    }
\end{scope}
\begin{scope}[yshift=-\h cm]
    \foreach \k in {0,2,4}{
        \filldraw[green!20, fill opacity=0.5] (\t+\k * 60:\r cm) 
            .. controls (0,.5*\h) .. (\t + \k * 60 + 60:\r cm) -- (\t+\k * 60:\r cm) ;
        \draw[green] (\t+\k * 60:\r cm) 
            .. controls (0,.5*\h) .. (\t + \k * 60 + 60:\r cm);
    }
\end{scope}
\draw[darkgreen!50] (-0.045,-.675*\h)--(\t+30:.44 cm);
\draw[darkgreen!50] (.22,-.6*\h) to[out=\t+120, in=-90, looseness=.5] (\t+150:.44 cm);
\draw[darkgreen!50] (-.18,-.585*\h) to[out=\t-75, in=-90, looseness=.5] (\t+270:.44 cm);
\begin{scope}[yshift=-.625*\h cm]
    \def \d {.25}; 
    \draw[darkgreen!50, dashed] (-98:\d cm) 
        -- (24:\d cm)
        -- (135:\d cm)
        -- (-98:\d cm) 
        ;
\end{scope}
\foreach \k in {0,1,2,3,4,5}{
    \draw[darkgreen] (\t + \k * 60:2cm) -- +(0,-\h);
}
\foreach \k in {0,2,4}{
\draw[
        postaction={decorate}, decoration={
            markings,
            mark=at position 0.7 with {\arrow{>}}
        }
     ]
    (\t+\k * 60:\R cm) 
    .. controls (0,0) ..  (\t + \k * 60 + 60:\R cm);
}
\end{tikzpicture}
}

\def\AlphaTwoC{
\begin{tikzpicture}[scale=0.85,yscale=.75,xscale=-1, draw=darkblue,
    decoration={
    markings},
    baseline={([yshift=-.8ex]current bounding box.center)}
    ]
\def \h{5};
\def \t{230}; 
\def \r{1.2}; 
\def \R{2}; 
%
\filldraw[black] (-1.4,-0.6) circle (2pt) node[black,scale=0.7, above]{$x$};
\begin{scope}[yshift=-\h cm]
    \foreach \k in {0,2,4}{
        \draw[double, postaction={decorate}, decoration={
            markings,
            mark=at position 0.5 with {\arrow{>}}
        }] (\t+\k * 60:\r cm) --  (\t + \k * 60 + 60:\r cm);
        \draw[postaction={decorate}, decoration={
            markings,
            mark=at position 0.5 with {\arrow{>}}
        }] (\t+\k*60 - 60:\r cm) -- (\t +\k*60:\r cm);
    }
    \foreach \k in {0,2,4}{
        \draw[postaction={decorate}, decoration={
            markings,
            mark=at position 0.7 with {\arrow{<}}
        }] (\t+\k*60:\r cm) -- (\t+\k*60:\R cm);
    }
    \foreach \k in {1,3,5}{
        \draw[postaction={decorate}, decoration={
            markings,
            mark=at position 0.7 with {\arrow{>}}
        }] (\t+\k*60:\r cm) -- (\t+\k*60:\R cm);
    }
\end{scope}
\begin{scope}[yshift=-\h cm]
    \foreach \k in {0,2,4}{
        \filldraw[green!20, fill opacity=0.5] (\t+\k * 60:\r cm) 
            .. controls (0,.5*\h) .. (\t + \k * 60 + 60:\r cm) -- (\t+\k * 60:\r cm) ;
        \draw[green] (\t+\k * 60:\r cm) 
            .. controls (0,.5*\h) .. (\t + \k * 60 + 60:\r cm);
    }
\end{scope}
\draw[darkgreen!50] (-0.045,-.675*\h)--(\t+30:.44 cm);
\draw[darkgreen!50] (.22,-.6*\h) to[out=\t+120, in=-90, looseness=.5] (\t+150:.44 cm);
\draw[darkgreen!50] (-.18,-.585*\h) to[out=\t-75, in=-90, looseness=.5] (\t+270:.44 cm);
\begin{scope}[yshift=-.625*\h cm]
    \def \d {.25}; 
    \draw[darkgreen!50, dashed] (-98:\d cm) 
        -- (24:\d cm)
        -- (135:\d cm)
        -- (-98:\d cm) 
        ;
\end{scope}
\foreach \k in {0,1,2,3,4,5}{
    \draw[darkgreen] (\t + \k * 60:2cm) -- +(0,-\h);
}
\foreach \k in {0,2,4}{
\draw[
        postaction={decorate}, decoration={
            markings,
            mark=at position 0.7 with {\arrow{>}}
        }
     ]
    (\t+\k * 60:\R cm) 
    .. controls (0,0) ..  (\t + \k * 60 + 60:\R cm);
}
\end{tikzpicture}
}

\def\AlphaTwo{
\begin{tikzpicture}[scale=0.85,yscale=.75,xscale=-1, draw=darkblue,
    decoration={
    markings},
    baseline={([yshift=-.8ex]current bounding box.center)}
    ]
\def \h{5};
\def \t{230}; 
\def \r{1.2}; 
\def \R{2}; 
%
\filldraw[black] (0.86,-5.2) circle (2pt) node[black,scale=0.7, above]{ $x$};
\begin{scope}[yshift=-\h cm]
    \foreach \k in {0,2,4}{
        \draw[double, postaction={decorate}, decoration={
            markings,
            mark=at position 0.5 with {\arrow{>}}
        }] (\t+\k * 60:\r cm) --  (\t + \k * 60 + 60:\r cm);
        \draw[postaction={decorate}, decoration={
            markings,
            mark=at position 0.5 with {\arrow{>}}
        }] (\t+\k*60 - 60:\r cm) -- (\t +\k*60:\r cm);
    }
    \foreach \k in {0,2,4}{
        \draw[postaction={decorate}, decoration={
            markings,
            mark=at position 0.7 with {\arrow{<}}
        }] (\t+\k*60:\r cm) -- (\t+\k*60:\R cm);
    }
    \foreach \k in {1,3,5}{
        \draw[postaction={decorate}, decoration={
            markings,
            mark=at position 0.7 with {\arrow{>}}
        }] (\t+\k*60:\r cm) -- (\t+\k*60:\R cm);
    }
\end{scope}
\begin{scope}[yshift=-\h cm]
    \foreach \k in {0,2,4}{
        \filldraw[green!20, fill opacity=0.5] (\t+\k * 60:\r cm) 
            .. controls (0,.5*\h) .. (\t + \k * 60 + 60:\r cm) -- (\t+\k * 60:\r cm) ;
        \draw[green] (\t+\k * 60:\r cm) 
            .. controls (0,.5*\h) .. (\t + \k * 60 + 60:\r cm);
    }
\end{scope}
\draw[darkgreen!50] (-0.045,-.675*\h)--(\t+30:.44 cm);
\draw[darkgreen!50] (.22,-.6*\h) to[out=\t+120, in=-90, looseness=.5] (\t+150:.44 cm);
\draw[darkgreen!50] (-.18,-.585*\h) to[out=\t-75, in=-90, looseness=.5] (\t+270:.44 cm);
\begin{scope}[yshift=-.625*\h cm]
    \def \d {.25}; 
    \draw[darkgreen!50, dashed] (-98:\d cm) 
        -- (24:\d cm)
        -- (135:\d cm)
        -- (-98:\d cm) 
        ;
\end{scope}
\foreach \k in {0,1,2,3,4,5}{
    \draw[darkgreen] (\t + \k * 60:2cm) -- +(0,-\h);
}
\foreach \k in {0,2,4}{
\draw[
        postaction={decorate}, decoration={
            markings,
            mark=at position 0.7 with {\arrow{>}}
        }
     ]
    (\t+\k * 60:\R cm) 
    .. controls (0,0) ..  (\t + \k * 60 + 60:\R cm);
}
\end{tikzpicture}
}

\def\TAlphaO{
\begin{tikzpicture}[scale=0.85,yscale=.75, xscale=-1, draw=darkblue,
    decoration={
    markings},
    baseline={([yshift=-.8ex]current bounding box.center)}
    ]
\def \h{5};
\def \t{50}; 
\def \r{1.2}; 
\def \R{2}; 
%
\begin{scope}[yshift=-\h cm]
    \foreach \k in {0,2,4}{
        \draw[double,  postaction={decorate}, decoration={
            markings,
            mark=at position 0.5 with {\arrow{<}}
        }] (\t+\k * 60:\r cm) --  (\t + \k * 60 + 60:\r cm);
        \draw[ postaction={decorate}, decoration={
            markings,
            mark=at position 0.5 with {\arrow{<}}
        }] (\t+\k*60 - 60:\r cm) -- (\t +\k*60:\r cm);
    }
    \foreach \k in {0,2,4}{
        \draw[postaction={decorate}, decoration={
            markings,
            mark=at position 0.7 with {\arrow{>}}
        }] (\t+\k*60:\r cm) -- (\t+\k*60:\R cm);
    }
    \foreach \k in {1,3,5}{
        \draw[postaction={decorate}, decoration={
            markings,
            mark=at position 0.7 with {\arrow{<}}
        }] (\t+\k*60:\r cm) -- (\t+\k*60:\R cm);
    }
\end{scope}
\begin{scope}[yshift=-\h cm]
    \foreach \k in {0,2,4}{
        \filldraw[green!20, fill opacity=0.5] (\t+\k * 60:\r cm) 
            .. controls (0,.5*\h) .. (\t + \k * 60 + 60:\r cm) -- (\t+\k * 60:\r cm) ;
        \draw[green] (\t+\k * 60:\r cm) 
            .. controls (0,.5*\h) .. (\t + \k * 60 + 60:\r cm);
    }
\end{scope}
\draw[darkgreen!50] (.03,-.575*\h) to[out=\t, in=-90, looseness=.4] (\t+30:.44 cm);
\draw[darkgreen!50] (-.265,-.635*\h) to[out=\t+120, in=-90, looseness=.5] (\t+155:.44 cm);
\draw[darkgreen!50] (.22,-.655*\h) to[out=\t-75, in=-100, looseness=.5] (\t+257.5:.46 cm);
\begin{scope}[yshift=-.625*\h cm]
    \def \d {.25}; 
    \draw[darkgreen!50, dashed] (82:\d cm) 
        -- (190:\d cm)
        -- (320:\d cm)
        -- (82:\d cm) 
        ;
\end{scope}
\foreach \k in {0,1,2,3,4,5}{
    \draw[darkgreen] (\t + \k * 60:2cm) -- +(0,-\h);
}
\foreach \k in {0,2,4}{
    \draw[
        postaction={decorate},
        decoration={
            markings,
            mark=at position 0.7 with {\arrow{<}}
        }
    ]
    (\t+\k * 60:\R cm)
    .. controls (0,0) ..
    (\t + \k * 60 + 60:\R cm);
}
\end{tikzpicture}
}

\def\TAlphaTwo{
\begin{tikzpicture}[scale=0.85,yscale=.75, xscale=-1, draw=darkblue,
    decoration={
    markings},
    baseline={([yshift=-.8ex]current bounding box.center)}
    ]
\def \h{5};
\def \t{50}; 
\def \r{1.2}; 
\def \R{2}; 
%
\filldraw[black] (-0.15,-5.5) circle (2pt) node[black,scale=0.7, above]{$x$};
\begin{scope}[yshift=-\h cm]
    \foreach \k in {0,2,4}{
        \draw[double,  postaction={decorate}, decoration={
            markings,
            mark=at position 0.5 with {\arrow{<}}
        }] (\t+\k * 60:\r cm) --  (\t + \k * 60 + 60:\r cm);
        \draw[ postaction={decorate}, decoration={
            markings,
            mark=at position 0.5 with {\arrow{<}}
        }] (\t+\k*60 - 60:\r cm) -- (\t +\k*60:\r cm);
    }
    \foreach \k in {0,2,4}{
        \draw[postaction={decorate}, decoration={
            markings,
            mark=at position 0.7 with {\arrow{>}}
        }] (\t+\k*60:\r cm) -- (\t+\k*60:\R cm);
    }
    \foreach \k in {1,3,5}{
        \draw[postaction={decorate}, decoration={
            markings,
            mark=at position 0.7 with {\arrow{<}}
        }] (\t+\k*60:\r cm) -- (\t+\k*60:\R cm);
    }
\end{scope}
\begin{scope}[yshift=-\h cm]
    \foreach \k in {0,2,4}{
        \filldraw[green!20, fill opacity=0.5] (\t+\k * 60:\r cm) 
            .. controls (0,.5*\h) .. (\t + \k * 60 + 60:\r cm) -- (\t+\k * 60:\r cm) ;
        \draw[green] (\t+\k * 60:\r cm) 
            .. controls (0,.5*\h) .. (\t + \k * 60 + 60:\r cm);
    }
\end{scope}
\draw[darkgreen!50] (-0.045,-.675*\h)--(\t+30:.44 cm);
\draw[darkgreen!50] (.22,-.6*\h) to[out=\t+120, in=-90, looseness=.5] (\t+150:.44 cm);
\draw[darkgreen!50] (-.18,-.585*\h) to[out=\t-75, in=-90, looseness=.5] (\t+270:.44 cm);
\begin{scope}[yshift=-.625*\h cm]
    \def \d {.25}; 
    \draw[darkgreen!50, dashed] (-98:\d cm) 
        -- (24:\d cm)
        -- (135:\d cm)
        -- (-98:\d cm) 
        ;
\end{scope}
\foreach \k in {0,1,2,3,4,5}{
    \draw[darkgreen] (\t + \k * 60:2cm) -- +(0,-\h);
}
\foreach \k in {0,2,4}{
    \draw[
        postaction={decorate},
        decoration={
            markings,
            mark=at position 0.7 with {\arrow{<}}
        }
    ]
    (\t+\k * 60:\R cm)
    .. controls (0,0) ..
    (\t + \k * 60 + 60:\R cm);
}
\end{tikzpicture}
}

\def\TAlphaTwoA{
\begin{tikzpicture}[scale=0.85,yscale=.75, xscale=-1, draw=darkblue,
    decoration={
    markings},
    baseline={([yshift=-.8ex]current bounding box.center)}
    ]
\def \h{5};
\def \t{50}; 
\def \r{1.2}; 
\def \R{2}; 
%
\filldraw[black] (1.6,-1) circle (2pt) node[black,scale=0.7, above]{$x$};
\begin{scope}[yshift=-\h cm]
    \foreach \k in {0,2,4}{
        \draw[double,  postaction={decorate}, decoration={
            markings,
            mark=at position 0.5 with {\arrow{<}}
        }] (\t+\k * 60:\r cm) --  (\t + \k * 60 + 60:\r cm);
        \draw[ postaction={decorate}, decoration={
            markings,
            mark=at position 0.5 with {\arrow{<}}
        }] (\t+\k*60 - 60:\r cm) -- (\t +\k*60:\r cm);
    }
    \foreach \k in {0,2,4}{
        \draw[postaction={decorate}, decoration={
            markings,
            mark=at position 0.7 with {\arrow{>}}
        }] (\t+\k*60:\r cm) -- (\t+\k*60:\R cm);
    }
    \foreach \k in {1,3,5}{
        \draw[postaction={decorate}, decoration={
            markings,
            mark=at position 0.7 with {\arrow{<}}
        }] (\t+\k*60:\r cm) -- (\t+\k*60:\R cm);
    }
\end{scope}
\begin{scope}[yshift=-\h cm]
    \foreach \k in {0,2,4}{
        \filldraw[green!20, fill opacity=0.5] (\t+\k * 60:\r cm) 
            .. controls (0,.5*\h) .. (\t + \k * 60 + 60:\r cm) -- (\t+\k * 60:\r cm) ;
        \draw[green] (\t+\k * 60:\r cm) 
            .. controls (0,.5*\h) .. (\t + \k * 60 + 60:\r cm);
    }
\end{scope}
\draw[darkgreen!50] (-0.045,-.675*\h)--(\t+30:.44 cm);
\draw[darkgreen!50] (.22,-.6*\h) to[out=\t+120, in=-90, looseness=.5] (\t+150:.44 cm);
\draw[darkgreen!50] (-.18,-.585*\h) to[out=\t-75, in=-90, looseness=.5] (\t+270:.44 cm);
\begin{scope}[yshift=-.625*\h cm]
    \def \d {.25}; 
    \draw[darkgreen!50, dashed] (-98:\d cm) 
        -- (24:\d cm)
        -- (135:\d cm)
        -- (-98:\d cm) 
        ;
\end{scope}
\foreach \k in {0,1,2,3,4,5}{
    \draw[darkgreen] (\t + \k * 60:2cm) -- +(0,-\h);
}
\foreach \k in {0,2,4}{
    \draw[
        postaction={decorate},
        decoration={
            markings,
            mark=at position 0.7 with {\arrow{<}}
        }
    ]
    (\t+\k * 60:\R cm)
    .. controls (0,0) ..
    (\t + \k * 60 + 60:\R cm);
}
\end{tikzpicture}
}

\def\TAlphaTwoC{
\begin{tikzpicture}[scale=0.85,yscale=.75, xscale=-1, draw=darkblue,
    decoration={
    markings},
    baseline={([yshift=-.8ex]current bounding box.center)}
    ]
\def \h{5};
\def \t{50}; 
\def \r{1.2}; 
\def \R{2}; 
%
\filldraw[black] (0.7,0.2) circle (2pt) node[black,scale=0.7, above]{$x$};
\begin{scope}[yshift=-\h cm]
    \foreach \k in {0,2,4}{
        \draw[double,  postaction={decorate}, decoration={
            markings,
            mark=at position 0.5 with {\arrow{<}}
        }] (\t+\k * 60:\r cm) --  (\t + \k * 60 + 60:\r cm);
        \draw[ postaction={decorate}, decoration={
            markings,
            mark=at position 0.5 with {\arrow{<}}
        }] (\t+\k*60 - 60:\r cm) -- (\t +\k*60:\r cm);
    }
    \foreach \k in {0,2,4}{
        \draw[postaction={decorate}, decoration={
            markings,
            mark=at position 0.7 with {\arrow{>}}
        }] (\t+\k*60:\r cm) -- (\t+\k*60:\R cm);
    }
    \foreach \k in {1,3,5}{
        \draw[postaction={decorate}, decoration={
            markings,
            mark=at position 0.7 with {\arrow{<}}
        }] (\t+\k*60:\r cm) -- (\t+\k*60:\R cm);
    }
\end{scope}
\begin{scope}[yshift=-\h cm]
    \foreach \k in {0,2,4}{
        \filldraw[green!20, fill opacity=0.5] (\t+\k * 60:\r cm) 
            .. controls (0,.5*\h) .. (\t + \k * 60 + 60:\r cm) -- (\t+\k * 60:\r cm) ;
        \draw[green] (\t+\k * 60:\r cm) 
            .. controls (0,.5*\h) .. (\t + \k * 60 + 60:\r cm);
    }
\end{scope}
\draw[darkgreen!50] (-0.045,-.675*\h)--(\t+30:.44 cm);
\draw[darkgreen!50] (.22,-.6*\h) to[out=\t+120, in=-90, looseness=.5] (\t+150:.44 cm);
\draw[darkgreen!50] (-.18,-.585*\h) to[out=\t-75, in=-90, looseness=.5] (\t+270:.44 cm);
\begin{scope}[yshift=-.625*\h cm]
    \def \d {.25}; 
    \draw[darkgreen!50, dashed] (-98:\d cm) 
        -- (24:\d cm)
        -- (135:\d cm)
        -- (-98:\d cm) 
        ;
\end{scope}
\foreach \k in {0,1,2,3,4,5}{
    \draw[darkgreen] (\t + \k * 60:2cm) -- +(0,-\h);
}
\foreach \k in {0,2,4}{
    \draw[
        postaction={decorate},
        decoration={
            markings,
            mark=at position 0.7 with {\arrow{<}}
        }
    ]
    (\t+\k * 60:\R cm)
    .. controls (0,0) ..
    (\t + \k * 60 + 60:\R cm);
}
\end{tikzpicture}
}

\def\TAlphaTwoB{
\begin{tikzpicture}[scale=0.85,yscale=.75, xscale=-1, draw=darkblue,
    decoration={
    markings},
    baseline={([yshift=-.8ex]current bounding box.center)}
    ]
\def \h{5};
\def \t{50}; 
\def \r{1.2}; 
\def \R{2}; 
%
\filldraw[black] (-1.5,-0.6) circle (2pt) node[black,scale=0.7, above]{$x$};
\begin{scope}[yshift=-\h cm]
    \foreach \k in {0,2,4}{
        \draw[double,  postaction={decorate}, decoration={
            markings,
            mark=at position 0.5 with {\arrow{>}}
        }] (\t+\k * 60:\r cm) --  (\t + \k * 60 + 60:\r cm);
        \draw[ postaction={decorate}, decoration={
            markings,
            mark=at position 0.5 with {\arrow{<}}
        }] (\t+\k*60 - 60:\r cm) -- (\t +\k*60:\r cm);
    }
    \foreach \k in {0,2,4}{
        \draw[postaction={decorate}, decoration={
            markings,
            mark=at position 0.7 with {\arrow{<}}
        }] (\t+\k*60:\r cm) -- (\t+\k*60:\R cm);
    }
    \foreach \k in {1,3,5}{
        \draw[postaction={decorate}, decoration={
            markings,
            mark=at position 0.7 with {\arrow{<}}
        }] (\t+\k*60:\r cm) -- (\t+\k*60:\R cm);
    }
\end{scope}
\begin{scope}[yshift=-\h cm]
    \foreach \k in {0,2,4}{
        \filldraw[green!20, fill opacity=0.5] (\t+\k * 60:\r cm) 
            .. controls (0,.5*\h) .. (\t + \k * 60 + 60:\r cm) -- (\t+\k * 60:\r cm) ;
        \draw[green] (\t+\k * 60:\r cm) 
            .. controls (0,.5*\h) .. (\t + \k * 60 + 60:\r cm);
    }
\end{scope}
\draw[darkgreen!50] (-0.045,-.675*\h)--(\t+30:.44 cm);
\draw[darkgreen!50] (.22,-.6*\h) to[out=\t+120, in=-90, looseness=.5] (\t+150:.44 cm);
\draw[darkgreen!50] (-.18,-.585*\h) to[out=\t-75, in=-90, looseness=.5] (\t+270:.44 cm);
\begin{scope}[yshift=-.625*\h cm]
    \def \d {.25}; 
    \draw[darkgreen!50, dashed] (-98:\d cm) 
        -- (24:\d cm)
        -- (135:\d cm)
        -- (-98:\d cm) 
        ;
\end{scope}
\foreach \k in {0,1,2,3,4,5}{
    \draw[darkgreen] (\t + \k * 60:2cm) -- +(0,-\h);
}
\foreach \k in {0,2,4}{
    \draw[
        postaction={decorate},
        decoration={
            markings,
            mark=at position 0.7 with {\arrow{<}}
        }
    ]
    (\t+\k * 60:\R cm)
    .. controls (0,0) ..
    (\t + \k * 60 + 60:\R cm);
}
\end{tikzpicture}
}

\def\TBetaO{
\begin{tikzpicture}[scale=0.85,yscale=-.75, xscale=1, draw=darkblue,
    decoration={
    markings},
    baseline={([yshift=-.8ex]current bounding box.center)}
    ]
\def \h{5};
\def \t{230}; 
\def \r{1.2}; 
\def \R{2}; 
%
\begin{scope}[yshift=-\h cm]
    \foreach \k in {0,2,4}{
        \draw[double,  postaction={decorate}, decoration={
            markings,
            mark=at position 0.5 with {\arrow{<}}
        }] (\t+\k * 60:\r cm) --  (\t + \k * 60 + 60:\r cm);
        \draw[ postaction={decorate}, decoration={
            markings,
            mark=at position 0.5 with {\arrow{<}}
        }] (\t+\k*60 - 60:\r cm) -- (\t +\k*60:\r cm);
    }
    \foreach \k in {0,2,4}{
    \draw[
        postaction={decorate},
        decoration={
            markings,
            mark=at position 0.7 with {\arrow{>}}
        }
    ]
      (\t+\k*60:\r cm) -- (\t+\k*60:\R cm);
    }
    \foreach \k in {1,3,5}{
        \draw[
        postaction={decorate},
        decoration={
            markings,
            mark=at position 0.7 with {\arrow{<}}
        }
    ]
       (\t+\k*60:\r cm) -- (\t+\k*60:\R cm);
    }
\end{scope}
\begin{scope}[yshift=-\h cm]
    \foreach \k in {0,2,4}{
        \filldraw[green!20, fill opacity=0.5] (\t+\k * 60:\r cm) 
            .. controls (0,.5*\h) .. (\t + \k * 60 + 60:\r cm) -- (\t+\k * 60:\r cm) ;
        \draw[green] (\t+\k * 60:\r cm) 
            .. controls (0,.5*\h) .. (\t + \k * 60 + 60:\r cm);
    }
\end{scope}
\draw[darkgreen!50] (-0.045,-.675*\h)--(\t+30:.44 cm);
\draw[darkgreen!50] (.22,-.6*\h) to[out=\t+120, in=-90, looseness=.5] (\t+150:.44 cm);
\draw[darkgreen!50] (-.18,-.585*\h) to[out=\t-75, in=-90, looseness=.5] (\t+270:.44 cm);
\begin{scope}[yshift=-.625*\h cm]
    \def \d {.25}; 
    \draw[darkgreen!50, dashed] (-98:\d cm) 
        -- (24:\d cm)
        -- (135:\d cm)
        -- (-98:\d cm) 
        ;
\end{scope}
\foreach \k in {0,1,2,3,4,5}{
    \draw[darkgreen] (\t + \k * 60:2cm) -- +(0,-\h);
}
\foreach \k in {0,2,4}{
    \draw[
        postaction={decorate},
        decoration={
            markings,
            mark=at position 0.7 with {\arrow{<}}
        }
    ]
 (\t+\k * 60:\R cm) 
    .. controls (0,0) ..  (\t + \k * 60 + 60:\R cm);
}
\end{tikzpicture}
}

\def\TBetaOne{
\begin{tikzpicture}[scale=0.85,yscale=-.75, xscale=1, draw=darkblue,
    decoration={
    markings},
    baseline={([yshift=-.8ex]current bounding box.center)}
    ]
\def \h{5};
\def \t{230}; 
\def \r{1.2}; 
\def \R{2}; 
%
\filldraw[black] (0.85,-5) circle (2pt) node[black,scale=0.7, above]{$x$};
\begin{scope}[yshift=-\h cm]
    \foreach \k in {0,2,4}{
        \draw[double,  postaction={decorate}, decoration={
            markings,
            mark=at position 0.5 with {\arrow{<}}
        }] (\t+\k * 60:\r cm) --  (\t + \k * 60 + 60:\r cm);
        \draw[ postaction={decorate}, decoration={
            markings,
            mark=at position 0.5 with {\arrow{<}}
        }] (\t+\k*60 - 60:\r cm) -- (\t +\k*60:\r cm);
    }
    \foreach \k in {0,2,4}{
    \draw[
        postaction={decorate},
        decoration={
            markings,
            mark=at position 0.7 with {\arrow{>}}
        }
    ]
      (\t+\k*60:\r cm) -- (\t+\k*60:\R cm);
    }
    \foreach \k in {1,3,5}{
        \draw[
        postaction={decorate},
        decoration={
            markings,
            mark=at position 0.7 with {\arrow{<}}
        }
    ]
       (\t+\k*60:\r cm) -- (\t+\k*60:\R cm);
    }
\end{scope}
\begin{scope}[yshift=-\h cm]
    \foreach \k in {0,2,4}{
        \filldraw[green!20, fill opacity=0.5] (\t+\k * 60:\r cm) 
            .. controls (0,.5*\h) .. (\t + \k * 60 + 60:\r cm) -- (\t+\k * 60:\r cm) ;
        \draw[green] (\t+\k * 60:\r cm) 
            .. controls (0,.5*\h) .. (\t + \k * 60 + 60:\r cm);
    }
\end{scope}
\draw[darkgreen!50] (-0.045,-.675*\h)--(\t+30:.44 cm);
\draw[darkgreen!50] (.22,-.6*\h) to[out=\t+120, in=-90, looseness=.5] (\t+150:.44 cm);
\draw[darkgreen!50] (-.18,-.585*\h) to[out=\t-75, in=-90, looseness=.5] (\t+270:.44 cm);
\begin{scope}[yshift=-.625*\h cm]
    \def \d {.25}; 
    \draw[darkgreen!50, dashed] (-98:\d cm) 
        -- (24:\d cm)
        -- (135:\d cm)
        -- (-98:\d cm) 
        ;
\end{scope}
\foreach \k in {0,1,2,3,4,5}{
    \draw[darkgreen] (\t + \k * 60:2cm) -- +(0,-\h);
}
\foreach \k in {0,2,4}{
    \draw[
        postaction={decorate},
        decoration={
            markings,
            mark=at position 0.7 with {\arrow{<}}
        }
    ]
 (\t+\k * 60:\R cm) 
    .. controls (0,0) ..  (\t + \k * 60 + 60:\R cm);
}
\end{tikzpicture}
}

\def\BetaOne{
\begin{tikzpicture}[scale=0.85, yscale=-.75, draw=darkblue,
    decoration={
    markings},
    baseline={([yshift=-.8ex]current bounding box.center)}
    ]
\def \h{5};
\def \t{50}; 
\def \r{1.2}; 
\def \R{2}; 
%
\filldraw[black] (-0.1,-5.2) circle (2pt) node[black,scale=0.7, above]{$x$};
\begin{scope}[yshift=-\h cm]
    \foreach \k in {0,2,4}{
        \draw[double, postaction={decorate}, decoration={
            markings,
            mark=at position 0.5 with {\arrow{>}}
        }] (\t+\k * 60:\r cm) --  (\t + \k * 60 + 60:\r cm);
        \draw[postaction={decorate}, decoration={
            markings,
            mark=at position 0.5 with {\arrow{>}}
        }] (\t+\k*60 - 60:\r cm) -- (\t +\k*60:\r cm);
    }
    \foreach \k in {0,2,4}{
        \draw[postaction={decorate}, decoration={
            markings,
            mark=at position 0.7 with {\arrow{<}}
        }] (\t+\k*60:\r cm) -- (\t+\k*60:\R cm);
    }
    \foreach \k in {1,3,5}{
        \draw[postaction={decorate}, decoration={
            markings,
            mark=at position 0.7 with {\arrow{>}}
        }] (\t+\k*60:\r cm) -- (\t+\k*60:\R cm);
    }
\end{scope}
\begin{scope}[yshift=-\h cm]
    \foreach \k in {0,2,4}{
        \filldraw[green!20, fill opacity=0.5] (\t+\k * 60:\r cm) 
            .. controls (0,.5*\h) .. (\t + \k * 60 + 60:\r cm) -- (\t+\k * 60:\r cm) ;
        \draw[green] (\t+\k * 60:\r cm) 
            .. controls (0,.5*\h) .. (\t + \k * 60 + 60:\r cm);
    }
\end{scope}
\draw[darkgreen!50] (.03,-.575*\h) to[out=\t, in=-90, looseness=.4] (\t+30:.44 cm);
\draw[darkgreen!50] (-.265,-.635*\h) to[out=\t+120, in=-90, looseness=.5] (\t+155:.44 cm);
\draw[darkgreen!50] (.22,-.655*\h) to[out=\t-75, in=-100, looseness=.5] (\t+257.5:.46 cm);
\begin{scope}[yshift=-.625*\h cm]
    \def \d {.25}; 
    \draw[darkgreen!50, dashed] (82:\d cm) 
        -- (190:\d cm)
        -- (320:\d cm)
        -- (82:\d cm) 
        ;
\end{scope}
\foreach \k in {0,1,2,3,4,5}{
    \draw[darkgreen] (\t + \k * 60:2cm) -- +(0,-\h);
}
\foreach \k in {0,2,4}{
    \draw[
        postaction={decorate},
        decoration={
            markings,
            mark=at position 0.7 with {\arrow{>}}
        }
    ]
    (\t+\k * 60:\R cm) 
    .. controls (0,0) ..  (\t + \k * 60 + 60:\R cm);
}
\end{tikzpicture}
}

\def\BetaO{
\begin{tikzpicture}[scale=0.85, yscale=-.75, draw=darkblue,
    decoration={
    markings},
    baseline={([yshift=-.8ex]current bounding box.center)}
    ]
\def \h{5};
\def \t{50}; 
\def \r{1.2}; 
\def \R{2}; 
%
\begin{scope}[yshift=-\h cm]
    \foreach \k in {0,2,4}{
        \draw[double, postaction={decorate}, decoration={
            markings,
            mark=at position 0.5 with {\arrow{>}}
        }] (\t+\k * 60:\r cm) --  (\t + \k * 60 + 60:\r cm);
        \draw[postaction={decorate}, decoration={
            markings,
            mark=at position 0.5 with {\arrow{>}}
        }] (\t+\k*60 - 60:\r cm) -- (\t +\k*60:\r cm);
    }
    \foreach \k in {0,2,4}{
        \draw[postaction={decorate}, decoration={
            markings,
            mark=at position 0.7 with {\arrow{<}}
        }] (\t+\k*60:\r cm) -- (\t+\k*60:\R cm);
    }
    \foreach \k in {1,3,5}{
        \draw[postaction={decorate}, decoration={
            markings,
            mark=at position 0.7 with {\arrow{>}}
        }] (\t+\k*60:\r cm) -- (\t+\k*60:\R cm);
    }
\end{scope}
\begin{scope}[yshift=-\h cm]
    \foreach \k in {0,2,4}{
        \filldraw[green!20, fill opacity=0.5] (\t+\k * 60:\r cm) 
            .. controls (0,.5*\h) .. (\t + \k * 60 + 60:\r cm) -- (\t+\k * 60:\r cm) ;
        \draw[green] (\t+\k * 60:\r cm) 
            .. controls (0,.5*\h) .. (\t + \k * 60 + 60:\r cm);
    }
\end{scope}
\draw[darkgreen!50] (.03,-.575*\h) to[out=\t, in=-90, looseness=.4] (\t+30:.44 cm);
\draw[darkgreen!50] (-.265,-.635*\h) to[out=\t+120, in=-90, looseness=.5] (\t+155:.44 cm);
\draw[darkgreen!50] (.22,-.655*\h) to[out=\t-75, in=-100, looseness=.5] (\t+257.5:.46 cm);
\begin{scope}[yshift=-.625*\h cm]
    \def \d {.25}; 
    \draw[darkgreen!50, dashed] (82:\d cm) 
        -- (190:\d cm)
        -- (320:\d cm)
        -- (82:\d cm) 
        ;
\end{scope}
\foreach \k in {0,1,2,3,4,5}{
    \draw[darkgreen] (\t + \k * 60:2cm) -- +(0,-\h);
}
\foreach \k in {0,2,4}{
    \draw[
        postaction={decorate},
        decoration={
            markings,
            mark=at position 0.7 with {\arrow{>}}
        }
    ]
    (\t+\k * 60:\R cm) 
    .. controls (0,0) ..  (\t + \k * 60 + 60:\R cm);
}
\end{tikzpicture}
}

\subsection{{$\GL_N$} foams for $6$-valent vertex categorification}
In this subsection, we introduce the $\GL_N$ foams which are used to categorify the $6$-valent vertex and decompose the hexagon web in $\operatorname{Kar}(\Foam^{\HexaBound})$. We develop basic properties of the $\GL_N$ foams for arbitrary $N$, which will be used in the following subsections when $N$ is specialized to $4$ and $5$.    
 
 First, we introduce notations for the $\GL_N$ foams in between webs $\Wone$, $\Wzero$, $\Wzeroa$, and $\Wtwo$ in Figure \ref{gl4foamdeef}.

\begin{figure}[htbp]
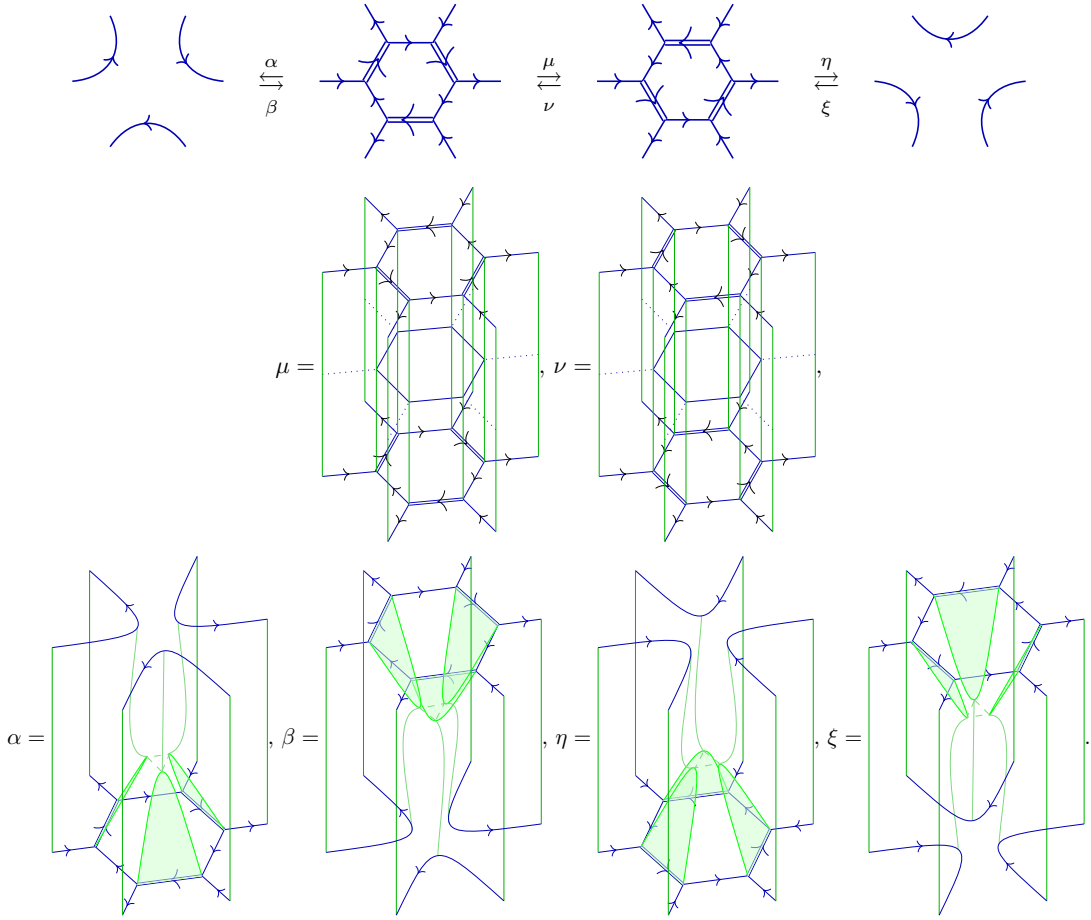

    \centering
     \[ \BenzenewebM 
      \quad \stackrel[\beta]{\alpha}  {\leftrightarrows} 
       \quad 
      \Benzeneweb  \quad \stackrel[\nu]{\mu}{\rightleftarrows}
      \quad \Benzeneweba   \quad \stackrel[\xi]{\eta}{\rightleftarrows}
      \quad \BenzenewebaM \]
      
\[
\scalebox{0.9}{$\displaystyle
    \mu=\TAlphaZBetaZDef, \,   \nu=\AlphaZTBetaZDef,
$}
\] 
\[
\scalebox{0.85}{$\displaystyle
    \alpha =  \AlphaO, \, \beta =  \BetaO, \, \eta =  \TAlphaO, \, \xi=  \TBetaO. 
$}
\] 
    \caption{The $\GL_N$ foams used to demonstrate the decomposition of the hexagon web. The horizontal facets in the middle of the foams $\mu, \nu$ are $3$-facets. The highlighted facets of foams $\alpha, \beta, \eta, \xi$ are $2$-facets.}
    \label{gl4foamdeef}
\end{figure}

\begin{lemma}\label{degcheck}
 When the foams $\mu, \nu, \beta, \alpha, \xi$ and $\eta$ are interpreted as $\GL_N$ foams their degrees are given by
    \[\foamdeg(\mu)= 0,\quad \foamdeg(\nu)= 0, \quad
    \foamdeg(\beta \circ \alpha)= -2N + 8,\quad \foamdeg(\xi \circ \eta)= -2N + 8.  \]
\end{lemma}

\begin{proof}
   This is a simple count and an application of Definition \ref{dfn:degreefoam}. Note that this definition does not work for foams with corners, so we have to close the boundary webs. Close the boundaries of $\beta \circ \alpha, \xi \circ \eta$ so that the boundary webs $\Wone$ and$ \Wtwo $ become a single circle  (see Notation \ref{notation webs 6pt boundary}). Close $\mu$ and $\nu$ following the induced closure of $\beta \circ \alpha$ from $\mu$ and $\xi \circ \eta$ from $\nu$. 
   
   First consider the foam $\mu$. Following Definition \ref{Foam-Compo-Conca} we modify this foam to $\overline{ \idfoam_{W_2}}\ast \mu$ (this corresponds to the above description) and count:

   \begin{itemize}
       \item six $(1,2,3)$ bindings, 
       \item twelve $(1,1,2)$ bindings, 
       \item nine $1$-facets with Euler characteristic $1$,
       \item six $2$-facets with Euler characteristic $1$,
       \item one $3$-facet with Euler characteristic $1$,
       \item and six $(1,1,1,2,2,3)$-vertices.
   \end{itemize}
   Following Definition \ref{dfn:degreefoam}, 
   \begin{align*}
       \foamdeg(\mu) &= - (1 \cdot (N -1) \cdot 9 + 2\cdot (N-2) \cdot 6 + 3 \cdot (N - 3) \cdot 1) +\\
       + &((2 + 3 \cdot(N - 3) )\cdot 6 + (1 + 2\cdot (N-2)) \cdot 12 ) - \\
       -&(3 + 3 \cdot(N - 3)) \cdot 6 = 0.
   \end{align*}

   In the case of $\beta \circ \alpha$ we modify the foam to $\overline{\idfoam_{\Wtwo}} \ast (\beta \circ \alpha) $ and count
   \begin{itemize}
       \item six $(1,1,2)$ bindings,
       \item two $1$-facets with Euler characteristic $1$,
       \item one $1$-facet with Euler characteristic $0$ (a tube), and
       \item six $2$-facets with Euler characteristic $1$.
   \end{itemize}
   The total degree is therefore
       \begin{align*}
       \foamdeg(\beta \circ \alpha) &= - (1 \cdot(N -1) \cdot 2 + 2\cdot (N-2) \cdot 6 ) + (1 + 2\cdot (N-2)) \cdot 6  \\ 
       &= -2N + 8.
   \end{align*}
   
   The degrees of the other two foams can be similarly calculated, as they are just rotations of the two foams for which we have done the computation. Hence, the degree of them is as written in the statement of the lemma and can be computed using the same count.
\end{proof}

\begin{thm} \label{GLNIndecompo}
     $\mathcal{E}(\Wone)$, and $\mathcal{E}(\Wtwo)$ are indecomposable objects in $Kar(\Foam^{\HexaBound})$. When $\nu\circ\mu$ is a projection in $Kar(\Foam^{\HexaBound})$, $\mathcal{E}(\nu\circ\mu)$ is an indecomposable object. Similarly, when $\mu\circ\nu$ is a projection, $\mathcal{E}(\mu\circ\nu)$ is an indecomposable object.
\end{thm}

\begin{proof}
      $\mathcal{E}(\Wone)$ and $\mathcal{E}(\Wtwo)$ are indecomposable due to Proposition \ref{CartanMatrix} and Lemma \ref{WebPairIndecompo}. 
     In order to prove that $\mathcal{E}(\nu\circ\mu)$ is indecomposable, we consider any degree zero foam $F$ from $\Wzero$ to $\Wzero$.
     Since $ \mu\circ  F  \in \Hom_{\Foam^{\HexaBound}}(\Wzero,\Wzeroa)$ and $\langle \overline{\Wzero} \Wzeroa  \rangle$ has coefficient $1$ in front of $q^{-\boundarydeg({\HexaBound})}$ (see Proposition \ref{prop:closed sixvertex evaluation}), we know that $ \mu F= k \mu $ for some $k\in \mathbb{Q}$ by Corollary \ref{WebPairDim}. So   
     \[ (\nu\circ\mu) \circ F \circ (\nu\circ\mu) 
        = \nu\circ ( \mu \circ F) \circ (\nu\circ\mu)
        =\nu\circ k \mu \circ (\nu\circ\mu) 
        =  k (\nu\circ\mu)  \circ (\nu\circ\mu) 
        =k(\nu\circ\mu) . \]
     Thus, by Proposition \ref{IndecompoCrit}, $\mathcal{E}(\nu\circ\mu)$ is indecomposable. Similarly, we can show that $\mathcal{E}(\mu\circ \nu)$ is indecomposable when $\mu\circ \nu$ is a projection. 
\end{proof}

Now we present two foam relations deduced from the foam evaluation formulas. The following proposition gives an alternative expression of the foam relation \cite[Equation 10]{RW-eval-foams}. The $\GL_3$ version is given by Proposition 2.22 in \cite{Khovanov2018FoamEA}.

\begin{proposition}\label{thickness1neckcut}
 The following neck-cutting relation holds for $\GL_N$ foams. 
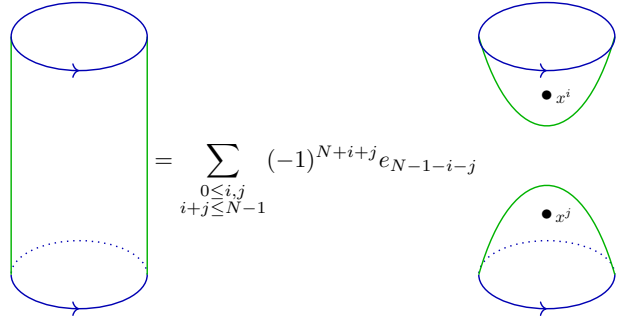
\[
\scalebox{0.9}{$\displaystyle
   \SingleCy =
 \sum_{\substack{0 \le i,j  \\  i+j\le N-1}
} (-1)^{N +i+j} e_{N-1-i-j} \SingleCupCapAb 
$}
\] 
\end{proposition}
\begin{proof}

Denote by $F$ the foam on the left side of the equality and by $(G_{ij})_{ij}$ the foams on the right hand side of the equality, where $i,j$ stand for the number of dots as in the picture --- $i$ dots on the top hemisphere and $j$ on the bottom one. 
We first analyze the case when a coloring $c$ of $F$ induces a coloring of $G_{ij}$'s, denote the induced coloring by $c$ as well. Without loss of generality assume that $c$ is the coloring by $1$. 
In particular, all hemispheres in $G_{ij}$'s are colored by $1$ as is the cylinder of $F$. Other cases follow analogously. 

Note that the change in $s$ when going from $F$ to $G_{ij}$'s is the change of Euler characteristics of the monocolored surfaces. 
With this coloring that change is equal to $(-1)$, as the Euler characteristic of $F_1$ changes by $2$. 
Using $e_{i} = 0$ for $i < 0$ gives
\begin{align*}
    \langle G_{ij}, c \rangle &= (-1)^{N +i+j}\frac{e_{N - 1 - i - j}\vara^{i+j}}{\prod_{k = 2}^N (\vara - \vark)}  (-1) \langle F, c \rangle \\
    \sum_{\substack{0 \le i,j  \\  i+j\le N}} \langle G_{ij}, c \rangle &=  \langle F, c \rangle \frac{\sum_{\substack{0 \le i,j  \\  i+j\le N-1}} (-1)^{N -1+i+j} e_{N -1 - i - j}\vara^{i+j}}{\prod_{k = 2}^N (\vara - \vark)} = \langle F, c \rangle.
\end{align*}
The last equality holds due to the following equality for (abstract) polynomials:
\begin{align*}
    P(m,\lambda) &= \prod_{k = 1}^m (\lambda - \vark) = \sum_{k = 0}^m (-1)^k e_k(\vara, \varb, \ldots, \varu{m})\lambda^{m-k} \\
    \frac{d}{d\lambda} P(m,\lambda) &= \sum_{l = 1}^m \frac{1}{(\lambda - \varu{l})} \prod_{k = 1}^m (\lambda - \vark) =  \sum_{k = 0}^{m-1} (m-k)(-1)^k e_k(\vara, \varb, \ldots, \varu{m})\lambda^{m-k - 1}\\
    \left.\frac{d}{d\lambda}\right|_{\lambda = \vara} P(m,\lambda) &= \prod_{k = 2}^m (\vara - \vark) =\sum_{k = 0}^{m-1} (m-k)(-1)^k e_k(\vara, \varb, \ldots, \varu{m})\vara^{m-k - 1}
\end{align*}
Lastly, set $m = N$ and the result follows.

Next consider the case where the hemispheres of $G_{ij}$'s are colored differently. Without loss of generality assume that the bottom hemisphere is colored by $2$ and the top by $1$. Name this coloring $c$. 
Denote by $G$ the foam that underlies $G_{ij}$'s, that is, $G_{ij}$ without decorations. 
We compute
\begin{align*}
    \langle G_{ij}, c \rangle &= (-1)^{N +i+j}e_{N -1- i - j}\vara^{i}\varb^j\langle G, c \rangle \\
    \sum_{\substack{0 \le i,j  \\  i+j\le N}} \langle G_{ij}, c \rangle &= \langle G, c \rangle \sum_{\substack{0 \le i,j  \\  i+j\le N-1}} (-1)^{N +i+j}e_{N -1- i - j}\vara^{i}\varb^j \\
    &= \langle G, c \rangle \sum_{k =0}^N (-1)^{k}e_{k}(\vara, \varb, \ldots, \varu{N})h_{N- k -1}(\vara, \varb) = 0.
\end{align*}
The last equality follows from $\sum_{i = 0}^N (-1)^i e_i(\vara, \varb, \ldots, \varu{m}) h_{N-i}(\vara, \varb, \ldots, \varu{m}) = 0$. Before applying that equality, we use formulas to express $e_i$'s and $h_i$'s in fewer variables to reduce $$\sum_{i = 0}^N (-1)^i e_i(\vara, \varb, \ldots, \varu{m}) h_{N-i}(\vara, \varb, \ldots, \varu{m}) = 0$$ to $$\sum_{i =0}^N (-1)^{i}e_{i}(\vara, \varb, \ldots, \varu{m})h_{N- i}(\vara, \varb) q_k(\vara, \varb, \ldots, \varu{m}),$$ where $q_k$ is some polynomial. The result then follows by an application of the above formulas.
\end{proof}

\begin{lemma}\label{dotsglngeneric}
    The following $\GL_N$ foam equality holds, assuming $k \ge 0$ and $e_i = 0$ if $i < 0$.
\[
\scalebox{0.85}{$\displaystyle
    \matchingbubble = e_{N-4-k}  \matchingbubbleA 
- e_{N-5-k} 
\left (  \matchingbubbleB  +
\matchingbubbleC + \matchingbubbleD \right)
$}
\] 

\end{lemma}
\begin{proof}
We perform $6$ neck cuttings that decompose the foam into a union of the identity foam on the planar matching web, three thickness-$(1,1,2)$ 
theta foams and one thickness-$1$ sphere, with some number of dots.

In order for the theta foams to evaluate to a nonzero value  $\pm 1$, there need to be $N-1$ dots on one $1$-facet and $N-2$ dots on the other. This follows from \cite[Equation 9]{RW-eval-foams} and dot migration, that is, having a dot on each of the $1$-facets is the same as having $e_2$ on the $2$-facet; this follows from $e_2(\vara, \varb) = \vara \cdot \varb$ and the foam evaluation formula. 
Additionally there need to be at least $N-1$ dots on the sphere. A sphere with $N-1$ dots evaluates to $-1$ and a sphere with $N$ dots evaluates to $-e_{1}$ by \cite[Equation 8]{RW-eval-foams}.

Using this discussion and Proposition \ref{thickness1neckcut} one obtains the desired result. 
\end{proof}

\subsection{Categorify \texorpdfstring{$\GL_4$}{GL_4} \texorpdfstring{$6$}{6}-valent vertex with \texorpdfstring{$\GL_4$}{GL_4} foams}
For $\GL_4$, Equation \eqref{GLN6VertexDef} becomes 
      \begin{equation*} 
      \SixVertex \ =  \ \Benzeneweb \  -  \ \BenzenewebM
        \  =  \ \Benzeneweba  \ -   \  \BenzenewebaM.
      \end{equation*} 
 In the following theorem, we provide a direct sum decomposition of the state space of $\Wzero$ (considered as a part of a closed web) into the state space of $\Wone$ and the state space of $\Wsix$ using $\GL_4$ foams. In particular, we provide a projection in the $\GL_4$ foam category with fixed boundary $\varepsilon$, which decategorifies to the $6$-vertex. 

\begin{remark}\label{rmk: notation for hexagon colorin}
    Recall Notation \ref{flapcolor on web}. By a light abuse of notation, we use the same shorthand \flapcolor[i_1][i_2][i_3][i_4][i_5][i_6] to denote the coloring of the foams appearing in the rest of the paper.  In particular, for a $(W_1, W_2)$-foam with corners $F$, denote by \flapcolor[i_1][i_2][i_3][i_4][i_5][i_6] the coloring of $F$ induced from the coloring of $W_1$, which we will also refer to as the ``bottom foam''. Note that a coloring of a boundary web often does not induce a unique coloring of a foam. The notation \flapcolor[i_1][i_2][i_3][i_4][i_5][i_6] only makes sense when $W_1$ has six boundary points.

    Moreover, a coloring of a foam may be induced by a coloring of another foam with shared boundary. 
    Suppose that $F_1, F_2$ are two $(W_1, W_2)$-foams, where $F_2$ is colored with a coloring $c$. 
    Inducing a coloring of $F_1$ is a two-step process. 
    First restrict the coloring $c$ of $F_2$ to $W_1$. Then induce a coloring of $F_1$ from a coloring of $W_1$. Again, in many cases the extension is not unique. This procedure makes sense for a pair consisting of a $(W_1,W_2)$-foam $F_1$ and a $(W_1, W_3)$-foam $F_2$ which share only one of the web boundaries $W_1$. 
\end{remark}

\begin{thm}\label{gl_4 sixvalent seam} Consider the $\GL_4$ foams $\alpha, \beta, \eta, \xi, \mu$, and $\nu$ (see Figure \ref{gl4foamdeef}) between the following webs: 
     \[ \BenzenewebM 
      \quad \stackrel[\beta]{\alpha}  {\leftrightarrows} 
       \quad 
      \Benzeneweb  \quad \stackrel[\nu]{\mu}{\rightleftarrows}
      \quad \Benzeneweba   \quad \stackrel[\xi]{\eta}{\rightleftarrows}
      \quad \BenzenewebaM \]

The following relations among morphisms in $\FoamFour^{\HexaBound}$ hold:

\begin{align*}
    \alpha \circ \nu =0 , \quad 
    \mu \circ \beta = 0, \quad 
    &\eta \circ \mu=0, \quad   
     \nu \circ \xi =0, \\
     \mu \circ \nu \circ \mu = \mu , \quad 
     &\nu \circ \mu \circ \nu = \nu , \\
    \alpha \circ \beta = \idfoam_{W_2}, &\quad
    \eta \circ \xi = \idfoam_{W_3},  \\
     \beta \circ \alpha + \nu \circ \mu = \idfoam_{W_0}, 
     & \quad
    \xi \circ \eta + \mu \circ \nu = \idfoam_{\Wzeroa}.
\end{align*}

\end{thm}

The following proposition is a consequence of 
Theorem \ref{gl_4 sixvalent seam} and its proof.

\begin{proposition}
\label{prop:gl4-kar-isom-objects}
The foams $P_0=\nu \circ \mu $ 
and $P_0'=\mu \circ \nu$ are projections in $\FoamFour^{\HexaBound}$. Furthermore, the associated objects $\mathcal{E}(P_0)$ and $\mathcal{E}(P_0')$ in $Kar(\FoamFour^{\HexaBound})$ are isomorphic via the inverse morphisms $\mu$ and $\nu$. 
\end{proposition}

\begin{thm} \label{GL4Decompo}
In $Kar(\FoamFour^{\HexaBound})$, we have the following direct sum decomposition of the object given by the hexagon web:
\[ \mathcal{E}\left( \Benzeneweb \right) = \mathcal{E}\left( \BenzenewebM \right)\oplus \mathcal{E}\left( P_0' \right).\]

So $\mathcal{E}\left( P_0 \right) \cong \mathcal{E}\left( P_0' \right)$ decategorifies to $\Wsix$, since  
  \[ \Benzeneweb = \BenzenewebM + \SixVertex. \]   
\end{thm}

\begin{proof}
Given the relations in Theorem \ref{gl_4 sixvalent seam} between the morphisms in $Kar(\FoamFour^{\HexaBound})$, this isomorphism between the object $\mathcal{E}(W_0)$ and the direct sum of objects $\mathcal{E}(W_2)$, and $\mathcal{E}(P_0')$ follows directly from the definition of direct sum of objects in a category. The projection maps are given by $\alpha$ and $\mu$, and the inclusion maps are given by $\beta$ and $\nu$. 

The degree computation follows from the degree computations of foams. In particular, this is a simpler version of Lemma \ref{degcheck}. A modification of that proof shows that $\foamdeg(\alpha) = 0$. This determines the rest of the foam degrees and they descend to the web degrees.
\end{proof}

We also have the direct sum decomposition of the state space of a closed web containing a (single) hexagon. 

\begin{remark}
The direct sum decomposition of the hexagon web in the above theorem is a decomposition into indecomposable objects due to Theorem \ref{GLNIndecompo}. 
\end{remark}

\begin{corollary}\label{cor:single hexagon closed gl4}
Given any web $W$ with boundary $\varepsilon$, we have the direct sum decomposition of the state space of $\overline{W} W_0$.

\[ \mathcal{F}(\overline{W} W_0 ) \cong \mathcal{F}(\overline{W} W_1)\oplus 
\left( \mathcal{F}(\overline{ \idfoam_W} \ast P_0' )(\mathcal{F}(\overline{W} W_0'))  \right)  \]
\end{corollary}

\begin{proof}
    This is a direct sum decomposition of the module 
    $\mathcal{F}(\overline{W} W_0 )$. The projection maps are given by $\overline{ \idfoam_W} \ast \alpha $ and $\overline{ \idfoam_W} \ast \mu$, and the inclusion maps are given by $\overline{ \idfoam_W} \ast \beta $ and $\overline{ \idfoam_W} \ast \nu $.
\end{proof}

The rest of this section is dedicated to the proof of Theorem \ref{gl_4 sixvalent seam}. We split the proof into two parts, Propositions \ref{orthogonalityprop} and \ref{idempotencyprop}, which are themselves divided into multiple lemmas.

\begin{proposition}\label{orthogonalityprop}
    The following equalities from Theorem \ref{gl_4 sixvalent seam} hold:
    \[
        \alpha \circ \nu=0,
        \quad 
        \eta \circ \mu=0,
        \quad 
        \mu \circ \beta=0,
        \quad
        \text{and} \quad  \nu  \circ \xi =0.
    \]
\end{proposition}

\begin{proof}
    We prove that $\alpha \circ \nu=0 $ by considering colorings of the foam. The other equalities follow a similar argument.

    Consider the colorings of $\nu$ where the boundary points of the bottom hexagon are colored as \flapcolor[i][i][i][i][i][i]. Suppose that the single edges in the middle of the bottom hexagon are colored by $j$. Then there are two possible colorings of the foam $\nu$ in this case, corresponding to the horizontal single facet being colored by $k$ or $l$. (Note that $i,j,k,l$ are distinct from each other.) See Figure \ref{iiiiii gl4 figure} for reference.

\def\AlphaZTBetaZDefK{
\begin{tikzpicture}[draw=darkblue, scale=.8, xscale=-1, yscale=1.5,
    mid arrow/.style={
            postaction={
                decorate,
                decoration={
                    markings,
                    mark=at position 0.5 with {      
                        \arrow[##1]{>}
                    }
                }
            }
        },
        mid arrow/.default=black, 
        baseline={([yshift=-.8ex]current bounding box.center)}
    ]
\def \r{2}; 
\def \R{4}; 
\def\t{-7}; 
\begin{scope}[yscale=.5, yshift=-10cm]
    \foreach \a in {0,2,4}{
        \draw[postaction={mid arrow}] (60*\a+60+\t:\r) -- (60*\a+\t:\r)  ;
        \draw[postaction={mid arrow}] (60*\a+\t:\R) -- (60*\a+\t:\r);
    }
    \foreach \a in {1,3,5}{
        \draw[postaction={mid arrow}, double] (60*\a+60+\t:\r) -- (60*\a+\t:\r) ;
        \draw[postaction={mid arrow}] (60*\a+\t:\r) -- (60*\a+\t:\R);
\draw (60*3+58+\t:0.8*\R) node[above,scale=0.7]{$i$};  
\draw (60*4+62+\t:0.8*\R) node[above,scale=0.7]{$i$};  
\draw (210+60+\t:0.8*\r) node[above,scale=0.7]{$j$};  
    }
\end{scope}
\begin{scope}[yshift=-2.5 cm, yscale=.5]
    \foreach \a in {0,1,2,3,4,5}{
        \draw (60*\a+\t:\r) --  (60*\a+60+\t:\r);
        \draw[dotted] (60*\a+\t:\R) -- (60*\a+\t:\r);
    }
\draw (0,0) node[scale=0.7]{$i,j,l$};
\end{scope}
\begin{scope}[yscale=.5]
    \foreach \a in {0,2,4}{
        \draw[double, postaction={mid arrow}] (60*\a+\t:\r) --  (60*\a+60+\t:\r);
        \draw[postaction={mid arrow}] (60*\a+\t:\R) -- (60*\a+\t:\r);
    }
    \foreach \a in {1,3,5}{
        \draw[postaction={mid arrow}] (60*\a+\t:\r) -- (60*\a+60+\t:\r);
        \draw[postaction={mid arrow}] (60*\a+\t:\r) -- (60*\a+\t:\R);
    }
    \foreach \a in {0,1,2,3,4,5}{
        \draw[darkgreen, very thin] (60*\a+\t:\R) -- ($(60*\a+\t:\R) + (0,-10)$);
        \draw[darkgreen, very thin] (60*\a+\t:\r) -- ($(60*\a+\t:\r) + (0,-10)$);
    }
\draw (\t:0.8*\R) node[below,scale=0.7]{$i$};  
\draw (60*0+58+\t:0.8*\R) node[below,scale=0.7]{$i$};  
\draw (60*1+61+\t:0.75*\R) node[right,scale=0.7]{$i$};  
\draw (60*2+60+\t:0.8*\R) node[below,scale=0.7]{$i$};  
\draw (30+60+\t:0.8*\r) node[below,scale=0.7]{$l$};  
\end{scope}
\end{tikzpicture}
}

\def\AlphaZTBetaZDefL{
\begin{tikzpicture}[draw=darkblue, scale=.8, xscale=-1, yscale=1.5,
    mid arrow/.style={
            postaction={
                decorate,
                decoration={
                    markings,
                    mark=at position 0.5 with {      
                        \arrow[##1]{>}
                    }
                }
            }
        },
        mid arrow/.default=black, 
        baseline={([yshift=-.8ex]current bounding box.center)}
    ]
\def \r{2}; 
\def \R{4}; 
\def\t{-7}; 
\begin{scope}[yscale=.5, yshift=-10cm]
    \foreach \a in {0,2,4}{
        \draw[postaction={mid arrow}] (60*\a+60+\t:\r) -- (60*\a+\t:\r)  ;
        \draw[postaction={mid arrow}] (60*\a+\t:\R) -- (60*\a+\t:\r);
    }
    \foreach \a in {1,3,5}{
        \draw[postaction={mid arrow}, double] (60*\a+60+\t:\r) -- (60*\a+\t:\r) ;
        \draw[postaction={mid arrow}] (60*\a+\t:\r) -- (60*\a+\t:\R);
\draw (60*3+58+\t:0.8*\R) node[above,scale=0.7]{$i$};  
\draw (60*4+62+\t:0.8*\R) node[above,scale=0.7]{$i$};  
\draw (210+60+\t:0.8*\r) node[above,scale=0.7]{$j$};  
    }
\end{scope}
\begin{scope}[yshift=-2.5 cm, yscale=.5]
    \foreach \a in {0,1,2,3,4,5}{
        \draw (60*\a+\t:\r) --  (60*\a+60+\t:\r);
        \draw[dotted] (60*\a+\t:\R) -- (60*\a+\t:\r);
    }
\draw (0,0) node[scale=0.7]{$i,j,k$};
\end{scope}
\begin{scope}[yscale=.5]
    \foreach \a in {0,2,4}{
        \draw[double, postaction={mid arrow}] (60*\a+\t:\r) --  (60*\a+60+\t:\r);
        \draw[postaction={mid arrow}] (60*\a+\t:\R) -- (60*\a+\t:\r);
    }
    \foreach \a in {1,3,5}{
        \draw[postaction={mid arrow}] (60*\a+\t:\r) -- (60*\a+60+\t:\r);
        \draw[postaction={mid arrow}] (60*\a+\t:\r) -- (60*\a+\t:\R);
    }
    \foreach \a in {0,1,2,3,4,5}{
        \draw[darkgreen, very thin] (60*\a+\t:\R) -- ($(60*\a+\t:\R) + (0,-10)$);
        \draw[darkgreen, very thin] (60*\a+\t:\r) -- ($(60*\a+\t:\r) + (0,-10)$);
    }
\draw (\t:0.8*\R) node[below,scale=0.7]{$i$};  
\draw (60*0+58+\t:0.8*\R) node[below,scale=0.7]{$i$};  
\draw (60*1+61+\t:0.75*\R) node[right,scale=0.7]{$i$};  
\draw (60*2+60+\t:0.8*\R) node[below,scale=0.7]{$i$};  
\draw (30+60+\t:0.8*\r) node[below,scale=0.7]{$k$};  
\end{scope}
\end{tikzpicture}
}

\begin{figure}[htbp]
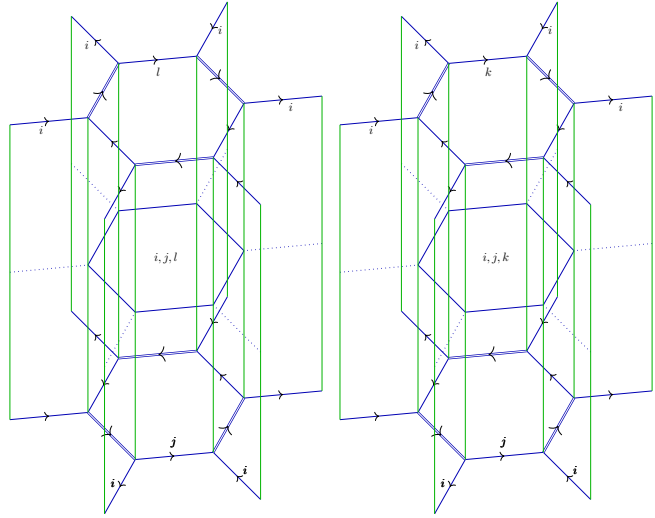

    \centering
\[
\scalebox{0.65}{$\displaystyle
\AlphaZTBetaZDefK \quad  \AlphaZTBetaZDefL
$}
\]   
    \caption{The two colorings of $\nu$ where the boundary points of the bottom hexagon are colored as \flapcolor[i][i][i][i][i][i].}
    \label{iiiiii gl4 figure}
\end{figure}

      The above two colorings of $\nu$ each induce a unique coloring of $\alpha \circ \nu$, denoted by $c_{ijl}$ and $c_{ijk}$. Explicitly comparing the monochrome and bichrome surfaces and signs of the intersection circles shows that $\langle \alpha \circ \nu, c_{ijl}\rangle + \langle \alpha \circ \nu, c_{ijk}\rangle =0 $. This computation follows a similar pattern to that of the proof of Lemma \ref{hexagonidempotent} in the case of the \flapcolor[i][j][k][i][j][k] coloring. 

    Now consider the colorings of $\nu$ where the boundary points of the bottom hexagon are colored as \flapcolor[i][i][j][j][i][i]. There are two possible colorings of the foam $\nu$. See Figure \ref{iijjiisl4 foam} for reference.

\def\AlphaZTBetaZDefKIJ{
\begin{tikzpicture}[draw=darkblue, scale=.8, xscale=-1, yscale=1.5,
    mid arrow/.style={
            postaction={
                decorate,
                decoration={
                    markings,
                    mark=at position 0.5 with {      
                        \arrow[##1]{>}
                    }
                }
            }
        },
        mid arrow/.default=black, 
        baseline={([yshift=-.8ex]current bounding box.center)}
    ]
\def \r{2}; 
\def \R{4}; 
\def\t{-7}; 
\begin{scope}[yscale=.5, yshift=-10cm]
    \foreach \a in {0,2,4}{
        \draw[postaction={mid arrow}] (60*\a+60+\t:\r) -- (60*\a+\t:\r)  ;
        \draw[postaction={mid arrow}] (60*\a+\t:\R) -- (60*\a+\t:\r);
    }
    \foreach \a in {1,3,5}{
        \draw[postaction={mid arrow}, double] (60*\a+60+\t:\r) -- (60*\a+\t:\r) ;
        \draw[postaction={mid arrow}] (60*\a+\t:\r) -- (60*\a+\t:\R);
\draw (60*3+58+\t:0.8*\R) node[above,scale=0.7]{$j$};  
\draw (60*4+62+\t:0.8*\R) node[above,scale=0.7]{$j$};  
\draw (210+60+\t:0.8*\r) node[above,scale=0.7]{$i$};  
    }
\end{scope}
\begin{scope}[yshift=-2.5 cm, yscale=.5]
    \foreach \a in {0,1,2,3,4,5}{
        \draw (60*\a+\t:\r) --  (60*\a+60+\t:\r);
        \draw[dotted] (60*\a+\t:\R) -- (60*\a+\t:\r);
    }
\draw (0,0) node[scale=0.7]{$i,j,l$};
\end{scope}
\begin{scope}[yscale=.5]
    \foreach \a in {0,2,4}{
        \draw[double, postaction={mid arrow}] (60*\a+\t:\r) --  (60*\a+60+\t:\r);
        \draw[postaction={mid arrow}] (60*\a+\t:\R) -- (60*\a+\t:\r);
    }
    \foreach \a in {1,3,5}{
        \draw[postaction={mid arrow}] (60*\a+\t:\r) -- (60*\a+60+\t:\r);
        \draw[postaction={mid arrow}] (60*\a+\t:\r) -- (60*\a+\t:\R);
    }
    \foreach \a in {0,1,2,3,4,5}{
        \draw[darkgreen, very thin] (60*\a+\t:\R) -- ($(60*\a+\t:\R) + (0,-10)$);
        \draw[darkgreen, very thin] (60*\a+\t:\r) -- ($(60*\a+\t:\r) + (0,-10)$);
    }
\draw (\t:0.8*\R) node[below,scale=0.7]{$i$};  
\draw (60*0+58+\t:0.8*\R) node[below,scale=0.7]{$i$};  
\draw (60*1+61+\t:0.75*\R) node[right,scale=0.7]{$i$};  
\draw (60*2+60+\t:0.8*\R) node[below,scale=0.7]{$i$};  
\draw (30+60+\t:0.8*\r) node[below,scale=0.7]{$l$};  
\end{scope}
\end{tikzpicture}
}

\def\AlphaZTBetaZDefLIJ{
\begin{tikzpicture}[draw=darkblue, scale=.8, xscale=-1, yscale=1.5,
    mid arrow/.style={
            postaction={
                decorate,
                decoration={
                    markings,
                    mark=at position 0.5 with {      
                        \arrow[##1]{>}
                    }
                }
            }
        },
        mid arrow/.default=black, 
        baseline={([yshift=-.8ex]current bounding box.center)}
    ]
\def \r{2}; 
\def \R{4}; 
\def\t{-7}; 
\begin{scope}[yscale=.5, yshift=-10cm]
    \foreach \a in {0,2,4}{
        \draw[postaction={mid arrow}] (60*\a+60+\t:\r) -- (60*\a+\t:\r)  ;
        \draw[postaction={mid arrow}] (60*\a+\t:\R) -- (60*\a+\t:\r);
    }
    \foreach \a in {1,3,5}{
        \draw[postaction={mid arrow}, double] (60*\a+60+\t:\r) -- (60*\a+\t:\r) ;
        \draw[postaction={mid arrow}] (60*\a+\t:\r) -- (60*\a+\t:\R);
\draw (60*3+58+\t:0.8*\R) node[above,scale=0.7]{$j$};  
\draw (60*4+62+\t:0.8*\R) node[above,scale=0.7]{$j$};  
\draw (210+60+\t:0.8*\r) node[above,scale=0.7]{$i$};  
    }
\end{scope}
\begin{scope}[yshift=-2.5 cm, yscale=.5]
    \foreach \a in {0,1,2,3,4,5}{
        \draw (60*\a+\t:\r) --  (60*\a+60+\t:\r);
        \draw[dotted] (60*\a+\t:\R) -- (60*\a+\t:\r);
    }
\draw (0,0) node[scale=0.7]{$i,j,k$};
\end{scope}
\begin{scope}[yscale=.5]
    \foreach \a in {0,2,4}{
        \draw[double, postaction={mid arrow}] (60*\a+\t:\r) --  (60*\a+60+\t:\r);
        \draw[postaction={mid arrow}] (60*\a+\t:\R) -- (60*\a+\t:\r);
    }
    \foreach \a in {1,3,5}{
        \draw[postaction={mid arrow}] (60*\a+\t:\r) -- (60*\a+60+\t:\r);
        \draw[postaction={mid arrow}] (60*\a+\t:\r) -- (60*\a+\t:\R);
    }
    \foreach \a in {0,1,2,3,4,5}{
        \draw[darkgreen, very thin] (60*\a+\t:\R) -- ($(60*\a+\t:\R) + (0,-10)$);
        \draw[darkgreen, very thin] (60*\a+\t:\r) -- ($(60*\a+\t:\r) + (0,-10)$);
    }
\draw (\t:0.8*\R) node[below,scale=0.7]{$i$};  
\draw (60*0+58+\t:0.8*\R) node[below,scale=0.7]{$i$};  
\draw (60*1+61+\t:0.75*\R) node[right,scale=0.7]{$i$};  
\draw (60*2+60+\t:0.8*\R) node[below,scale=0.7]{$i$};  
\draw (30+60+\t:0.8*\r) node[below,scale=0.7]{$k$};  
\end{scope}
\end{tikzpicture}
}

\begin{figure}[htbp]
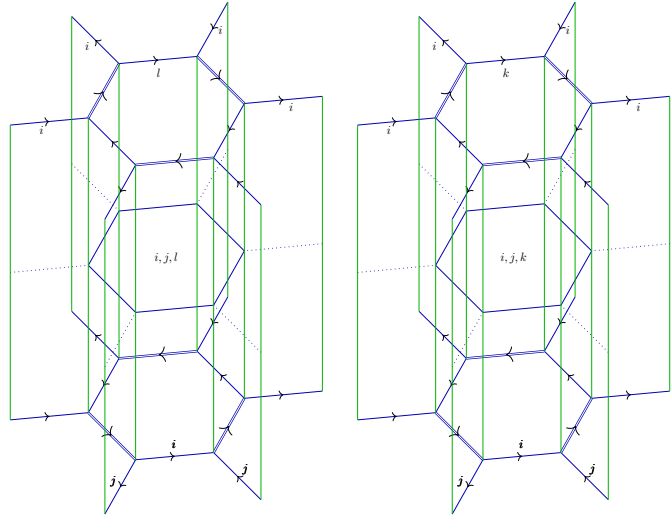

    \centering
\[
\scalebox{0.65}{$\displaystyle
\AlphaZTBetaZDefKIJ \qquad \AlphaZTBetaZDefLIJ
$}
\]   
    \caption{The two colorings of $\nu$ when the boundary points of the bottom hexagon are colored as \flapcolor[i][i][j][j][i][i].}
    \label{iijjiisl4 foam}
\end{figure}

 Again, it can be shown that the corresponding two colorings of $\alpha \circ \nu$ cancel each other out. 

Moreover, when the boundary points of the bottom hexagon of $\nu$ are colored by \flapcolor[i][i][j][j][k][k], \flapcolor[i][i][k][j][j][k], or \flapcolor[i][j][k][i][j][k], there is no way to color the foam $\alpha \circ \nu$. 
Summing over all possible colorings of the foam, it follows that $\alpha \circ \nu=0$. 

 \end{proof}

\begin{proposition}\label{idempotencyprop}
    The following two equalities from Theorem \ref{gl_4 sixvalent seam} hold:
    \[
        \beta \circ \alpha + \nu \circ \mu = \idfoam_{\Wzero}
        \quad 
        \text{and}
        \quad \xi \circ \eta + \mu \circ \nu =\idfoam_{\Wzeroa}.
    \]
\end{proposition}

We split the proof of this proposition into three lemmas.

\begin{lemma}\label{lindependence}
    There is a linear dependence between $\nu$ and $\nu \circ \mu \circ \nu $. Additionally, there is a linear dependence between $\alpha$ and $\alpha  \circ  \beta \circ \alpha$ (see Figure \ref{gl4foamdeef}).
\end{lemma}
\begin{proof}
    Observe that the total boundary of $\idfoam_{\Wzero}$, $\beta \circ \alpha$, and $\nu \circ \mu$ are all given by $\BenzeneBound$, which by Proposition \ref{prop:closed sixvertex evaluation} evaluates to 
    \[2[2]^2[3]^2[4]+[2]^2[3][4] = 2q^{-9} + \sum_{k>-9} a_k q^k. \]

    On the other hand, Lemma \ref{degcheck} shows that the degrees of the three foams above are all equal to $0$. Since \(\boundarydeg(\HexaBound)=9\), degree-zero morphisms correspond to the coefficient in evaluation degree \(-9\). Therefore the discussion in Section \ref{foam with boundary} implies that there is a linear relation between the foams. That is, there exist \(a,b,c \in \Q\), not all zero, such that
    \[ a\cdot \idfoam_{\Wzero} + b \cdot \beta \circ \alpha + c \cdot \nu \circ \mu=0. \]

    Precomposing with $\alpha$ completes the proof. 
    Explicitly,
    \[ 
        \alpha  \circ (a\cdot \idfoam_{\Wzero} + b \cdot \beta \circ \alpha + c \cdot \nu \circ \mu) = 0. 
    \]
     Since $ \alpha \circ \nu=0$ by Proposition \ref{orthogonalityprop}, the following must hold
    \[
        a \cdot \alpha + b \cdot \alpha  \circ  \beta \circ \alpha =0. 
    \]
     Similarly, 
     \[ 
        (a\cdot \idfoam_{\Wzero} + b \cdot \beta \circ \alpha + c \cdot \nu \circ \mu) \circ \nu = 0, 
    \]
     and again by Proposition \ref{orthogonalityprop}, the last equality needed to complete the proof follows
     \[
        a \cdot \nu + c \cdot \nu \circ \mu \circ \nu =0. 
    \]
\end{proof}
\begin{remark}\label{abcglobalremark}
    We preserve the variables $a,b,c$ in what follows. The next lemmas will establish that $-a=b=c=1$, completing the proof of Proposition \ref{idempotencyprop}.
\end{remark}

\begin{notation}\label{notation:table coloring contents}
    A few distinct surfaces appear in the proofs of the following lemmas.
    The \emph{saddle} is the usual neighborhood of an index-1 critical point, viewed as a disk embedded in $\mathbb{D}^2 \times [0,1]$. We can depict this disk by its projection onto the $\mathbb{D}^2$ factor as 
            \begin{tikzpicture}[scale=.5, baseline=-.5ex]
                \filldraw[white!90!black] (0,0) circle (1cm);
                \draw[dotted] (0,0) circle (1cm);
                \draw[ultra thick] (1,0) arc (0:90:1cm);
                \draw[ultra thick] (-1,0) arc (180:270:1cm);
            \end{tikzpicture},
        where the thick arcs are projections of boundary arcs embedded in $S^1 \times \{1\}$ and the dotted arcs are projections of boundary arcs embedded in $S^1 \times \{0\}$.
     The \emph{monkey saddle} (a.k.a. `triple saddle') is a disk embedded in $\mathbb{D}^2\times I$ that projects to 
            \begin{tikzpicture}[scale=.5, baseline=-.5ex]
                \filldraw[white!90!black] (0,0) circle (1cm);
                \draw[dotted] (0,0) circle (1cm);
                \draw[ultra thick] (0:1cm) arc (0:60:1cm);
                \draw[ultra thick] (120:1cm) arc (120:180:1cm);
                \draw[ultra thick] (240:1cm) arc (240:300:1cm);
            \end{tikzpicture}.
      The \emph{doubled monkey saddle} is obtained by stacking two monkey saddles vertically, resulting in a surface homeomorphic to a thrice-punctured sphere. 
\end{notation}

We will use both Notations \ref{notation:table coloring contents} and \ref{flapcolor on web} extensively in what follows.

\begin{lemma}\label{hexagonidempotent}
    The following equality holds:  $ \nu \circ \mu \circ \nu = \nu $ (see Figure \ref{gl4foamdeef}).
\end{lemma}
\begin{proof}

This is a direct foam evaluation. We are going to first analyze the most involved case in detail. This is the case of $\nu \circ \mu \circ \nu = \nu$ with the boundary points of the bottom web all colored with the same color \flapcolor[i][i][i][i][i][i], as depicted in Figure \ref{fig iiiiii gl4}.

\def\HexagonIdempA{
\begin{tikzpicture}[draw=darkblue, scale=.4, xscale=-1, yscale=1.5,
    mid arrow/.style={
            postaction={
                decorate,
                decoration={
                    markings,
                    mark=at position 0.5 with {      
                        \arrow[##1]{>}
                    }
                }
            }
        },
        mid arrow/.default=black, 
        baseline={([yshift=-.8ex]current bounding box.center)}
    ]
\def \r{2}; 
\def \R{4}; 
\def\t{-7}; 
\begin{scope}[yscale=.5, yshift=-10cm]
    \foreach \a in {0,2,4}{
        \draw[postaction={mid arrow}] (60*\a+60+\t:\r) -- (60*\a+\t:\r)  ;
        \draw[postaction={mid arrow}] (60*\a+\t:\R) -- (60*\a+\t:\r);
    }
    \foreach \a in {1,3,5}{
        \draw[postaction={mid arrow}, double] (60*\a+60+\t:\r) -- (60*\a+\t:\r) ;
        \draw[postaction={mid arrow}] (60*\a+\t:\r) -- (60*\a+\t:\R);
    }
\draw (60*3+58+\t:0.8*\R) node[above,scale=0.7]{$i$};  
\draw (60*4+62+\t:0.8*\R) node[above,scale=0.7]{$i$};  
\draw (210+60+\t:0.8*\r) node[above,scale=0.7]{$l$};  
\end{scope}
\begin{scope}[yshift=-2.5 cm, yscale=.5]
    \foreach \a in {0,1,2,3,4,5}{
        \draw (60*\a+\t:\r) --  (60*\a+60+\t:\r);
        \draw[dotted] (60*\a+\t:\R) -- (60*\a+\t:\r);
    }
   \draw (0,0.15) node[scale=0.7]{$i,j,l$};
\end{scope}
\begin{scope}[yscale=.5]
    \foreach \a in {0,2,4}{
        \draw[double, postaction={mid arrow}] (60*\a+\t:\r) --  (60*\a+60+\t:\r);
        \draw[postaction={mid arrow}] (60*\a+\t:\R) -- (60*\a+\t:\r);
    }
    \foreach \a in {1,3,5}{
        \draw[postaction={mid arrow}] (60*\a+\t:\r) -- (60*\a+60+\t:\r);
        \draw[postaction={mid arrow}] (60*\a+\t:\r) -- (60*\a+\t:\R);
    }
    \foreach \a in {0,1,2,3,4,5}{
        \draw[darkgreen, very thin] (60*\a+\t:\R) -- ($(60*\a+\t:\R) + (0,-10)$);
        \draw[darkgreen, very thin] (60*\a+\t:\r) -- ($(60*\a+\t:\r) + (0,-10)$);
    }
\end{scope}

\def \r{2}; 
\def \R{4}; 
\def\t{180-7}; 

\begin{scope}[ yscale=.5, yshift=5 cm]
    \foreach \a in {0,1,2,3,4,5}{
        \draw (60*\a+\t:\r) --  (60*\a+60+\t:\r);
        \draw[dotted] (60*\a+\t:\R) -- (60*\a+\t:\r);
    }
   \draw (0,0.15) node[scale=0.7]{$i,j,l$};
\end{scope}
\begin{scope}[yscale=.5,yshift=10 cm]
    \foreach \a in {0,2,4}{
        \draw[double, postaction={mid arrow}] (60*\a+\t:\r) --  (60*\a+60+\t:\r);
        \draw[postaction={mid arrow}] (60*\a+\t:\R) -- (60*\a+\t:\r);
    }
    \foreach \a in {1,3,5}{
        \draw[postaction={mid arrow}] (60*\a+\t:\r) -- (60*\a+60+\t:\r);
        \draw[postaction={mid arrow}] (60*\a+\t:\r) -- (60*\a+\t:\R);
    }
    \foreach \a in {0,1,2,3,4,5}{
        \draw[darkgreen, very thin] (60*\a+\t:\R) -- ($(60*\a+\t:\R) + (0,-10)$);
        \draw[darkgreen, very thin] (60*\a+\t:\r) -- ($(60*\a+\t:\r) + (0,-10)$);
    }
\end{scope}


\def \r{2}; 
\def \R{4}; 
\def\t{-7}; 

\begin{scope}[ yscale=.5, yshift=15cm]
    \foreach \a in {0,1,2,3,4,5}{
        \draw (60*\a+\t:\r) --  (60*\a+60+\t:\r);
        \draw[dotted] (60*\a+\t:\R) -- (60*\a+\t:\r);
    }
   \draw (0,0.15) node[scale=0.7]{$i,j,l$};
\end{scope}
\begin{scope}[yscale=.5, yshift=20cm]
    \foreach \a in {0,2,4}{
        \draw[double, postaction={mid arrow}] (60*\a+\t:\r) --  (60*\a+60+\t:\r);
        \draw[postaction={mid arrow}] (60*\a+\t:\R) -- (60*\a+\t:\r);
    }
    \foreach \a in {1,3,5}{
        \draw[postaction={mid arrow}] (60*\a+\t:\r) -- (60*\a+60+\t:\r);
        \draw[postaction={mid arrow}] (60*\a+\t:\r) -- (60*\a+\t:\R);
    }
    \foreach \a in {0,1,2,3,4,5}{
        \draw[darkgreen, very thin] (60*\a+\t:\R) -- ($(60*\a+\t:\R) + (0,-10)$);
        \draw[darkgreen, very thin] (60*\a+\t:\r) -- ($(60*\a+\t:\r) + (0,-10)$);
    }
\draw (\t:0.8*\R) node[below,scale=0.7]{$i$};  
\draw (60*0+58+\t:0.8*\R) node[below,scale=0.7]{$i$};  
\draw (60*1+61+\t:0.8*\R) node[below,scale=0.7]{$i$};  
\draw (60*2+60+\t:0.8*\R) node[below,scale=0.7]{$i$};  
\draw (30+60+\t:0.8*\r) node[below,scale=0.7]{$j$};  
\end{scope}

\end{tikzpicture}
}

\def\HexagonIdempAO{
\begin{tikzpicture}[draw=darkblue, scale=.4, xscale=-1, yscale=1.5,
    mid arrow/.style={
            postaction={
                decorate,
                decoration={
                    markings,
                    mark=at position 0.5 with {      
                        \arrow[##1]{>}
                    }
                }
            }
        },
        mid arrow/.default=black, 
        baseline={([yshift=-.8ex]current bounding box.center)}
    ]
\def \r{2}; 
\def \R{4}; 
\def\t{-7}; 
\begin{scope}[yscale=.5, yshift=-10cm]
    \foreach \a in {0,2,4}{
        \draw[postaction={mid arrow}] (60*\a+60+\t:\r) -- (60*\a+\t:\r)  ;
        \draw[postaction={mid arrow}] (60*\a+\t:\R) -- (60*\a+\t:\r);
    }
    \foreach \a in {1,3,5}{
        \draw[postaction={mid arrow}, double] (60*\a+60+\t:\r) -- (60*\a+\t:\r) ;
        \draw[postaction={mid arrow}] (60*\a+\t:\r) -- (60*\a+\t:\R);
    }
\draw (60*3+58+\t:0.8*\R) node[above,scale=0.7]{$i$};  
\draw (60*4+62+\t:0.8*\R) node[above,scale=0.7]{$i$};  
\draw (210+60+\t:0.8*\r) node[above,scale=0.7]{$l$};  
\end{scope}
\begin{scope}[yshift=-2.5 cm, yscale=.5]
    \foreach \a in {0,1,2,3,4,5}{
        \draw (60*\a+\t:\r) --  (60*\a+60+\t:\r);
        \draw[dotted] (60*\a+\t:\R) -- (60*\a+\t:\r);
    }
   \draw (0,0.15) node[scale=0.7]{$i,k,l$};
\end{scope}
\begin{scope}[yscale=.5]
    \foreach \a in {0,2,4}{
        \draw[double, postaction={mid arrow}] (60*\a+\t:\r) --  (60*\a+60+\t:\r);
        \draw[postaction={mid arrow}] (60*\a+\t:\R) -- (60*\a+\t:\r);
    }
    \foreach \a in {1,3,5}{
        \draw[postaction={mid arrow}] (60*\a+\t:\r) -- (60*\a+60+\t:\r);
        \draw[postaction={mid arrow}] (60*\a+\t:\r) -- (60*\a+\t:\R);
    }
    \foreach \a in {0,1,2,3,4,5}{
        \draw[darkgreen, very thin] (60*\a+\t:\R) -- ($(60*\a+\t:\R) + (0,-10)$);
        \draw[darkgreen, very thin] (60*\a+\t:\r) -- ($(60*\a+\t:\r) + (0,-10)$);
    }
\end{scope}

\def \r{2}; 
\def \R{4}; 
\def\t{180-7}; 

\begin{scope}[ yscale=.5, yshift=5 cm]
    \foreach \a in {0,1,2,3,4,5}{
        \draw (60*\a+\t:\r) --  (60*\a+60+\t:\r);
        \draw[dotted] (60*\a+\t:\R) -- (60*\a+\t:\r);
    }
   \draw (0,0.15) node[scale=0.7]{$i,k,l$};
\end{scope}
\begin{scope}[yscale=.5,yshift=10 cm]
    \foreach \a in {0,2,4}{
        \draw[double, postaction={mid arrow}] (60*\a+\t:\r) --  (60*\a+60+\t:\r);
        \draw[postaction={mid arrow}] (60*\a+\t:\R) -- (60*\a+\t:\r);
    }
    \foreach \a in {1,3,5}{
        \draw[postaction={mid arrow}] (60*\a+\t:\r) -- (60*\a+60+\t:\r);
        \draw[postaction={mid arrow}] (60*\a+\t:\r) -- (60*\a+\t:\R);
    }
    \foreach \a in {0,1,2,3,4,5}{
        \draw[darkgreen, very thin] (60*\a+\t:\R) -- ($(60*\a+\t:\R) + (0,-10)$);
        \draw[darkgreen, very thin] (60*\a+\t:\r) -- ($(60*\a+\t:\r) + (0,-10)$);
    }
\end{scope}


\def \r{2}; 
\def \R{4}; 
\def\t{-7}; 

\begin{scope}[ yscale=.5, yshift=15cm]
    \foreach \a in {0,1,2,3,4,5}{
        \draw (60*\a+\t:\r) --  (60*\a+60+\t:\r);
        \draw[dotted] (60*\a+\t:\R) -- (60*\a+\t:\r);
    }
   \draw (0,0.15) node[scale=0.7]{$i,j,l$};
\end{scope}
\begin{scope}[yscale=.5, yshift=20cm]
    \foreach \a in {0,2,4}{
        \draw[double, postaction={mid arrow}] (60*\a+\t:\r) --  (60*\a+60+\t:\r);
        \draw[postaction={mid arrow}] (60*\a+\t:\R) -- (60*\a+\t:\r);
    }
    \foreach \a in {1,3,5}{
        \draw[postaction={mid arrow}] (60*\a+\t:\r) -- (60*\a+60+\t:\r);
        \draw[postaction={mid arrow}] (60*\a+\t:\r) -- (60*\a+\t:\R);
    }
    \foreach \a in {0,1,2,3,4,5}{
        \draw[darkgreen, very thin] (60*\a+\t:\R) -- ($(60*\a+\t:\R) + (0,-10)$);
        \draw[darkgreen, very thin] (60*\a+\t:\r) -- ($(60*\a+\t:\r) + (0,-10)$);
    }
\draw (\t:0.8*\R) node[below,scale=0.7]{$i$};  
\draw (60*0+58+\t:0.8*\R) node[below,scale=0.7]{$i$};  
\draw (60*1+61+\t:0.8*\R) node[below,scale=0.7]{$i$};  
\draw (60*2+60+\t:0.8*\R) node[below,scale=0.7]{$i$};  
\draw (30+60+\t:0.8*\r) node[below,scale=0.7]{$j$};  
\end{scope}

\end{tikzpicture}
}

\def\HexagonIdempAOO{
\begin{tikzpicture}[draw=darkblue, scale=.4, xscale=-1, yscale=1.5,
    mid arrow/.style={
            postaction={
                decorate,
                decoration={
                    markings,
                    mark=at position 0.5 with {      
                        \arrow[##1]{>}
                    }
                }
            }
        },
        mid arrow/.default=black, 
        baseline={([yshift=-.8ex]current bounding box.center)}
    ]
\def \r{2}; 
\def \R{4}; 
\def\t{-7}; 
\begin{scope}[yscale=.5, yshift=-10cm]
    \foreach \a in {0,2,4}{
        \draw[postaction={mid arrow}] (60*\a+60+\t:\r) -- (60*\a+\t:\r)  ;
        \draw[postaction={mid arrow}] (60*\a+\t:\R) -- (60*\a+\t:\r);
    }
    \foreach \a in {1,3,5}{
        \draw[postaction={mid arrow}, double] (60*\a+60+\t:\r) -- (60*\a+\t:\r) ;
        \draw[postaction={mid arrow}] (60*\a+\t:\r) -- (60*\a+\t:\R);
    }
\draw (60*3+58+\t:0.8*\R) node[above,scale=0.7]{$i$};  
\draw (60*4+62+\t:0.8*\R) node[above,scale=0.7]{$i$};  
\draw (210+60+\t:0.8*\r) node[above,scale=0.7]{$l$};  
\end{scope}
\begin{scope}[yshift=-2.5 cm, yscale=.5]
    \foreach \a in {0,1,2,3,4,5}{
        \draw (60*\a+\t:\r) --  (60*\a+60+\t:\r);
        \draw[dotted] (60*\a+\t:\R) -- (60*\a+\t:\r);
    }
   \draw (0,0.15) node[scale=0.7]{$i,j,l$};
\end{scope}
\begin{scope}[yscale=.5]
    \foreach \a in {0,2,4}{
        \draw[double, postaction={mid arrow}] (60*\a+\t:\r) --  (60*\a+60+\t:\r);
        \draw[postaction={mid arrow}] (60*\a+\t:\R) -- (60*\a+\t:\r);
    }
    \foreach \a in {1,3,5}{
        \draw[postaction={mid arrow}] (60*\a+\t:\r) -- (60*\a+60+\t:\r);
        \draw[postaction={mid arrow}] (60*\a+\t:\r) -- (60*\a+\t:\R);
    }
    \foreach \a in {0,1,2,3,4,5}{
        \draw[darkgreen, very thin] (60*\a+\t:\R) -- ($(60*\a+\t:\R) + (0,-10)$);
        \draw[darkgreen, very thin] (60*\a+\t:\r) -- ($(60*\a+\t:\r) + (0,-10)$);
    }
\end{scope}

\def \r{2}; 
\def \R{4}; 
\def\t{180-7}; 

\begin{scope}[ yscale=.5, yshift=5 cm]
    \foreach \a in {0,1,2,3,4,5}{
        \draw (60*\a+\t:\r) --  (60*\a+60+\t:\r);
        \draw[dotted] (60*\a+\t:\R) -- (60*\a+\t:\r);
    }
   \draw (0,0.15) node[scale=0.7]{$i,j,k$};
\end{scope}
\begin{scope}[yscale=.5,yshift=10 cm]
    \foreach \a in {0,2,4}{
        \draw[double, postaction={mid arrow}] (60*\a+\t:\r) --  (60*\a+60+\t:\r);
        \draw[postaction={mid arrow}] (60*\a+\t:\R) -- (60*\a+\t:\r);
    }
    \foreach \a in {1,3,5}{
        \draw[postaction={mid arrow}] (60*\a+\t:\r) -- (60*\a+60+\t:\r);
        \draw[postaction={mid arrow}] (60*\a+\t:\r) -- (60*\a+\t:\R);
    }
    \foreach \a in {0,1,2,3,4,5}{
        \draw[darkgreen, very thin] (60*\a+\t:\R) -- ($(60*\a+\t:\R) + (0,-10)$);
        \draw[darkgreen, very thin] (60*\a+\t:\r) -- ($(60*\a+\t:\r) + (0,-10)$);
    }
\end{scope}


\def \r{2}; 
\def \R{4}; 
\def\t{-7}; 

\begin{scope}[ yscale=.5, yshift=15cm]
    \foreach \a in {0,1,2,3,4,5}{
        \draw (60*\a+\t:\r) --  (60*\a+60+\t:\r);
        \draw[dotted] (60*\a+\t:\R) -- (60*\a+\t:\r);
    }
   \draw (0,0.15) node[scale=0.7]{$i,j,k$};
\end{scope}
\begin{scope}[yscale=.5, yshift=20cm]
    \foreach \a in {0,2,4}{
        \draw[double, postaction={mid arrow}] (60*\a+\t:\r) --  (60*\a+60+\t:\r);
        \draw[postaction={mid arrow}] (60*\a+\t:\R) -- (60*\a+\t:\r);
    }
    \foreach \a in {1,3,5}{
        \draw[postaction={mid arrow}] (60*\a+\t:\r) -- (60*\a+60+\t:\r);
        \draw[postaction={mid arrow}] (60*\a+\t:\r) -- (60*\a+\t:\R);
    }
    \foreach \a in {0,1,2,3,4,5}{
        \draw[darkgreen, very thin] (60*\a+\t:\R) -- ($(60*\a+\t:\R) + (0,-10)$);
        \draw[darkgreen, very thin] (60*\a+\t:\r) -- ($(60*\a+\t:\r) + (0,-10)$);
    }
\draw (\t:0.8*\R) node[below,scale=0.7]{$i$};  
\draw (60*0+58+\t:0.8*\R) node[below,scale=0.7]{$i$};  
\draw (60*1+61+\t:0.8*\R) node[below,scale=0.7]{$i$};  
\draw (60*2+60+\t:0.8*\R) node[below,scale=0.7]{$i$};  
\draw (30+60+\t:0.8*\r) node[below,scale=0.7]{$j$};  
\end{scope}

\end{tikzpicture}
}

\def\HexagonIdempB{
\begin{tikzpicture}[draw=darkblue, scale=.4, xscale=-1, yscale=1.5,
    mid arrow/.style={
            postaction={
                decorate,
                decoration={
                    markings,
                    mark=at position 0.5 with {      
                        \arrow[##1]{>}
                    }
                }
            }
        },
        mid arrow/.default=black, 
        baseline={([yshift=-.8ex]current bounding box.center)}
    ]
\def \r{2}; 
\def \R{4}; 
\def\t{-7}; 
\begin{scope}[yscale=.5, yshift=-10cm]
    \foreach \a in {0,2,4}{
        \draw[postaction={mid arrow}] (60*\a+60+\t:\r) -- (60*\a+\t:\r)  ;
        \draw[postaction={mid arrow}] (60*\a+\t:\R) -- (60*\a+\t:\r);
    }
    \foreach \a in {1,3,5}{
        \draw[postaction={mid arrow}, double] (60*\a+60+\t:\r) -- (60*\a+\t:\r) ;
        \draw[postaction={mid arrow}] (60*\a+\t:\r) -- (60*\a+\t:\R);
    }
\draw (60*3+58+\t:0.8*\R) node[above,scale=0.7]{$i$};  
\draw (60*4+62+\t:0.8*\R) node[above,scale=0.7]{$i$};  
\draw (210+60+\t:0.8*\r) node[above,scale=0.7]{$l$};  
\end{scope}
\begin{scope}[yshift=-2.5 cm, yscale=.5]
    \foreach \a in {0,1,2,3,4,5}{
        \draw (60*\a+\t:\r) --  (60*\a+60+\t:\r);
        \draw[dotted] (60*\a+\t:\R) -- (60*\a+\t:\r);
    }
   
   \draw (0,0.15) node[scale=0.7]{$i,j,l$};
\end{scope}
\begin{scope}[yscale=.5]
    \foreach \a in {0,2,4}{
        \draw[double, postaction={mid arrow}] (60*\a+\t:\r) --  (60*\a+60+\t:\r);
        \draw[postaction={mid arrow}] (60*\a+\t:\R) -- (60*\a+\t:\r);
    }
    \foreach \a in {1,3,5}{
        \draw[postaction={mid arrow}] (60*\a+\t:\r) -- (60*\a+60+\t:\r);
        \draw[postaction={mid arrow}] (60*\a+\t:\r) -- (60*\a+\t:\R);
    }
    \foreach \a in {0,1,2,3,4,5}{
        \draw[darkgreen, very thin] (60*\a+\t:\R) -- ($(60*\a+\t:\R) + (0,-10)$);
        \draw[darkgreen, very thin] (60*\a+\t:\r) -- ($(60*\a+\t:\r) + (0,-10)$);
    }
\draw (\t:0.8*\R) node[below,scale=0.7]{$i$};  
\draw (60*0+58+\t:0.8*\R) node[below,scale=0.7]{$i$};  
\draw (60*1+61+\t:0.8*\R) node[below,scale=0.7]{$i$};  
\draw (60*2+60+\t:0.8*\R) node[below,scale=0.7]{$i$};  
\draw (30+60+\t:0.8*\r) node[below,scale=0.7]{$j$};  
\end{scope}
\end{tikzpicture}
}

\begin{figure}[htbp]
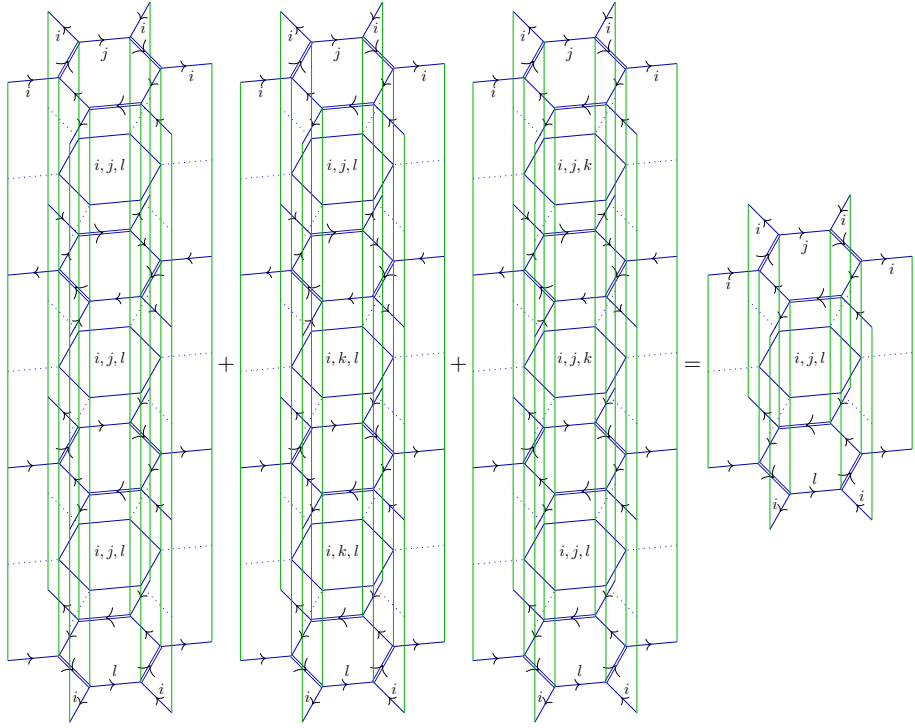

    \centering
\[
\scalebox{0.85}{$\displaystyle
\HexagonIdempA + \HexagonIdempAO + \HexagonIdempAOO =\HexagonIdempB
$}
\] 
    \caption{\flapcolor[i][i][i][i][i][i] coloring of $\nu $ (on the right side of the equals sign) and the induced colorings of $\nu \circ \mu \circ \nu $ (on the left).}
    \label{fig iiiiii gl4}
\end{figure}

Denote the foam appearing on the left side of the above equality by $F_1$ and the foam on the right side by $F_2$. 
Notice that $F_1=\nu \circ \mu \circ \nu $ and $F_2=\nu$. Notice that in Figure \ref{fig iiiiii gl4} any coloring $c$ of $F_2$ on the right induces a coloring of $F_1$. 

Denote the presented colorings in Figure \ref{fig iiiiii gl4} of $F_1$ by $c_0, c_1, c_2$ from left to right; these colorings are induced by the coloring $c$ of $F_2$. In particular, in $c_0$, the bottom benzene web is colored by $i$ and $l$, and the horizontal hexagons are all colored by $i,j,l$. It is evident that $c_0, c_1$ and $c_2$ are the only colorings induced by $c$. Using the figure we obtain Table \ref{[i][i][i][i][i][i] coloring of nu gl4}.

    \begin{table}[h!] 
\centering
\arrayrulecolor{black} 
\begin{tabular}{|l|l|l|}
\hline
{\ul }           & $(F_1, c_0)$                                                            & $(F_2, c)$              \\
\hline 
$F_{ij}$         & $\mathbb{D}^2 \sqcup \mathbb{D}^2 \sqcup \mathbb{D}^2$ & $\mathbb{D}^2 \sqcup \mathbb{D}^2 \sqcup \mathbb{D}^2$ \\
\arrayrulecolor[gray]{0.8} \hline
$F_{ik}$         & doubled monkey saddle & monkey saddle \\
\arrayrulecolor[gray]{0.8} \hline
$F_{il}$         & $\mathbb{D}^2 \sqcup \mathbb{D}^2 \sqcup \mathbb{D}^2$ & $\mathbb{D}^2 \sqcup \mathbb{D}^2 \sqcup \mathbb{D}^2$ \\
\arrayrulecolor[gray]{0.8} \hline
$F_{jk}$         & $\mathbb{D}^2$ (cup) $\sqcup \mathbb{S}^2$ & $\mathbb{D}^2$ (cup)  \\
\arrayrulecolor[gray]{0.8} \hline
$F_{jl}$         & $\mathbb{S}^1 \times \mathbb{D}^1$ & $\mathbb{S}^1 \times \mathbb{D}^1$ \\
\arrayrulecolor[gray]{0.8} \hline
$F_{kl}$         & $\mathbb{D}^2$ (cup) $\sqcup \mathbb{S}^2$ & $\mathbb{D}^2$ (cup)  \\
\arrayrulecolor{black} \hline
$ {F_i}$ & doubled monkey saddle & monkey saddle \\
\arrayrulecolor[gray]{0.8} \hline
$ {F_j}$ & $\mathbb{D}^2$ (cup) $\sqcup \mathbb{S}^2$ & $\mathbb{D}^2$ (cup)  \\
\arrayrulecolor[gray]{0.8} \hline
$ {F_k}$ & $\emptyset$ & $\emptyset$ \\
\arrayrulecolor[gray]{0.8}\hline
$ {F_l}$ & $\mathbb{D}^2$ (cap) $\sqcup \mathbb{S}^2$ & $\mathbb{D}^2$ (cap) \\
\arrayrulecolor{black} \hline
$\theta_{ij}$    & three cups $\sqcup  \mathbb{S}^1 \sqcup \mathbb{S}^1 \sqcup \mathbb{S}^1$ (positive if $j > i$) & three cups \\
\arrayrulecolor[gray]{0.8} \hline
$\theta_{ik}$    & $\emptyset$ & $\emptyset$ \\
\arrayrulecolor[gray]{0.8} \hline
$\theta_{il}$    & three caps $\sqcup  \mathbb{S}^1 \sqcup \mathbb{S}^1 \sqcup \mathbb{S}^1$ (positive if $i > l$) & three caps \\
\arrayrulecolor[gray]{0.8} \hline
$\theta_{jk}$    & $\emptyset$ & $\emptyset$ \\
\arrayrulecolor[gray]{0.8}\hline
$\theta_{jl}$    & $\mathbb{S}^1 \sqcup \mathbb{S}^1 \sqcup \mathbb{S}^1$ & $\mathbb{S}^1$ \\
\arrayrulecolor[gray]{0.8} \hline
$\theta_{kl}$    & $\emptyset$ & $\emptyset$ \\
\arrayrulecolor{black} \hline 
\end{tabular}
\caption{\flapcolor[i][i][i][i][i][i] coloring of $\nu$}\label{[i][i][i][i][i][i] coloring of nu gl4}
\end{table}

As explained in the above paragraph (see also Figure \ref{fig iiiiii gl4}), for a fixed coloring $c$ of the foam $F_2$, there are three induced colorings of the foam $F_1$. 
The coloring $c_1$ is obtained from $c_0$ by a Kempe move along the \kempe[k][l] sphere and $c_2$ by a Kempe move along the  \kempe[j][k] sphere. As noted, $F_{jk}$ and $F_{kl}$ are a disjoint union of a hemisphere and a sphere; we are performing the Kempe move only along the sphere in each of the two cases. 
It is straightforward to check that the described  \kempe[j][k] Kempe move simply swaps the $j$ and $k$ indices of bichrome surfaces. The argument is similar for the \kempe[k][l] Kempe move. This tells us how to modify Table \ref{[i][i][i][i][i][i] coloring of nu gl4} for $(F_1, c_1)$ and $(F_1, c_2)$.

In order to make the analysis clearer, set $i = 1$, $j = 2$, $k = 3$, $l = 4$. Other cases follow analogously. 
By the above discussion and Table \ref{[i][i][i][i][i][i] coloring of nu gl4}, using Lemma \ref{rwlemma2.19}, we compute

\begin{align*}
    &\frac{\langle F_1,c_0 \rangle + \langle F_1,c_1 \rangle  +  \langle F_1,c_2 \rangle }{\langle F_2, c \rangle}=\\
    &= \frac{s(F_1,c_0)}{s(F_2,c)} \cdot \left(\frac{(\vara  - \varc )^2}{(\varb  - \varc )(\varc  - \vard )} - \frac{(\vara  - \varb )^2}{(\varb  - \varc )(\varb  - \vard )} - \frac{(\vara  - \vard )^2}{(\varc  - \vard )(\varb  - \vard )}\right) \\
    &= \frac{s(F_1,c_0)}{s(F_2,c)} (-1) \\
    &= 1,
\end{align*}
where the last equality follows by computing $s(F_1, c_0) - s(F_2, c)$ using Table \ref{[i][i][i][i][i][i] coloring of nu gl4}.

Next, we analyze the case where the boundary points of the bottom web are colored as \flapcolor[j][i][i][i][i][j]. Assume that the horizontal hexagons in the coloring $c_0$ of $F_1$ are all colored by $i,j,l$ and that the color $k$ appears on the bottom benzene web (but not on the top one). 
Once again, call the foam on the right side of the equality $F_2 = \nu$ and the one on the left $F_1 = \nu \circ \mu \circ \nu$. We obtain Table \ref{[j][i][i][i][i][j] color nu gl4}.

\begin{table}[h!]
\centering
\arrayrulecolor{black} 
\begin{tabular}{|l|l|l|}
\hline
{\ul }           & $(F_1, c_0)$                                                            & $(F_2, c)$              \\
\hline 
$F_{ij}$         & $\mathbb{D}^2 \sqcup \mathbb{D}^2 \sqcup \mathbb{D}^2$ & $\mathbb{D}^2 \sqcup \mathbb{D}^2 \sqcup \mathbb{D}^2$ \\
\arrayrulecolor[gray]{0.8} \hline
$F_{ik}$         & connected sum of a saddle &  saddle \\
\arrayrulecolor[gray]{0.8} \hline
$F_{il}$         & $\mathbb{D}^2  \sqcup \mathbb{D}^2 $ & $\mathbb{D}^2  \sqcup \mathbb{D}^2 $ \\
\arrayrulecolor[gray]{0.8} \hline
$F_{jk}$         & $\mathbb{D}^2$ & $\mathbb{D}^2$ \\
\arrayrulecolor[gray]{0.8} \hline
$F_{jl}$         &  $\mathbb{D}^2$ & $\mathbb{D}^2$ \\
\arrayrulecolor[gray]{0.8} \hline
$F_{kl}$         & $\mathbb{D}^2$ (cap) $\sqcup \mathbb{S}^2$ & $\mathbb{D}^2$ (cap)\\
\arrayrulecolor{black} \hline
$ {F_i}$ & connected sum of a saddle &  saddle \\
\arrayrulecolor[gray]{0.8} \hline
$ {F_j}$ & $\mathbb{D}^2$ & $\mathbb{D}^2$ \\
\arrayrulecolor[gray]{0.8} \hline
$ {F_k}$ & $\emptyset$ & $\emptyset$ \\
\arrayrulecolor[gray]{0.8}\hline
$ {F_l}$ & $\mathbb{D}^2$ (cap) $\sqcup \mathbb{S}^2$ & $\mathbb{D}^2$ (cap) \\
\arrayrulecolor{black} \hline
$\theta_{ij}$    & three cups $\sqcup  \mathbb{S}^1 \sqcup \mathbb{S}^1 \sqcup \mathbb{S}^1$ (positive if $j > i$) & three cups \\
\arrayrulecolor[gray]{0.8} \hline
$\theta_{ik}$    & $\emptyset$ & $\emptyset$ \\
\arrayrulecolor[gray]{0.8} \hline
$\theta_{il}$    & two caps $\sqcup  \mathbb{S}^1 \sqcup \mathbb{S}^1 $ & two caps \\
\arrayrulecolor[gray]{0.8} \hline
$\theta_{jk}$    & $\emptyset$ & $\emptyset$ \\
\arrayrulecolor[gray]{0.8}\hline
$\theta_{jl}$    & one cap $\sqcup  \mathbb{S}^1 $ (positive if $l > j$) & one cap  \\
\arrayrulecolor[gray]{0.8} \hline
$\theta_{kl}$    & $\emptyset$ & $\emptyset$ \\
\arrayrulecolor{black} \hline 
\end{tabular}
\caption{\flapcolor[j][i][i][i][i][j] coloring of $\nu$}\label{[j][i][i][i][i][j] color nu gl4}
\end{table}

There is one extra coloring $c_1$ of $F_1$ coming from performing a  \kempe[k][l] Kempe move along the sphere. Taking that into account, the sum of foam evaluations for these colorings is (again only considering the case $i =1, j =2, k =3, l =4$)
\begin{align*}
    \frac{\langle F_1,c_0 \rangle + \langle F_1,c_1 \rangle }{\langle F_2, c \rangle} &= \frac{s(F_1,c_0)}{s(F_2,c)} \cdot \left(\frac{(\vara  - \vard )}{(\varc  - \vard )} - \frac{(\vara  - \varc )}{(\varc  - \vard )}\right) \\
    &= \frac{s(F_1,c_0)}{s(F_2,c)} \\
    &= 1.
\end{align*}

Lastly, consider the case where the boundary points of the bottom web are colored as \flapcolor[i][j][k][i][j][k]. Here, the horizontal hexagons must be colored by $i,j,k$ (the same holds for the two cases of the next sentence). The case \flapcolor[i][i][j][j][k][k] is almost identical to the case \flapcolor[i][j][k][i][j][k], so we omit it. 

\begin{table}[h!]
\centering
\arrayrulecolor{black} 
\begin{tabular}{|l|l|l|}
\hline
{\ul }           & $(F_1, c_0)$                                                            & $(F_2, c)$              \\
\hline 
$F_{ij}$         & $\mathbb{D}^2 \sqcup \mathbb{D}^2 $ & $\mathbb{D}^2 \sqcup \mathbb{D}^2 $ \\
\arrayrulecolor[gray]{0.8} \hline
$F_{ik}$         & $\mathbb{D}^2 \sqcup \mathbb{D}^2 $ & $\mathbb{D}^2 \sqcup \mathbb{D}^2 $ \\
\arrayrulecolor[gray]{0.8} \hline
$F_{il}$         & $\mathbb{D}^2 $ & $\mathbb{D}^2 $ \\
\arrayrulecolor[gray]{0.8} \hline
$F_{jk}$         & $\mathbb{D}^2 \sqcup \mathbb{D}^2 $ & $\mathbb{D}^2 \sqcup \mathbb{D}^2 $ \\
\arrayrulecolor[gray]{0.8} \hline
$F_{jl}$         &  $\mathbb{D}^2$ & $\mathbb{D}^2$ \\
\arrayrulecolor[gray]{0.8} \hline
$F_{kl}$         & $\mathbb{D}^2 $ & $\mathbb{D}^2 $ \\
\arrayrulecolor{black} \hline
$ {F_i}$ & $\mathbb{D}^2 $ & $\mathbb{D}^2 $  \\
\arrayrulecolor[gray]{0.8} \hline
$ {F_j}$ & $\mathbb{D}^2$ & $\mathbb{D}^2$ \\
\arrayrulecolor[gray]{0.8} \hline
$ {F_k}$ & $\mathbb{D}^2$ & $\mathbb{D}^2$ \\
\arrayrulecolor[gray]{0.8}\hline
$ {F_l}$ & $\emptyset$ & $\emptyset$ \\
\arrayrulecolor{black} \hline
$\theta_{ij}$    & two twisted strands & two strands \\
\arrayrulecolor[gray]{0.8} \hline
$\theta_{ik}$    & two caps $\sqcup  \mathbb{S}^1 \sqcup \mathbb{S}^1 $ & two caps  \\
\arrayrulecolor[gray]{0.8} \hline
$\theta_{il}$    & $\emptyset$ & $\emptyset$\\
\arrayrulecolor[gray]{0.8} \hline
$\theta_{jk}$    & two cups $\sqcup  \mathbb{S}^1 \sqcup \mathbb{S}^1 $ & two cups  \\
\arrayrulecolor[gray]{0.8}\hline
$\theta_{jl}$    & $\emptyset$ & $\emptyset$ \\
\arrayrulecolor[gray]{0.8} \hline
$\theta_{kl}$    & $\emptyset$ & $\emptyset$ \\
\arrayrulecolor{black} \hline 
\end{tabular}
\caption{\flapcolor[i][j][k][i][j][k] coloring of $\nu$}\label{[i][j][k][i][j][k] coloring of nu gl4}
\end{table}

This time, there is only one induced coloring $c_0$ of $F_1$ coming from $(F_2, c)$. A direct comparison of columns in Table \ref{[i][j][k][i][j][k] coloring of nu gl4} shows that the evaluation agrees. We have now shown that $ \nu \circ \mu \circ \nu = \nu$.
\end{proof}

\begin{lemma}
    The following equality holds: $\alpha  \circ  \beta \circ \alpha = \alpha$ (see Figure \ref{gl4foamdeef}).
\end{lemma}
\begin{remark}
     One could alternatively show that $\beta \circ \alpha$ is an idempotent. This would follow essentially the same computation; in the tables below, some additional semicircles would appear and some of the sheets would change into saddles. 
     We instead prove that $\alpha  \circ  \beta \circ \alpha = \alpha$, as there are fewer colorings on $\alpha \circ \beta \circ \alpha$: there is only one choice of the additional color at the matching web on the bottom instead of a choice both at the bottom and top. 
\end{remark}
\begin{proof}
For simplicity, write  $F_1 = \alpha \circ \beta \circ \alpha$ and $F_2 = \alpha$. There are three boundary colorings that give a well defined coloring of $F_2$. We present the tables for those.

First, consider the colorings where the boundary points of the bottom web of $F_2$ are colored by \flapcolor[i][i][i][i][i][i]. Assume that the sphere appearing in the middle is colored by $k$ and that the additional color on the bottom hemisphere is $l$. We obtain Table \ref{[i][i][i][i][i][i] coloring of alpha gl4}. 

\begin{table}[h!]
\centering
\arrayrulecolor{black} 
\begin{tabular}{|l|l|l|}
\hline
{\ul }           & $(F_1, c_0)$                                                            & $(F_2, c)$              \\
\hline 
$F_{ij}$         & $\mathbb{D}^2 \sqcup \mathbb{D}^2  \sqcup \mathbb{D}^2 $ & $\mathbb{D}^2 \sqcup \mathbb{D}^2  \sqcup \mathbb{D}^2 $ \\
\arrayrulecolor[gray]{0.8} \hline
$F_{ik}$         & two stacked monkey saddles met at the saddle & $\mathbb{D}^2 \sqcup \mathbb{D}^2 \sqcup \mathbb{D}^2  $ \\
\arrayrulecolor[gray]{0.8} \hline
$F_{il}$         &   monkey saddle  &  monkey saddle \\
\arrayrulecolor[gray]{0.8} \hline
$F_{jk}$         & $\mathbb{S}^2 $ & $\emptyset$ \\
\arrayrulecolor[gray]{0.8} \hline
$F_{jl}$         &   $\mathbb{D}^2$ (cap) &  $\mathbb{D}^2$ (cap) \\
\arrayrulecolor[gray]{0.8} \hline
$F_{kl}$         & $\mathbb{S}^2 \sqcup $  $\mathbb{D}^2$ (cap) &  $\mathbb{D}^2$ (cap)  \\
\arrayrulecolor{black} \hline
$ {F_i}$ & $\mathbb{D}^2 \sqcup \mathbb{D}^2 \sqcup  \mathbb{D}^2 $ & $\mathbb{D}^2 \sqcup \mathbb{D}^2 \sqcup  \mathbb{D}^2  $  \\
\arrayrulecolor[gray]{0.8} \hline
$ {F_j}$ &$\emptyset$ & $\emptyset$ \\
\arrayrulecolor[gray]{0.8} \hline
$ {F_k}$ & $\mathbb{S}^2$ & $\emptyset$ \\
\arrayrulecolor[gray]{0.8}\hline
$ {F_l}$ &  $\mathbb{D}^2$ (cap) &  $\mathbb{D}^2$ (cap)  \\
\arrayrulecolor{black} \hline
$\theta_{ij}$    &$\emptyset$ & $\emptyset$\\
\arrayrulecolor[gray]{0.8} \hline
$\theta_{ik}$    & $\mathbb{S}^1 \sqcup \mathbb{S}^1 \sqcup  \mathbb{S}^1 $ positive if $k > i$ & $\emptyset$  \\
\arrayrulecolor[gray]{0.8} \hline
$\theta_{il}$    & three cups  & three cups \\
\arrayrulecolor[gray]{0.8} \hline
$\theta_{jk}$    & $\emptyset$ & $\emptyset$  \\
\arrayrulecolor[gray]{0.8}\hline
$\theta_{jl}$    & $\emptyset$ & $\emptyset$ \\
\arrayrulecolor[gray]{0.8} \hline
$\theta_{kl}$    & $\emptyset$ & $\emptyset$ \\
\arrayrulecolor{black} \hline 
\end{tabular}
\caption{\flapcolor[i][i][i][i][i][i] coloring of $\alpha$}\label{[i][i][i][i][i][i] coloring of alpha gl4}
\end{table}

There are three induced colorings of $F_1$ for each coloring $c$ of $F_2$. We obtain $c_1$ from $c_0$ by a \kempe[k][l] Kempe move along the sphere and $c_2$ by a  \kempe[j][k] Kempe move along the sphere. 
Summing them up as in the previous lemma, the desired result follows. 
In that computation we use the fact that the Euler characteristic change from $F_{ik}$ in $(F_1, c_0)$ and $F_{ik}$ in $(F_2, c)$ is $4$ and remains $4$ if $c_0$ in the previous sentence is replaced by $c_1$ or $c_2$. 
This is because both (a) connect-summing two surfaces and (b) attaching a tube to a surface change the Euler characteristic by $-2$.  
Hence it does not matter if $\mathbb{D}^2 \sqcup \mathbb{D}^2 \sqcup \mathbb{D}^2  $ in $F_{ik}$ of $(F_2, c)$ are connected or not.

Next, consider the colorings where the boundary points of the bottom web of $F_2$ are colored as  \flapcolor[j][i][i][i][i][j]. 
Assume that in the coloring $c_0$ of $F_1$, the sphere appearing in the middle is colored by $k$ and that the additional color on the bottom hemisphere is $l$. This gives us Table \ref{[j][i][i][i][i][j] coloring of alpha gl4}.

\begin{table}[h!]
\centering
\arrayrulecolor{black} 
\begin{tabular}{|l|l|l|}
\hline
{\ul }           & $(F_1, c_0)$                                                            & $(F_2, c)$              \\
\hline 
$F_{ij}$         & $\mathbb{D}^2 \sqcup \mathbb{D}^2  \sqcup \mathbb{D}^2 $ & $\mathbb{D}^2 \sqcup \mathbb{D}^2  \sqcup \mathbb{D}^2 $ \\
\arrayrulecolor[gray]{0.8} \hline
$F_{ik}$         & two stacked saddles met at the saddle & $\mathbb{D}^2 \sqcup \mathbb{D}^2 $ \\
\arrayrulecolor[gray]{0.8} \hline
$F_{il}$         &  saddle  &  saddle \\
\arrayrulecolor[gray]{0.8} \hline
$F_{jk}$         &  $\mathbb{D}^2 $ & $\mathbb{D}^2$ \\
\arrayrulecolor[gray]{0.8} \hline
$F_{jl}$         &  $\mathbb{D}^2 $ & $\mathbb{D}^2$  \\
\arrayrulecolor[gray]{0.8} \hline
$F_{kl}$         & $\mathbb{S}^2 \sqcup$  $\mathbb{D}^2$ (cap)  &  $\mathbb{D}^2$ (cap)  \\
\arrayrulecolor{black} \hline
$ {F_i}$ & $\mathbb{D}^2 \sqcup \mathbb{D}^2$ & $\mathbb{D}^2 \sqcup \mathbb{D}^2 $  \\
\arrayrulecolor[gray]{0.8} \hline
$ {F_j}$ &  $\mathbb{D}^2 $ & $\mathbb{D}^2$ \\
\arrayrulecolor[gray]{0.8} \hline
$ {F_k}$ & $\mathbb{S}^2$ & $\emptyset$ \\
\arrayrulecolor[gray]{0.8}\hline
$ {F_l}$ &  $\mathbb{D}^2$ (cap)  &  $\mathbb{D}^2$ (cap)   \\
\arrayrulecolor{black} \hline
$\theta_{ij}$    & cup & cup \\
\arrayrulecolor[gray]{0.8} \hline
$\theta_{ik}$    & $\mathbb{S}^1 \sqcup \mathbb{S}^1 $ & $\emptyset$  \\
\arrayrulecolor[gray]{0.8} \hline
$\theta_{il}$    & two cups  & two cups \\
\arrayrulecolor[gray]{0.8} \hline
$\theta_{jk}$    & $\mathbb{S}^1 $ positive if $k > j$  & $\emptyset$  \\
\arrayrulecolor[gray]{0.8}\hline
$\theta_{jl}$    & $\emptyset$ & $\emptyset$ \\
\arrayrulecolor[gray]{0.8} \hline
$\theta_{kl}$    & $\emptyset$ & $\emptyset$ \\
\arrayrulecolor{black} \hline 
\end{tabular}
\caption{\flapcolor[j][i][i][i][i][j] coloring of $\alpha$}\label{[j][i][i][i][i][j] coloring of alpha gl4}
\end{table}

For each coloring $c$ of $F_2$, there are two colorings of $F_1$, depending on how we color the $2$-facets. We present a coloring $c_0$ in Table \ref{[j][i][i][i][i][j] coloring of alpha gl4}. The other coloring is obtained from $c_0$ by a \kempe[k][l] Kempe move along the sphere. Summing up their evaluations gives the desired result.

Lastly, consider the coloring where the boundary points of the bottom web of $F_2$ are colored as \flapcolor[i][i][j][j][k][k]. 
In this case, the coloring of $F_1$ is uniquely determined by the coloring $c$ of $F_2$: the middle sphere and the bottom hemisphere must both be colored by $l$, giving us Table \ref{[i][i][j][j][k][k] coloring of alpha gl4}. 
\begin{table}[h!]
\centering
\arrayrulecolor{black} 
\begin{tabular}{|l|l|l|}
\hline
{\ul }           & $(F_1, c_0)$                                                            & $(F_2, c)$              \\
\hline 
$F_{ij}$         & $\mathbb{D}^2 \sqcup \mathbb{D}^2  $ & $\mathbb{D}^2 \sqcup \mathbb{D}^2  $ \\
\arrayrulecolor[gray]{0.8} \hline
$F_{ik}$         & $\mathbb{D}^2 \sqcup \mathbb{D}^2  $ & $\mathbb{D}^2 \sqcup \mathbb{D}^2  $\\
\arrayrulecolor[gray]{0.8} \hline
$F_{il}$         & $\mathbb{D}^2 $ & $\mathbb{D}^2$ \\
\arrayrulecolor[gray]{0.8} \hline
$F_{jk}$         &  $\mathbb{D}^2 \sqcup \mathbb{D}^2  $ & $\mathbb{D}^2 \sqcup \mathbb{D}^2  $\\
\arrayrulecolor[gray]{0.8} \hline
$F_{jl}$         &  $\mathbb{D}^2 $ & $\mathbb{D}^2$  \\
\arrayrulecolor[gray]{0.8} \hline
$F_{kl}$         & $\mathbb{D}^2 $ & $\mathbb{D}^2$  \\
\arrayrulecolor{black} \hline
$ {F_i}$ & $\mathbb{D}^2 $ & $\mathbb{D}^2$  \\
\arrayrulecolor[gray]{0.8} \hline
$ {F_j}$ &  $\mathbb{D}^2 $ & $\mathbb{D}^2$ \\
\arrayrulecolor[gray]{0.8} \hline
$ {F_k}$ &$\mathbb{D}^2 $ & $\mathbb{D}^2$ \\
\arrayrulecolor[gray]{0.8}\hline
$ {F_l}$ &  $\mathbb{D}^2$ (cap) $\sqcup$ $\mathbb{S}^2$  &  $\mathbb{D}^2$ (cap)  \\
\arrayrulecolor{black} \hline
$\theta_{ij}$    & cup & cup \\
\arrayrulecolor[gray]{0.8} \hline
$\theta_{ik}$    & $\mathbb{S}^1 $ positive if $k > i$  & $\emptyset$  \\
\arrayrulecolor[gray]{0.8} \hline
$\theta_{il}$    & two cups  & two cups \\
\arrayrulecolor[gray]{0.8} \hline
$\theta_{jk}$    & $\emptyset$ & $\emptyset$  \\
\arrayrulecolor[gray]{0.8}\hline
$\theta_{jl}$    & $\mathbb{S}^1 $ positive if $l > j$  & $\emptyset$ \\
\arrayrulecolor[gray]{0.8} \hline
$\theta_{kl}$    & $\mathbb{S}^1 $ positive if $l > k$  & $\emptyset$  \\
\arrayrulecolor{black} \hline 
\end{tabular}
\caption{\flapcolor[i][i][j][j][k][k] coloring of $\alpha$}\label{[i][i][j][j][k][k] coloring of alpha gl4}
\end{table}

Summarizing, we have shown $\alpha \circ \beta \circ \alpha = \alpha$. 
\end{proof}

 \begin{proof}[Proof of Proposition \ref{idempotencyprop}.]
    Recall the notation of Remark \ref{abcglobalremark} and Lemma \ref{lindependence}. 
    The last two lemmas imply $(a+b)\cdot \alpha =0$ and $(a+c) \cdot \nu=0$. Since $\alpha, \nu \ne 0$, it follows that $a+b=a+c=0$. So, 
    \[ 
        \idfoam_{\Wzero} =  \beta \circ \alpha +    \nu \circ \mu. 
    \]
     We have thus proven $\beta \circ \alpha +  \nu \circ \mu = \idfoam_{\Wzero}$. 
     A similar argument shows $\xi \circ \eta + \mu \circ \nu =\idfoam_{\Wzeroa}$.
 \end{proof}

\begin{proof}[Proof of Theorem \ref{gl_4 sixvalent seam}]
    This is the direct consequence of Proposition \ref{orthogonalityprop} and Proposition \ref{idempotencyprop}.
\end{proof}

\def\HexaPairing
{\begin{tric}
\draw [scale=0.7, decoration={markings,mark=at position 0.5 with {\arrow{>}}},postaction={decorate}] 
      (0,0)..controls(1.7,-0.6)and(2,-1.4)..(2,-2)..controls(2,-2.6)and(1.7,-3.4)..(0,-4);
\draw [scale=0.7, decoration={markings,mark=at position 0.5 with {\arrow{>}}},postaction={decorate}]      
     (0,-4)..controls(0.8,-3.4)and(1.2,-2.6).. (1.2,-2)..controls(1.2,-1.4)and(0.8,-0.6)..(0,0);
\draw [scale=0.7, decoration={markings,mark=at position 0.5 with {\arrow{>}}},postaction={decorate}]       
      (0,0)..controls(-0.8,-0.6)and(-1.2,-1.4)..(-1.2,-2)..controls(-1.2,-2.6)and(-0.8,-3.4)..(0,-4);
\draw [scale=0.7, decoration={markings,mark=at position 0.5 with {\arrow{>}}},postaction={decorate}]       
     (0,-4)..controls(-1.7,-3.4)and(-2,-2.6)..(-2,-2)..controls(-2,-1.4)and(-1.7,-0.6)..(0,0);
\draw [scale=0.7, decoration={markings,mark=at position 0.5 with {\arrow{>}}},postaction={decorate}]       
      (0,0)..controls(0.25,-0.6)and(0.4,-1.4)..(0.4,-2)..controls(0.4,-2.6)and(0.25,-3.4)..(0,-4);
\draw [scale=0.7, decoration={markings,mark=at position 0.5 with {\arrow{>}}},postaction={decorate}]      
     (0,-4)..controls(-0.25,-3.4)and(-0.4,-2.6)..(-0.4,-2) ..controls(-0.4,-1.4)and(-0.25,-0.6)..(0,0);
     
\end{tric}
}

\subsection{Categorifying the \texorpdfstring{$\GL_5$}{GL_5} \texorpdfstring{$6$}{6}-valent vertex with \texorpdfstring{$\GL_5$}{GL_5} foams}

In this subsection we will modify the results of the previous subsection to the $\GL_5$ foams.   

\begin{figure}[htbp]
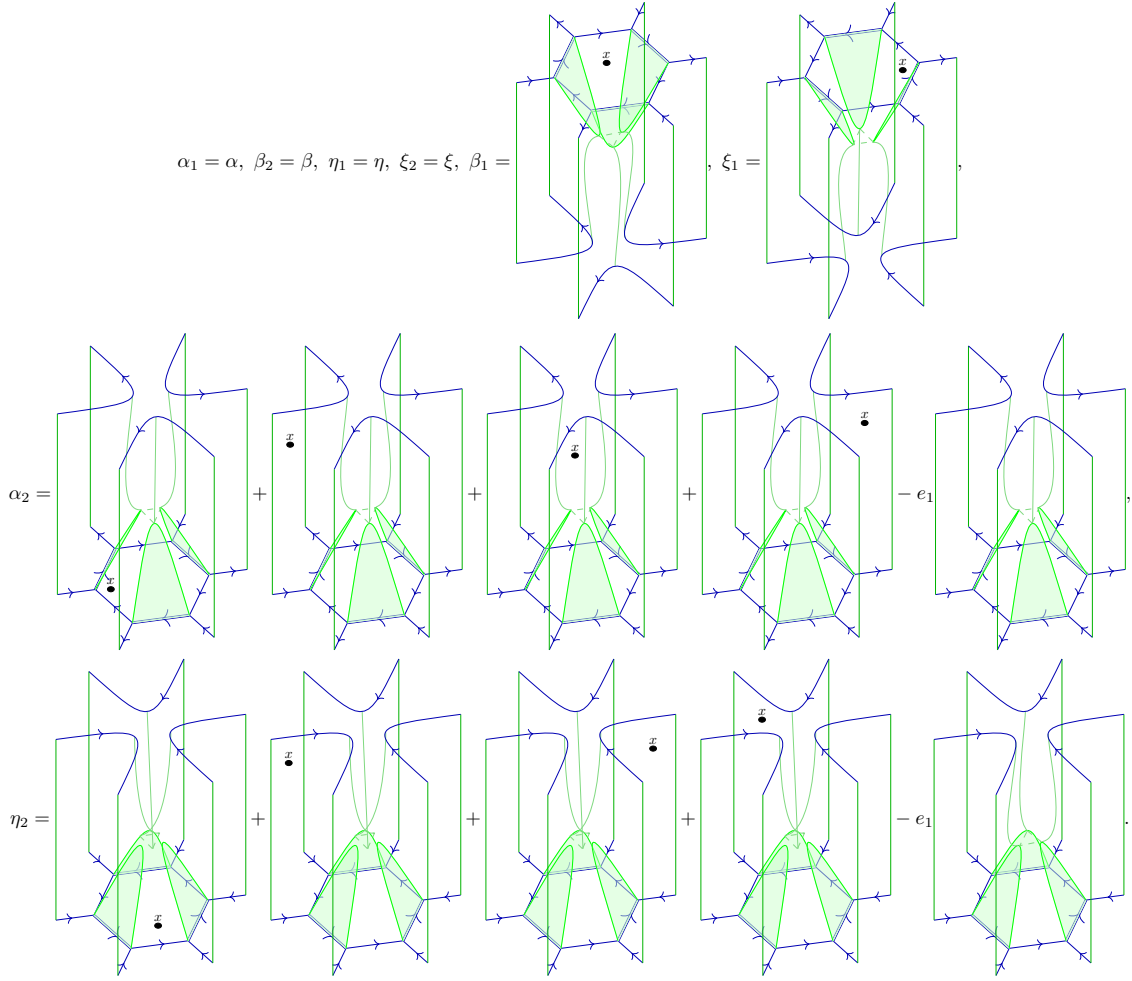

    \centering
\[
\scalebox{0.75}{$\displaystyle
\alpha_1 = \alpha, \ \beta_2 = \beta, \ \eta_1 =  \eta, \  \xi_2=  \xi,  \ \beta_1 =  \BetaOne, \ \xi_1=  \TBetaOne,
$}
\] 
\[
\scalebox{0.75}{$\displaystyle
\alpha_2 =  \AlphaTwo + \AlphaTwoA + \AlphaTwoB + \AlphaTwoC - e_1\AlphaO,  
$}
\] 
\[
\scalebox{0.75}{$\displaystyle
 \eta_2 =  \TAlphaTwo + \TAlphaTwoA + \TAlphaTwoB + \TAlphaTwoC - e_1\TAlphaO.
$}
\] 
    \caption{The $\GL_5$ foams used to demonstrate the decomposition of the $6$-valent vertex; compare with Figure \ref{gl4foamdeef}.
    The horizontal hexagons in $\mu, \nu$ are $3$-facets. The highlighted sections of the foams $\alpha_i, \beta_i, \xi_i, \eta_i$ for $i = 1, 2$ are $2$-facets.}
    \label{gl5foamdeef}
\end{figure}

\begin{figure}[htbp]
    \centering
    \begin{tikzcd}
\BenzenewebM  &
 &
 & \BenzenewebaM  
\\
&
 \Benzeneweb   
 \ar[ul,leftarrow, "\beta_1",shift left] 
 \ar[ul,rightarrow, "\alpha_1"',shift right] 
 \ar[dl,leftarrow, "\beta_2",shift left] 
 \ar[dl,rightarrow, "\alpha_2"',shift right] 
 \ar[r,rightarrow, "\mu",shift left] 
 \ar[r,leftarrow, "\nu"',shift right] & 
\Benzeneweba 
\ar[ur,rightarrow, "\eta_1",shift left] 
 \ar[ur,leftarrow, "\xi_1"',shift right] 
 \ar[dr,rightarrow, "\eta_2",shift left] 
 \ar[dr,leftarrow, "\xi_2"',shift right] &
\\
\BenzenewebM &
&
& \BenzenewebaM  
\end{tikzcd}
\caption{Morphisms among $\GL_5$ webs in Theorem \ref{gl_5 sixvalent seam}.}
\label{fig:gl5 tikz commutative diagram between webs}
\end{figure}
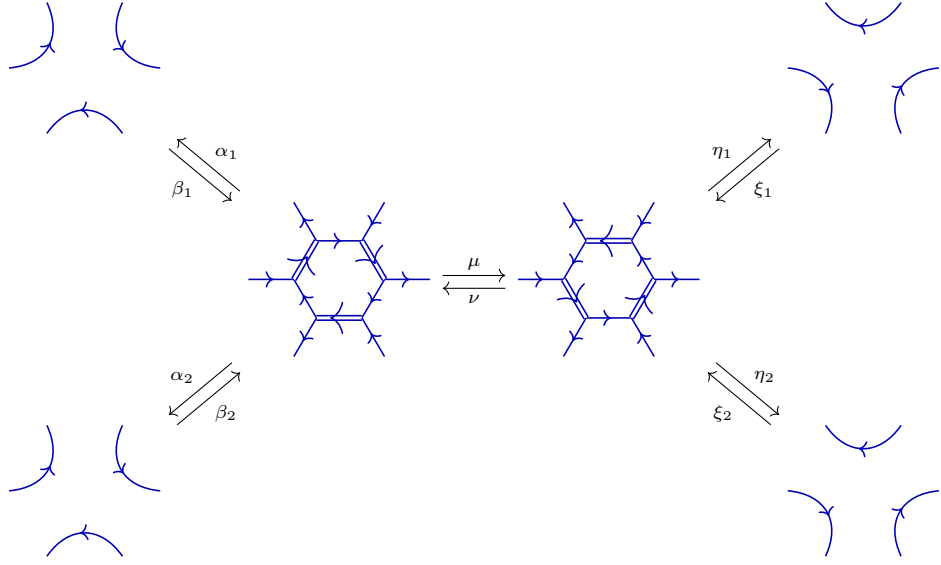

\begin{thm}\label{gl_5 sixvalent seam} 
The $\GL_5$ foams $\alpha_i, \beta_i, \eta_i, \xi_i, \mu$, and $\nu$ ($i=1,2$) shown in Figure \ref{gl5foamdeef} fit into the commutative diagram in Figure \ref{fig:gl5 tikz commutative diagram between webs} and satisfy the following relations in $\FoamFive^{\HexaBound}$:

\begin{align*}
    \alpha_i \circ \nu =0 , \quad 
    \mu \circ \beta_i = 0, \quad  
    &\eta_i \circ \mu=0, \quad   
     \nu \circ \xi_i =0, \quad  i=1,2,  \\
     \alpha_1 \circ \beta_2 = 0, \quad \alpha_2 \circ \beta_1 = 0, \quad
     & \eta_1 \circ \xi_2 = 0, \quad  \eta_2 \circ \xi_1 =0, \\ 
     \mu \circ \nu \circ \mu = \mu , \quad 
     &\nu \circ \mu \circ \nu = \nu , \\
    \alpha_i \circ \beta_i = \idfoam_{\Wtwo}, \quad
    &\eta_i \circ \xi_i = \idfoam_{\Wthree},  \quad  i=1,2, \\
     \beta_1 \circ \alpha_1 + \beta_2 \circ \alpha_2  + \nu \circ \mu = \idfoam_{\Wzero}, 
      \quad
    &\xi_1 \circ \eta_1 + \xi_2 \circ \eta_2 + \mu \circ \nu = \idfoam_{\Wzeroa}.
\end{align*}
\end{thm}

The foams $\nu, \mu, \beta_2, \xi_2, \alpha_1, \eta_1$ are the same as the foams in Theorem \ref{gl_4 sixvalent seam} when viewed only as \textit{pictures}. The difference comes from the fact that these are now $\GL_5$ foams, whereas they were $\GL_4$ foams in the previous subsection.

We have the analogous statement to Proposition \ref{prop:gl4-kar-isom-objects}, as a consequence of Theorem \ref{gl_5 sixvalent seam}:

\begin{proposition}
The foams $P_1=\nu \circ \mu$ 
and $P_2=\mu \circ \nu$ are projections in $\FoamFive^{\HexaBound}$. Furthermore, the associated objects $\mathcal{E}(P_1)$ and $\mathcal{E}(P_2)$ in $Kar(\FoamFive^{\HexaBound})$ are isomorphic via the inverse morphisms $\mu$ and $\nu$. 
\end{proposition}

\begin{thm} \label{GL5Decompo}
In $Kar(\FoamFive^{\HexaBound})$, with $i\in\{1,2\}$, 
\[ \mathcal{E}\left( \Benzeneweb \right) \cong \mathcal{E}\left( \BenzenewebM \right)\{-1\}\oplus \mathcal{E}\left( \BenzenewebM \right)\{1\}\oplus \mathcal{E}\left( P_i \right). \]

So $\mathcal{E}\left( P_i \right)$ decategorifies to $\Wsix$, since 
\[
    \Benzeneweb = [2]\BenzenewebM + \SixVertex.
\]
  Furthermore, this is a decomposition into indecomposable summands due to Theorem \ref{GLNIndecompo}.
\end{thm}

\begin{proof}
    First, we know that the direct sum decomposition holds by the equations in Theorem \ref{gl_5 sixvalent seam}. We then verify the formal grading of the objects by checking the degrees of the morphisms giving the decomposition. By Lemma \ref{degcheck} for $\GL_5$, we know that $\foamdeg(\alpha_1) = \foamdeg(\beta_2) = -1$ and $\foamdeg(\mu)= \foamdeg(\nu)= 0$. By definition of $\alpha_2$ and $\beta_1$, $\foamdeg(\alpha_2) = \foamdeg(\alpha_1) + 2=1$ and  $\foamdeg(\beta_1) = \foamdeg(\beta_2) + 2=1$.
\end{proof}

\begin{remark}\label{GL5Indecompo}
The components of the decomposition of the hexagon web in the above theorem are in fact indecomposable, due to 
Theorem \ref{GLNIndecompo}. 
\end{remark}

\begin{corollary}\label{cor:single hexagon gl5}
 For any given web $W$ with boundary $\varepsilon$, we have the following direct sum decomposition of the state space of $\overline{W} \Wzero$:
\[ \mathcal{F}(\overline{W} \Wzero ) \cong 
\mathcal{F}(\overline{W} \Wone )\{ -1 \}
\oplus \mathcal{F}(\overline{W} \Wone )\{ 1 \} 
\oplus \left( \mathcal{F}(\overline{\idfoam_W} \ast P_2  )(\mathcal{F}( \overline{W} \Wzeroa))  \right)  \]
\end{corollary}   

\begin{proof}
    This is a direct sum decomposition of the module 
    $\mathcal{F}(\overline{W} W_0 )$. The projection maps are given by $\overline{ \idfoam_W} \ast \alpha_1 $, $\overline{ \idfoam_W} \ast  \alpha_2$, and $\overline{ \idfoam_W} \ast \mu $, and the inclusion maps are given by $\overline{ \idfoam_W} \ast \beta_1  $, $ \overline{ \idfoam_W} \ast \beta_2 $, and $\overline{ \idfoam_W}\ast  \nu  $.
\end{proof}

The rest of this section is dedicated to the proof of Theorem \ref{gl_5 sixvalent seam}. The proof is split into multiple lemmas, and we preserve the notation of Theorem \ref{gl_5 sixvalent seam} throughout the section. Note that some of the relations are not proven as they follow directly from Lemma \ref{dotsglngeneric}.

\begin{lemma}\label{orthogonalidempotentsgl5}
The following two pairs of equations from Theorem \ref{gl_5 sixvalent seam} hold:
\begin{align*}
     \alpha_1 \circ \beta_2 = 0, \quad \alpha_2 \circ \beta_1 = 0
\end{align*}
\end{lemma}

\begin{proof}
    This follows from an application of Lemma \ref{dotsglngeneric} to the case where $k = 0$ and $k = 1$ for $\alpha_2 \circ \beta_1$, and $k = 2$ for $\alpha_1 \circ \beta_2$.
\end{proof}

\begin{lemma}
    For $i \in \{1,2\}$, we have the following  orthogonality conditions from Theorem \ref{gl_5 sixvalent seam} among foams: $\alpha_i \circ \nu =0 $, $\mu \circ \beta_i = 0$,  $\eta_i \circ \mu=0$, $\nu \circ \xi_i =0$. 
\end{lemma}
\begin{proof}
    This is a modification of the case of $\GL_4$ foams of Proposition \ref{orthogonalityprop} to $\GL_5$ foams. Adopt the colorings from that proposition. 
    We will comment on the proof of $\mu \circ \beta_i = 0$; the other relations follow from this one. First, focus on the case of $\mu \circ \beta_1 = 0$. The difference when compared to the proof of Proposition \ref{orthogonalityprop} is that we can now perform two Kempe moves when the boundary points of the bottom web are colored as \flapcolor[i][i][i][i][i][i]. 
    In addition to the  \kempe[k][l] Kempe move along the $k$ sphere, we can also perform the  \kempe[k][m] Kempe move. In summary,
\begin{align*}
    F\left( 1 -\frac{(\varu{k} - \varu{m})}{(\varu{l}- \varu{m})}+ \frac{(\varu{k} - \varu{l})}{(\varu{l}- \varu{m})} \right)&= 0 .
\end{align*}
Here $F$ represents the contribution from the part of the foam that remains unchanged. 
Checking that the Kempe move only changes this part of the equations follows from similar computations to those in the previous two sections. 
The signs follow from Lemma \ref{rwlemma2.19}. Note that there is an additional case in the evaluation of $\mu \circ \beta_2$ where there is a dot on the hemisphere. In that case the above equation transforms to:
\begin{align*}
    F\left( \varu{k} -\frac{(\varu{k} - \varu{m})}{(\varu{l}- \varu{m})} \varu{l} + \frac{(\varu{k} - \varu{l})}{(\varu{l}- \varu{m})} \varu{m} \right)&= 0 .
\end{align*}
In the first term the additional dot contributes $\varu{k}$, since the other two terms are obtained by performing a Kempe move involving the surface where the dot lies, $\varu{k}$ transforms to $\varu{l}$ or $\varu{m}$.

The evaluation for other colorings can be generalized from the $\GL_4$ case in a similar manner. 
\end{proof}

\begin{lemma}
    The following two foam equalities hold: $\mu \circ \nu \circ \mu = \mu$ and $\nu \circ \mu \circ \nu = \nu$. This can be interpreted as the proof of idempotency of $\mu \circ \nu $ and $\nu \circ \mu $.
\end{lemma}
\begin{proof}
For the $\mu \circ \nu \circ \mu = \mu$ this is a generalization of Lemma \ref{hexagonidempotent} to $\GL_5$ foams (the other relation is just a rotation). The difference between the $\GL_5$ and $\GL_4$ settings once again boils down to additional Kempe moves. The additional Kempe moves balance out the contribution from the bichrome surfaces that involve the fifth color $m$ not seen in the $\GL_4$ case. 

In what follows we use the notation of Lemma \ref{hexagonidempotent} and additionally set $\varu{m} = \vare$. We will analyze the case when the boundary points of the bottom web are colored as \flapcolor[i][i][i][i][i][i]; see Figure \ref{fig iiiiii gl4} for more details.  The signs follow from Table \ref{[i][i][i][i][i][i] coloring of nu gl4} and Lemma \ref{rwlemma2.19}.

First, we can modify the three terms from Lemma \ref{hexagonidempotent} to include the color $m = 5$. As before $i = 1$, $j = 2$, $k = 3$, and $l =4$. We obtain 
\begin{align*}
    & F \left(\frac{-(\vara  - \varc )^2(\vara - \vare)^2}{(\varb  - \varc )(\varc  - \vard )(\varb - \vare)(\vard - \vare)} \right. + \\ & + \left. \frac{(\vara  - \varb )^2(\vara - \vare)^2}{(\varb  - \varc )(\varb  - \vard )(\varc - \vare)(\vard - \vare)}  + \frac{(\vara  - \vard )^2(\vara - \vare)^2}{(\varc  - \vard )(\varb  - \vard )(\varc - \vare)(\varb - \vare)}\right),
\end{align*}
where the second and third terms come from  \kempe[j][k] and \kempe[l][k] Kempe moves along spheres, as before.

Next, we may perform \kempe[j][m] and \kempe[l][m]  Kempe moves along the $j$ and $l$ monochrome spheres appearing under coloring $c_0$ in Table \ref{[i][i][i][i][i][i] coloring of nu gl4}. This gives us two additional terms
\begin{align*}
    F\left(\frac{(\vara  - \varc )^2(\vara - \varb)^2}{(\varb  - \vare )(\varc  - \vare )(\varb- \vard)(\varc - \vard)} + \frac{(\vara  - \varc )^2(\vara - \vard)^2}{(\varc  - \vare )(\vard  - \vare )(\varb - \varc)(\varb - \vard)} \right).
\end{align*}

Lastly, we can perform two consecutive Kempe moves along the spheres. There are two such instances originating from the coloring $c_0$  (see Table \ref{[i][i][i][i][i][i] coloring of nu gl4}). Namely, we may first perform a \kempe[j][k] Kempe move followed by an \kempe[l][m] Kempe move. Or, we may first perform an \kempe[l][k] Kempe move followed by a \kempe[k][m] Kempe move. This gives us two additional terms
\begin{align*}
    -F\left( \frac{(\vara-\varb)^2(\vara-\vard)^2}{(\varb-\varc)(\varb-\vard)(\vard-\vare)(\varc-\vare)} + \frac{(\vara-\varb)^2(\vara-\vard)^2}{(\varb-\vard)(\varb-\vare)(\varc-\vard)(\varc-\vare)}\right).
\end{align*}

Summing up all seven terms returns $1 \cdot F$, proving the desired identity for the boundary coloring \flapcolor[i][i][i][i][i][i]. Other colorings can be generalized from $\GL_4$ to $\GL_5$ analogously.
\end{proof}

\begin{remark}
    The authors believe that the above lemma holds when the foams are interpreted as $\GL_N$ foams in general. This should be provable by induction, or by using an action of the symmetric group on $N$ variables, along with some symmetric function theory. This was also communicated to the authors by E.~Wagner. 
\end{remark}

\begin{lemma}
    The following equalities hold: $\beta_1 \circ \alpha_1 + \beta_2 \circ \alpha_2  + \nu \circ \mu = \idfoam_{\Wzero}$ and  $\xi_1 \circ \eta_1 + \xi_2 \circ \eta_2 + \mu \circ \nu = \idfoam_{\Wzeroa}$ (see Figure \ref{gl5foamdeef}). 
\end{lemma}
\begin{proof}
    We will show that $\beta_1 \circ \alpha_1 + \beta_2 \circ \alpha_2  + \nu \circ \mu = \idfoam_{\Wzero}$. This is a direct foam evaluation. As we have already presented most of the foam evaluations involving these foams in the $\GL_4$ case, we only present the most involved case in detail. The rest follows by a similar, but easier computation. Denote the fifth color by $m$, so that our colors are $i,j,k,l,m$. 

    Denote by $G_1$ the foam used in $\beta_2 \circ \alpha_1$ (note the change of indices in order to momentarily forget about the decorations), and let $G_2 =  \nu \circ \mu $. 
    Color the boundary points of the bottom web of $I$ with the color as \flapcolor[i][i][i][i][i][i], and denote this coloring by $c$. Let $c_1$ be the induced coloring of $G_1$ and $c_2$ the induced coloring of $G_2$, where the color appearing in the middle of $G_2$ is chosen to be $k$. 
    That is, assume that the foam $G_2$ carries the colors $i,j,k$ on the horizontal hexagons. We then obtain Table \ref{[i][i][i][i][i][i] main coloring gl5}. 

    We note that there are colorings of $G_1$ and $G_2$ that are not induced by a coloring of $\idfoam_{\Wzero}$. The contributions of those colorings with the same boundary coloring cancel each other out in the sum $\beta_1 \circ \alpha_1 + \beta_2 \circ \alpha_2  + \nu \circ \mu $.

\begin{table}[h!]
\centering
\arrayrulecolor{black} 
\begin{tabular}{|l|l|l|l|}
\hline
{\ul }           & $(\idfoam_{\Wzero},c)$   & $(G_1, c_1)$    & $(G_2, c_2)$         \\
\hline 
$F_{ij}$         & $\mathbb{D}^2 \sqcup \mathbb{D}^2 \sqcup \mathbb{D}^2 $& doubled monkey saddle& $\mathbb{D}^2 \sqcup \mathbb{D}^2 \sqcup \mathbb{D}^2 $ \\
\arrayrulecolor[gray]{0.8} \hline
$F_{ik}$         & $\mathbb{D}^2 \sqcup \mathbb{D}^2 \sqcup \mathbb{D}^2 $ & $\mathbb{D}^2 \sqcup \mathbb{D}^2 \sqcup \mathbb{D}^2 $ & $\mathbb{D}^2 \sqcup \mathbb{D}^2 \sqcup \mathbb{D}^2 $\\
\arrayrulecolor[gray]{0.8} \hline
$F_{il}$         & $\mathbb{D}^2 \sqcup \mathbb{D}^2 \sqcup \mathbb{D}^2 $ & $\mathbb{D}^2 \sqcup \mathbb{D}^2 \sqcup \mathbb{D}^2 $ & doubled monkey saddle\\
\arrayrulecolor[gray]{0.8} \hline
$F_{im}$         & $\mathbb{D}^2 \sqcup \mathbb{D}^2 \sqcup \mathbb{D}^2 $ & $\mathbb{D}^2 \sqcup \mathbb{D}^2 \sqcup \mathbb{D}^2 $ & doubled monkey saddle\\
\arrayrulecolor[gray]{0.8} \hline
$F_{jk}$         & $\mathbb{S}^1 \times [0,1]$ & $\mathbb{D}^2$ (cup) $\sqcup$ $\mathbb{D}^2$ (cap) & $\mathbb{S}^1 \times [0,1]$  \\
\arrayrulecolor[gray]{0.8} \hline
$F_{jl}$         & $\mathbb{S}^1 \times [0,1]$ & $\mathbb{D}^2$ (cup) $\sqcup$ $\mathbb{D}^2$ (cap) & $\mathbb{D}^2$ (cup) $\sqcup$ $\mathbb{D}^2$ (cap) \\
\arrayrulecolor[gray]{0.8} \hline
$F_{jm}$         & $\mathbb{S}^1 \times [0,1]$ & $\mathbb{D}^2$ (cup) $\sqcup$ $\mathbb{D}^2$ (cap) & $\mathbb{D}^2$ (cup) $\sqcup$ $\mathbb{D}^2$ (cap) \\
\arrayrulecolor[gray]{0.8} \hline
$F_{kl}$         &$\emptyset$ & $\emptyset$ & $\mathbb{S}^2$  \\
\arrayrulecolor[gray]{0.8} \hline
$F_{km}$         &$\emptyset$ & $\emptyset$ & $\mathbb{S}^2$  \\
\arrayrulecolor{black} \hline
$F_i$ & $\mathbb{D}^2 \sqcup \mathbb{D}^2 \sqcup \mathbb{D}^2 $ & $\mathbb{D}^2 \sqcup \mathbb{D}^2 \sqcup \mathbb{D}^2 $ & doubled monkey saddle \\
\arrayrulecolor[gray]{0.8} \hline
$F_j$ & $\mathbb{S}^1 \times [0,1]$ & $\mathbb{D}^2$ (cup) $\sqcup$ $\mathbb{D}^2$ (cap) & $\mathbb{D}^2$ (cup) $\sqcup$ $\mathbb{D}^2$ (cap)  \\
\arrayrulecolor[gray]{0.8} \hline
$F_k$ & $\emptyset$ & $\emptyset$ & $\mathbb{S}^2$ \\
\arrayrulecolor[gray]{0.8}\hline
$F_l$ & $\emptyset$ & $\emptyset$ & $\emptyset$\\
\arrayrulecolor[gray]{0.8}\hline
$F_m$ & $\emptyset$ & $\emptyset$ & $\emptyset$\\
\arrayrulecolor{black} \hline
$\theta_{ij}$    & six strands & \makecell{three cups and caps \\ positive if $i > j$} & \makecell{three cups and caps \\ positive if $i > j$} \\ 
\arrayrulecolor[gray]{0.8} \hline
$\theta_{ik}$    & $\emptyset$ & $\emptyset$ & $\mathbb{S}^1 \sqcup \mathbb{S}^1 \sqcup  \mathbb{S}^1  $ positive if $k > i$\\
\arrayrulecolor[gray]{0.8} \hline
$\theta_{il}$    & $\emptyset$ & $\emptyset$ & $\emptyset$\\
\arrayrulecolor[gray]{0.8} \hline
$\theta_{im}$    & $\emptyset$ & $\emptyset$ & $\emptyset$\\
\arrayrulecolor[gray]{0.8} \hline
$\theta_{jk}$    & $\emptyset$ & $\emptyset$ & $\mathbb{S}^1 \sqcup \mathbb{S}^1$ \\
\arrayrulecolor[gray]{0.8}\hline
$\theta_{jl}$    & $\emptyset$ & $\emptyset$ & $\emptyset$ \\
\arrayrulecolor[gray]{0.8} \hline
$\theta_{jm}$    & $\emptyset$ & $\emptyset$ & $\emptyset$ \\
\arrayrulecolor[gray]{0.8} \hline
$\theta_{kl}$    & $\emptyset$ & $\emptyset$ & $\emptyset$ \\
\arrayrulecolor[gray]{0.8} \hline
$\theta_{km}$    & $\emptyset$ & $\emptyset$ & $\emptyset$ \\
\arrayrulecolor[gray]{0.8} \hline
$\theta_{lm}$    & $\emptyset$ & $\emptyset$ & $\emptyset$ \\
\arrayrulecolor{black} \hline 
\end{tabular}
\caption{\flapcolor[i][i][i][i][i][i] coloring of $\beta_1 \circ \alpha_1 + \beta_2 \circ \alpha_2  + \nu \circ \mu = \idfoam_{\Wzero}$ as a $\GL_5$ foam. 
} \label{[i][i][i][i][i][i] main coloring gl5}
\end{table}

For $G_2$, there are two additional colorings induced by the coloring $c$ of $\idfoam_{\Wzero}$, corresponding to a  \kempe[k][l] or  \kempe[k][m] Kempe move along the spheres that appear. Denote them by $c_{2l}$ and $c_{2m}$ respectively. We restrict ourselves to the case where $i = 1$, $j = 2$, $k = 3$, $l= 4$, and $m = 5$; the other cases can be checked analogously. Using considerations as in the previous section and adding the contribution from decorations to $G_1$ we compute that
\begin{align*}
    & \frac{\langle G_2, c_2 \rangle + \langle G_2, c_{2l} \rangle + \langle G_2, c_{2m} \rangle + \langle \beta_1 \circ \alpha_1, c_1 \rangle + \langle \beta_2 \circ \alpha_2, c_1 \rangle}{\langle \idfoam_{\Wzero}, c \rangle} = \\
    &\frac{(\vari - \varu{m})^2 (\vari - \varu{l})^2}
     {(\varu{j} - \varu{l})(\varu{k} - \varu{m})(\varu{k} - \varu{l})(\varu{j} - \varu{m})}
-
\frac{(\vari - \varu{m})^2 (\vari - \varu{k})^2}
     {(\varu{j} - \varu{k})(\varu{l} - \varu{m})(\varu{k} - \varu{l})(\varu{j} - \varu{m})}
 - \\
&-
\frac{(\vari - \varu{k})^2 (\vari - \varu{l})^2}
     {(\varu{j} - \varu{l})(\varu{k} - \varu{m})(\varu{m} - \varu{l})(\varu{j} - \varu{k})}
+
\frac{(\varu{j} + 2\vari - \varu{k} - \varu{l} - \varu{m})(\vari - \varu{j})^2}
     {(\varu{j} - \varu{m})(\varu{j} - \varu{k})(\varu{j} - \varu{l})} \\
     &= 1,
\end{align*}
where the factor $(\varu{j} + 2\vari - \varu{k} - \varu{l} - \varu{m})$ comes from the decorations of $\beta_1 \circ \alpha_1$ and $ \beta_2 \circ \alpha_2$ when compared to $G_1$ (which has no decorations). We have also dropped the sign factor; it is easy to check that that factor is $1$ for the choice of values $(i,j,k,l,m) = (1,2,3,4,5)$. 
\end{proof}

\section*{Future directions}

As a natural generalization of the results from the previous sections, we propose the following categorification of the $\GL_N$ $6$-valent vertex with $\GL_N$ foams for any $N \ge 4$. 

\begin{conjecture} \label{gl_n sixvalent seam} 
For $N \ge 4$, there exists a direct sum decomposition of the hexagon web in $Kar(\Foam^{\HexaBound})$ as 
\[
 \mathcal{E}(\Wzero) \cong \mathcal{E}(\Wone)\{4-N\} \oplus \mathcal{E}(\Wone)\{6-N\}  \oplus \dots \oplus \mathcal{E}(\Wone)\{N-4\} \oplus \mathcal{E}(\mu \circ \nu).
\]

Hence $\mathcal{E}(\mu \circ \nu)$ categorifies the $6$-valent vertex $\Wsix$. 
\end{conjecture}

If the conjecture holds, the hexagon web is decomposed into indecomposable objects in $Kar(\Foam^{\HexaBound})$ due to Theorem \ref{GLNIndecompo}.
We have proved Conjecture \ref{gl_n sixvalent seam} when $N= 4, 5$ in this paper, and a similar construction with $\GL_N$ foams is expected when proving the conjecture for $N \ge 6$. 
In~\cite{Robert2015}, Robert gave a characterization of the $\SL_3$ webs which are decomposable as objects in the $\SL_3$ foam category. He further proposed initial steps on decomposing such $\SL_3$ webs into indecomposable summands. Conjecture \ref{gl_n sixvalent seam} can be seen as the first step for the program of studying $\GL_N$ webs which are decomposable in the $\GL_N$ foam category as well as their decomposition when $N>3$. 
In addition, the conjecture also hints at the construction of a unique $\GL_N$ web basis in the Karoubi completion of the corresponding $\GL_N$ foam category.

\providecommand{\bysame}{\leavevmode\hbox to3em{\hrulefill}\thinspace}
\providecommand{\MR}{\relax\ifhmode\unskip\space\fi MR }
\providecommand{\MRhref}[2]{%
  \href{http://www.ams.org/mathscinet-getitem?mr=#1}{#2}
}
\providecommand{\href}[2]{#2}

\end{document}